\documentclass[11pt,a4paper,onecolumn,notitlepage]{article}

\usepackage{CJKutf8}
\usepackage{cmap} 
\usepackage[pdfpagemode=None,
pagebackref=true,
bookmarksopen=false,
hyperindex=true,
colorlinks=true,
linkcolor=blue,
citecolor=blue]{hyperref}
\hypersetup{pdfauthor={Ryshard-Pavel Kostecki}}

\ifx\VBver\sthundefined
\hypersetup{pdftitle={Va\u{\i}nberg--Br\`{e}gman relative entropy and quasinonexpansive operators}}
\else
\hypersetup{pdftitle={Va\u{\i}nberg--Br\`{e}gman relative entropy and quasinonexpansive operators (v.\VBver)}}
\fi

\usepackage{tabularx} 
\usepackage{amsfonts} 
\usepackage{amsmath} 
\usepackage{mathrsfs} 
\usepackage{amssymb} 
\usepackage[X2,T2A,T1]{fontenc} 

\usepackage[2cell,all]{xy} 
\usepackage{graphicx} 
\usepackage{bm}

\DeclareMathAlphabet{\mathpzc}{OT1}{pzc}{m}{it}

\usepackage{geometry}
\usepackage[toc]{multitoc}
\usepackage{makeidx}

\usepackage[UKenglish]{isodate}
\cleanlookdateon

\usepackage{empheq}
\usepackage{sectsty} 
\usepackage{scalerel}

\ifx\VBver\sthundefined
\usepackage{enumitem}
\usepackage{bibspacing}
\usepackage{rpk}
\renewcommand{\rpkmark}[1]{#1}
\renewcommand{\rpkmarkmath}[1]{#1}
\else
\usepackage{undertilde} 
\usepackage{../../enumitem}
\usepackage{../../bibspacing}
\usepackage{../../rpk}
\renewcommand{\rpkmark}[1]{#1}
\renewcommand{\rpkmarkmath}[1]{#1}
\fi

\usepackage{setspace}
\usepackage{caption}
\usepackage{array}
\usepackage{makecell}
\renewcommand*{\backref}[1]{}
\renewcommand*{\backrefalt}[4]{%
  \ifcase #1 %
    \relax
  \or
    $\uparrow$~#2.
  \else
    $\uparrow$~#2.
  \fi%
}
\allsectionsfont{\normalsize}
\begin{document}
\title{{\vskip -1cm}{\huge Va\u{\i}nberg--Br\`{e}gman relative entropy\\and quasinonexpansive operators\vskip 0.2cm}}
\author{Ryshard-Pavel Kostecki\\
{\scriptsize\textit{Research Center for Quantum Information, Slovak Academy of Sciences}}\vspace{-0.15cm}\\{\scriptsize\textit{D\'{u}bravsk\'{a} cesta 9, 84511 Bratislava, Slovakia}}\\{\scriptsize\texttt{kostecki@fuw.edu.pl}}}
\date{{\scriptsize 18 November 2025\footnote{{\footnotesize v2: 16 February 2026.}}}}
\maketitle      
\thispagestyle{empty}
\vspace*{-0.9cm}
\begin{center}
\textit{to the memory of Yuri\u{\i} I. Manin}\\
\vspace*{0.2cm}
\end{center}

\begin{abstract}
{\noindent%
We review the theory of Va\u{\i}nberg--Br\`{e}gman relative entropies and quasinonexpansive operators on reflexive Banach spaces, and obtain several new results. We also develop an extension of this theory to nonreflexive Banach spaces, which is a joint generalisation of the reflexive Banach space approach and the finite-dimensional information geometric approach. In the reflexive case, we study generalised pythagorean inequality, as well as norm-to-norm, uniform, and Lipschitz--H\"{o}lder continuity, of (left and right) entropic projections, proximal maps, and resolvents. We also provide a detailed study of a special (`gauge') family of Va\u{\i}nberg--Br\`{e}gman geometries and operators that is tightly related with the geometric properties of the underlying Banach space norm. The extended theory belongs to the intersection of convex theoretic and homeomorphic approaches to nonlinear analysis. Its models are constructed, using integration theory on order unit spaces, via nonlinear embeddings into reflexive rearrangement invariant spaces. E.g., we compute the exponent parameters of Lipschitz--H\"{o}lder continuity of the extended entropic projections and resolvents, and establish composability of a suitable class of nonlinear quasinonexpansive operators, over normal state spaces of JBW- and W$^*$-algebras, determined by `gauge' Va\u{\i}nberg--Br\`{e}gman geometries over, respectively, nonassociative and noncommutative L$_p$ spaces, and extended via Mazur embeddings. Other examples of extended Va\u{\i}nberg--Br\`{e}gman geometries feature the (commutative and noncommutative) Lozanovski\u{\i} factorisation map, generalised spin factors, finite dimensional base normed spaces, and convex spectral functions on unitarily invariant ideals of compact operators. We also discuss several categories of entropic projections and quasinonexpansive operators naturally appearing in this framework.}%
\end{abstract}
\vspace{5pt}
{\footnotesize
\renewcommand{\contentsname}{\vspace{-30pt}}
\addtocontents{toc}{\protect\begin{multicols}{2}}
\tableofcontents
\addtocontents{toc}{\protect\end{multicols}}
\addtocontents{toc}{\vspace{-30pt}}
}
\newpage
 \newgeometry{
 a4paper,
 total={210mm,297mm},
 left=20mm,
 right=20mm,
 top=20mm,
 bottom=20mm,
 }
\section{Introduction}\label{section.introduction}
\subsection{Va\u{\i}nberg--Br\`{e}gman geometry...}
A property of a Banach space will be said to be \textit{linear norm-geometric} if{}f it is invariant under a linear isometry into another Banach space. For any Banach space $(X,\n{\cdot}_X)$, both $\n{\cdot}_X$ and $\n{\cdot}^2_X$ are convex functions. Various linear norm-geometric properties of a Banach space  $(X,\n{\cdot}_X)$ which quantify convexity (resp., differentiability) of its norm are dual to the corresponding linear norm-geometric properties quantifying differentiability (resp., convexity) of a norm of a Banach dual space $(X^\star,\n{\cdot}_{X^\star})$. One of the key features of nonlinear convex analysis on Banach spaces is a generalisation of these relationships from the pairs ($\n{\cdot}_X$, $\n{\cdot}_{X^\star}$) of Banach dual norms to the Mandelbrojt--Fenchel dual pairs ($\Psi$, $\Psi^\lfdual$) of convex functions, acting, respectively, on $(X,\n{\cdot}_X)$ and $(X^\star,\n{\cdot}_{X^\star})$. 

A special role in a passage from norm geometry to convex geometry on a Banach space $(X,\n{\cdot}_X)$ is played by the convex functions $\Psi_\varphi(x):=\int_0^{\n{x}_X}\dd t\,\varphi(x)$ $\forall x\in X$ (with $\varphi$ positive, strictly increasing, continuous, $\varphi(0)=0$, and $\lim_{t\ra\infty}\varphi(t)=\infty$; such $\varphi$ is called a \textit{gauge})\footnote{Statements of this paragraph hold also for a more general class of \textit{quasigauges}, defined as nondecreasing functions $\varphi:\RR^+\ra[0,\infty]$ such that $\varphi\not\equiv0$ and $\exists s>0$ $\lim_{t\ra^+s}\varphi(t)<\infty$. However, in this case, each of the linear norm-geometric properties of $\n{\cdot}_X$ requires to impose some additional conditions on $\varphi$ (see Proposition \ref{prop.psi.quasigauge.geometry}), so it is more straightforward to discuss the key ideas of $\Psi_\varphi$ while assuming that $\varphi$ is a gauge.}, which are generating duality mappings $j_\varphi$ through their subdifferential: $\partial\Psi_\varphi=:j_\varphi:X\ra 2^{X^\star}$. The functions  $\Psi_\varphi$ can be seen as a generalisation of $\frac{1}{2}\n{\cdot}^2_X$ ($=\Psi_\varphi$ with $\varphi(t)=t$), allowing to characterise linear norm-geometric properties of an underlying Banach space in terms of properties of $\Psi_\varphi$ (instead of properties of $\n{\cdot}_X$). In general, the following triad is equivalent: convexity (resp., differentiability) properties of $\n{\cdot}_X$, convexity (resp., differentiability) properties of $\Psi_\varphi$, monotonicity (resp., continuity) properties of $j_\varphi$ (see Proposition \ref{prop.psi.varphi.geometry}). Since linearity of $j_\varphi$ for $\varphi(t)=t$ is equivalent with $(X,\n{\cdot}_X)$ being a Hilbert space \cite[Prop. 2]{Golomb:Tapia:1972}, $j_\varphi$ can be seen as a generically nonlinear map. For Banach spaces $(X,\n{\cdot}_X)$ with Gateaux differentiable $\n{\cdot}_X$, the corresponding duality mappings $j_\varphi$ are functions (i.e. singleton-valued maps), and take the particularly useful form $j_\varphi=\DG\Psi_\varphi:X\ra X^\star$ (with $\DG$ denoting Gateaux derivative). 

In essence, the Va\u{\i}nberg--Br\`{e}gman geometry on a reflexive Banach space $(X,\n{\cdot}_X)$ takes two further steps: moving from the specific Gateaux differentiable convex function $\Psi_\varphi:X\ra\RR^+$ to any Gateaux differentiable convex function $\Psi:X\ra\,]-\infty,\infty]$, and moving from the (symmetric) metric distance $d_{\n{\cdot}_X}(x,y):=\n{x-y}_X$ $\forall x,y\in X$ to the (asymmetric) Va\u{\i}nberg--Br\`{e}gman functional \cite[Eqn. (8.5)]{Vainberg:1956}
\begin{equation}
D_\Psi(x,y):=\Psi(x)-\Psi(y)-\duality{x-y,\DG\Psi(y)}_{X\times X^\star}\;\forall(x,y)\in X\times\intefd{\Psi},
\label{eqn.bregman.function.introduction}
\end{equation}
with $D_\Psi(x,y):=\infty$ $\forall y\in X\setminus\intefd{\Psi}$, $\duality{\cdot,\cdot}_{X\times X^\star}$ denoting the Banach space duality, $\efd(\Psi)$ denoting a domain of finiteness of $\Psi$, and $\INT$ denoting a topological interior operator on the subsets of $X$ with respect to the topology of $\n{\cdot}_X$. (Due to this asymmetry, most of objects in the Va\u{\i}nberg--Br\`{e}gman geometry exist in two, left and right, versions.) As a result, it provides an alternative setting of geometric properties on $X$, quantified in terms of a convex function $\Psi$ instead of $\n{\cdot}_X$. In particular (see Table \ref{table.comparison.metric.bregman}): $D_\Psi$ acts as an analogue of $(d_{\n{\cdot}_X})^2$; left and right $D_\Psi$-projections (onto left and right $D_\Psi$-Chebysh\"{e}v sets $K$, respectively),
\begin{equation}
y\mapsto\LPPP^{D_\Psi}_K(y):=\arginff{x\in K}{D_\Psi(x,y)}\;\;\mbox{ and }\;\;y\mapsto\RPPP^{D_\Psi}_K(y):=\arginff{x\in K}{D_\Psi(y,x)},
\end{equation} respectively, act as analogues of metric projections (onto Chebysh\"{e}v sets $K$)
\begin{equation}
y\mapsto\PPP_K^{d_{\n{\cdot}_X}}(y):=\arginff{x\in K}{\n{x-y}_X};
\end{equation}
left and right $D_\Psi$-(quasi)nonexpansive operators act as analogues of $\n{\cdot}_X$-(quasi)nonexpansive operators, etc. For $(X,\n{\cdot}_X)$ given by a Hilbert space and $\Psi=\frac{1}{2}\n{\cdot}_X^2$ one has $D_\Psi(x,y)=\frac{1}{2}\n{x-y}_X^2$, hence metric and Va\u{\i}nberg--Br\`{e}gman geometry coincide in this case, however it is no longer so in more general cases.\footnote{Strictly speaking, the Va\u{\i}nberg--Br\`{e}gman theory on reflexive Banach spaces is a generalisation of the theory of metric projections and norm-nonexpansive operators on Hilbert spaces, and is analogous to (and is generally \textit{better behaved}, see \cite[\S{}\S4--5, \S{}\S7--8]{Alber:1993}, than) a corresponding theory of metric projections and norm-nonexpansive operators on Banach spaces. In particular, while metric projections on Hilbert space as well as left and right $D_\Psi$-projections on reflexive Banach spaces satisfy the pythagorean theorem (see Table \ref{table.comparison.metric.bregman}), it is not so for metric projections on reflexive Banach spaces.} Each choice of $\Psi$ provides a specific ``probing'' of the structure of a Banach space, and it also establishes a particular convention of statistical inference on it (e.g., $\Psi=\frac{1}{2}\n{\cdot}_X^2$ corresponds to $D_\Psi$-projections encoding the optimal estimation on $(X,\n{\cdot}_X)$ in the sense of least squares). The class of Va\u{\i}nberg--Br\`{e}gman geometries determined by $\Psi=\Psi_\varphi$ provides thus an intermediate stage between the Banach space norm geometry (characterised by the properties of $\Psi_\varphi$) and the Va\u{\i}nberg--Br\`{e}gman geometry in general. Beyond the realms of $\Psi=\Psi_\varphi$, the properties of Va\u{\i}nberg--Br\`{e}gman geometry are no longer directly related to linear norm-geometric characteristics of the Banach space. However, it remains a rich geometric theory on its own, with a deep role played in convex nonlinear analysis on Banach spaces, since the formula \eqref{eqn.bregman.function.introduction}, defining $D_\Psi$, essentially encapsulates the first order Taylor expansion of the convex function $\Psi$ (and, thus, the knowledge about the global minimum of $\Psi$).\footnote{According to \cite[p. 69]{Boyd:Vandenberghe:2004}: \cytat{[t]his is perhaps the most important property of convex functions, and explains some of the remarkable properties of convex functions and convex optimization problems}.}

Br\`{e}gman and Chencov have independently discovered a characteristic geometric property of additive decomposition of a certain family of relative entropies under entropic projection onto a (suitably understood) affine (resp., convex) closed subset $K\subseteq M$, called a \textit{generalised pythagorean equation} (resp., \textit{inequality}). Chencov \cite[Eqns. (11), (15), Thm. 1]{Chencov:1968}  discovered it for $D_1$ and $\RPPP^{D_1}_K$,
\begin{equation}
	D_1(\omega,\phi)\geq D_\Psi(\omega,\RPPP^{D_1}_K(\omega))+D_\Psi(\RPPP^{D_1}_K(\omega),\phi)\;\;\forall(\omega,\phi)\in M\times K.
\label{right.pyth.D.1}
\end{equation}
where $D_1(\omega,\phi):=\int\mu(\phi-\omega+\omega\log\frac{\omega}{\phi})$ is the Kullback--Leibler information \cite[Eqn. (2.4)]{Kullback:Leibler:1951}, $\omega,\phi\in(L_1(\X,\mu))^+$, $(\X,\mu)$ is a localisable measure space, $M=(S(L_1(\X,\mu),\n{\cdot}_1))^+$ is a (not necessarily finite-dimensional) set of probability densities, and $K$ is given by the finite-dimensional convex set of exponential families, i.e. for $m\in\NN$, a convex closed subset $\Theta\subseteq\RR^m$, $p_0(\xx)\in M$, and $q_i(\xx)\in M$ $\forall i\in\{1,\ldots,m\}$,
\begin{equation}
K=\left\{p(\xx,\theta)\in M\st p(\xx,\theta)=\frac{p_0(\xx)\exp(\sum_{i=1}^m q_i(\xx)\theta_i)}{\int_\X\mu\, p_0(\xx)\exp(\sum_{i=1}^m q_i(\xx)\theta_i)},\;\theta=(\theta_1,\ldots,\theta_m)\in\Theta\right\}.
\label{eqn.exponential.family.chencov}
\end{equation}
On the other hand, Br\`{e}gman \cite[Lem. 1]{Bregman:1966} (=\cite[Lem. 1, \S2.2]{Bregman:1966:PhD}) discovered it for $D_\Psi$ and $\LPPP^{D_\Psi}_K$,
\begin{equation}
	D_\Psi(\omega,\phi)\geq D_\Psi(\omega,\LPPP^{D_\Psi}_K(\phi))+D_\Psi(\LPPP^{D_\Psi}_K(\phi),\phi)\;\;\forall(\omega,\phi)\in K\times M,
\label{gen.pyth.intro}
\end{equation}
where $D_\Psi$ is a Va\u{\i}nberg--Br\`{e}gman functional, and $K$ is a convex closed subset of $M=\RR^n$. In both cases, the passage from convex to affine $K$ implies replacing $\geq$ by $=$. The fact that $D_1$ belongs to the family $D_\Psi$  for atomic finite $(\X,\mu)$ was established in \cite[p. 1021]{Bregman:1966} (=\cite[p. 15]{Bregman:1966:PhD}). A generalisation of right pythagorean inequality \eqref{right.pyth.D.1} for $D_\Psi$ and $\RPPP^{D_\Psi}_K$,
\begin{equation}
	D_\Psi(\omega,\phi)\geq D_\Psi(\omega,\RPPP^{D_\Psi}_K(\omega))+D_\Psi(\RPPP^{D_\Psi}_K(\omega),\phi)\;\;\forall(\omega,\phi)\in M\times K,
\label{right.pyth.D.Psi}
\end{equation}
where $M$ is a reflexive Banach space $(X,\n{\cdot}_X)$, and $K\subseteq X$ is such that $\DG\Psi(K)$ is convex and closed in $X^\star$, has been obtained in \cite[Prop. 4.11]{MartinMarquez:Reich:Sabach:2012}. Under some suitable conditions on $\Psi$, left (resp., right) pythagorean inequality \textit{characterises} left (resp., right) $D_\Psi$-projections, cf. Proposition \ref{prop.left.pythagorean} (resp., \ref{prop.right.pythagorean}).

Given the nonlinearity of $\LPPP^{D_\Psi}_K$ and $\RPPP^{D_\Psi}_K$, as well as asymmetry and nonquadraticity of $D_\Psi$, the left and right generalised pythagorean equations are a highly remarkable generalisation of the ancient equation $a^2+b^2=c^2$, as well as its cartesian and hilbertian analogues,
\begin{equation}
\n{x-y}^2_\H=\n{x-\PPP^{d_{\n{
\cdot}_\H}}_K(y)}^2_\H+\n{\PPP^{d_{\n{
\cdot}_\H}}_K(y)-y}^2_\H\;\forall (x,y)\in K\times\H,
\end{equation}
for any convex closed subset $K$ of a Hilbert space $\H$. As shown in Table \ref{table.comparison.metric.bregman}, this generalisation extends to a wide range of structures. 
\begin{table}[!h]
{\scriptsize 
\noindent 
\setcellgapes{2pt}
\makegapedcells
\begin{tabularx}{\textwidth}{|>{\raggedright}p{2.55cm}||>{\raggedright}p{5.2cm}|>{\raggedright\arraybackslash}*{1}{X|}}\hline
geometry
&
norm
&
Va\u{\i}nberg--Br\`{e}gman
\\\hline
domain
&
Hilbert space $(\H,\s{\cdot,\cdot}_\H)$
&
reflexive Banach space $(X,\n{\cdot}_X)$
\\\hline\hline
convex function
&
$\frac{1}{2}\n{\cdot}_\H^2$
&
$\Psi$
\\\hline
orthogonality functional
&
$\s{\cdot,\cdot}_\H$
&
$\duality{\cdot,\DG\Psi(\cdot)}_{X\times X^\star}$
\\\hline
relative quantification
&
$(d_{\n{\cdot}_\H}(x,y))^2:=\n{x-y\,}_\H^2$ $\forall x,y\in\H$
&
$D_\Psi(x,y)$ $\forall x,y\in X$
\\\hline
cosine equation
&
$\n{x-z}_\H^2=\n{x-y}_\H^2+\n{y-z}_\H^2-2\s{x-y,z-y}_\H\;\forall x,y,z\in\H$
&
$D_\Psi(z,x)=D_\Psi(z,y)+D_\Psi(y,x)-\duality{z-y,\DG\Psi(x)-\DG\Psi(y)}_{X\times X^\star}\;\forall x,y,z\in X$
\\\hline
projection
&
$\PPP_K^{d_{\n{\cdot}_\H}}:=\arginff{x\in K}{\n{x-\,\cdot\,}_\H}$ 
&
$\LPPP^{D_\Psi}_K:=\arginff{x\in K}{D_\Psi(x,\,\cdot\,)}$ $\forall x\in X$,
\\
&$=\arginff{x\in K}{\n{x-\,\cdot\,}_\H^2}$ $\forall x\in\H$
&$\RPPP^{D_\Psi}_K:=\arginff{x\in K}{D_\Psi(\,\cdot\,,x)}$ $\forall x\in X$
\\\hline
orthogonal decomposition
&
$\left\{
\begin{array}{l}
\hspace{-0.2cm}P_L^{d_{\n{\cdot}_\H}}+P_{L^\bot}^{d_{\n{\cdot}_\H}}=\II_\H\\
\hspace{-0.25cm}\s{y,P^{d_{\n{\cdot}_\H}}_{L^\bot}(x)}_\H=0\;
\forall (x,y)\in \H\times L,
\end{array}
\right.$

$\left\{
\begin{array}{l}
\hspace{-0.2cm}\PPP^{d_{\n{\cdot}_\H}}_K+\PPP^{d_{\n{\cdot}_\H}}_{K^\circ}=\II_\H\\
\hspace{-0.25cm}\s{\PPP^{d_{\n{\cdot}_\H}}_K(x),\PPP^{d_{\n{\cdot}_\H}}_{K^\circ}(x)}_\H\!=0\;\forall x\in\H,
\end{array}
\right.$ 
for linear subspace $L\subseteq\H$, convex closed cone $K\subseteq\H$, $L^\bot:=\{y\in\H\mid\s{x,y}_\H=0\;\forall x\in L\}$, $K^\circ:=\{y\in\H\mid\s{x,y}_\H\leq0\;\forall x\in K\}$
&
$\left\{
\begin{array}{l}
\hspace{-0.2cm}\LPPP^{D_\Psi}_K+(\DG\Psi)^{\inver}\circ\hat{\PPP}^{\Psi^\lfdual}_{K^\circ}\circ\DG\Psi=\id_X\\
\hspace{-0.25cm}\duality{\LPPP^{D_\Psi}_K(x),\hat{\PPP}^{\Psi^\lfdual}_{K^\circ}\circ\DG\Psi(x)}_{X\times X^\star}=0\;\forall x\in X,
\end{array}
\right.$
\hspace{-0.2cm}$\left\{
\begin{array}{l}\hspace{-0.2cm}\hat{\PPP}^\Psi_{(\DG\Psi(C))^\circ}+\RPPP^{D_\Psi}_C=\id_X\\
\hspace{-0.25cm}\duality{(\DG\Psi^\lfdual)^{\inver}\circ\RPPP^{D_\Psi}_C(y),\hat{\PPP}^\Psi_{(\DG\Psi(C))^\circ}(y)}_{X\times X^\star}\!=0\;
\forall y\in X,
\end{array}
\right.$
for convex closed cone $K\subseteq X$, convex closed cone $\DG\Psi(C)\subseteq X^\star$, $K^\circ:=\{y\in X^\star\mid\duality{x,y}_{X\times X^\star}\leq0\;\forall x\in K\}$, $\hat{\PPP}^{\Psi^\lfdual}_{K^\circ}(y):=\arginff{z\in K^\circ}{\Psi^\lfdual(y-z)}$ $\forall y\in X^\star$; if $K$ (resp., $\DG\Psi(C)$) is replaced by a linear subspace $L\subseteq X$ (resp., $\DG\Psi(L)\subseteq X^\star$), then $(\cdot)^\circ$ is replaced by $(\cdot)^\bot$, where $M^\bot:=\{y\in X^\star\mid\duality{x,y}_{X\times X^\star}=0\;\forall x\in M\}$
\\\hline
pythagorean theorem for projections
&
$\n{x-y}^2_\H\geq\n{x-\PPP_K^{d_{\n{\cdot}_\H}}(y)}_\H^2+\n{\PPP_K^{d_{\n{\cdot}_\H}}(y)-y}_\H^2$ $\forall (x,y)\in K\times\H$ $\forall$ Chebysh\"{e}v $K$
&
$D_\Psi(x,y)\geq D_\Psi(x,\LPPP_K^{D_\Psi}(y))+D_\Psi(\LPPP_K^{D_\Psi}(y),y)$ $\forall(x,y)\in K\times X$ $\forall$ left $D_\Psi$-Chebysh\"{e}v $K$, $D_\Psi(x,y)\geq D_\Psi(x,\RPPP_K^{D_\Psi}(x))+D_\Psi(\RPPP_K^{D_\Psi}(x),y)$ $\forall(x,y)\in X\times K$ $\forall$ right $D_\Psi$-Chebysh\"{e}v $K$
\\\hline
completely nonexpansive maps
&
$\n{T(x)-T(y)}_\H\leq\n{x-y}_\H$ $\forall x,y\in K$
&
$D_\Psi(T(x),T(y))\leq D_\Psi(x,y)$ $\forall x,y\in K$
\\\hline
quasinonexpansive maps
&
$\n{x-T(y)}_\H\leq\n{x-y}_\H$ $\forall (x,y)\in \Fix(T)\times K$
&
$D_\Psi(x,T(y))\leq D_\Psi(x,y)$ $\forall(x,y)\in\Fix(T)\times K$,
$D_\Psi(T(x),y)\leq D_\Psi(x,y)$ $\forall(x,y)\in K\times\Fix(T)$
\\\hline
proximal maps
&
$\prox^{d_{\n{\cdot}_\H}}_{\lambda,f}:=\arginff{x\in\efd(f)}{f(x)-\frac{\lambda}{2}\n{x-\,\cdot\,}^2}$
&
$\lprox^{D_\Psi}_{\lambda,f}:=\arginff{x\in\efd(f)\cap\efd(\Psi)}{f(x)+\lambda D_\Psi(x,\,\cdot\,)}$, $\rprox^{D_\Psi}_{\lambda,f}:=\arginff{x\in\efd(f)\cap\intefd{\Psi}}{f(x)+\lambda D_\Psi(\,\cdot\,,x)}$
\\\hline
pythagorean theorem for proximal maps
&
$\n{x-y}^2_\H\geq\n{x-\prox^{d_{\n{\cdot}_\H}}_{\lambda,f}(y)}_\H^2+\n{\prox^{d_{\n{\cdot}_\H}}_{\lambda,f}(y)-y}_\H^2$ $\forall (x,y)\in\Fix(\prox^{d_{\n{\cdot}_\H}}_{\lambda,f})\times\H$, $\Fix(\prox^{d_{\n{\cdot}_\H}}_{\lambda,f})=\arginff{x\in\efd(f)}{f(x)}$ 
&
$D_\Psi(x,y)\geq D_\Psi(x,\lprox^{D_\Psi}_{\lambda,f}(y))+D_\Psi(\lprox^{D_\Psi}_{\lambda,f}(y),y)$ $\forall(x,y)\in\Fix(\lprox^{D_\Psi}_{\lambda,f})\times\intefd{\Psi}$, $\Fix(\lprox^{D_\Psi}_{\lambda,f})=\intefd{\Psi}\cap\arginff{x\in \efd(f)}{f(x)}$, $D_\Psi(x,y)\geq D_\Psi(x,\rprox^{D_\Psi}_{\lambda,f}(x))+D_\Psi(\rprox^{D_\Psi}_{\lambda,f}(x),y)$ $\forall(x,y)\in \intefd{\Psi}\times\Fix(\rprox^{D_\Psi}_{\lambda,f})$, $\Fix(\rprox^{D_\Psi}_{\lambda,f})=\DG\Psi^\lfdual(\lprox^{D_{\Psi^\lfdual}}_{f\circ\DG\Psi^\lfdual})$
\\\hline
resolvents
&
$\res_{\lambda T}:=(\id_X+\lambda T)^{\inver}$
&
$\lres^{\Psi}_{\lambda T}:=(\DG\Psi+\lambda T)\circ\DG\Psi$, $\rres^{\Psi}_{\lambda T}:=(\id_{X^\star}+\lambda T\circ\DG\Psi^\lfdual)^{\inver}$
\\\hline
pythagorean theorem for resolvents
&
$\n{x-y}^2_\H\geq\n{x-\res_{\lambda T}(y)}_\H^2+\n{\res_{\lambda T}(y)-y}_\H^2$ $\forall (x,y)\in\Fix(\res_{\lambda T})\times\H$, $\Fix(\res_{\lambda T})=T^\inver(0)$
&
$D_\Psi(x,y)\geq D_\Psi(x,\lres^\Psi_{\lambda T}(y))+D_\Psi(\lres^\Psi_{\lambda T}(y),y)\;\forall(x,y)\in\Fix(\lres^\Psi_{\lambda T})\times\intefd{\Psi}$, $\Fix(\lres^\Psi_{\lambda T})=\intefd{\Psi}\cap T^\inver(0)$, $D_{\Psi^\lfdual}(x,y)\geq D_{\Psi^\lfdual}(x,\rres^\Psi_{\lambda T}(x))+D_{\Psi^\lfdual}(\rres^\Psi_{\lambda T}(x),y)\;\forall(x,y)\in\intefd{\Psi^\lfdual}\times\Fix(\rres^\Psi_{\lambda T})$, $\Fix(\rres^{\Psi}_{\lambda T})=\DG\Psi(\intefd{\Psi}\cap T^\inver(0))$
\\\hline
\end{tabularx}
}
\caption{{\small Comparison of metric geometry on a Hilbert space and Va\u{\i}nberg--Br\`{e}gman geometry on a reflexive Banach space. For interpretation of $\duality{\cdot,\DG\Psi(\cdot)}_{X\times X^\star}$ as orthogonality in a Gateaux differentiable reflexive Banach space $(X,\n{\cdot}_X)$ for $\Psi=\Psi_\varphi$, see Remark \ref{remark.varphi}.\ref{remark.varphi.xiii}. For  comparison of characterisations of metric projections in Hilbert and Banach spaces with characterisations of left and right $D_\Psi$-projections, including their respective continuity properties, see Remark \ref{remark.varphi}.\ref{remark.varphi.iii}. For $Z\in\{\H,X\}$ and $T:X\ra 2^{X}$, $\Fix(T):=\{x\in Z\mid T(x)=x\}$. We denote $P_L^{d_{\n{\cdot}_{\H}}}:=\PPP_L^{d_{\n{\cdot}_{\H}}}$ for a linear subspace $L\subseteq\H$, since in such case $\PPP_L^{d_{\n{\cdot}_{\H}}}$ coincide with the bounded linear projection operators on $\H$.}}
\label{table.comparison.metric.bregman}
\end{table}
\subsection{...and quasinonexpansive operators}
Let $(X,\n{\cdot}_X)$ be a reflexive Banach space. In general, a map $T:K\ra\intefd{\Psi}$ with $\varnothing\neq K\subseteq\intefd{\Psi}\subseteq X$ is said to be \textit{left} (resp., \textit{right}) \textit{$D_\Psi$-quasinonexpansive} on $\Fix(T):=\{x\in K\mid T(x)=x\}$ if{}f
\begin{align}
D_\Psi(x,T(y))&\leq D_\Psi(x,y)\;\;\forall(x,y)\in\Fix(T)\times K\\
\mbox{(resp., }D_\Psi(T(x),y)&\leq D_\Psi(x,y)\;\;\forall(x,y)\in K\times\Fix(T)\mbox{)}.
\end{align}
Let $\lambda\in\,]0,\infty[$. An important example of left (resp., right) $D_\Psi$-quasinonexpansive maps is provided by left (resp., right) $D_\Psi$-resolvents of monotone operators $W:X\ra 2^{X^\star}$ \cite[Lem. 1]{Eckstein:1993} \cite[Def. 3.7]{Bauschke:Borwein:Combettes:2003} (resp., \cite[Def. 5.3]{MartinMarquez:Reich:Sabach:2012}),
\begin{align}
\lres^{\Psi}_{\lambda W}&:=(\DG\Psi+\lambda W)\circ\DG\Psi,\\
\mbox{(resp., }\rres^{\Psi}_{\lambda W}&:=(\id_{X^\star}+\lambda W\circ\DG\Psi^\lfdual)^{\inver}\mbox{)}.
\end{align}
Their special case,\footnote{Strictly speaking, $\rprox^{D_\Psi}_{\lambda,f}$ is a map $X\ra 2^X$, while $\rres^{\Psi}_{\lambda W}$ is a map $X^\star\ra 2^{X^\star}$, so the former cannot be a special case of the latter. However, under some mild conditions on $\Psi$, \eqref{eqn.resolvent.left.right} and \eqref{eqn.prox.left.right} allow for the above conceptual simplification. Apart from simplification of terminology, this perspective proves to be useful in applications. E.g., in Proposition \ref{prop.varphi.resolvent.uniform.continuity}.\ref{prop.varphi.resolvent.uniform.continuity.iv} we show that $\rprox^{D_\Psi}_{\lambda,f}$ and $\rres^{\Psi}_{\lambda W}$, for $\Psi=\Psi_{\varphi_{1,\beta}}$, exhibit exactly the same value $t$ of $t$-Lipschitz--H\"{o}lder continuity. On the other hand, while the generalised pythagorean theorem for $\rprox^{D_\Psi}_{\lambda,f}$ features $D_\Psi$, its variant for  $\rres^{\Psi}_{\lambda W}$ features $D_{\Psi^\lfdual}$.} given by the left (resp., right) $D_\Psi$-proximal maps for suitable $f:X\ra\,]-\infty,\infty]$ \cite[Eqn. (13)]{Censor:Zenios:1992} \cite[Def. 3.16]{Bauschke:Borwein:Combettes:2003} (resp., \cite[Def. 3.7]{Bauschke:Combettes:Noll:2006} \cite[Def. 3.3]{Laude:Ochs:Cremers:2020}),
\begin{align}
\lprox^{D_\Psi}_{\lambda,f}&:=\arginff{x\in\efd(f)\cap\efd(\Psi)}{f(x)+\lambda D_\Psi(x,\,\cdot\,)}\\
\mbox{(resp., }\rprox^{D_\Psi}_{\lambda,f}&:=\arginff{x\in\efd(f)\cap\intefd{\Psi}}{f(x)+\lambda D_\Psi(\,\cdot\,,x)}\mbox{)},
\end{align}
generalises left (resp., right) $D_\Psi$-projections. Quite noticeably, it turns out that the left $D_\Psi$-resolvents (resp., left $D_\Psi$-proximal maps) also satisfy the generalised pythagorean theorem \cite[Lem. 1]{Eckstein:1993} \cite[Props. 3.3.(i), 3.13.(iv).(b), 3.21.(vi), 3.22.(ii).(b), 3.23.(v).(b), Cor. 3.25]{Bauschke:Borwein:Combettes:2003}
\begin{align}
D_\Psi(x,y)&\geq D_\Psi(x,\lres^\Psi_{\lambda W}(y))+D_\Psi(\lres^\Psi_{\lambda W}(y),y)\;\;\forall(x,y)\in\Fix(\lres^\Psi_{\lambda W})\times\intefd{\Psi};\\
\hspace{-0.2cm}\mbox{(resp., }D_\Psi(x,y)&\geq D_\Psi(x,\lprox^{D_\Psi}_{\lambda,f}(y))+D_\Psi(\lprox^{D_\Psi}_{\lambda,f}(y),y)\;\;\forall(x,y)\in\Fix(\lprox^{D_\Psi}_{\lambda,f})\times\intefd{\Psi}\mbox{)}.
\end{align}
In Proposition \ref{prop.gen.pyth.thm.rres} (resp., \ref{prop.gen.pyth.thm.rprox}) we complete this theoretical landscape, establishing the generalised pythagorean theorem for right $D_\Psi$-resolvents (resp., right $D_\Psi$-proximal maps)

For $\varnothing\neq K\subseteq\intefd{\Psi}\subseteq X$, $T:K\ra\intefd{\Psi}$, and $\cl$ denoting a topological closure operator on the subsets of $X$ with respect to the topology of $\n{\cdot}_X$, consider a topological generalisation of $\Fix(T)$ given by \cite[p. 313]{Reich:1996}
\begin{equation}
\aFix(T):=\{x\in\cl(K)\st\exists\{x_n\in K\mid n\in\NN\},\;x_n\mbox{ converges weakly to }x,\;\lim_{n\ra\infty}\n{x_n-T(x_n)}_X\}.
\end{equation}
$T:K\ra\intefd{\Psi}$ such that
\begin{align}
D_\Psi(x,T(y))&\leq D_\Psi(x,y)\;\forall(x,y)\in\aFix(T)\times K\\
\mbox{(resp., }D_\Psi(T(x),y)&\leq D_\Psi(x,y)\;\forall(x,y)\in K\times\aFix(T)\mbox{)}
\end{align}
and, for any $y\in\aFix(T)$ and any bounded $\{x_n\in K\mid n\in\NN\}$, \cite[Def. 3.2]{Censor:Reich:1996} \cite[pp. 313--314]{Reich:1996} (resp., \cite[Def. 2.3.(iv)]{MartinMarquez:Reich:Sabach:2012}) 
\begin{align}
\lim_{n\ra\infty}(D_\Psi(y,x_n)-D_\Psi(y,T(x_n)))=0\;\;&\limp\;\;
\lim_{n\ra\infty}D_\Psi(T(x_n),x_n)=0\label{eqn.left.afix.qnexp}\\
\mbox{(resp., }\lim_{n\ra\infty}(D_\Psi(x_n,y)-D_\Psi(T(x_n),y))=0
\;\;&\limp\;\;\lim_{n\ra\infty}D_\Psi(x_n,T(x_n))=0\mbox{)},\label{eqn.right.afix.qnexp}
\end{align}
will be called a \underline{l}eft (resp., \underline{r}ight) \textit{\underline{s}trongly} $D_\Psi$-\underline{q}uasinonexpansive map, while the set of all such maps will be denoted by $\lsq(\Psi,K)$ (resp., $\rsq(\Psi,K)$). Condition \eqref{eqn.left.afix.qnexp} (resp., \eqref{eqn.right.afix.qnexp}) is essentially a topological version of the left (resp., right) generalised pythagorean theorem. Under some conditions on $\Psi$, cf. Definition \ref{def.LSQ.RSQ.compositional} and Proposition \ref{prop.lsq.rsq.old}.\ref{prop.lsq.rsq.old.i}--\ref{prop.lsq.rsq.old.ii}, the suitable finite subsets of $\lsq(\Psi,K)$ (resp., $\rsq(\Psi,K)$) admit composition \cite[Lems. 1, 2]{Reich:1996} \cite[Prop. 3.3]{MartinMarquez:Reich:Sabach:2013:BSN} (resp., \cite[Props. 4.4, 6.6]{MartinMarquez:Reich:Sabach:2013:BSN}): for any set $\{T_i:K\ra K\mid T_i\in\lsq(\Psi,K)$ (resp., $\rsq(\Psi,K)$), $i\in\{1,\ldots,m\}, m\in\NN\}$,
\begin{equation}
\textstyle\bigcap_{i=1}^m\aFix(T_i)\neq\varnothing\neq\aFix(T_m\circ\cdots\circ T_1)\;\;\limp\;\;T_m\circ\cdots\circ T_1\in\lsq(\Psi,K)\mbox{ (resp., }\rsq(\Psi,K)\mbox{)},
\end{equation}
with $\aFix(T_m\circ\cdots\circ T_1)\subseteq\bigcap_{i=1}^m\aFix(T_i)$. Such $\Psi$ will be called \textit{LSQ-compositional} (resp., \textit{RSQ-compositional}). Under some additional conditions on LSQ-compositional (resp., RSQ-compositional) $\Psi$, cf. Proposition \ref{prop.lsq.rsq.old}.\ref{prop.lsq.rsq.old.iii}--\ref{prop.lsq.rsq.old.iv}, $\LPPP_C^{D_{\Psi}}:K\ra\intefd{\Psi}$ belongs to $\lsq(\Psi,K)$ with $\aFix(\LPPP_C^{D_\Psi})=\Fix(\LPPP_C^{D_\Psi})=C$, and $\lres_T^\Psi\in\lsq(\Psi,X)$ (resp., $\RPPP^{D_\Psi}_C: K\ra\intefd{\Psi}$ belongs to $\rsq(\Psi,K)$ with $\aFix(\RPPP_C^{D_\Psi})=\Fix(\RPPP_C^{D_\Psi})=C$, and $\rres_T^\Psi\in\rsq(\Psi^\lfdual,X^\star)$). Such $\Psi$ will be called \textit{LSQ-adapted} (resp., \textit{RSQ-adapted})\footnote{More precisely, the abstract notion of an RSQ-adapted $\Psi$ is not sufficient for $\rres_T^\Psi\in\rsq(\Psi^\lfdual,X^\star)$, however the latter property is implied by all of the known sufficient conditions for RSQ-adaptedness of $\Psi$.}.

In the context of information theory (or statistical inference), $D_\Psi(\omega,\phi)$ can be interpreted as a quantification of relative information content of the information state $\omega$ with respect to the information state $\phi$. Hence, the generalised pythagorean equation \eqref{gen.pyth.intro} can be interpreted as a nonlinear additive decomposition of  quantification $D_\Psi(\omega,\phi)$ of an ``information gain'' (or an ``uncertainty loss'') from $\phi$ to $\omega$ into $D_\Psi(\LPPP^{D_\Psi}_K(\phi),\phi)$, interpreted as quantification of an ``information gain due to learning of constraints'' $K$, and $D_\Psi(\omega,\LPPP^{D_\Psi}_K(\phi))$, interpreted as quantification of an ``uncertainty loss within the constraints''. Analogous interpretation holds for right $D_\Psi$-projections, left and right $D_\Psi$-proximity maps,  as well as left and right $D_\Psi$-resolvents. Thus, left and right pythagorean inequalities allow for an additive decomposition of information contained in ``data'' into ``signal'' plus ``noise'' under a vast range of nonlinear quasinonexpansive operators. Due to \eqref{eqn.left.afix.qnexp}--\eqref{eqn.right.afix.qnexp}, the quasinonexpansive maps in $\lsq(\Psi,K)$ and $\rsq(\Psi,K)$ can be seen as a suitable topological generalisation of such inferences.
\subsection{New results}
Our work has five interconnected layers: 1) new results for $D_\Psi$ on reflexive Banach spaces $(X,\n{\cdot}_X)$; 2) construction of suitable categories of left and right $D_\Psi$-projections as well as of $\lsq(\Psi,K)$ and $\rsq(\Psi,K)$ maps; 3) new results for $D_{\Psi_\varphi}$ on reflexive Banach spaces $(X,\n{\cdot}_X)$; 4) developing the theory of extended Va\u{\i}nberg--Br\`{e}gman functionals $D_{\ell,\Psi}$ and quasinonexpansive maps over nonreflexive Banach spaces, together with the corresponding categories; 5) deriving a range of results for $D_\Psi$ and $D_{\ell,\Psi}$ in particular models, with a main focus on (nonreflexive) preduals of W$^*$- and JBW-algebras.

\subsubsection{Reflexive setting}

We contribute to the general theory of $D_\Psi$ on reflexive Banach space $(X,\n{\cdot}_X)$ with several new results, including an extension of a characterisation of right $D_\Psi$-projections by means of generalised pythagorean inequality to not necessarily finite $\Psi$ (Proposition \ref{prop.right.pythagorean}.\ref{prop.right.pythagorean.i}), proving a generalised pytha\-go\-re\-an inequality for right $D_\Psi$-proximal maps (Proposition \ref{prop.gen.pyth.thm.rprox}) and right $D_\Psi$-resolvents (Proposition \ref{prop.gen.pyth.thm.rres}), providing the right version of Al'ber's generalised orthogonal decompositions (Proposition \ref{prop.alber.decomposition}.\ref{prop.alber.decomposition.ii}), and establishing sufficient conditions for: norm-to-norm continuity of right $D_\Psi$-projections (Proposition \ref{prop.continuity.psi}.\ref{prop.continuity.psi.iii}--\ref{prop.continuity.psi.iv}) and right $D_\Psi$-proximal maps (Proposition \ref{prop.norm.continuity.left.right.prox}.\ref{prop.norm.continuity.left.right.prox.ii}); Lipschitz--H\"{o}lder continuity of left and right $D_\Psi$-resolvents (Proposition \ref{prop.lipschitz.hoelder.left.resolvent}), left and right $D_\Psi$-proximal maps (Proposition \ref{prop.lipschitz.hoelder.left.resolvent}), as well as left and right $D_\Psi$-projections (Corollary \ref{cor.Lipschitz.Hoelder.projections.Psi}).

\subsubsection{Categories}

Under some restrictions, the sets of left and right $D_\Psi$-projections, as well as the sets $\lsq(\Psi,K)$ and $\rsq(\Psi,K)$, can be transformed into suitable categories (Definitions \ref{def.cats.breg.proj} and \ref{def.LSQ.RSQ.cats}). Taking closed convex sets in $(X,\n{\cdot}_X)$ as objects, left $D_{\Psi}$-projections onto such sets as morphisms, and defining the composition of morphisms as a left $D_\Psi$-projection onto the intersection of constraint sets of the composed left $D_\Psi$-projections, we obtain the category $\lCvx(\Psi)$. By restriction of objects to affine closed sets, we obtain the category $\lAff(\Psi)$. By taking objects given by $\DG\Psi$-convex $\DG\Psi$-closed sets in $(X,\n{\cdot}_X)$ (i.e. sets such that their images in $(X^\star,\n{\cdot}_{X^\star})$ under the map $\DG\Psi:X\ra X^\star$ are convex and closed), morphisms given by their right $D_\Psi$-projections, and defining composition of two morphisms as a right $D_\Psi$-projection onto the intersection of the constraint sets of these morphisms, we obtain the category $\rbarCvx(\Psi)$. Its restriction to $\DG\Psi$-affine $\DG\Psi$-closed sets is denoted $\rbarAff(\Psi)$. These categories are well defined provided that we admit an empty set as an object and an empty arrow, $\ulcorner\varnothing\urcorner$, as a morphism available between any two objects. The superscript $^\subseteq$ will denote a restriction of the above categories to the case when the composition of morphisms is different from $\ulcorner\varnothing\urcorner$ only when the codomain of second projection is a subset of the codomain of first projection. Assuming analogous conditions on the $\aFix(T)$ sets of left strongly $D_\Psi$-quasinonexpansive maps, and restricting them to be convex and closed, gives rise to the category $\LSQcvxsub$. An analogous construction for right strongly $D_\Psi$-quasinonexpansive maps gives a category $\RbarSQcvxsub$. Hence, LSQ-compositionality (resp., RSQ-compositionality) of $\Psi$ is a condition allowing to define the category $\LSQcvxsub$ (resp., $\RbarSQcvxsub$). If $\Psi$ is LSQ-adapted (resp., RSQ-adapted), then there is an embedding functor from $\lCvx^\subseteq(\Psi)$ to $\LSQcvxsub$ (resp., from $\rbarCvx^\subseteq(\Psi)$ to $\RbarSQcvxsub$), as well as a functor right adjoint to it, which assigns to each $T$ its fixed point set $\Fix(T)$ (Definition \ref{def.LSQ.RSQ.functors}.\ref{def.LSQ.RSQ.functors.iii}--\ref{def.LSQ.RSQ.functors.vi} and Proposition \ref{prop.adjointness.cvx}.\ref{prop.adjointness.cvx.ii}--\ref{prop.adjointness.cvx.iii}). The categories $\lCvx(\Psi)$ and $\rbarCvx(\Psi)$ are equivalent, and the same is true for $\LSQcvxsub(\Psi)$ and $\RbarSQcvxsub(\Psi)$, since the second element of each of these pairs is defined by means of the first one, through the Euler--Legendre transformation implemented by $\DG\Psi$ (Propositions \ref{cor.functors.breg.proj}.\ref{cor.functors.breg.proj.i} and \ref{prop.adjointness.cvx}.\ref{prop.adjointness.cvx.i}).

\subsubsection{Gauges and quasigauges $\varphi$}

For $D_\Psi$ with $\Psi=\Psi_\varphi$ and a gauge $\varphi$, we provide the first systematic study of the properties of this functional, obtaining an array of new results, including the sufficient conditions on linear norm-geometric properties of $(X,\n{\cdot}_X)$ for: zone consistency and characterisation of left and right $D_{\Psi_\varphi}$-projections by a generalised pythagorean inequality (Proposition \ref{prop.left.right.psi.varphi}); norm-to-norm continuity of left and right $D_{\Psi_\varphi}$-projections (Proposition \ref{prop.continuity}), left and right $D_{\Psi_\varphi}$-proximal maps (Proposition \ref{prop.varphi.prox.norm.continuity}), and left and right $D_{\Psi_\varphi}$-resolvents (Proposition \ref{prop.varphi.resolvent.norm.continuity}), as well as uniform continuity and Lipschitz--H\"{o}lder continuity of left and right $D_{\Psi_\varphi}$-projections (Proposition \ref{prop.varphi.uniform.continuity}) and of left and right $D_{\Psi_\varphi}$-proximal maps and $D_{\Psi_\varphi}$-resolvents (Proposition \ref{prop.varphi.resolvent.uniform.continuity}); LSQ-compositionality, LSQ-adaptedness, RSQ-compositionality, and RSQ-adaptedness of $\Psi_\varphi$ (Proposition \ref{prop.varphi.compositional}). (See Table \ref{table.cases.i-v} for more detailed discussion of, and structural view on, these notions and results.) Furthermore, we characterise the Euler--Legendre property of $\Psi_\varphi$ by strict convexity and Gateaux differentiability of $(X,\n{\cdot}_X)$ (Proposition \ref{prop.legendre}). We also deliver several new results for $\Psi=\Psi_\varphi$ with a quasigauge $\varphi$, which has never been considered before in the context of $D_\Psi$, including a characterisation of the Euler--Legendre property of $\Psi_\varphi$ (Proposition \ref{prop.legendre.quasigauge}) and sufficient conditions on $\varphi$ for a characterisation of left and right $D_{\Psi_\varphi}$-projections by the respective generalised pythagorean theorem (Proposition \ref{prop.left.right.psi.quasigauge}). Taken together, these results establish a strong bridge between the Va\u{\i}nberg--Br\`{e}gman and norm geometries of reflexive Banach spaces. They also play a key role in providing nontrivial functional analytic and operator algebraic models of the (reflexive and extended) Va\u{\i}nberg--Br\`{e}gman geometry in Section \ref{section.models}.

\subsubsection{Extension $\ell$}

While the generalised pythagorean equation \eqref{right.pyth.D.Psi} has the same form as, and is completely analogous to, \eqref{gen.pyth.intro}, the original result \eqref{right.pyth.D.1} of Chencov is more subtle, containing an additional layer of abstraction. More specifically, \eqref{right.pyth.D.Psi}, provided in \cite[Prop. 4.11]{MartinMarquez:Reich:Sabach:2012} and Proposition \ref{prop.right.pythagorean}, deals with $M$ given by the reflexive Banach space $(X,\n{\cdot}_X)$, with convexity and closure specified, via $\DG\Psi$, in terms of structure of $(X^\star,\n{\cdot}_{X^\star})$, and it is also applicable to a special case of $D_1$ defined on $\RR^n$. On the other hand, Chencov's result deals with $D_1$ over $(L_1(\X,\mu),\n{\cdot}_1)$, which is a nonreflexive Banach space, and it specifies convexity not in terms of the linear structure of $L_1(\X,\mu)$, but in terms of exponential families \eqref{eqn.exponential.family.chencov}: for a fixed $n\in\NN$, it uses a nonlinear exponential coordinate map from an $n$-dimensional subset of $L_1(\X,\mu)$ to a convex closed subset of a reflexive space $\RR^n$. The same issue of a ``suitable definition'' of convexity and closure of the sets for a generalised pythagorean theorem over a nonreflexive Banach space appeared in \cite[Prop. 8.1, Prop. 8.2]{Jencova:2005}, for a family $D_\gamma$ of relative entropies over (nonreflexive) preduals $(\N_\star,\n{\cdot}_{\N_\star})$ of arbitrary W$^*$-algebras $\N$: in order to establish the generalised pythagorean theorem for left $D_\gamma$-projections on subsets $K$ of $(\N_\star,\n{\cdot}_{\N_\star})$, the convexity and closure of $K$ was specified in terms of the structure of (reflexive) noncommutative $(L_{1/\gamma}(\N),\n{\cdot}_{1/\gamma})$ space with $\gamma\in\,]0,1[$, via Mazur maps $\ell_\gamma:\N_\star\ra L_{1/\gamma}(\N)$ (cf. Corollary \ref{cor.d.gamma}.\ref{cor.d.gamma.ii} and Remark \ref{remark.lp.breg.proj}.\ref{remark.lp.breg.proj.i} for details).

In order to extend applicability of the Va\u{\i}nberg--Br\`{e}gman geometry beyond the class of reflexive Banach spaces $(X,\n{\cdot}_X)$, we generalise the above observations, and consider a bijective map $\ell$ from a subset $U$ of a (generally not reflexive) Banach space $(Y,\n{\cdot}_Y)$ to a subset $\ell(U)$ of $X$. In this sense, $\ell$ establishes a global nonlinear coordinate system on $U$, modelled in $X$. Given $\ell$ and $\Psi$, the \textit{extended Va\u{\i}nberg--Br\`{e}gman functional} on $U$ is defined by $D_{\ell,\Psi}(\phi,\psi):=D_\Psi(\ell(\phi),\ell(\psi))$. The suitable properties of the subsets of $U$ become expressed in terms of the corresponding properties of their $\ell$-embeddings into $(X,\n{\cdot}_X)$, with the convex analytic results proved for $D_\Psi$ on $(X,\n{\cdot}_X)$ pulled back into the corresponding results for $D_{\ell,\Psi}$. Hence,  consideration of convexity (resp., affinity; closure; $\DG\Psi$-convexity; $\DG\Psi$-affinity; $\DG\Psi$-closure) of sets in $(X,\n{\cdot}_X)$ leads us to consideration of $\ell$-convexity (resp., $\ell$-affinity; $\ell$-closure; $(\DG\Psi\circ\ell)$-convexity; $(\DG\Psi\circ\ell)$-affinity; $(\DG\Psi\circ\ell)$-closure) of sets in $(Y,\n{\cdot}_Y)$, and dealing with the left $D_{\ell,\Psi}$-projections onto $\ell$-convex $\ell$-closed sets, right $D_{\ell,\Psi}$-projections onto $(\DG\Psi\circ\ell)$-convex $(\DG\Psi\circ\ell)$-closed sets, sets $\lsq(\ell,\Psi,C)$ and $\rsq(\ell,\Psi,C)$ of left and right strongly $D_{\ell,\Psi}$-quasinonexpansive maps, respectively, etc. (In particular, we obtain categories $\lCvx(\ell,\Psi)$ and $\rbarCvx(\ell,\Psi)$ (as well as their affine and $^\subseteq$- subcategories), $\LSQcvxsub(\ell,\Psi)$ and $\RbarSQcvxsub(\ell,\Psi)$, and the corresponding functorial relationships between them.)

In other words, each choice of $(\ell,\Psi)$ sets up a specific way of probing the relationship between quantitative and geometric properties of a subset $U$ of $(Y,\n{\cdot}_Y)$ in terms of the Va\u{\i}nberg--Br\`{e}gman geometry of $\ell(U)$ in $(X,\n{\cdot}_X)$, i.e. as perceived through the lenses of a nonlinear embedding (`coordinate system') $\ell:U\ra X$ and a convex function $\Psi:X\ra\,]-\infty,\infty]$ (`loss/bias criterion of statistical discrimination'). This way, by moving from $D_\Psi$ to $D_{\ell,\Psi}=D_\Psi\circ(\ell,\ell)$, we provide a partial reconcilliation between (arbitrary dimensional) Chencov's and Va\u{\i}nberg--Br\`{e}gman approaches by means of a setting, which we call the \textit{extended Va\u{\i}nberg--Br\`{e}gman geometry}. 

While the results on uniqueness and existence of $D_\Psi$-projections can be pulled back without any additional topological assumptions on $\ell$ (since bijectivity of $\ell$ allows to use the topology induced by the norm topology of $(X,\n{\cdot}_X)$), providing sufficient conditions for the stability of $D_{\ell,\Psi}$-projections, expressed in terms of their norm-to-norm continuity or uniform continuity, requires to impose the corresponding continuity conditions on $\ell$. As a result, the best behaved sector of the extended Va\u{\i}nberg--Br\`{e}gman geometry (which we identify as `well adapted models') belongs to an intersection of the convex analysis on reflexive Banach spaces with the nonlinear homeomorphic theory of (arbitrary) Banach spaces. In particular, if $\ell$ is norm-to-norm continuous, then the $\ell$-closed sets in $(X,\n{\cdot}_X)$ coincide with the sets closed with respect to the topology of $\n{\cdot}_X$. Hence, the only inevitable price to pay for the extension from $D_\Psi$ to $D_{\ell,\Psi}$ is the replacement of convexity by $\ell$-convexity.
\subsection{Models}
A pair $(\ell,\Psi)$ will be called a \textit{model} of an extended Va\u{\i}nberg--Br\`{e}gman geometry if{}f it satisfies the following axioms: $D_\Psi$ is an information (i.e. $D_\Psi(x,y)=0$ $\iff$ $x=y$), left and right $D_\Psi$-projections (onto, respectively, left and right $D_\Psi$-Chebysh\"{e}v subsets of $(X,\n{\cdot}_X)$) exist, and they satisfy the corresponding generalised pythagorean inequalities. For $\Psi=\Psi_\varphi$, further strengthening of the model (allowing more axioms to be satisfied) is implemented by strengthening of the geometric properties of the corresponding norm $\n{\cdot}_X$, as shown in Table \ref{table.cases.i-v}.
\begin{table}[!h]
{\small\ \\
\noindent
\begin{tabularx}{\textwidth}{p{1.0cm}|*{1}{X}|p{3.3cm}}
&
&
\textit{geometry of }$(X,\n{\cdot}_X)$
\\
\textit{case}&
\textit{properties of the Va\u{\i}nberg--Br\`{e}gman geometry (satisfied in a given case)}
&
\textit{for $\Psi=\Psi_\varphi$}
\\\hline\hline
I
&
$D_\Psi$ and $D_{\Psi^\lfdual}$ are informations, left and right $D_\Psi$-projections are single-valued and are characterised by the respective generalised pythagorean inequalities, left and right $D_\Psi$-resolvents are single-valued and satisfy the respective generalised pythagorean inequalities
&
\vspace{0.97cm}
$\bSC\cap\bG\cap\bR$
\\
&
&\hspace{0.7cm}$\hookuparrow$
\\
II
&
left and right $D_\Psi$-projections are norm-to-norm continuous + all of above
&
$\bSC\cap\bRRS\cap\bF\cap\bR$
\\
&
&\hspace{0.7cm}$\hookuparrow$
\\
III$^{\mathrm{R}}$/ /III$^{\mathrm{L}}$
&
(norm-to-norm continuity of right $D_\Psi$-resolvents + LSQ- and RSQ-com\-po\-sability/norm-to-norm continuity of left $D_\Psi$-resolvents) + all of above
&
$\bSC\cap\bRRS\cap\bUF$/ $/\bUC\cap\bF$
\\
&
&\hspace{0.7cm}$\hookuparrow$
\\
IV
&
RSQ-adaptedness + all of above
&
$\bUC\cap\bUF$
\\
&
&\hspace{0.7cm}$\hookuparrow$
\\
V$^{\mathrm{L}}$/V$^{\mathrm{R}}$
&
left/right $D_\Psi$-projections and left/right $D_\Psi$-resolvents are uniformly continuous on bounded subsets + all of above
&
{$\frac{1}{\beta}$-$\bUC\cap\bUF$/\hspace{0.5cm}\;\;\;\;\;\;\;} {/$\bUC\cap\frac{1}{\beta}$-$\bUF$}
\\
&
&\hspace{0.7cm}$\hookuparrow$
\\
VI
&
left and right $D_\Psi$-projections and $D_\Psi$-resolvents are Lipschitz--H\"{o}lder continuous + all of above
&
$\frac{1}{\beta}$-$\bUC\cap\frac{1}{\gamma}$-$\bUF$
\\
II$^{\mathrm{b}}$
&
LSQ-adaptedness + all of case I
&
$\bSC\cap\bUF$
\end{tabularx}
}
\caption{{\small Different cases of Va\u{\i}nberg--Br\`{e}gman geometric axioms, and the corresponding conditions on the Banach geometry of $(X,\n{\cdot}_X)$ which are sufficient for these axioms to hold for $\Psi=\Psi_\varphi$ with a gauge $\varphi$. Notation: $\bSC$ := strict convexity, $\bG$ := Gateaux differentiability, $\bR$ := reflexivity, $\bRRS$ := the Radon--Riesz--Shmul'yan property, $\bF$ := Fr\'{e}chet differentiability, ($\frac{1}{\beta}$-)$\bUF$ := ($\frac{1}{\beta}$-)uniform Fr\'{e}chet differentiability, ($\frac{1}{\beta}$-)$\bUC$ := ($\frac{1}{\beta}$-)uniform convexity. These results are obtained in Propositions \ref{prop.left.right.psi.varphi}, \ref{prop.continuity}, \ref{prop.varphi.compositional}, \ref{prop.varphi.resolvent.norm.continuity}, \ref{prop.varphi.resolvent.uniform.continuity}, and \ref{prop.varphi.uniform.continuity}. While for any model $(\ell,\Psi)$ we require left and right $D_\Psi$-projections to \textit{satisfy} the respective generalised pythagorean inequalities, case I strengthens this to a requirement of \textit{characterisation}. Case II$^{\mathrm{b}}$ is included in case III$^{\mathrm{R}}$, but not in case III$^{\mathrm{L}}$.}}
\label{table.cases.i-v}
\end{table}

While the definition of an extended Va\u{\i}nberg--Br\`{e}gman model does not impose any additional conditions on a bijection $\ell:U\ra\ell(U)\subseteq X$, this is not longer so for its further refinements. Borrowing a terminology from \cite[p. 231]{Dubuc:1979}, we will call a model $(\ell,\Psi)$, not necessarily with $\Psi=\Psi_\varphi$, to be \textit{well adapted} (resp., \textit{uniformly well adapted}; \textit{Lipschitz--H\"{o}lder well adapted}) if{}f it satisfies all axioms of case IV (resp., V; VI) and $\ell$ is a norm-to-norm (resp., uniform; Lipschitz--H\"{o}lder)  homeomorphism. For $(\ell,\Psi=\Psi_\varphi)$ with a gauge $\varphi$ it turns out that the properties of case I--IV models do not depend on the choice of $\varphi$, but (essentially due to Proposition \ref{prop.psi.varphi.geometry}.\ref{prop.psi.varphi.geometry.vii}--\ref{prop.psi.varphi.geometry.viii}) such dependence appears for case V$^\mathrm{L}$/V$^\mathrm{R}$ (so, also for case VI) models (Propositions \ref{prop.varphi.resolvent.uniform.continuity} and \ref{prop.varphi.uniform.continuity}). Hence, a passage from norm-to-norm continuity of left and right $D_{\Psi_\varphi}$-projections, $D_{\Psi_\varphi}$-proximal maps, and $D_{\Psi_\varphi}$-resolvents to their uniform continuity includes becoming sensitive to the properties of a particular $\varphi$. More specifically, case V$^{\mathrm{L}}$ and case V$^{\mathrm{R}}$ models are specified only for the particular family of $\varphi$, given by $\varphi(t)=\varphi_{1,\beta}(t):=t^{1/\beta-1}$. Furthermore, since uniform homeomorphy of unit balls of two Banach spaces does not imply uniform homeomorphy of these spaces, a passage to uniformly well adapted models amounts, in practice (and for $\Psi=\Psi_\varphi$), to restriction of considerations to the extended Va\u{\i}nberg--Br\`{e}gman geometry of a unit ball (and its subsets). On the other hand, under generalisation from gauges to quasigauges, there is already a split of case I models into left and right cases (I$^\mathrm{L}$/I$^\mathrm{R}$), dependently on the particular properties of a quasigauge (Proposition \ref{prop.left.right.psi.quasigauge}).

The examples of well adapted case IV models are given by $(\ell_\gamma,\Psi_\varphi)$, with an arbitrary gauge $\varphi$ and the Mazur map $\ell_\gamma(\phi):=\phi^\gamma$, for $\gamma\in\,]0,1[$, mapping preduals of arbitrary W$^*$-algebras $\N$ and semifinite JBW-algebras $A$ into the corresponding $L_{1/\gamma}$ spaces over these algebras (Propositions \ref{prop.dope.gauge.gamma.wstar} and \ref{prop.dope.gauge.gamma.jbw}). The resulting extended Va\u{\i}nberg--Br\`{e}gman geometries over the respective preduals of W$^*$- and JBW-algebras depend only on the choice of $\gamma$ and $\varphi$. From the perspective of elementary differential and convex geometric properties of the Banach space norm, the structure of $(L_{1/\gamma}(\N),\n{\cdot}_{1/\gamma})$ and $(L_{1/\gamma}(A,\tau),\n{\cdot}_{1/\gamma})$ spaces does not exhibit major variability in the range of $\gamma\in\,]0,1[\setminus\{\frac{1}{2}\}$, and does not depend on the type of $\N$ and $A$.\footnote{In comparison, as we show in the sequel work \cite{Kostecki:2026:Orlicz} (cf. \cite{Kostecki:2025:Poznan} for an announcement of some of its main results), the extended Va\u{\i}nberg--Br\`{e}gman geometries $(\ell_\orlicz,\Psi_\varphi)$, induced over preduals $\N_\star$ of semifinite W$^*$-algebras $\N$, via Kaczmarz maps $\ell_\orlicz$, from the geometry of $p$-Amemiya norm on noncommutative Orlicz spaces, $(L_{\orlicz}(\N,\tau),\n{\cdot}_{\orlicz,p})$ with $p\in[1,\infty]$, exhibit variability over the type of $\N$, over the geometric properties of the Orlicz function $\orlicz$, and over $p\in\{\{1\},\,]1,\infty[,\,\{\infty\}\}$. In this context, two main virtues of providing an independent analysis for noncommutative and nonassociative $L_{1/\gamma}$ spaces in the current paper are: to have the range of results available for type III (i.e. not semifinite) W$^*$-algebras and for semifinite JBW-algebras.} In consequence, some qualitative differences between the extended Va\u{\i}nberg--Br\`{e}gman geometries for different choices of $(\gamma,\varphi)$ show up only at the level of case V$^{\mathrm{L}}$/V$^{\mathrm{R}}$ and case VI models, through dependence of the uniform and Lipschitz--H\"{o}lder continuity of $j_\varphi$ and $\ell_\gamma$ on the available values $(p,q)$ of $p$-uniform convexity and $q$-uniform Fr\'{e}chet differentiability of $(L_{1/\gamma}(\N),\n{\cdot}_{1/\gamma})$ and $(L_{1/\gamma}(A,\tau),\n{\cdot}_{1/\gamma})$. Case V$^{\mathrm{L}}$ (resp., V$^{\mathrm{R}}$) models are given by $(\ell_\gamma,\Psi_\varphi)$ with $L_{1/\gamma}$ spaces over arbitrary W$^*$-algebras and semifinite JBW-algebras for $\varphi=\varphi_{1,\beta}$ with $(\gamma,\beta)\in(]0,\frac{1}{2}]\times\,]0,\gamma])\cup([\frac{1}{2},1[\times\,]0,\frac{1}{2}])$ (resp., $(\gamma,\beta)\in(]0,\frac{1}{2}]\times[\frac{1}{2},1[)\cup([\frac{1}{2},1[\,\times[\gamma,1[)$), cf. Proposition \ref{prop.uniform.Lp} and Corollary \ref{cor.r.unif.cont.proj.na.Lp}. Among the case V$^{\mathrm{L}}$ and case V$^{\mathrm{R}}$ models we specify a class of uniformly and Lipschitz--H\"{o}lder well adapted models, by restriction of $\ell_\gamma$ to a unit ball of the respective predual. For case VI models we calculate the \rpkmark{modulus of continuity} (more specifically, the value of $t$ for $t$-Lipschitz--H\"{o}lder continuity) of left and right $D_{\ell_\gamma,\Psi_{\varphi_{1,\beta}}}$-projections, as well as left $D_{\ell_\gamma,\Psi_{\varphi_{1,\beta}}}$-resolvents.

We give also three different examples of models with $\Psi=\Psi_\varphi$ for a gauge $\varphi$, but with $\ell$ different from the Mazur (as well as Kaczmarz) map. Proposition \ref{prop.oscr} provides well adapted (resp., uniformly and Lipschitz--H\"{o}lder well adapted) case V$^{\mathrm{L}}$ and V$^{\mathrm{R}}$ models, with $(X,\n{\cdot}_X)$ given by uniformly convex and uniformly Fr\'{e}chet differentiable (resp., $p$-uniformly convex and $q$-uniformly Fr\'{e}chet differentiable) Banach function space over a localisable measure space $(\X,\mu)$, with $\ell$ given by an inverse $\ell_X$ of the Lozanovski\u{\i} factorisation map ${\ell_X}^{\inver}$, and the domain of $\ell_X$ given by the unit sphere $S(L_1(\X,\mu),\n{\cdot}_1)$. Proposition \ref{prop.nc.oscr} provides an analogue of this result for $(X,\n{\cdot}_X)$ given by uniformly convex and uniformly Fr\'{e}chet differentiable (resp., $p$-uniformly convex and $q$-uniformly Fr\'{e}chet differentiable) noncommutative rearrangement invariant Banach space over type I$_n$ W$^*$-algebra $\N$ with $n\in\NN$ (i.e. $\N=\BH$ for $\dim\H=n$, and $(X,\n{\cdot}_X)$ is a space $\MNC$ of $n\times n$ matrices over $\CC$, equipped with a unitary invariant matrix norm). Proposition \ref{prop.gen.spin.factors} provides case I model for preduals of generalised spin factors $(X^\star\oplus\RR,\n{\cdot}_{X^\star\oplus\RR})$, with $\ell$ given by a mapping $\ell_{/\RR}$ from a base of a base normed space $(X\oplus\RR,\n{\cdot}_{X\oplus\RR})$ into a unit ball of a reflexive Banach space $(X,\n{\cdot}_X)$. This result goes beyond the realms of W$^*$- and JBW-algebras (since a generalised spin factor is a JBW-algebra if{}f $(X,\n{\cdot}_X)$ is a Hilbert space), and provides an example of application of our framework to the general base normed spaces (in particular, Proposition \ref{prop.gen.spin.factors} characterises the presence of Alfsen--Shultz spectral duality between $(X\oplus\RR,\n{\cdot}_{X\oplus\RR})$ and $(X^\star\oplus\RR,\n{\cdot}_{X^\star\oplus\RR})$ by the condition that $\Psi_\varphi$ is Euler--Legendre, which is equivalent to Gateaux differentiability and strict convexity of $(X,\n{\cdot}_X)$).

Finally, we specify also some examples of case I models featuring $\Psi\neq\Psi_\varphi$. Proposition \ref{prop.hilbert.operator.psi} deals with self-adjoint parts of preduals of arbitrary W$^*$-algebras and with preduals of semifinite JBW-algebras using $\ell=\ell_{1/2}$ (with a codomain $\H$ denoting, respectively, either a self-adjoint part of a noncommutative $L_2$ space or a nonassociative $L_2$ space) with $\Psi(x):=\frac{1}{2}\s{Tx,x}_\H$, where $T$ is a continuous linear map satisfying $\exists\lambda>0$ $\forall x,y\in\H$ $\s{Tx-Ty,y-x}_\H\geq\lambda\n{x-y}^2_\H$. Proposition \ref{prop.finite.spectral.EL.case.I} deals with preduals of type I$_n$ W$^*$-algebras, with $n\in\NN$, using $\Psi$ given by a spectral convex Euler--Legendre function on the space of finite dimensional self-adjoint matrices (with three formerly known examples of such functions given in Example \ref{ex.spectral.convex}.\ref{ex.spectral.convex.i}--\ref{ex.spectral.convex.iii}, and two new examples given in Corollary \ref{cor.finite.spectral.EL}.\ref{cor.finite.spectral.EL.iv}--\ref{cor.finite.spectral.EL.v}). Proposition \ref{prop.BRLZ.spectral.EL} provides an extension of this approach to Schatten classes $(\schatten(\H))^\sa$ of self-adjoint compact operators on a separable Hilbert space $\H$ (with $((\schatten(\H))^\sa,\n{\cdot}_{(\schatten(\H))^\sa})$ uniformly convex and uniformly Fr\'{e}chet differentiable and $\ell$ given by the inverse $\ell_{(\schatten(\H))^\sa}$ of Lozanovski\u{\i} factorisation map), resulting in models $(\ell,\Psi)$ on ${\N_\star}^\sa=((\BH)^\sa)_\star$ such that $\Psi$ is Euler--Legendre, while it is not assumed to be $\Psi_\varphi$. (However, while this construction is a natural extension of already established results, we are currently missing examples of $\Psi\neq\Psi_\varphi$ at this level of generality.)
\subsection{Plan of the paper}\label{section.plan}
Section \ref{section.background} covers background definitions and properties that are extensively used in the rest of this paper. Section \ref{section.background.convex} introduces key notions from convex analysis on Banach spaces. Section \ref{section.background.bregman.projections} presents the properties of left and right Va\u{\i}nberg--Br\`{e}gman projections and quasinonexpansive maps, used in Section \ref{section.convex.new.general}. Section \ref{section.gauge.functions} discusses the properties of Banach norm geometry as characterised by $\Psi_\varphi$, which will be used in Section \ref{section.convex.new.varphi}.

Section \ref{section.convex.new} contains new results in the Va\u{\i}nberg--Br\`{e}gman theory on reflexive Banach spaces together with the theory of extended Va\u{\i}nbeg--Br\`{e}gman geometries on arbitrary Banach spaces. In Section \ref{section.convex.new.general} we establish results applicable to arbitrary Va\u{\i}nberg--Br\`{e}gman functionals on reflexive Banach spaces. Section \ref{section.convex.new.varphi} is concerned with the special case, when $\Psi$ is given by an integral of a gauge (or quasigauge) function $\varphi$, i.e. $\Psi=\Psi_\varphi$. Section \ref{section.convex.new.d.ell.psi} encapsulates the results of Sections \ref{section.convex.new.general} and \ref{section.convex.new.varphi} into the general setting of extended Va\u{\i}nberg--Br\`{e}gman geometry, based on composition of $D_\Psi$ on reflexive Banach space $(X,\n{\cdot}_X)$ with a nonlinear homeomorphism $\ell$ from a subset of an arbitrary Banach space $(Y,\n{\cdot}_Y)$. Section \ref{section.categories.d.psi.proj} provides construction of suitable categories of left and right $D_{\ell,\Psi}$-projections, as well as categories of left and right strongly $D_{\ell,\Psi}$-quasinonexpansive maps, together with some elementary results about functors between them. 

Section \ref{section.models} applies the results of Section \ref{section.convex.new} to the case when $(Y,\n{\cdot}_Y)$ is (in principle) a base normed space, and $(X,\n{\cdot}_X)$ is a suitable reflexive space constructed over $(Y,\n{\cdot}_Y)$. Nearly all of results in Section \ref{section.models} are obtained in a setting when $(Y,\n{\cdot}_Y)$ is a predual of an arbitrary W$^*$-algebra or of a semifinite JBW-algebra and $(Y,\n{\cdot}_Y)$ is, respectively, noncommutative or nonassociative rearrangement invariant space over this algebra. In particular, in Section \ref{section.L.gamma.Mazur} we consider $\Psi=\Psi_\varphi$ with $(X,\n{\cdot}_X)$ given by noncommutative and nonassociative $(L_{1/\gamma},\n{\cdot}_{1/\gamma})$ spaces for $\gamma\in\,]0,1[$, and $\ell$ given by the corresponding Mazur maps $\ell_\gamma$. To show the flexibility of the framework (and its applicability for the general statistical theory on base normed spaces), in Section \ref{section.other.models.Psi.varphi} we provide an example beyond JBW-algebraic setting, with $(Y,\n{\cdot}_Y)$ given by preduals of generalised spin factors. Apart from this, Section \ref{section.other.models.Psi.varphi} contains also an example of $\ell$ given by the Lozanovski\u{\i} uniform homeomorphism $\ell_X$, which goes far beyond the realms of Mazur (resp., Kaczmarz) maps into $L_{1/\gamma}$ (resp., Orlicz) spaces, and is applicable to embedding of $S(L_1(\X,\mu),\n{\cdot}_1)$ into \textit{any} uniformly convex and uniformly Fr\'{e}chet differentiable function space over a localisable measure space $(\X,\mu)$. We also provide an analogue of this result for preduals of type I$_n$ W$^*$-algebras $\N$. While the results of Sections \ref{section.L.gamma.Mazur}--\ref{section.other.models.Psi.varphi} rely on an assumption $\Psi=\Psi_\varphi$, in Section \ref{section.lewis.type.models} we provide few examples of models with $\Psi\neq\Psi_\varphi$.

There is essentially no new results in Section \ref{section.background},\footnote{With a possible exceptions of Corollaries \ref{cor.proj.adaptedness} and \ref{cor.r.unif.convex.j.lambda}, which are quite straightforward, but we have not found them in the literature.} yet we provide some new definitions, which package formerly known properties into suitable objects (e.g., the notions of left/right pythagorean $D_\Psi$ and LSQ/RSQ-adapted/compositional $\Psi$). Contrary to that, all definitions and propositions/corollaries in Sections \ref{section.convex.new} and \ref{section.models} are new, either completely or in their (extended) range of generality (with the exception of Propositions \ref{prop.norm.continuity.left.right.prox}.\ref{prop.norm.continuity.left.right.prox.i}, \ref{prop.alber.decomposition}.\ref{prop.alber.decomposition.i}+\ref{prop.alber.decomposition.iii}, \ref{prop.varphi.uniform.continuity}.\ref{prop.varphi.uniform.continuity.i}, \ref{prop.oscr}.\ref{prop.oscr.i}, \ref{prop.nc.oscr}.\ref{prop.nc.oscr.i}, \ref{prop.weakly.compact.base}.\ref{prop.weakly.compact.base.iv}, \ref{prop.hilbert.operator.psi}.\ref{prop.hilbert.operator.psi.i}--\ref{prop.hilbert.operator.psi.iv}, and equivalence of \ref{prop.gen.spin.factors.1} and \ref{prop.gen.spin.factors.2} in Proposition \ref{prop.gen.spin.factors}). The detailed discussion of relationships with the formerly known results is provided in the remarks at the end of each subsection.
\section{Background definitions and properties}\label{section.background}
In what follows, $(X,\n{\cdot}_X)$ will denote a Banach space over $\KK\in\{\RR,\CC\}$ \cite[\S1]{Banach:1922}, $B(X,\n{\cdot}_X):=\{x\in X\mid \n{x}_X\leq1\}$, $S(X,\n{\cdot}_X):=\{x\in X\mid\n{x}_X=1\}$. If $\lambda\in\RR$ and $Z\subseteq X$, then $\lambda Z:=\cup_{x\in Z}\{\lambda x\}$. $(X^\star,\n{\cdot}_{X^\star})$ will denote a Banach space of continuous linear functions $X\ra\KK$, equipped with a norm $\n{y}_{X^\star}:=\sup\{\ab{y(x)}\mid\n{x}_X\leq1\}$ $\forall y\in X^\star$ \cite[p. 62]{Helly:1921}, and will be called a \df{Banach dual} of $(X,\n{\cdot}_X)$ (with respect to a bilinear duality $\duality{x,y}_{X\times X^\star}:=y(x)\in\KK$ $\forall x\in X$ $\forall y\in X^\star$). $(X,\n{\cdot}_X)$ is called \df{reflexive} \cite[pp. 219--220]{Hahn:1927} if{}f $J_X:X\ni x\mapsto\duality{x,\,\cdot\,}_{X\times X^\star}\in X^\star{}^\star$ is an isometric isomorphism. By default, unless stated otherwise, all references to continuity and closure/openness of sets will be understood in the sense of the norm topology of an underlying Banach space. In particular, given $(X,\n{\cdot}_X)$, for any $Y\subseteq X$, $\INT(Y)$ (resp., $\cl(Y)$) will denote a topological interior (resp., closure) of $Y$ with respect to the topology of $\n{\cdot}_X$. If $(X,\n{\cdot}_X)$ and $(Y,\n{\cdot}_Y)$ are Banach spaces, then $(X,\n{\cdot}_X)\sqsubseteq(Y,\n{\cdot}_Y)$ will denote a continuous embedding $X\subseteq Y$. (In more general situation, the reference to a particular norm will be provided in a subscript of $\cl$ and $\INT$.) Given Banach spaces $(X,\n{\cdot}_X)$ and $(Y,\n{\cdot}_Y)$, $Z\subseteq X$, $W\subseteq Y$, a function $f:Z\ra W$ is said to be: \df{uniformly continuous} on $Z$ if{}f
\begin{equation}
\forall\epsilon_1>0\;\exists\epsilon_2>0\;\forall x,y\in Z\;\;\n{x-y}_X<\epsilon_2\;\limp\;\n{f(x)-f(y)}_Y<\epsilon_1;
\end{equation}
\df{$t$-Lipschitz--H\"{o}lder continuous}\footnote{This condition, both for $t=1$ and $t<1$ has been introduced first by Lipschitz, respectively, in \cite[Eqn. (2)]{Lipschitz:1876} and \cite[Eqn. (2$^\star$)]{Lipschitz:1876}, published 6 years before H\"{o}lder's \cite[pp. 17--18]{Hoelder:1882}.} on $Z$ for $t\in\,]0,\infty[$ (called an \df{exponent} of $f$) if{}f
\begin{equation}
\exists c>0\;\;\forall x,y\in Z\;\;\n{f(x)-f(y)}_Y\leq c\n{x-y}_X^t;
\label{eqn.t.hoelder}
\end{equation}
\df{Lipschitz continuous} on $Z$ if{}f it satisfies \eqref{eqn.t.hoelder} with $t=1$. If $t\in\,]0,1]$ and $f$ is $t$-Lipschitz--H\"{o}lder continuous on $Z$, then it is uniformly continuous on $Z$ and $r$-Lipschitz-H\"{o}lder continuous on $Z$ $\forall r\in\,]0,t[$. For any Banach spaces $(X,\n{\cdot}_X)$ and $(Y,\n{\cdot}_Y)$, any uniform homeomorphism $\alpha:S(Y,\n{\cdot}_Y)\ra S(X,\n{\cdot}_X)$ extends to a uniform homeomorphism \cite[Prop. 2.9]{Odell:Schlumprecht:1994}
\begin{equation}
B(Y,\n{\cdot}_Y)\ni x\mapsto
\left\{
\begin{array}{ll}
\n{x}_Y\alpha\left(\frac{x}{\n{x}_Y}\right)\in B(X,\n{\cdot}_X)
&\st x\in B(Y,\n{\cdot}_Y)\setminus\{0\}\\
0&\st x=0.
\end{array}
\right.
\label{eqn.homeo.ext}
\end{equation}
Conversely \cite[p. 197]{Benyamini:Lindenstrauss:2000}, for any uniform homeomorphism $\alpha:B(Y,\n{\cdot}_Y)\ra B(X,\n{\cdot}_X)$ there exists a corresponding uniform homeomorphism
\begin{equation}
S(Y,\n{\cdot}_Y)\ni x\mapsto\frac{\alpha(x)}{\n{\alpha(x)}_X}\in S(X,\n{\cdot}_X).
\label{eqn.unif.homeo.induced.on.spheres}
\end{equation}
If $\alpha:S(Y,\n{\cdot}_Y)\ra S(X,\n{\cdot}_X)$ (resp., $\alpha:B(Y,\n{\cdot}_Y)\ra B(X,\n{\cdot}_X)$) is $t$-Lipschitz--H\"{o}lder continuous, then the map \eqref{eqn.homeo.ext} (resp., \eqref{eqn.unif.homeo.induced.on.spheres}) is $t$-Lipschitz--H\"{o}lder continuous \cite[Lem. 3.1]{Adzhiev:2014:I}.

$\Psi:X\ra\,]-\infty,\infty]$ will be called: \df{proper} if{}f $\efd(\Psi):=\{x\in X\mid\Psi(x)\neq\infty\}\neq\varnothing$; \df{coercive} if{}f $\lim_{\n{x}_X\ra\infty}\Psi(x)=\infty$; \df{supercoercive} if{}f $\lim_{\n{x}_X\ra\infty}\frac{\Psi(x)}{\n{x}_X}=\infty$; \df{lower semicontinuous} if{}f $\{x\in X\mid\Psi(x)\leq\lambda\}$ is closed $\forall\lambda\in\RR$ if{}f $f(x)\leq\liminf_{y\ra x}f(y)$; \df{convex} if{}f
\begin{equation}
x\neq y\;\limp\;\Psi(\lambda x+(1-\lambda)y)\leq\lambda\Psi(x)+(1-\lambda)\Psi(y)\;\;\forall x,y\in\efd(\Psi)\;\forall\lambda\in\,]0,1[
\label{eqn.convex.function}
\end{equation}
(this is equivalent to the definition based on the same inequality, with quantifiers changed to $\forall x,y\in X$ $\forall\lambda\in[0,1]$, with the conventions $\infty+\infty\equiv\infty$, $0\cdot\infty\equiv\infty$, and without $x\neq y$ assumption); \df{strictly convex} if{}f \eqref{eqn.convex.function} holds under the same quantifiers and with $\leq$ replaced by $<$. The set of all proper, convex, lower semicontinuous functions $\Psi:X\ra\,]-\infty,\infty]$ will be denoted by $\pcl(X,\n{\cdot}_X)$. If $\Psi\in\pcl(X,\n{\cdot}_X)$, then $\Psi$ is continuous on $\intefd{\Psi}$ \cite[Cor. 7C]{Rockafellar:1966:level}. For any $K\subseteq X$, an \df{indicator function} of $K$ on $X$ is given by $\iota_K(x):=\left\{\begin{array}{ll}
0&\st x\in K\\
\infty&\st x\not\in K
\end{array}\right.$ $\forall x\in X$. If $K$ is nonempty, convex, and closed, then $\iota_K\in\pcl(X,\n{\cdot}_X)$ \cite[p. 2897]{Moreau:1962:prox} \cite[p. 23]{Rockafellar:1963}.

For convenience of notation, from now on, and until the end of Section \ref{section.background} (as well as in entire Section \ref{section.convex.new}), we will assume that $\KK=\RR$ (all results and formulas of those Sections are applicable for the case $\KK=\CC$ under replacement of $\duality{\cdot,\cdot}_{X\times X^\star}$ by $\re\duality{\cdot,\cdot}_{X\times X^\star}$). Conventions $\inf\varnothing=\infty$ and $0\notin\NN$ will be applied everywhere. For any set $Z$, $2^Z$ will denote the set of all subsets of $Z$.
\subsection{Convex analytic preliminaries}\label{section.background.convex}
\subsubsection{Differentiability and Mandelbrojt--Fenchel duality}
For any $T:X\ra 2^{X^\star}$, $\efd(T):=\{x\in X\mid T(x)\neq\varnothing\}$. The \df{subdifferential} of a proper $\Psi:X\ra\,]-\infty,\infty]$ is \cite[Def. 2-G]{Rockafellar:1963} \cite[Eqn. (1)]{Moreau:1963} \cite[Def. 4]{Minty:1964}
\begin{equation}
\partial\Psi(x):=\{y\in X^\star\mid\Psi(z)-\Psi(x)\geq\duality{z-x,y}_{X\times X^\star}\;\forall z\in X\}\;\;\forall x\in X.
\label{eqn.subdiff}
\end{equation}
Hence, $\partial\Psi(x)=\varnothing$ $\forall x\in X\setminus\efd(\Psi)$ and $\efd(\partial\Psi)=\{x\in\efd(\Psi)\mid\partial\Psi(x)\neq\varnothing\}$. If $\Psi:X\ra\,]-\infty,\infty]$ is proper, then the \df{right Gateaux derivative} of $\Psi$ at $x\in\efd(\Psi)$ in the direction $h\in X$ reads \cite[p. 53]{Ascoli:1932}
\begin{equation}
	\efd(\Psi)\times X\ni(x,h)\mapsto\DG_+\Psi(x,h):=\lim_{t\ra^+0}(\Psi(x+th)-\Psi(x))/t\in\,]-\infty,\infty],
\label{eqn.right.gateaux}
\end{equation}
and it exists $\forall h\in X$. If $\Psi\in\pcl(X,\n{\cdot}_X)$, then $\DG_+\Psi(x,\,\cdot\,)$ is Lipschitz continuous and finite $\forall x\in\intefd{\Psi}$ (cf., e.g., \cite[Cor. 1.1.6]{Butnariu:Iusem:2000}). $\Psi\in\pcl(X,\n{\cdot}_X)$ is called \df{Gateaux differentiable} at $x\in\intefd{\Psi}$ \cite[p. 311]{Gateaux:1914} if{}f $\DG_+\Psi(x,y)=-\DG_+\Psi(x,-y)$ $\forall y\in X$. In such case $\DG_+\Psi(x,\,\cdot\,)$ is linear, so it defines a bounded linear operator $\DG_+\Psi(x,y)=:\duality{y,\DG\Psi(x)}_{X\times X^\star}$ $\forall y\in X$ \cite[Def. 3]{Vainberg:1952}. (If $n\in\NN$ and $x\in X=\RR^n$, then $\DG\Psi(x)=\grad\Psi(x):=(\frac{\partial}{\partial x^1},\ldots,\frac{\partial}{\partial x^n})\Psi(x)$.) A set of all $\Psi\in\pcl(X,\n{\cdot}_X)$ that are Gateaux differentiable on $\intefd{\Psi}\neq\varnothing$ will be denoted $\pclg(X,\n{\cdot}_X)$. If $\Psi\in\pcl(X,\n{\cdot}_X)$ is Gateaux differentiable at $x\in\intefd{\Psi}$, then $\partial\Psi(x)=\{\DG\Psi(x)\}$ \cite[p. 20]{Moreau:1963:etude}. Combined with \eqref{eqn.subdiff}, this gives
\begin{equation}
	\Psi(x)-\Psi(y)-\duality{x-y,\DG\Psi(y)}_{X\times X^\star}\geq0\;\;\forall(x,y)\in X\times\intefd{\Psi}.
\label{eqn.bregman.from.subdiff}
\end{equation}

$\Psi\in\pcl(X,\n{\cdot}_X)$ will be called: \df{Fr\'{e}chet differentiable} at $x\in\intefd{\Psi}$ if{}f \cite[p. 808]{Frechet:1925} \cite[p. 309]{Frechet:1925:ENS} it is Gateaux differentiable at $x$ and $\DG\Psi(x)$ is uniformly continuous on $S(X,\n{\cdot}_X)$; \df{uniformly Gateaux differentiable} on $\varnothing\neq K\subseteq\efd(\Psi)$ \cite[p. 4]{Shioji:1994} (=\cite[p. 643]{Shioji:1995}) \cite[Def. 1]{Aze:Penot:1995} \cite[p. 207]{Zalinescu:2002} if{}f there exists $g:[0,\infty[\,\ra[0,\infty]$ with $g(0)=0$ and $\lim_{t\ra^+0}\frac{g(t)}{t}=0$ such that
\begin{equation}
\forall x\in K\;\forall y\in X\;\forall\lambda\in\,]0,1[\;\;\Psi(x)+\lambda(1-\lambda)g(\n{y}_X)\geq(1-\lambda)\Psi(x-\lambda y)+\lambda(x+(1-\lambda)y);
\end{equation}
\df{essentially Gateaux differentiable} \cite[Def. 5.2.(i), Thm. 5.6]{Bauschke:Borwein:Combettes:2001} if{}f $\intefd{\Psi}\neq\varnothing$ and $\partial\Psi(x)=\{*\}$ $\forall x\in\efd(\partial\Psi)$; \df{essentially strictly convex} \cite[Def. 5.2.(ii)]{Bauschke:Borwein:Combettes:2001} if{}f $\Psi$ is strictly convex on every convex subset of $\efd(\partial\Psi)$ and
\begin{equation}
	\exists\varepsilon>0\;\;\forall x\in\efd((\partial\Psi)^{\inver})\;\;\sup\{\n{(\partial\Psi)^{\inver}(x+\varepsilon y)}_X\mid y\in X,\;\n{y}_X\leq1\}<\infty,
\end{equation}
where $(\partial\Psi)^{\inver}(y):=\{x\in X\mid y\in\partial\Psi(x)\}$; \df{Euler--Legendre} \cite[Def. 5.2.(iii)]{Bauschke:Borwein:Combettes:2001} if{}f it is essentially Gateaux differentiable and essentially strictly convex. If $\Psi\in\pcl(X,\n{\cdot}_X)$ and $\efd(\Psi)=X$, then $\Psi$ will be called: \df{uniformly Fr\'{e}chet differentiable} on $X$ (resp., on $\varnothing\neq K\subseteq X$) \cite[p. 4]{Shioji:1994} (=\cite[p. 643]{Shioji:1995}) \cite[Thm. 3.5.6]{Zalinescu:2002} if{}f $\Psi$ is Fr\'{e}chet differentiable at any $x\in X$ (resp., $x\in K$), and $\DG\Psi$ is uniformly continuous on $X$ (resp., $\DG\Psi(x)(h)$ exists in uniform convergence $\forall(x,h)\in K\times S(X,\n{\cdot}_X)$); \df{uniformly Fr\'{e}chet differentiable on bounded subsets} of $X$ \cite[p. 221]{Zalinescu:2002} if{}f $\Psi$ is uniformly Fr\'{e}chet differentiable on $\lambda B(X,\n{\cdot}_X)$ $\forall\lambda>0$. If $\Psi\in\pcl(X,\n{\cdot}_X)$ and $\efd(\Psi)\neq\{*\}$, then $\Psi$ is called \df{uniformly convex} at $x\in\efd(\Psi)$ (resp., on $X$; on bounded subsets of $X$) \cite[p. 231]{Asplund:1967:averaged} \cite[Def. 2.1]{Zalinescu:1983} (resp., \cite[p. 997]{Levitin:Polyak:1966} \cite[Def. 1]{Vladimirov:Nesterov:Chekanov:1978}; \cite[p. 221]{Zalinescu:2002})\footnote{For equivalence of different formulas used in these definitions, see \cite[Rem. 2.1]{Zalinescu:1983} and \cite[Lem. 2.2]{Butnariu:Iusem:Zalinescu:2003}.} if{}f 
\begin{equation}
\forall t\in\,]0,\infty[\;\;\inf\left\{\textstyle\frac{1}{2}\Psi(x)+\textstyle\frac{1}{2}\Psi(y)-\Psi\left(\frac{x+y}{2}\right)\mid y\in\efd(\Psi),\;\n{y-x}_X=t\right\}>0
\end{equation}
(resp., with $x\in\efd(\Psi)$ replaced by: $y,x\in\efd(\Psi)$; $y\in\efd(\Psi)$, $x\in\lambda B(X,\n{\cdot}_X)$ (together with $\forall\lambda>0$ condition stated outside of $\inf\{\ldots\}$)).

If $\Psi\in\pcl(X,\n{\cdot}_X)$ is Fr\'{e}chet differentiable at $x\in\intefd{\Psi}$, then $\DG\Psi(x)$ will be denoted by $\DF\Psi(x)$. If $\Psi\in\pcl(X,\n{\cdot}_X)$ is essentially Gateaux differentiable, then $\Psi\in\pclg(X,\n{\cdot}_X)$ \cite[Thm. 5.6.(iv)--(v)]{Bauschke:Borwein:Combettes:2001}. If $\varnothing\neq C\subseteq X$ is open and convex, and $\Psi:C\ra\RR$ is convex, continuous, and Gateaux differentiable on $C$, then ($\Psi$ is Fr\'{e}chet differentiable on $C$ if{}f $\DG\Psi$ is norm-to-norm continuous) \cite[Prop. 2.8]{Phelps:1989}.

For a proper $\Psi:X\ra\,]-\infty,\infty]$,
\begin{equation}
X^\star\ni y\mapsto\Psi^\lfdual(y):=\sup_{x\in X}\{\duality{x,y}_{X\times X^\star}-\Psi(x)\}\in\,]-\infty,\infty],
\label{eqn.fenchel.dual}
\end{equation}
which will be called the \df{Mandelbrojt--Fenchel dual} of $\Psi$ \cite[Eqn. (1)]{Mandelbrojt:1939} \cite[p. 75]{Fenchel:1949} \cite[p. 8]{Moreau:1962},\footnote{\label{footnote.Mandelbrojt.Fenchel}\cite[Eqn. (1)]{Mandelbrojt:1939} considered convex $\Psi:\RR\ra\RR$ and $X=\RR$, while \cite[p. 75]{Fenchel:1949} considered convex and lower semicontinuous $\Psi:K\ra[0,\infty]$ with $\lim_{x\ra y}\Psi(x)=\infty$ $\forall y\in\cl(K)\setminus\INT(K)$, as well as $\sup_{x\in K}\{\ldots\}$ instead of $\sup_{x\in X}\{\ldots\}$, for a convex $\varnothing\neq K\subseteq\RR^n$. \cite[Thm. (p. 977)]{Mandelbrojt:1939} contains an error, asserting that $\Psi^\lfdual$ is a convex real valued function for \textit{any} convex function $\Psi:\RR\ra\RR$ (cf. \cite[Footn. 1]{Fenchel:1949}). \cite[p. 8]{Moreau:1962} introduced \eqref{eqn.fenchel.dual} in a full form. In subsequent references we list first Fenchel's result for $K\subseteq\RR^n$ (if available), and then its generalisation to $(X,\n{\cdot}_X)$.} satisfies $\Psi^\lfdual\in\pcl(X^\star,\n{\cdot}_{X^\star})$ \cite[Thm. 5]{Hoermander:1954} \cite[p. 9]{Moreau:1962} (cf. also \cite[Thm. 3.6]{Broendsted:1964}). If $\Psi\in\pcl(X,\n{\cdot}_X)$, then $(\Psi^\lfdual)^\lfdual|_{J_X(X)}=\Psi$ \cite[Thm. (p. 75)]{Fenchel:1949} \cite[Thm. 3.13]{Broendsted:1964}. Furthermore, by \eqref{eqn.fenchel.dual}, \df{Fenchel inequality} holds \cite[p. 75]{Fenchel:1949} \cite[p. 13]{Broendsted:1964}:
\begin{equation}
\Psi(x)+\Psi^\lfdual(y)-\duality{x,y}_{X\times X^\star}\geq0\;\;\forall(x,y)\in X\times X^\star,
\label{eqn.fenchel.ineq}
\end{equation}
with $=$ attained if{}f $y\in\partial\Psi(x)$. If $\Psi\in\pclg(X,\n{\cdot}_X)$ then $\DG\Psi:\intefd{\Psi}\ra\DG\Psi(\intefd{\Psi})$ is a bijection, with $(\DG\Psi)^{\inver}=J_X^{\inver}\circ\DG\Psi^\lfdual$. If $(X,\n{\cdot}_X)$ is reflexive, and $\Psi\in\pcl(X,\n{\cdot}_X)$ is essentially Gateaux differentiable, then $\DG\Psi(\intefd{\Psi})=\intefd{\Psi^\lfdual}$ and $(\DG\Psi)^{\inver}=\DG\Psi^\lfdual$. 

If $(X,\n{\cdot}_X)$ is reflexive, then $\Psi\in\pcl(X,\n{\cdot}_X)$ is essentially Gateaux differentiable if{}f $\Psi^\lfdual$ is essentially strictly convex \cite[Thm. 5.4]{Bauschke:Borwein:Combettes:2001}, hence \cite[Cor. 5.5]{Bauschke:Borwein:Combettes:2001} $\Psi$ is Euler--Legendre if{}f $\Psi^\lfdual$ is Euler--Legendre. Furthermore, if $(X,\n{\cdot}_X)$ is reflexive, and $\Psi\in\pcl(X,\n{\cdot}_X)$, then $\Psi$ is Euler--Legendre if{}f $\Psi\in\pclg(X,\n{\cdot}_X)$, $\efd(\DG\Psi)=\intefd{\Psi}$, $\Psi^\lfdual\in\pclg(X^\star,\n{\cdot}_{X^\star})$, $\efd(\DG\Psi^\lfdual)=\intefd{\Psi^\lfdual}$ \cite[\S2.1]{Reich:Sabach:2009}. $\Psi\in\pcl(X,\n{\cdot}_X)$ is uniformly Gateaux differentiable (resp., uniformly convex) on $X$ if{}f $\Psi^\lfdual$ is uniformly convex (resp., uniformly Gateaux differentiable) on $X^\star$ \cite[Cor. 2.8]{Aze:Penot:1995}. If $(X,\n{\cdot}_X)$ is reflexive, $\Psi\in\pcl(X,\n{\cdot}_X)$, and $\intefd{\Psi}\neq\varnothing$ (resp., $\intefd{\Psi^\lfdual}\neq\varnothing$), then $\Psi$ is uniformly Fr\'{e}chet differentiable on $\intefd{\Psi}$ (resp., uniformly convex on $X$) if{}f $\Psi^\lfdual$ is uniformly convex on $X^\star$ (resp., uniformly Fr\'{e}chet differentiable on $\intefd{\Psi^\lfdual}$ \cite[Thm. 2.2]{Zalinescu:1983}.
\subsubsection{Va\u{\i}nberg--Br\`{e}gman functional}
Dependently on a purpose, the \df{Va\u{\i}nberg--Br\`{e}gman functional} on $(X,\n{\cdot}_X)$ is defined either as \cite[Eqn. (2)]{Butnariu:Iusem:1997}
\begin{equation}
D_\Psi^+:X\times X\ni(x,y)\mapsto
        \left\{
                \begin{array}{ll}
                        \Psi(x)-\Psi(y)-\DG_+\Psi(y;x-y)
												&\st y\in\efd(\Psi)\\
                        \infty
												&\st\mbox{otherwise}
                \end{array}
        \right.\in[0,\infty],
\label{eqn.bregman.plus}
\end{equation}
for any $\Psi\in\pcl(X,\n{\cdot}_X)$, or as \cite[Eqn. (8.5)]{Vainberg:1956} 
\begin{equation}
D_\Psi:X\times X\ni(x,y)\mapsto
        \left\{
                \begin{array}{ll}
                        \Psi(x)-\Psi(y)-\duality{x-y,\DG\Psi(y)}_{X\times X^\star}
												&\st y\in\intefd{\Psi}\\
                        \infty
												&\st\mbox{otherwise}
                \end{array}
        \right.\in[0,\infty],
\label{eqn.vainberg.bregman}
\end{equation}
for any $\Psi\in\pclg(X,\n{\cdot}_X)$. (Nonnegativity of a codomain of $D^+_\Psi$ and $D_\Psi$ follows from convexity of $\Psi$.) 

Definition \eqref{eqn.vainberg.bregman} implies $\forall\Psi,\Psi_1,\Psi_2\in\pclg(X,\n{\cdot}_X)$ $\forall x,y\in\intefd{\Psi}$ $\forall z,w\in X$ $\forall\lambda_1,\lambda_2\geq0$ $\forall\lambda_3,\lambda_4\in\RR$
\begin{equation}
\Psi_3(x):=\lambda_3x+\lambda_4\;\;\limp\;\;D_{\lambda_1\Psi_1+\lambda_2\Psi_2+\Psi_3}=\lambda_1D_{\Psi_1}+\lambda_2D_{\Psi_2}\;\;\textup{\cite[p. 16]{Bregman:1966:PhD}},
\label{eqn.affine.scaling}
\end{equation}
\begin{align}
D_\Psi(x,y)+D_\Psi(y,x)&=\duality{x-y,\DG\Psi(x)-\DG\Psi(y)}_{X\times X^\star}\;\;{\textup{\cite[p. 328]{Censor:Lent:1981}}},
\label{eqn.symmetric.bregman}\\
D_\Psi(z,x)&=D_\Psi(z,y)+D_\Psi(y,x)-\duality{z-y,\DG\Psi(x)-\DG\Psi(y)}_{X\times X^\star}\;\;{\textup{\cite[Lem. 3.1]{Chen:Teboulle:1993}}},
\label{eqn.generalised.cosine}\\
D_\Psi(z,x)+D_\Psi(w,y)&=D_\Psi(z,y)+D_\Psi(w,x)-\duality{z-w,\DG\Psi(x)-\DG\Psi(y)}_{X\times X^\star}\;\;{\textup{\cite[Rem. 3.5]{Bauschke:Lewis:2000}}}.
\label{eqn.quadruple.property}
\end{align}
Equations \eqref{eqn.affine.scaling}--\eqref{eqn.quadruple.property} hold also for $D_\Psi$ (resp., $\duality{\cdot,\DG\Psi(y)}_{X\times X^\star}$) replaced by $D^+_\Psi$ (resp., $\DG_+\Psi(y,\,\cdot\,)$). Equation \eqref{eqn.symmetric.bregman} implies the formula for the class of Va\u{\i}nberg--Br\`{e}gman functionals that are symmetric with respect to interchange of variables: $\forall x,y\in\intefd{\Psi}$
\begin{equation}
D_\Psi(x,y)=D_\Psi(y,x)\;\;\iff\;\;D_\Psi(x,y)=\textstyle\frac{1}{2}\duality{x-y,\DG\Psi(x)-\DG\Psi(y)}_{X\times X^\star}.
\end{equation}

Comparison of \eqref{eqn.vainberg.bregman} with \eqref{eqn.bregman.from.subdiff}, while applying the equality case of \eqref{eqn.fenchel.ineq}, gives
\begin{equation}
D_\Psi(x,y)=\Psi(x)+\Psi^\lfdual(\DG\Psi(y))-\duality{x,\DG\Psi(y)}_{X\times X^\star}\;\;\forall(x,y)\in X\times\intefd{\Psi}.
\label{eqn.bregman.fenchel}
\end{equation}
When equipped with an additional condition, $D_\Psi(x,y)=\infty$ $\forall(x,y)\in X\times(X\setminus\intefd{\Psi})$, \eqref{eqn.bregman.fenchel} becomes a definition of $D_\Psi$ equivalent to \eqref{eqn.vainberg.bregman}. If $\Psi^\lfdual$ is Gateaux differentiable on $\varnothing\neq\DG\Psi(\intefd{\Psi})\subseteq\intefd{\Psi^\lfdual}$, then \cite[Lem. 7.3]{Bauschke:Borwein:Combettes:2001} \cite[Lem. 3.2]{Luo:Meng:Wen:Yao:2019}
\begin{equation}
D_\Psi(x,y)=D_{\Psi^\lfdual}(\DG\Psi(y),\DG\Psi(x))\;\;\forall x,y\in\intefd{\Psi}.
\label{eqn.bregman.fenchel.duality}
\end{equation}
$D_\Psi$ is said to be \df{jointly convex} \cite[\S1]{Bauschke:Borwein:2001} if{}f $(x,y)\mapsto D_\Psi(x,y)$ is convex on $\intefd{\Psi}\times\intefd{\Psi}$.

$\Psi\in\pcl(X,\n{\cdot}_X)$ is called: \df{totally convex} at $x\in\efd(\Psi)$ if{}f \cite[2.2]{Butnariu:Censor:Reich:1997} \cite[p. 62]{Butnariu:Iusem:1997} 
\begin{equation}
\nu_\Psi(x,t):=\inf\{D^+_\Psi(y,x)\mid y\in\efd(\Psi),\;\n{y-x}_X=t\}>0\;\;\forall t\in\,]0,\infty[\,,
\end{equation}
or, equivalently \cite[Prop. 2.2]{Resmerita:2004}, if{}f
\begin{equation}
\lim_{n\ra\infty}D^+_\Psi(y_n,x)=0\;\limp\;\lim_{n\ra\infty}\n{y_n-x}_X=0\;\;\;\forall\{y_n\in\efd(\Psi)\mid n\in\NN\};
\end{equation}
\df{totally convex on bounded subsets} of $X$ \cite[Lem. 2.1.2]{Butnariu:Iusem:2000} if{}f
\begin{equation}
\inf\{\nu_\Psi(x,t)\mid x\in Y\cap\efd(\Psi)\}>0\;\;\forall t\in\,]0,\infty[\;\;\forall\mbox{ bounded }\varnothing\neq Y\subseteq X.
\label{eqn.totally.convex.bounded}
\end{equation}
If $\efd(\Psi)\neq\{*\}$, then \eqref{eqn.totally.convex.bounded} is equivalent \cite[Prop. 4.2]{Butnariu:Iusem:Zalinescu:2003} to \df{sequential consistency} of $\Psi$, defined as \cite[Def. 2.1.(vi)]{Censor:Lent:1981} \cite[Cor. 4.9.(iii)]{Butnariu:Iusem:1997} $\forall\{y_n\in\efd(\Psi)\mid n\in\NN\}$ $\forall$ bounded $\{x_n\in\efd(\Psi)\mid n\in\NN\}$
\begin{equation}
\lim_{n\ra\infty}D_\Psi(y_n,x_n)=0\;\;\limp\;\;\lim_{n\ra\infty}\n{y_n-x_n}_X=0.
\end{equation}
\subsubsection{Monotone maps}
If $(X,\n{\cdot}_X)$ is a Banach space, then the \df{graph} of $T:X\ra 2^{X^\star}$ is given by $\Graph(T):=\{(x,y)\in X\times X^\star\mid y\in T(X)\}$, while $T^\inver(y):=\{x\in X\mid y\in T(x)\}$ $\forall y\in X^\star$. If, furthermore, $\varnothing\neq K\subseteq\efd(T)$, then $T$ is called: \df{monotone} on $K$ if{}f \cite[Thm. 7]{Vainberg:Kachurovskii:1959} \cite[Eqn. (1)]{Kachurovskii:1960} \cite[p. 0]{Zarantonello:1960} \cite[p. 341]{Minty:1962}
\begin{equation}
\duality{x-y,v-w}_{X\times X^\star}\geq0\;\forall x,y\in K\;\forall v\in T(x)\;\forall w\in T(y);
\end{equation}
\df{strictly monotone} on $K$ if{}f
\begin{equation}
x\neq y\;\;\limp\;\;\duality{x-y,v-w}_{X\times X^\star}>0\;\forall x,y\in K\;\forall v\in T(x)\;\forall w\in T(y);
\end{equation}
\df{strongly monotone} on $K$ if{}f \cite[p. 12]{Zarantonello:1960} $\exists\lambda>0$
\begin{equation}
\duality{x-y,v-w}_{X\times X^\star}\geq\lambda\n{x-y}^2_X\;\forall x,y\in K\;\forall v\in T(x)\;\forall w\in T(y);
\end{equation}
\df{$f$-uniformly monotone} on $K$ if{}f \cite[p. 203]{Vainberg:1961} there exists a strictly increasing $f:\RR^+\ra\RR^+$ with $f(0)=0$ such that
\begin{equation}
\duality{x-y,v-w}_{X\times X^\star}\geq\n{x-y}_Xf(\n{x-y}_X)\;\forall x,y\in K\;\forall v\in T(x)\;\forall w\in T(y);
\label{eqn.f.uniformly.monotone}
\end{equation}
\df{maximally monotone} on $K$ if{}f it is monotone on $K$ and its graph is not contained in the graph of any other map from $X$ to $2^{X^\star}$ that is monotone on $K$. For any Banach space $(X,\n{\cdot}_X)$, if $f\in\pcl(X,\n{\cdot}_X)$, then $\partial f$ is maximally monotone \cite[Thm. 4]{Rockafellar:1969} \cite[Thm. A]{Rockafellar:1970:maximal}.
\subsubsection{Remarks and examples}
\begin{remark}\label{remark.EL.Bregman.definitions}
\begin{enumerate}[nosep,label=(\roman*)]
\item\label{remark.EL.Bregman.definitions.i} Transformation $\dd(z(x,y)-px-qy)=-x\dd p-y\dd q$, with $p=\frac{\partial z(x,y)}{\partial x}$ and $q=\frac{\partial z(x,y)}{\partial y}$, was introduced by Euler \cite[Probl. 11 (Part I)]{Euler:1770} and Legendre \cite[p. 347]{Legendre:1787}. Since Legendre's work appeared 17 years later then Euler's, it seems to be quite adequate to include Euler in the terminology. Under generalisation to $\RR^n$ with $n\in\NN$, the \df{Euler--Legendre transformation} $(\Psi,\theta)\mapsto(\Psi^\legendre,\eta)$ of a strictly convex and differentiable function $\Psi:\RR^n\ra\RR$ is defined by 
\begin{equation}
\left\{
\begin{array}{rl}
\eta:=&\!\!\!\!\grad\Psi(\theta)\;\forall\theta\in\RR^n\\
\Psi^{\legendre}(\eta):=&\!\!\!\!\duality{\eta,\theta}_{\RR^n\times\RR^n}-\Psi(\theta)=\duality{\eta,(\grad\Psi)^{\inver}(\eta)}_{\RR^n\times\RR^n}-\Psi((\grad\Psi)^{\inver})(\eta))\;\forall\eta\in\RR^n,
\end{array}
\right.
\end{equation}
with its inverse given by 
\begin{equation}
\left\{
\begin{array}{rl}
\theta=&\!\!\!\!\grad(\Psi^\legendre)(\eta)\\
\Psi=&\!\!\!\!\Psi^\legendre{}^\legendre.
\end{array}
\right.
\end{equation}
If one restricts the domain of $\Psi$ to an open set $U\subseteq\RR^n$ and allows $\Psi$ to take infinite values, there appears a question about the optimal conditions to be imposed on $\Psi$ to guarantee  uniqueness of transformation in both directions, while preserving structural symmetry between $\Psi$ and $\Psi^\legendre$. Following a remark in \cite[p. 77]{Fenchel:1949}, Rockafellar \cite[Thm. C-K]{Rockafellar:1963} \cite[Thm. 1]{Rockafellar:1967} showed that if $\varnothing\neq U\subseteq\RR^n$ is open and convex, while $\Psi:U\ra\,]-\infty,\infty]$ is strictly convex, differentiable on $U$, and satisfies
\begin{equation}
\lim_{t\ra^+0}\textstyle\frac{\dd}{\dd t}\Psi(tx+(1-t)y)=-\infty\;\;\forall(x,y)\in U\times(\cl(U)\setminus U),
\end{equation}
then $\grad\Psi$ is a bijection on $U$, $\grad(\Psi^\lfdual)=(\grad\Psi)^{\inver}$ on $\grad\Psi(U)$, and $\Psi^\lfdual$ satisfies on $\grad(U)$ the same conditions as $\Psi$ on $U$. The notion of (Euler--)Legendre function introduced in \cite[Def. 5.2.(iii)]{Bauschke:Borwein:Combettes:2001} reduces to Rockafellar's for $X=\RR^n$ \cite[Thm. 5.11.(iii)]{Bauschke:Borwein:Combettes:2001}. See \cite[\S3]{Borwein:Vanderwerff:2001} for an extension of essentially Gateaux differentiable and Euler--Legendre functions to nonreflexive Banach spaces, as well as to other types of differentiability. On the other hand, strengthening to essential Fr\'{e}chet differentiability and, correspondingly, to Fr\'{e}chet--(Euler--)Legendre functions was developed in \cite{Borwein:Vanderwerff:2010:Frechet,Stroemberg:2011,Volle:Zalinescu:2013}.
\item\label{remark.EL.Bregman.definitions.ii} For $y\in\intefd{\Psi}$, the function \eqref{eqn.vainberg.bregman} was introduced by Va\u{\i}nberg in \cite[Eqn. (8.5)]{Vainberg:1956}{}\footnote{Numbered as Eqn. (8.4) in the English translation of this book.}, already at the Banach space level of generality (and, as such, it was further discussed in a series of works by Va\u{\i}nberg \cite[Lem. 1]{Vainberg:1965} \cite[Lem. 6.1]{Vainberg:1970} \cite[Eqn (0.1)]{Vainberg:1972} and Kachurovski\u{\i} \cite[Thm. 3]{Kachurovskii:1966}). For $X=\RR$, \eqref{eqn.vainberg.bregman} appeared independently in \cite[Eqn. (4.1)]{Brunk:Ewing:Utz:1957}, in the context of the problem of minimisation of
\begin{equation}
D^\mu_\Psi(x,y):=\int_\X\mu(\xx)D_\Psi(x(\xx),y(\xx)),
\end{equation}
for $x,y:\X\ra\RR$ with $\X\subseteq\RR^n$ and $n\in\NN$, over a measure space $(\RR^n,\mho_{\mathrm{Borel}}(\RR^n),\mu)$, considered in \cite[Thm. 4.2]{Brunk:Ewing:Utz:1957}. Independently, \eqref{eqn.vainberg.bregman} appeared in \cite[p. 1021]{Bregman:1966} and \cite[Eqn. (1.4)]{Bregman:1967} (=\cite[Eqn. (2.1)]{Bregman:1966:PhD}) for $X=\RR^n$ (with a convex set $\varnothing\neq C\subseteq\RR^n$, and with $\Psi:C\ra\RR$ differentiable on $C$ and strictly convex), where it was used in the context of minimisation of $D_\Psi$. Attribution of the name `Br\`{e}gman' to $D_\Psi$ and $D_\Psi^+$  goes back to \cite[\S2]{Censor:Lent:1981}. Correspondingly, it is fair to call $D^\mu_\Psi$ the \df{Brunk--Ewing--Utz functional}. It has been further investigated in \cite{Jones:Trutzer:1989,Jones:Byrne:1990}, and, under generalisation to any countably finite measure space, in \cite{Csiszar:1995,Csiszar:Matus:2008,Csiszar:Matus:2009,Csiszar:Matus:2012:kybernetika,Csiszar:Matus:2016}.
\item\label{remark.EL.Bregman.definitions.iii} The function \eqref{eqn.vainberg.bregman}, and its related properties, can be further generalised to $D^\partial_\Psi$ by replacing $\DG\Psi$ by a function $\DDD^\partial\Psi$, defined as a selection from the set $\partial\Psi$ ranging over $\efd(\Psi)$ (instead of assuming that this set is globally a singleton). See \cite{Kiwiel:1997:free,Kiwiel:1997:proximal,Kazimierski:2010,Burger:2016,Sprung:2019} for further discussion of this direction.
\item\label{remark.EL.Bregman.definitions.iv} If $k\in\NN$, $\Psi_i:\RR^{n_i}\ra\,]-\infty,\infty]$ is Euler--Legendre (resp., totally convex with $\bigcap_{i=1}^k\efd(\Psi_i)\neq\varnothing$) $\forall n_i\in\NN$ $\forall i\in\{1,\ldots,k\}$ and $\lambda_i>0$ $\forall i\in\{1,\ldots,k\}$, then $\Psi:\prod_{i=1}^k\RR^{n_i}\ra\,]-\infty,\infty]$, defined by $\Psi=\sum_{i=1}^k\lambda_i\Psi_i$, is Euler--Legendre \cite[Cor. 5.13]{Bauschke:Borwein:1997} (resp., totally convex \cite[Prop. 1.2.7]{Butnariu:Iusem:2000}). In such case $D_\Psi=\sum_{i=1}^k\lambda_iD_{\Psi_i}$ \cite[Lem. 3.1]{Censor:Elfving:1994}. Examples \ref{ex.EL.totalconvex.Rn.Psi}.\ref{ex.EL.totalconvex.Rn.Psi.i}--\ref{ex.EL.totalconvex.Rn.Psi.v} provide special cases of both these theorems in action, while $\Psi$ in Example \ref{ex.EL.totalconvex.Rn.Psi}.\ref{ex.EL.totalconvex.Rn.Psi.vi} is both Euler--Legendre and totally convex, but not decomposable into a weighted sum of $\Psi_i$.
\end{enumerate}
\end{remark}

\begin{example}\label{ex.EL.totalconvex.Rn.Psi}
Let $n\in\NN$.
\begin{enumerate}[nosep,label=(\roman*)]
\item\label{ex.EL.totalconvex.Rn.Psi.i} Let $\Psi(x)=\sum_{i=1}^n\gamma\ab{x_i}^{1/\gamma}=:\gamma\n{x}_{1/\gamma}^{1/\gamma}$ on $X=\efd(\Psi)=\RR^n$, $\gamma\in\,]0,1[$. This implies $\Psi^\lfdual(y)=(1-\gamma)\n{y}_{1/(1-\gamma)}^{1/(1-\gamma)}$ \cite[p. 106]{Rockafellar:1970} \cite[Ex. 2]{Eckstein:1993}. $\Psi$ is Euler--Legendre \cite[Ex. 6, Cor. 5.13]{Bauschke:Borwein:1997}. From $\grad\Psi(x)=\sum_{i=1}^n\sgn(x_i)\ab{x_i}^{1/\gamma-1}$ outside of the points where $x_j=0$ for some $j\in\{1,\ldots,n\}$, it follows that $\forall x\in\RR^n$ $\forall y\in\RR^n\setminus\{(y_1,\ldots,y_n)\in\RR^n\mid\exists i\in\{1,\ldots,n\}\;y_i=0\}$
\begin{equation}
D_\Psi(x,y)=\sum_{i=1}^n\left(\gamma\ab{x_i}^{1/\gamma}-\gamma\ab{y_i}^{1/\gamma}-(x_i-y_i)\ab{y_i}^{1/\gamma-1}\sgn(y_i)\right).
\end{equation}
 $D_\Psi$ is jointly convex only for $\gamma\in\,]0,\frac{1}{2}]$ \cite[Ex. 4.2]{Bauschke:Borwein:2001}.
\item\label{ex.EL.totalconvex.Rn.Psi.ii} Let $\Psi(x)=\sum_{i=1}^n(x_i\log(x_i)-x_i)$ if $x\geq0$ and $\Psi(x)=\infty$ otherwise (with $0\log0\equiv0$). This implies $\intefd{\Psi}=(\RR^n)^+_0:=\{x\in\RR^n\mid x_i>0\;\forall i\in\{1,\ldots,n\}\}$, $\grad\Psi(x)=\log(x)$, and $\Psi^\lfdual(y)=\exp(y)$ (cf., e.g., \cite[p. 105]{Rockafellar:1970}). $\Psi$ is Euler--Legendre \cite[Ex. 6.5, Cor. 5.13]{Bauschke:Borwein:1997}. $D_\Psi$ is equal to \cite[p. 1021]{Bregman:1966} (=\cite[p. 15]{Bregman:1966:PhD}) the finite dimensional denormalised \df{Kullback--Leibler information} \cite[Eqn. (2.4)]{Kullback:Leibler:1951},
\begin{equation}
D_{\Psi}(x,y)=\left\{\begin{array}{ll}\sum_{i=1}^n(y_i-x_i+x_i(\log(x_i)-\log(y_i)))
&\st(x,y)\in(\RR^n)^+\times(\RR^n)^+_0\\
\infty&\st\,\mbox{otherwise.}\end{array}\right.
\end{equation}
$D_\Psi$ is jointly convex (cf., e.g., \cite[p. 34]{Bauschke:Borwein:2001}) and satisfies $D_\Psi(\lambda x,\lambda y)=\lambda D_\Psi(x,y)$ $\forall\lambda>0$. $\Psi$ is totally convex on $\intefd{\Psi}$ \cite[Prop. 2.5.(i)]{Butnariu:Censor:Reich:1997} \cite[Prop. 9]{Borwein:Reich:Sabach:2011} and supercoercive. For any nonempty, closed, convex $K\subseteq\intefd{\Psi}$, $\LPPP^{D_\Psi}_K$ is continuous \cite[Lem. 3.1]{Butnariu:Censor:Reich:1997}.\footnote{For the notions of left $D_\Psi$-projection $\LPPP^{D_\Psi}_K$ and right $D_\Psi$-Chebysh\"{e}v set, see Definition \ref{def.d.projections}.\ref{def.d.projections.i}--\ref{def.d.projections.ii}.}   Convex closed subsets $C\subseteq\RR^n$ with $C\cap\intefd{\Psi}$ are right $D_\Psi$-Chebysh\"{e}v \cite[Ex. 2.16.(ii)]{Bauschke:Noll:2002}. For $n=2$, $C=\{(\ee^\lambda,\ee^{2\lambda})\mid\lambda\in[0,1]\}$ is a nonconvex right $D_\Psi$-Chebysh\"{e}v set with convex $(\grad\Psi)(C)$ \cite[Ex. 7.5]{Bauschke:Wang:Ye:Yuan:2009}. 
\item\label{ex.EL.totalconvex.Rn.Psi.iii} Let $\Psi(x)=-\sum_{i=1}^n\log(x_i)$ on $\efd(\Psi)=(\RR^n)^+_0$ and $\Psi(x)=\infty$ otherwise.\footnote{The works of Burg, often referenced in this context, consider only the continuous analogue of this $\Psi$, given by $x\mapsto-\int_{t_1}^{t_2}\dd t\log(x(t))$, $t_1,t_2\in\RR$ \cite[Slide 6]{Burg:1967} \cite[p. 1]{Burg:1975}.} This gives $\grad\Psi(x)=-\frac{1}{x}$ and $\Psi^\lfdual(y)=-\sum_{i=1}^n\log(-y_i)-n$ on $\efd(\Psi^\lfdual)=\,]-\infty,0[^n$, and thus \cite[Eqn. (57)]{Censor:Lent:1987} $D_\Psi$ is equal to the \df{Pinsker information} \cite[Eqn. (4)]{Pinsker:1960:DAN} \cite[Eqn. (10.5.4)]{Pinsker:1960}\footnote{As opposed to Pinsker's \cite[Eqn. (4)]{Pinsker:1960:DAN} and \cite[Eqn. (10.5.4)]{Pinsker:1960}, which are featuring \eqref{eqn.pinsker.relative.entropy} explicitly, the paper by Itakura and Saito, published 8 years later and often referenced for introduction of this $D_\Psi$, contains only a formula $2\log(2\pi)+\frac{1}{2\pi}\int_{-\pi}^{\pi}\dd t(\log(y(t))+\frac{x(t)}{y(t)})$ \cite[Eqn. (7)]{Itakura:Saito:1968}.},
\begin{equation}
D_\Psi(x,y)=\sum_{i=1}^n\left(-\log\frac{x_i}{y_i}+\frac{x_i}{y_i}-1\right)\;\;\forall(x,y)\in(\RR^n)^+_0\times(\RR^n)^+_0.
\label{eqn.pinsker.relative.entropy}
\end{equation}
$\Psi$ is Euler--Legendre \cite[Ex. 6.7, Cor. 5.13]{Bauschke:Borwein:1997} \cite[\S8.1]{Reem:Reich:DePierro:2019}. $D_\Psi$ is not jointly convex \cite[Thm. 3.11.(i), Ex. 3.14]{Bauschke:Borwein:2001}. It satisfies $D_\Psi(\lambda x,\lambda y)=D_\Psi(x,y)$ $\forall\lambda>0$.
\item\label{ex.EL.totalconvex.Rn.Psi.iv} Let $\Psi(x)=\sum_{i=1}^n(x_i\log(x_i)+(1-x_i)\log(1-x_i))$ on $\efd(\Psi)=[0,1]^n$ and $\Psi(x)=\infty$ otherwise \cite[Eqn. (60)]{Kapur:1972}. $\Psi^\lfdual(y)=\log(1+\exp(y))$ on $\efd(\Psi^\lfdual)=\RR^n$, and $(\grad\Psi(y))_i=\log(\frac{y_i}{1-y_i})$. $\Psi$ is Euler--Legendre \cite[Ex. 6.6, Cor. 5.13]{Bauschke:Borwein:1997} and totally convex \cite[Prop. 11]{Borwein:Reich:Sabach:2011}. The resulting Va\u{\i}nberg--Br\`{e}gman functional reads \cite[p. 142]{Kivinen:Warmuth:1999:boosting}
\begin{equation}
D_\Psi(x,y)=\sum_{i=1}^n\left(x_i\log\left(\frac{x_i}{y_i}\right)+(1-x_i)\log\left(\frac{1-x_i}{1-y_i}\right)\right),
\end{equation}
and it is jointly convex \cite[Ex. 3.5]{Bauschke:Borwein:2001}. Convex closed subsets $C\subseteq X$ with $C\cap\intefd{\Psi}$ are right $D_\Psi$-Chebysh\"{e}v \cite[Ex. 2.16.(iii)]{Bauschke:Noll:2002}. 
\item\label{ex.EL.totalconvex.Rn.Psi.v} Let 
\begin{equation}
\Psi(x)=\Psi_\alpha(x):=\left\{\begin{array}{ll}
\frac{1}{\alpha-1}\sum_{i=1}^n(x_i^\alpha-1)&\st x\in[0,\infty[^n,\;\alpha\in\,]0,1[\\
\frac{1}{1-\alpha}\sum_{i=1}^n(x_i^\alpha-1)&\st x\in\,]0,\infty[^n,\;\alpha\in\,]-\infty,0[\\
\infty&\st\mbox{otherwise}
\end{array}
\right. \textup{\cite[Eqn. (37)]{Reem:Reich:DePierro:2019}}.
\label{eqn.Psi.alpha}
\end{equation}

$\Psi_\alpha$ is Euler--Legendre \cite[\S7.2]{Reem:Reich:DePierro:2019}, and gives \cite[Eqn. (38)]{Reem:Reich:DePierro:2019}
\begin{equation}
\hspace{-0.3cm}D_{\Psi_\alpha}(x,y)=\left\{
\begin{array}{ll}
\frac{1}{1-\alpha}\sum_{i=1}^n(-x_i^\alpha+(1-\alpha)y_i^\alpha+\alpha y_i^{\alpha-1}x_i)
&\st(x,y)\in(\RR^n)^+\times(\RR^n)_0^+,\;\alpha\in\,]0,1[\\
\frac{1}{\alpha-1}\sum_{i=1}^n(-x_i^\alpha+(1-\alpha)y_i^\alpha+\alpha y_i^{\alpha-1}x_i)
&\st(x,y)\in(\RR^n)_0^+\times(\RR^n)_0^+,\;\alpha\in\,]-\infty,0[\\
\infty&\st\mbox{otherwise}.
\end{array}
\right.
\label{eqn.D.Psi.alpha}
\end{equation}
$-\frac{2^{\alpha-1}}{2^{\alpha-1}-1}(\alpha-1)\Psi_\alpha$ with $\alpha>0$ (resp., $-\Psi_\alpha$ with $\alpha\in\RR$) was introduced in \cite[Thm. 1]{Havrda:Charvat:1967} (resp., \cite[Eqn. (1)]{Tsallis:1988}). Denoting by $D_{\hat{\Psi}_\alpha}$ the first case in \eqref{eqn.D.Psi.alpha}, $\frac{\alpha-1}{\alpha}D_{\hat{\Psi}_\alpha}$ with $\alpha>1$ (resp., $(1-\alpha)D_{\hat{\Psi}_\alpha}$ for $\alpha\in\,]-\infty,0[\,\cup\,]0,1[$; $D_{\hat{\Psi}_\alpha}$ for $\alpha\in\,]0,1[$) was introduced in \cite[Eqn. (7)]{Jones:Trutzer:1989} (resp., \cite[Thm. 4]{Csiszar:1991} and \cite[Eqns. (1.7), (1.8), (1.12)]{Csiszar:1995}; \cite[Ex. 3.1.3]{Teboulle:1992}). Since $D_\Psi$ is invariant under addition of affine function to $\Psi$, and scales linearly under positive linear scaling of $\Psi$, there are several closely related functions $\Psi$, giving rise to the same $D_{\Psi_\alpha}$ (or $D_{\hat{\Psi}_\alpha}$), up to a positive scaling. For example, $\Psi(x)=\frac{1}{\alpha(\alpha-1)}\sum_{i=1}^n(x_i^\alpha-\alpha x_i+\alpha-1)$ \cite[Eqn. (2.1)]{Liese:Vajda:1987} or $\Psi(x)=\frac{1}{\alpha-1}\sum_{i=1}^n(x_i^\alpha-\alpha x_i)$ \cite[Ex. 3.1.3]{Teboulle:1992}. In particular, $\Psi(x)=\frac{1}{\alpha(1-\alpha)}\sum_{i=1}^nx^\alpha$ is Euler--Legendre for $(x,\alpha)\in(\RR^n\times\,]1,\infty[)\cup((\RR^n)^+\times\,]0,1[)\cup((\RR^n)_0^+\times\,]-\infty,0])\cup(-(\RR^n)^+\times\,]0,1[)\cup(-(\RR^n)_0^+\times\,]-\infty,0[)$ \cite[Thm. 5]{Woo:2017}.
\item\label{ex.EL.totalconvex.Rn.Psi.vi} Let $\Psi(x)=\frac{1}{2}\n{x}^2_{1/\gamma}=\frac{1}{2}\left(\sum_{i=1}^n\ab{x_i}^{1/\gamma}\right)^{2\gamma}$ on $\efd(\Psi)=\RR^n$ for $\gamma\in[\frac{1}{2},1[$ \cite[Eqn. (III.2.2)]{Nemirovskii:Yudin:1979} \cite[\S8]{BenTal:Margalit:Nemirovskii:2001}. This gives $(\grad\Psi(y))_i=\sign(y_i)\ab{y_i}^{1/\gamma-1}\n{y}_{1/\gamma}^{2-1/\gamma}$ outside the points where $y_i=0$, and thus \cite[p. 1784]{Kivinen:Warmuth:Hassibi:2006} $\forall(x,y)\in\RR^n\times(\RR^n\setminus\{(y_1,\ldots,y_n)\in\RR^n\mid\exists i\in\{1,\ldots,n\}\;y_i=0\})$
\begin{equation}
D_\Psi(x,y)=\textstyle\frac{1}{2}\n{x}^{2}_{1/\gamma}-\textstyle\frac{1}{2}\n{y}^2_{1/\gamma}-\n{y}^{2-1/\gamma}_\gamma\sum_{i=1}^n(x_i-y_i)\sign(y_i)\ab{y_i}^{1/\gamma-1}.
\label{eqn.Kivinen.Warmuth.Hassibi.D.Psi}
\end{equation}
$\Psi$ is strictly convex \cite[Lem. 8.1]{BenTal:Margalit:Nemirovskii:2001} (and this implies total convexity), as well as supercoercive and Euler--Legendre \cite[\S11.3]{Reem:Reich:DePierro:2019}. 
\end{enumerate}
\end{example}

\begin{example}\label{ex.spectral.convex}
Given a Hilbert space $\H$ with $\dim\H\in\NN$, $\K:=(\schatten_2(\H))^\sa:=\{x\in\BH\mid\sqrt{\tr_\H(x^*x)}\leq\infty,\;x=x^*\}$ with $\dim\K=:n$, equipped with an inner product $\s{x,y}_\K:=\tr_\H(xy)$ $\forall x,y\in\K$, becomes a real Hilbert space of $n\times n$ self-adjoint matrices. Let $S_n$ denote the group of all $n\times n$ permutation matrices $\RR^n\ra\RR^n$ (representing the group of all bijections of the set $\{1,\ldots,n\}$ into itself), called a \df{symmetric group}, and let $U_n$ denote the group of all $n\times n$ unitary matrices $\K\ra\K$, called a \df{unitary group}. For $x\in\K$, let $\lewis(x):=(\lewis_1(x),\ldots,\lewis_n(x))\in\RR^n$ denote the vector of eigenvalues of $x$ ordered nonincreasingly. For any $n\times n$ matrix $x$ (resp., for $x\in\RR^n$), let $\diag(x)$ denote the diagonal matrix with elements given by a diagonal of $x$ (resp., by elements of $x$). Given $C\subseteq\RR^n$, $\lewis^{\inver}(C):=\{x\in\K\mid\lewis(x)\in C\}$ is called a \df{spectral set}. If $s(C)=C$ $\forall s\in S_n$, then $\lewis^{\inver}(C)$ is closed (resp., convex) if{}f $C$ is closed (resp., convex) \cite[Thm. 8.4]{Lewis:1996:group}. A function $f:\RR^n\ra\,]-\infty,\infty]$ will be called \df{symmetric}, and $f\circ\lewis:\K\ra\,]-\infty,\infty]$ will be called \df{spectral}, if{}f $f(s(x))=f(x)$ $\forall s\in S_n$. For any symmetric $f$, $f\circ\lewis$ is convex if{}f $f$ is convex \cite[Thm. (p. 276)]{Davis:1957}. Furthermore, for any symmetric $f$ \cite[Thm. 2.3, Cors. 2.4, 3.2, 3.3]{Lewis:1996:hermitian} \cite[Thms. 8.1, 5.4]{Lewis:1996:group} \cite[Fact 7.14, Prop. 7.19.(ii)]{Bauschke:Borwein:1997}: 
\begin{enumerate}[nosep,label=\arabic*)]
\item\label{ex.spectral.convex.1} $f^\lfdual$ is symmetric; 
\item\label{ex.spectral.convex.2} $f\in\pcl(\RR^n)$ if{}f $f\circ\lewis\in\pcl(\K)$; $(f\circ\lewis)(u^*xu)=(f\circ\lewis)(x)$ $\forall x\in\H$ $\forall u\in U_n$; 
\item\label{ex.spectral.convex.3} $(f\circ\lewis)^\lfdual=f^\lfdual\circ\lewis$; 
\item\label{ex.spectral.convex.4} $\efd(f\circ\lewis)=\lewis^{\inver}(\efd(f))$; 
\item\label{ex.spectral.convex.5} $\intefd{f\circ\lewis}=\lewis^{\inver}(\intefd{f})$;
\item\label{ex.spectral.convex.6} if $f\in\pcl(\RR^n)$, then $f\circ\lewis$ is differentiable at $x$ if{}f $f$ is differentiable at $\lewis(x)$; 
\item\label{ex.spectral.convex.7} $f$ is essentially strictly convex (resp., essentially Gateaux differentiable; Euler--Legendre) if{}f $f\circ\lewis$  is essentially strictly convex (resp., essentially Gateaux differentiable; Euler--Legendre); 
\item\label{ex.spectral.convex.8} \sloppy if $f\in\pcl(\RR^n)$ and $f\circ\lambda$ is differentiable at $y\in\intefd{f\circ\lewis}$ then $\grad(f\circ\lewis)(y)=v(\diag(\grad f(\lewis(y))))v^*$ $\forall v\in U_n$ such that $v^*yv=\diag(\lewis(y))$, and $\grad(f\circ\lewis)(u^*yu)=u^*\grad(f\circ\lewis)(y)u$ $\forall u\in U_n$.
\end{enumerate}
In consequence, $D_\Psi(x,y)=D_\Psi(u^*xu,u^*yu)$ $\forall x\in\K$ $\forall y\in\intefd{\Psi}$ $\forall u\in U_n$ for all spectral Euler--Legendre $\Psi$ \cite[Cor. 7.21]{Bauschke:Borwein:1997}.

The examples of symmetric Euler--Legendre functions are:
\begin{enumerate}[nosep,label=(\roman*)]
\item\label{ex.spectral.convex.i} $f(x)=\sum_{i=1}^n\gamma\ab{x_i}^{1/\gamma}=:\gamma\n{x}^{1/\gamma}_{1/\gamma}$ on $\efd(f)=\intefd{f}=\RR^n$ for $\gamma\in\,]0,1[$, which gives \cite[p. 171]{Lewis:1996:hermitian} spectral Euler--Legendre $(f\circ\lewis)(\xi)=\gamma\n{\lewis(\xi)}^{1/\gamma}_{1/\gamma}=\sum_{i=1}^n\gamma\ab{\lewis_i(\xi)}^{1/\gamma}=\gamma\tr_\H(\ab{\xi}^{1/\gamma})=\gamma\n{\xi}_{1/\gamma}^{1/\gamma}$ on $\efd(f\circ\lewis)=\schatten_2(\H)$. Under a restriction of a domain of $\zeta$ to $(\schatten_2(\H))_0^+=\{x\in\K\mid x\mbox{ is positive definite}\}$, the corresponding Va\u{\i}nberg--Br\`{e}gman functional reads
\begin{equation}
D_{f\circ\lewis}(\xi,\zeta)=\tr_\H(\gamma\ab{\xi}^{1/\gamma}-\gamma\zeta^{1/\gamma}-(\xi-\zeta)\zeta^{1/\gamma-1})\;\forall(\xi,\zeta)\in(\schatten_2(\H))^\sa\times(\schatten_2(\H))_0^+;
\end{equation}
\item\label{ex.spectral.convex.ii} $f(x)=\sum_{i=1}^n(x_i\log(x_i)-x_i)$ if $x\geq0$ and $f(x)=\infty$ otherwise, which gives \cite[Ex. 7.29]{Bauschke:Borwein:1997} spectral Euler--Legendre $(f\circ\lewis)(\xi)=\tr_\H(\xi\log\xi-\xi)$ with $\efd(f\circ\lewis)=(\schatten_2(\H))^+=\{x\in\K\mid x\mbox{ is positive semi-definite}\}$, $\intefd{f\circ\lewis}=(\schatten_2(\H))^+_0$, and $\grad(f\circ\lewis)(\xi)=\log(\xi)$. The corresponding Va\u{\i}nberg--Br\`{e}gman functional reads  \cite[Def. 1]{Umegaki:1961}
\begin{equation}
D_{f\circ\lewis}(\xi,\zeta)=\tr_\H(\xi(\log\xi-\log\zeta)-\xi-\zeta)\;\;\forall(\xi,\zeta)\in(\schatten_2(\H))^+\times(\schatten_2(\H))^+_0;
\end{equation}
\item\label{ex.spectral.convex.iii} $f(x)=-\sum_{i=1}^n\log(x_i)$ on $\efd(f)=(\RR^n)^+_0$ and $f(x)=\infty$ otherwise, which gives \cite[pp. 170--171]{Lewis:1996:hermitian} spectral Euler--Legendre $(f\circ\lewis)(\xi)=-\log\det(\xi)$ $\forall\xi\in\efd(f\circ\lewis)=(\schatten_2(\H))^+_0$ and $(f\circ\lewis)(\xi)=\infty$ otherwise (cf., e.g., \cite[Thm. 3.2.(iv)]{Nesterov:Nemirovskii:1989}). It satisfies $\grad(f\circ\lewis)(\xi)=-\xi^{-1}$. Its Mandelbrojt--Fenchel dual is $(f\circ\lewis)^\lfdual(\xi)=-n-\log\det(-\xi)$ on $\efd((f\circ\lewis)^\lfdual)=-(\schatten_2(\H))^+_0$ \cite[p. 171]{Lewis:1996:hermitian}, and satisfies $\grad((f\circ\lewis)^\lfdual)(\xi)=-\xi^{-1}$. The corresponding Va\u{\i}nberg--Br\`{e}gman functional reads \cite[\S5]{James:Stein:1961}
\begin{equation}
D_{f\circ\lewis}(\xi,\zeta)=\s{\xi,\zeta^{-1}}_\K-\log\det(\xi\zeta^{-1})-n=h(\zeta^{-1/2}\xi\zeta^{-1/2})-n,
\end{equation}
for $h(\xi):=\tr_\K(\xi)-\log\det(\xi)$.
\end{enumerate}
\end{example}
\subsection{Va\u{\i}nberg--Br\`{e}gman projections and quasinonexpansive maps}\label{section.background.bregman.projections}
\subsubsection{Projections}
\begin{definition}\label{def.information}
For any set $Z$, $D:Z\times Z\ra[0,\infty]$ will be called an \df{information} on $Z$ (and $-D$ will be called a \df{relative entropy} on $Z$) if{}f $D(x,y)=0\iff x=y$ $\forall x,y\in Z$.
\end{definition}

\begin{proposition}\label{prop.information}
Let $\Psi\in\pcl(X,\n{\cdot}_X)$. Then:
\begin{enumerate}[nosep,label=(\roman*)]
\item\label{prop.information.i}if $\Psi\in\pclg(X,\n{\cdot}_X)$, then $D_\Psi$ is an information on $X$ if{}f $\Psi$ is strictly convex on $\intefd{\Psi}$ \textup{\cite[Prop. 1.1.9]{Butnariu:Iusem:2000}};
\item\label{prop.information.ii} if $(X,\n{\cdot}_X)$ is reflexive and $\Psi$ is essentially strictly convex, then $D^+_\Psi$ is an information on $X$ \textup{\cite[Lem. 7.3.(vi)]{Bauschke:Borwein:Combettes:2001}}.
\end{enumerate}
\end{proposition}

\begin{definition}\label{def.d.projections}
Let $\Psi\in\pclg(X,\n{\cdot}_X)$, $y\in\intefd{\Psi}$, and $K\subseteq X$ with $\varnothing\neq K\cap\intefd{\Psi}$.
\begin{enumerate}[nosep,label=(\roman*)]
\item\label{def.d.projections.i} If the set $\arginff{x\in K}{D_\Psi(x,y)}$ is a singleton, then its element will be denoted $\LPPP^{D_\Psi}_K(y)$, and called a \df{left $D_\Psi$-projection} of $y$ onto $K$ \textup{\cite[p. 1019]{Bregman:1966} \cite[\S1.II, \S2.2]{Bregman:1966:PhD} \cite[p. 620]{Bregman:1967}}, while $K$ will be called a \df{left $D_\Psi$-Chebysh\"{e}v} set \textup{\cite[Def. 3.28]{Bauschke:Borwein:Combettes:2003}}.
\item\label{def.d.projections.ii} If $K\subseteq\intefd{\Psi}$ and the set $\arginff{x\in K}{D_\Psi(y,x)}$ is a singleton, then its element will be denoted $\RPPP^{D_\Psi}_K(y)$, and called a \df{right $D_\Psi$-projection} of $y$ onto $K$ \textup{\cite[Def. 3.1, Lem. 3.5]{Bauschke:Noll:2002}}, while $K$ will be called a \df{right $D_\Psi$-Chebysh\"{e}v} set \textup{\cite[Def. 1.7]{Bauschke:Macklem:Wang:2011}}.
\item\sloppy\label{def.d.projections.iii} $\LPPP^{D_\Psi}_K$ (resp., $\RPPP^{D_\Psi}_K$) will be called \df{zone consistent} (with respect to the class of sets $K$ which are under consideration) \textup{\cite[Def. 3.1.(i)]{Censor:Lent:1981}} if{}f $\LPPP^{D_\Psi}_K(\intefd{\Psi})\subseteq\intefd{\Psi}$ (resp., $\RPPP^{D_\Psi}_K(\intefd{\Psi})\subseteq\intefd{\Psi}$) for any $K$ (in the given class).
\end{enumerate}
\end{definition}

\begin{definition}\label{def.pythagorean}
Let $\Psi\in\pclg(X,\n{\cdot}_X)$, $\varnothing\neq K\subseteq X$. $D_\Psi$ will be called:
\begin{enumerate}[nosep,label=\alph*)]
\item\label{def.pythagorean.a} \df{left pythagorean} on $K$ if{}f $K$ is left $D_\Psi$-Chebysh\"{e}v and, for any $x\in\intefd{\Psi}$ and any $w\in K$, the following conditions are equivalent:
\begin{enumerate}[nosep,label=(\roman*)]
\item\label{def.pythagorean.a.i} $w=\LPPP^{D_\Psi}_K(x)$;
\item\label{def.pythagorean.a.ii} $w$ is the unique solution of the variational inequality
\begin{equation}
\duality{z-y,\DG\Psi(x)-\DG\Psi(z)}_{X\times X^\star}\geq0\;\;\forall y\in K;
\label{eqn.left.pyth.ii}
\end{equation}
\item\label{def.pythagorean.a.iii} $w$ is the unique solution of the variational inequality
\begin{equation}
	D_\Psi(y,z)+D_\Psi(z,x)\leq D_\Psi(y,x)\;\;\forall y\in K;
\label{eqn.left.pyth.iii}
\end{equation}
\end{enumerate}
\item\label{def.pythagorean.b} \df{right pythagorean} on $K$ if{}f $K$ is right $D_\Psi$-Chebysh\"{e}v and, for any $x\in\intefd{\Psi}$ and any $w\in K$, the following conditions are equivalent:
\begin{enumerate}[nosep,label=(\roman*)]
\item\label{def.pythagorean.b.i} $w=\RPPP^{D_\Psi}_K(x)$;
\item\label{def.pythagorean.b.ii} $w$ is the unique solution of the variational inequality
\begin{equation}
\duality{x-z,\DG\Psi(x)-\DG\Psi(y)}_{X\times X^\star}\geq0\;\;\forall y\in K;
\end{equation}
\item\label{def.pythagorean.b.iii} $w$ is the unique solution of the variational inequality
\begin{equation}
D_\Psi(x,z)+D_\Psi(z,y)\leq D_\Psi(x,y)\;\;\forall y\in K.
\label{eqn.right.pyth.ineq}
\end{equation}
\end{enumerate}
\end{enumerate}
\end{definition}

\begin{proposition}\label{prop.left.pythagorean}
Let $(X,\n{\cdot}_X)$ be reflexive, $\Psi\in\pclg(X,\n{\cdot}_X)$ and $\varnothing\neq K\subseteq X$ be closed and convex. Then:
\begin{enumerate}[nosep,label=(\roman*)]
\item\label{prop.left.pythagorean.i} $K$ is left $D_\Psi$-Chebysh\"{e}v, and $D_\Psi$ is left pythagorean on $K$ if any of (generally, inequivalent) conditions holds:
\begin{enumerate}[nosep,label=\alph*)]
\item\label{prop.left.pythagorean.i.a} $\Psi$ is totally convex on $\efd(\Psi)$, $K\subseteq\intefd{\Psi}$ \textup{\cite[Prop. 2.1.5]{Butnariu:Iusem:2000}+\cite[Prop. 4.1.(i)]{Butnariu:Resmerita:2006}(=\cite[Cor. 4.4]{Butnariu:Resmerita:2006})}; or
\item\label{prop.left.pythagorean.i.b} $\Psi$ is strictly convex on $\efd(\Psi)$ and supercoercive, $K\cap\intefd{\Psi}\neq\varnothing$ \textup{\cite[Prop. 2.2]{Alber:Butnariu:1997}+\cite[Prop. 4.1.(ii)]{Butnariu:Resmerita:2006}}; or
\item\label{prop.left.pythagorean.i.c} $\Psi$ is Euler--Legendre, $K\cap\intefd{\Psi}\neq\varnothing$ \textup{\cite[Prop. 3.16]{Bauschke:Borwein:1997}+\cite[Thm. 7.8]{Bauschke:Borwein:Combettes:2001}(=\cite[Cor. 3.35]{Bauschke:Borwein:Combettes:2003})};
\end{enumerate}
\item\label{prop.left.pythagorean.ii} if any of the conditions \ref{prop.left.pythagorean.i}.\ref{prop.left.pythagorean.i.a}--\ref{prop.left.pythagorean.i}.\ref{prop.left.pythagorean.i.c} holds, and $K$ is affine, then 
\begin{equation}
	D_\Psi(x,\LPPP^{D_\Psi}_K(y))+D_\Psi(\LPPP^{D_\Psi}_K(y),y)=D_\Psi(x,y)\;\;\forall(x,y)\in K\times\intefd{\Psi};
\label{eqn.left.pyth}
\end{equation}
\item\label{prop.left.pythagorean.iii} if (condition \ref{prop.left.pythagorean.i}.\ref{prop.left.pythagorean.i.c} holds), or (condition \ref{prop.left.pythagorean.i}.\ref{prop.left.pythagorean.i.a} or condition \ref{prop.left.pythagorean.i}.\ref{prop.left.pythagorean.i.b} holds, and $K\subseteq\intefd{\Psi}$), then $\LPPP^{D_\Psi}_K$ is zone consistent \textup{\cite[Cor. 7.9]{Bauschke:Borwein:Combettes:2001}}.
\end{enumerate}
\end{proposition}

\begin{corollary}\label{cor.information}
If any of the conditions \ref{prop.left.pythagorean.i.a}--\ref{prop.left.pythagorean.i.c} in Proposition \ref{prop.left.pythagorean}.\ref{prop.left.pythagorean.i} holds, then $D_\Psi$ is an information on $X$.
\end{corollary}
\begin{proof}
The condition of strict convexity on $\efd(\Psi)$ in Proposition \ref{prop.left.pythagorean}.\ref{prop.left.pythagorean.i}.\ref{prop.left.pythagorean.i.b} implies Proposition \ref{prop.information}.\ref{prop.left.pythagorean.i}. The condition of \ref{prop.left.pythagorean}.\ref{prop.left.pythagorean.i}.\ref{prop.left.pythagorean.i.c} implies Proposition \ref{prop.information}.\ref{prop.left.pythagorean.ii}, which in turn implies that $D_\Psi$ is an information on $X$. Since total convexity on $\efd(\Psi)$ implies strict convexity on $\intefd{\Psi}$ \cite[Prop. 3.1.(i)]{Butnariu:Iusem:1997}, the condition of Proposition \ref{prop.left.pythagorean}.\ref{prop.left.pythagorean.i}.\ref{prop.left.pythagorean.i.a} implies Proposition \ref{prop.information}.\ref{prop.left.pythagorean.i}.
\end{proof}

\begin{proposition}\label{prop.Luo.Meng.Wen.Yao} \textup{\cite[Lem. 3.2]{Luo:Meng:Wen:Yao:2019}}
If $\Psi\in\pclg(X,\n{\cdot}_X)$, $\Psi^\lfdual$ is Gateaux differentiable on $\varnothing\neq\DG\Psi(\intefd{\Psi})\subseteq\intefd{\Psi^\lfdual}$, $K\subseteq\intefd{\Psi}$, and $\DG\Psi(K)$ is convex and closed, then
\begin{equation}
	\RPPP^{D_\Psi}_K(x)=(J_X)^{\inver}\circ\DG\Psi^\lfdual\circ\LPPP^{D_{\Psi^\lfdual}}_{\DG\Psi(K)}\circ\DG\Psi(x)\;\;\forall x\in\intefd{\Psi}.
\label{eqn.left.to.right}
\end{equation}
\end{proposition}

\begin{remark}\label{remark.projections}
\begin{enumerate}[nosep,label=(\roman*)]
\item\label{remark.projections.i} The notion of a \df{Chebysh\"{e}v set}, defined as a subset $K$ of a Banach space $(X,\n{\cdot}_X)$ such that $\arginff{x\in K}{\n{x-y}_X}=\{*\}=:\{\PPP^{d_{\n{\cdot}_X}}_K(y)\}$ $\forall y\in X$, was introduced in \cite[\S{}A2]{Klee:1953} (implicitly) and \cite[p. 17]{Efimov:Stechkin:1958} (explicitly). This name refers to Chebysh\"{e}v's paper \cite{Chebyshev:1859}, where first nontrivial examples of such sets were considered. See Remarks \ref{remark.varphi}.\ref{remark.varphi.iii}, \ref{remark.varphi}.\ref{remark.varphi.vii}, and \ref{remark.varphi}.\ref{remark.varphi.ix} for further discussion.
\item\label{remark.projections.ii} In principle, Definitions \ref{def.d.projections} and \ref{def.pythagorean} could be formulated more generally, by dropping an assumption $\Psi\in\pclg(X,\n{\cdot}_X)$ and with $D_\Psi$ (resp., $\DG\Psi$; $\intefd{\Psi}$) replaced either by $D^+_\Psi$ (resp., $\DG_+\Psi$; $\efd(\Psi)$) or by $D^\partial_\Psi$ (resp., $\DDD^\partial\Psi$; $\efd(\Psi)$). (This is the reason for putting Gateaux differentiability into brackets in the third paragraph of Section \ref{section.introduction}.) While this is a tempting possibility in the context of a general axiomatic scheme, as well as in the light of references listed in Remark \ref{remark.EL.Bregman.definitions}.\ref{remark.EL.Bregman.definitions.iii}, there are currently no substantial geometric results available for such a degree of generality.
\item\label{remark.projections.iii} Both left and right $D_\Psi$-projections are idempotent: if $K$ is left (resp., right) $D_\Psi$-Chebysh\"{e}v, then $\LPPP^{D_\Psi}_K\circ\LPPP^{D_\Psi}_K=\LPPP^{D_\Psi}_K$ (resp.,  $\RPPP^{D_\Psi}_K\circ\RPPP^{D_\Psi}_K=\RPPP^{D_\Psi}_K$). 
\item\label{remark.projections.iv} The naming convention of Definition \ref{def.information} follows Wiener's dictum that the \cytat{amount of information is the negative of the quantity usually defined as entropy} \cite[p. 76]{Wiener:1948}, and agrees with: R\'{e}nyi's \cytat{measure of the amount of information} \cite[p. 554]{Renyi:1961}, Umegaki's definition of \cytat{information} on state spaces of type I W$^*$-algebras as $D_1(\rho,\sigma)=\tr_\H(\rho(\log\rho-\log\sigma))$ \cite[Def. 1]{Umegaki:1961} \cite[Def. 1]{Umegaki:1962}, Csisz\'{a}r's definition of \cytat{relative information} $D_\fff$ in \cite[p. 86]{Csiszar:1963}, as well as with the sign, ordering, and naming conventions used throughout \cite{Bratteli:Robinson:1979}. It also avoids terminological confusions: when, e.g., both $-\tr_\H(\rho\log\rho)$ and $D_1(\rho,\sigma)$ are called an `entropy' (sign ambiguity); when both $D(\rho,\sigma)$ and $D(\rho,\sigma)+D(\sigma,\rho)$ are called a `divergence' (both \cite[p. 81]{Kullback:Leibler:1951} and \cite[Def. 2]{Umegaki:1962} explicitly distinguish between $D_1$ and \cytat{divergence}, the latter defined as a symmetrisation $D_1(\rho,\sigma)+D_1(\sigma,\rho)$ \cite[Eqn. (1)]{Jeffreys:1946})\footnote{Furthermore, the notion of \cytat{divergence} has already other meanings in differential calculus and in renormalisation.}; when any $D$ as well as only a symmetric $D$ satisfying triangle inequality are called a `distance'. Mathematically, our definition of \cytat{information} coincides with the definition of a \cytat{contrast functional} in \cite[p. 794]{Eguchi:1983} and of a \cytat{distance} in \cite[p. 161]{Csiszar:1995}, with the property (I) of $D$ in \cite[p. 1019]{Bregman:1966} (=\cite[p. 5]{Bregman:1966:PhD}), and can be seen as turning \cite[Lem. 2.1]{Censor:Lent:1981} into an axiom.
\item\label{remark.projections.v} First special case of left $D_\Psi$-projection for nonsymmetric $D_\Psi$, with $D_\Psi$ given by the Kullback--Leibler information (i.e. $D_\Psi$ on $\RR^n\times\RR^n$ with $\Psi(x_i)=\sum_{i=1}^n(x_i\log(x_i)-x_i)$ and $n\in\NN$), was independently introduced in \cite[p. 32]{Sanov:1957} and \cite[Ch. 3.2]{Kullback:1959}. First special case of right $D_\Psi$-projection for nonsymmetric $D_\Psi$, with $D_\Psi$ given by the Kullback--Leibler information, was introduced in \cite[Eqn. (16)]{Chencov:1968} (cf. also \cite[Def. 22.2]{Chencov:1972}). First instance of a right pythagorean (in)equality \eqref{eqn.right.pyth.ineq}, together with its interpretation as \cytat{nonsymmetrical analogue of the theorem of Pythagoras}, was established in \cite[Thm. 1]{Chencov:1968} for the Kullback--Leibler information. The corresponding special case of \eqref{eqn.left.pyth.iii} was first considered implicitly in \cite[p. 1021]{Bregman:1966} and explicitly in \cite[Thm. 2.2]{Csiszar:1975}.
\end{enumerate}
\end{remark}
\subsubsection{Quasinonexpansive maps}
\begin{definition}\label{def.fixed.point.set}
For any $(X,\n{\cdot}_X)$, let $\varnothing\neq K\subseteq X$, $T:K\ra X$. Then:
\begin{enumerate}[nosep,label=(\roman*)]
\item\label{def.fixed.point.set.i} $x\in K$ is called a \df{fixed point} of $T$ if{}f $T(x)=x$; a set of all fixed points of $T$ will be denoted $\Fix(T)$;
\item\label{def.fixed.point.set.ii} $x\in\cl(K)$ is called an \df{asymptotic fixed point} of $T$ if{}f there exists $\{x_n\in K\mid n\in\NN\}$ which converges weakly to $x$, and $\lim_{n\ra\infty}\n{x_n-T(x_n)}_X=0$ \textup{\cite[p. 313]{Reich:1996}}. The set of all asymptotic fixed points of $T$ will be denoted $\aFix(T)$.
\end{enumerate}
\end{definition}

\begin{definition}\label{def.nonexpansive.maps}
Let $\varnothing\neq K\subseteq\intefd{\Psi}$, $\Psi\in\pclg(X,\n{\cdot}_X)$, $T:K\ra\intefd{\Psi}$ will be called:
\begin{enumerate}[nosep,label=(\roman*)]
\item\label{def.nonexpansive.maps.i} \df{completely nonexpansive} with respect to $\Psi$ and $K$ if{}f \textup{\cite[2.1.7]{Butnariu:Iusem:2000}}
\begin{equation}
D_\Psi(T(x),T(y))\leq D_\Psi(x,y)\;\;\forall x,y\in K;
\end{equation}
\item\label{def.nonexpansive.maps.ii} \df{left strongly quasinonexpansive} with respect to $\Psi$ and $K$, if{}f \textup{\cite[Def. 3.2]{Censor:Reich:1996} \cite[pp. 313--314]{Reich:1996}} $\varnothing\neq\aFix(T)\subseteq\efd(\Psi)$,
\begin{equation}
D_\Psi(y,T(x))\leq D_\Psi(y,x)\;\;\forall(y,x)\in\aFix(T)\times K,
\end{equation}
and, for any $y\in\aFix(T)$ and any bounded $\{x_n\in K\mid n\in\NN\}$, 
\begin{equation}
\lim_{n\ra\infty}(D_\Psi(y,x_n)-D_\Psi(y,T(x_n)))=0\;\;\limp\;\;
\lim_{n\ra\infty}D_\Psi(T(x_n),x_n)=0;
\end{equation}
\item\label{def.nonexpansive.maps.iii} \df{right strongly quasinonexpansive} with respect to $\Psi$ and $K$ if{}f \textup{\cite[Def. 2.3.(iv)]{MartinMarquez:Reich:Sabach:2012}} $\varnothing\neq\aFix(T)\subseteq\intefd{\Psi}$,
\begin{equation}
D_\Psi(T(x),y)\leq D_\Psi(x,y)\;\;\forall(x,y)\in K\times\aFix(T),
\end{equation}
and, for any $y\in\aFix(T)$ and any bounded $\{x_n\in K\mid n\in\NN\}$,
\begin{equation}
\lim_{n\ra\infty}(D_\Psi(x_n,y)-D_\Psi(T(x_n),y))=0
\;\;\limp\;\;
\lim_{n\ra\infty}D_\Psi(x_n,T(x_n))=0;
\end{equation}
\item\label{def.nonexpansive.maps.iv} \df{left firmly nonexpansive} with respect to $\Psi$ and $K$ if{}f \textup{\cite[Def. 3]{Brohe:Tossings:2000}} \textup{\cite[Def. 3.4, Prop. 3.5.(iv)]{Bauschke:Borwein:Combettes:2003}} $\forall x,y\in K$
\begin{equation}
D_\Psi(T(x),T(y))+D_\Psi(T(y),T(x))+D_\Psi(T(x),x)+D_\Psi(T(y),y)\leq D_\Psi(T(x),y)+D_\Psi(T(y),x);
\end{equation}
\item\label{def.nonexpansive.maps.v} \df{right firmly nonexpansive} with respect to $\Psi$ and $K$ if{}f \textup{\cite[Def. 2.3.(i$^*$)]{MartinMarquez:Reich:Sabach:2012}} $\forall x,y\in K$
\begin{equation}
D_\Psi(T(x),T(y))+D_\Psi(T(y),T(x))+D_\Psi(x,T(x))+D_\Psi(y,T(x))\leq D_\Psi(x,T(y))+D_\Psi(y,T(x)).
\end{equation}
\end{enumerate}
The set of all left (resp., right) strongly quasinonexpansive maps with respect to $\Psi$ and $K$ will be denoted $\lsq(\Psi,K)$ (resp., $\rsq(\Psi,K)$). The set of all left (resp., right) firmly nonexpansive maps with respect to $\Psi$ and $K$ will be denoted $\lfn(\Psi,K)$ (resp., $\rfn(\Psi,K)$). The set of all completely nonexpansive maps with respect to $\Psi$ and $K$ will be denoted $\cn(\Psi,K)$.
\end{definition}

\begin{definition}\label{def.LSQ.RSQ.compositional}
If $(X,\n{\cdot}_X)$ is reflexive, then $\Psi\in\pclg(X,\n{\cdot}_X)$ will be called \df{LSQ-compositional} (resp., \df{RSQ-compositional}) \df{on} $\varnothing\neq K\subseteq\intefd{\Psi}$ if{}f, for any set $\{T_i:K\ra K\mid T_i\in\lsq(\Psi,K)$ (resp., $\rsq(\Psi,K)$), $i\in\{1,\ldots,m\}, m\in\NN\}$ such that $\bigcap_{i=1}^m\aFix(T_i)\neq\varnothing$:
\begin{enumerate}[nosep,label=(\roman*)]
\item\label{def.LSQ.RSQ.compositional.i} $\aFix(T_m\circ\cdots\circ T_1)\subseteq\bigcap_{i=1}^m\aFix(T_i)$;
\item\label{def.LSQ.RSQ.compositional.ii} if $\aFix(T_m\circ\cdots\circ T_1)\neq\varnothing$, then $T_m\circ\cdots\circ T_1\in\lsq(\Psi,K)$ (resp., $\rsq(\Psi,K)$).
\end{enumerate}
\sloppy The set $\lsq(\Psi,K)$ (resp., $\rsq(\Psi,K)$) will be called \df{composable} if{}f $\Psi$ is LSQ-(resp., RSQ-)com\-po\-si\-tio\-nal on $K$. $\Psi$ will be called \df{LSQ-compositional} (resp., \df{RSQ-com\-po\-si\-tio\-nal}) if{}f it is LSQ-(resp., RSQ-)compositional on any $K\subseteq\intefd{\Psi}$.
\end{definition}

\begin{definition}\label{def.LSQ.RSQ.preadapted}
If $(X,\n{\cdot}_X)$ is reflexive, then $\Psi\in\pclg(X,\n{\cdot}_X)$ will be called:
\begin{enumerate}[nosep,label=(\roman*)]
\item\label{def.LSQ.RSQ.preadapted.i} \df{LSQ-preadapted} on a set $\varnothing\neq K\subseteq\intefd{\Psi}$ iff $T\in\lfn(\Psi,K)$ $\limp$ ($T\in\lsq(\Psi,K)$, $\aFix(T)=\Fix(T)$, and $\Fix(T)$ is convex and closed); 
\item\label{def.LSQ.RSQ.preadapted.ii} \df{RSQ-preadapted} on a set $\varnothing\neq K\subseteq\intefd{\Psi}$ iff $T\in\rfn(\Psi,K)$ $\limp$ ($T\in\rsq(\Psi,K)$, $\aFix(T)=\Fix(T)$, and $\DG\Psi(\Fix(T))$ is convex and closed).
\end{enumerate}
 \end{definition}

\begin{proposition}\label{prop.lsq.rsq.old}
If $(X,\n{\cdot}_X)$ is reflexive and $\Psi\in\pclg(X,\n{\cdot}_X)$, then:
\begin{enumerate}[nosep,label=(\roman*)]
\item\label{prop.lsq.rsq.old.i} $\Psi$ is LSQ-compositional if any of (generally, inequivalent) conditions holds:
\begin{enumerate}[nosep,label=\alph*)]
\item\label{prop.lsq.rsq.old.i.a} $\Psi:X\ra\RR$ is uniformly convex on $X$, $\DG\Psi$ is (bounded and uniformly continuous) on bounded subsets of $X$ \textup{\cite[Lems. 1, 2]{Reich:1996}}; or
\item\label{prop.lsq.rsq.old.i.b} $\Psi:X\ra\RR$ is (bounded, uniformly Fr\'{e}chet differentiable, totally convex) on bounded subsets of $X$, $\DG(\Psi^\lfdual)$ is bounded on bounded subsets of $\efd(\Psi^\lfdual)=X^\star$ \textup{\cite[Prop. 3.3]{MartinMarquez:Reich:Sabach:2013:BSN}};
\end{enumerate}
\item\label{prop.lsq.rsq.old.ii} $\Psi$ is RSQ-compositional if any of (generally, inequivalent) conditions holds:
\begin{enumerate}[nosep,label=\alph*)]
\item\label{prop.lsq.rsq.old.ii.a} $\Psi:X\ra\RR$ is (bounded, uniformly Fr\'{e}chet differentiable, totally convex) on bounded subsets of $X$ \textup{\cite[Prop. 4.4]{MartinMarquez:Reich:Sabach:2013:BSN}}; or
\item\label{prop.lsq.rsq.old.ii.b} $\efd(\Psi^\lfdual)=X^\star$, $\Psi$ is Euler--Legendre, $\Psi^\lfdual$ is totally convex on bounded subsets of $X^\star$, $\DG\Psi$ is (bounded and uniformly continuous) on bounded subsets of $\intefd{\Psi}$, $\DG\Psi^\lfdual$ is (bounded and uniformly continuous) on bounded subsets of $X^\star$ \textup{\cite[Prop. 6.6]{MartinMarquez:Reich:Sabach:2013:BSN}};
\end{enumerate}
\item\label{prop.lsq.rsq.old.iii} $\Psi:X\ra\RR$ is LSQ-preadapted on any convex closed $\varnothing\neq K\subseteq X$ if $\Psi$ is Euler--Legendre and (bounded and uniformly Fr\'{e}chet differentiable) on bounded subsets of $X$ \textup{\cite[Lems. 15.5, 15.6]{Reich:Sabach:2011}+\cite[Rem. 2.1.3]{Sabach:2012}};
\item\label{prop.lsq.rsq.old.iv} $\Psi$ is RSQ-preadapted on any $\varnothing\neq K\subseteq X$ if $\Psi:X\ra\RR$ is Euler--Legendre, and any of (generally, inequivalent) conditions holds:
\begin{enumerate}[nosep,label=\alph*)]
\item\label{prop.lsq.rsq.old.iv.a} $\Psi$ is uniformly continuous on bounded subsets of $X$, $\DG\Psi$ is weakly sequentially continuous\footnote{For any Banach space $(X,\n{\cdot}_X)$, $f:X\ra X^\star$ is called \df{weakly sequentially continuous} if{}f a weak convergence of $\{x_n\in X\mid n\in\NN\}$ to $x\in X$ implies a weak convergence of $\{f(x_n)\mid n\in\NN\}$ to $f(x)$.} \textup{\cite[Props. 3.3, 3.6]{MartinMarquez:Reich:Sabach:2012}}; or
\item\label{prop.lsq.rsq.old.iv.b} $\Psi^\lfdual$ is (uniformly Fr\'{e}chet differentiable and bounded) on bounded subsets of $\intefd{\Psi^\lfdual}\neq\varnothing$, $\DG\Psi$ is uniformly continuous on bounded subsets of $X$ \textup{\cite[Prop. 3.3, Rem. 3.7]{MartinMarquez:Reich:Sabach:2012}};
\end{enumerate}
\item\label{prop.lsq.rsq.old.v} if $\Psi$ is Euler--Legendre, $\varnothing\neq K\subseteq\intefd{\Psi}$, $T:K\ra\intefd{\Psi}$, $\DG\Psi$ and $\DG\Psi^\lfdual$ are (uniformly continuous and bounded) on bounded subsets of $\intefd{\Psi}$ and $\intefd{\Psi^\lfdual}\neq\varnothing$, respectively, then $T\in\rsq(\Psi,K)$ if{}f $\DG\Psi\circ T\circ\DG\Psi^\lfdual\in\lsq(\Psi,\DG\Psi(K))$ \textup{\cite[Fact 6.5]{MartinMarquez:Reich:Sabach:2013:BSN}}.
\end{enumerate}
\end{proposition}

\begin{definition}\label{def.resolvent.prox}
Let $(X,\n{\cdot}_X)$ be a Banach space, $\Psi\in\pclg(X,\n{\cdot}_X)$, $\lambda\in\,]0,\infty[$. 
\begin{enumerate}[nosep,label=(\roman*)]
\item\label{def.resolvent.prox.i} If $T:X\ra 2^{X^\star}$ and $\Graph(T)\neq\varnothing$, then the \df{left} (resp., \df{right}) \df{$D_\Psi$-resolvent} of $\lambda T$ is defined as \textup{\cite[Lem. 1]{Eckstein:1993} \cite[Def. 3.7]{Bauschke:Borwein:Combettes:2003}} (resp., \textup{\cite[Def. 5.3]{MartinMarquez:Reich:Sabach:2012}})
\begin{align}
\lres^\Psi_{\lambda T}&:=(\DG\Psi+\lambda T)\circ\DG\Psi:X\ra 2^X\\
\mbox{(resp., }\rres^\Psi_{\lambda T}&:=(\id_{X^\star}+\lambda T\circ\DG\Psi^\lfdual)^{\inver}:X^\star\ra 2^{X^\star}\mbox{)}.
\end{align}
\item\label{def.resolvent.prox.ii} If $f:X\ra\,]-\infty,\infty]$ is proper, then the \df{left} (resp., \df{right}) \df{$D_\Psi$-proximal map} of index $\lambda$ is defined as \textup{\cite[Eqn. (13)]{Censor:Zenios:1992} \cite[Def. 3.16]{Bauschke:Borwein:Combettes:2003}} (resp., \textup{\cite[Def. 3.7]{Bauschke:Combettes:Noll:2006} \cite[Def. 3.3]{Laude:Ochs:Cremers:2020}})
\begin{align}
\lprox^{D_\Psi}_{\lambda,f}:X\ni y&\mapsto\arginff{x\in\efd(f)\cap\efd(\Psi)}{f(x)+\lambda D_\Psi(x,y)}\in 2^X\\
\mbox{(resp., }\rprox^{D_\Psi}_{\lambda,f}:X\ni y&\mapsto\arginff{x\in\efd(f)\cap\intefd{\Psi}}{f(x)+\lambda D_\Psi(y,x)}\in 2^X\mbox{)},
\end{align}
whenever the argument of $\inf\{\ldots\}$ is finite.
\end{enumerate}
\end{definition}

\begin{proposition}\label{prop.prox.res.left.right}
Let $(X,\n{\cdot}_X)$ be reflexive and let $\Psi\in\pclg(X,\n{\cdot}_X)$ be Euler--Legendre. Then:
\begin{enumerate}[nosep,label=(\roman*)]
\item\label{prop.prox.res.left.right.i} if $T:X\ra2^{X^\star}$, then
\begin{equation}
\rres^{\Psi}_T=\DG\Psi\circ\lres_T^\Psi\circ\DG\Psi^\lfdual
\label{eqn.resolvent.left.right}
\end{equation}
with $\efd(\rres^{\Psi}_T)=\DG\Psi(\efd(\lres^\Psi_T))$ and $\ran(\rres^{\Psi}_T)=\DG\Psi(\ran(\lres^\Psi_T))$ \textup{\cite[Lem. 5.4]{MartinMarquez:Reich:Sabach:2012}};
\item\label{prop.prox.res.left.right.ii} if $f:X\ra\,]-\infty,\infty]$ is proper and satisfies $\efd(f)\cap\intefd{\Psi}\neq\varnothing$, then \textup{\cite[Lem. 3.4]{Laude:Ochs:Cremers:2020}+\cite[Prop. 3.23.(v).(b)]{Bauschke:Borwein:Combettes:2003}}\footnote{While \cite[Lem. 3.4]{Laude:Ochs:Cremers:2020} is stated and proved for $X=\RR^n$, its extension from $\RR^n$ to reflexive $(X,\n{\cdot}_X)$ is straightforward, with exactly the same proof, due to \cite[Prop. 3.23.(v).(b)]{Bauschke:Borwein:Combettes:2003}.}%
\begin{equation}
\rprox^{D_\Psi}_{\lambda,f}(x)=\DG\Psi^\lfdual\circ\lprox^{D_{\Psi^\lfdual}}_{\lambda,f\circ\DG\Psi^\lfdual}\circ\DG\Psi(x)\;\;\forall x\in\intefd{\Psi}.
\label{eqn.prox.left.right}
\end{equation}
\end{enumerate}
\end{proposition}

\begin{proposition}\label{prop.proj.prox.res}
Let $(X,\n{\cdot}_X)$ be a Banach space, let $\Psi\in\pclg(X,\n{\cdot}_X)$, and let $\varnothing\neq K\subseteq X$ be such that $K\cap\efd(\Psi)\neq\varnothing$ and ($K\cap\efd(\Psi)\subseteq\intefd{\Psi}$ or $K\cap\efd(\partial\Psi)\subseteq\intefd{\Psi}$ or $K\subseteq\intefd{\Psi}$ or $\efd(\Psi)$ is open or ($\intefd{\Psi}\cap K\neq\varnothing$ and $\Psi$ is essentially Gateaux differentiable)).
\begin{enumerate}[nosep,label=(\roman*)]
\item\label{prop.proj.prox.res.i} If $\lambda\in\,]0,\infty[$, $f\in\pcl(X,\n{\cdot}_X)$, and $K=\efd(f)$, then $\lprox^{D_\Psi}_{\lambda,f}=\lres^\Psi_{\lambda\partial f}$ \textup{\cite[Props. 3.22.(ii).(a), 3.23]{Bauschke:Borwein:Combettes:2003}}.
\item\label{prop.proj.prox.res.ii} If $K$ is closed and convex, and $K\cap\intefd{\Psi}\neq\varnothing$, then $\LPPP^{D_\Psi}_K=\lprox^{D_\Psi}_{1,\iota_K}=\lres^\Psi_{\partial\iota_K}$ with $\Fix(\lres^{\Psi}_{\partial\iota_K})=K\cap\intefd{\Psi}$ \textup{\cite[Props. 3.32, 3.33]{Bauschke:Borwein:Combettes:2003}}.
\end{enumerate}
\end{proposition}

\begin{proposition}\label{prop.res.pythagorean}
Let $(X,\n{\cdot}_X)$ be a Banach space, let $\Psi\in\pclg(X,\n{\cdot}_X)$, and let $T:X\ra 2^{X^\star}$ be monotone. Then:
\begin{enumerate}[nosep,label=(\roman*)]
\item\label{prop.res.pythagorean.i} $\lres^\Psi_{\lambda T}\in\lfn(\Psi,X)$ with  $\efd(\lres^\Psi_{\lambda T})\subseteq\intefd{\Psi}\supseteq\ran(\lres^\Psi_{\lambda T})$ and $\Fix(\lres^\Psi_{\lambda T})=\intefd{\Psi}\cap T^\inver(0)$ \textup{\cite[Prop. 3.8.(i)--(iii),(iv).(a)]{Bauschke:Borwein:Combettes:2003}}\footnote{\label{footnote.lambda.T.independence}This reference assumes $\lambda=1$, however the proofs of the corresponding properties do not change under generalisation to $\lambda\in\,]0,\infty[$, since if $T:X\ra 2^{X^\star}$ is monotone, then $\lambda T$ is monotone, $\efd(T)=\efd(\lambda T)$, and $(\lambda T)^\inver(0)=T^\inver(0)$.};
\item\label{prop.res.pythagorean.ii} if $\ran(\DG\Psi)\subseteq\ran(\DG\Psi+T)$ and $\Psi$ is strictly convex on $\intefd{\Psi}$, then \textup{\cite[Prop. 3.8.(iv).(b)--(c)]{Bauschke:Borwein:Combettes:2003}} $\lres^\Psi_T$ is single-valued on $\efd(\lres^\Psi_T)$, $\Fix(\lres^\Psi_T)$ is convex, and \textup{\cite[Lem. 1]{Eckstein:1993} \cite[Prop. 3.3.(i)]{Bauschke:Borwein:Combettes:2003}}
\begin{equation}
D_\Psi(x,y)\geq D_\Psi(x,\lres^\Psi_T(y))+D_\Psi(\lres^\Psi_T(y),y)\;\;\forall(x,y)\in\Fix(\lres^\Psi_T)\times\intefd{\Psi};
\label{eqn.resolvent.pythagorean.theorem}
\end{equation}
\item\label{prop.res.pythagorean.iii} if $(X,\n{\cdot}_X)$ is reflexive, $\lambda\in\,]0,\infty[$, $T$ is maximally monotone with $\efd(T)\subseteq\intefd{\Psi}$, $\Psi$ is Euler--Legendre, and $\efd(\Psi^\lfdual)=X^\star$, then $\lres^\Psi_{\lambda T}$ is single-valued on $\efd(\lres^\Psi_{\lambda T})$, and \eqref{eqn.resolvent.pythagorean.theorem} holds for $\lres^\Psi_T$ replaced by $\lres^\Psi_{\lambda T}$ \textup{\cite[Prop. 3.13.(iv).(b)]{Bauschke:Borwein:Combettes:2003}};
\item\label{prop.res.pythagorean.iv} if $(X,\n{\cdot}_X)$ is reflexive, $\lambda\in\,]0,\infty[$, $f\in\pcl(X,\n{\cdot}_X)$, $\Psi$ is Euler--Legendre, and $\intefd{\Psi}\cap\efd(f)\neq\varnothing$, then $\lprox^{D_\Psi}_{\lambda,f}$ is single-valued on $\efd(\lprox^{D_\Psi}_{\lambda,f})=\intefd{\Psi}$ and satisfies \eqref{eqn.resolvent.pythagorean.theorem} with $\lres^\Psi_T$ replaced by $\lprox^{D_\Psi}_{\lambda,f}$, and with $\Fix(\lprox^{D_\Psi}_{\lambda,f})=\intefd{\Psi}\cap\arginff{x\in X}{f(x)}$ \textup{\cite[Props. 3.21.(vi), 3.22.(ii).(b), 3.23.(v).(b), Cor. 3.25]{Bauschke:Borwein:Combettes:2003}};
\item\sloppy\label{prop.res.pythagorean.v} if $(X,\n{\cdot}_X)$ is reflexive, $\lambda\in\,]0,\infty[$, $\intefd{\Psi}\cap\efd(T)\neq\varnothing$, $\Psi$ is strictly convex on $\intefd{\Psi}$ and Euler--Legendre, then $\efd(\rres^\Psi_{\lambda T})\subseteq\intefd{\Psi^\lfdual}\supseteq\ran(\rres^\Psi_{\lambda T})$, $\Fix(\rres^\Psi_{\lambda T})=\DG\Psi(\intefd{\Psi}\cap T^\inver(0))$, $\rres^\Psi_{\lambda T}$ is single-valued on $\efd(\rres^\Psi_{\lambda T})$, and $\rres^\Psi_{\lambda T}\in\rfn(\Psi^\lfdual,X^\star)$ \textup{\cite[Lem. 5.4, Prop. 5.5]{MartinMarquez:Reich:Sabach:2012}}\textup{\footref{footnote.lambda.T.independence}}.
\end{enumerate}
\end{proposition}

\begin{definition}\label{def.adaptedness}
For reflexive $(X,\n{\cdot}_X)$, $\Psi\in\pclg(X,\n{\cdot}_X)$, and $\varnothing\neq K\subseteq\intefd{\Psi}$, $\Psi$ will be called:
\begin{enumerate}[nosep,label=(\roman*)]
\item\label{def.adaptedness.i} \df{LSQ-adapted} on $K$ if{}f, for any convex and closed $\varnothing\neq C\subseteq K$, $\LPPP^{D_\Psi}_C:K\ra\intefd{\Psi}$ belongs to $\lsq(\Psi,K)$, with $\aFix(\LPPP^{D_\Psi}_C)=\Fix(\LPPP^{D_\Psi}_C)=C$;
\item\label{def.adaptedness.ii} \df{RSQ-adapted} on $K$ if{}f, for any $\varnothing\neq C\subseteq K$ such that $\DG\Psi(C)$ is convex and closed, $\RPPP^{D_\Psi}_C:K\ra\intefd{\Psi}$ belongs to $\rsq(\Psi,K)$, with $\aFix(\RPPP^{D_\Psi}_C)=\Fix(\RPPP^{D_\Psi}_C)=C$.
\end{enumerate}
\end{definition}

\begin{corollary}\label{cor.proj.adaptedness}
Let $(X,\n{\cdot}_X)$ be reflexive, $\Psi\in\pclg(X,\n{\cdot}_X)$, $\varnothing\neq K\subseteq\intefd{\Psi}$.
\begin{enumerate}[nosep,label=(\roman*)]
\item\label{cor.proj.adaptedness.i} If $\Psi$ is LSQ-preadapted on $K$, then:
\begin{enumerate}[nosep,label=\alph*)]
\item\label{cor.proj.adaptedness.i.a} $\Psi$ is LSQ-adapted on $K$;
\item\label{cor.proj.adaptedness.i.b} if $K=X$, then $\lres_T^\Psi\in\lsq(\Psi,X)$ for any monotone $T:X\ra2^{X^\star}$.
\end{enumerate}
\item\label{cor.proj.adaptedness.ii} If $\Psi$ is RSQ-preadapted on $K$, then:
\begin{enumerate}[nosep,label=\alph*)]
\item\label{cor.proj.adaptedness.ii.a} if $\DG\Psi$ is bounded on bounded subsets of $\intefd{\Psi}$, and $\DG\Psi^\lfdual$ is bounded on bounded subsets of $\intefd{\Psi^\lfdual}$, then $\Psi$ is RSQ-adapted;
\item\label{cor.proj.adaptedness.ii.b} if $K=X$, $T:X\ra 2^{X^\star}$ is monotone, $\intefd{\Psi}\cap\efd(T)\neq\varnothing$, $\Psi$ is strictly convex on $\intefd{\Psi}$ and Euler--Legendre, then $\rres_T^\Psi\in\rsq(\Psi^\lfdual,X^\star)$.
\end{enumerate}
\end{enumerate}
\end{corollary}
\begin{proof}
\begin{enumerate}[nosep]
\item[(i)] Follows from Propositions \ref{prop.res.pythagorean}.\ref{prop.res.pythagorean.i} and \ref{prop.proj.prox.res}.\ref{prop.proj.prox.res.ii}.
\item[(ii)] Follows from \ref{cor.proj.adaptedness.i}, Proposition \ref{prop.Luo.Meng.Wen.Yao}, \cite[Fact 6.2]{MartinMarquez:Reich:Sabach:2013:BSN}\footnote{This result disproves an earlier claim in \cite[Prop. 2.7.(iv)]{MartinMarquez:Reich:Sabach:2012}, which stated the same consequence, but without assuming that $\DG\Psi$ (resp., $\DG\Psi^\lfdual$) is bounded on bounded subsets of $\intefd{\Psi}$ (resp., $\intefd{\Psi^\lfdual}$), and without an explicit proof.}, and Proposition \ref{prop.res.pythagorean}.\ref{prop.res.pythagorean.v}. 
\end{enumerate}
\end{proof}

\begin{definition}\label{def.lsq.rsq.adapted}
If $(X,\n{\cdot}_X)$ is reflexive, and $\Psi\in\pclg(X,\n{\cdot}_X)$ together with $\varnothing\neq C\subseteq K\subseteq\intefd{\Psi}$ satisfy the conditions of Corollary \ref{cor.proj.adaptedness}.\ref{cor.proj.adaptedness.i}.\ref{cor.proj.adaptedness.i.a} (resp., \ref{cor.proj.adaptedness}.\ref{cor.proj.adaptedness.ii}.\ref{cor.proj.adaptedness.ii.a}), then $\LPPP^{D_\Psi}_C$(resp., $\RPPP^{D_\Psi}_C$) will be said to be \df{adapted}.
\end{definition}

\begin{remark}\label{rem.quasinonexpansive}
\begin{enumerate}[nosep,label=(\roman*)]
\item\label{rem.quasinonexpansive.i} In general, without some additional conditions, neither $\LPPP^{D_\Psi}_K$ nor $\RPPP^{D_\Psi}_K$ will belong to $\cn(\Psi,K)$. Consider $\Psi$ given in Example \ref{ex.EL.totalconvex.Rn.Psi}.\ref{ex.EL.totalconvex.Rn.Psi.ii}. If $(\RR^n)^+_1:=\{x\in(\RR^n)^+\mid\n{x}_1=1\}$, and $x,y\in(\RR^n)^+_0$ such that $\n{x}_1=\n{y}_1<1$, then $D_\Psi(\LPPP^{D_\Psi}_{(\RR^n)^+_1}(x),\LPPP^{D_\Psi}_{(\RR^n)^+_1}(y))>D_\Psi(x,y)$, which implies that $\LPPP^{D_\Psi}_{(\RR^n)^+_1}$ is not an element of $\cn(\Psi)$ \cite[2.1.7]{Butnariu:Iusem:2000}. Furthermore, there is no $\hat{\Psi}:\RR^n\ra\,]-\infty,\infty]$ such that $D_\Psi(x,y)=D_{\hat{\Psi}}(y,x)$ $\forall x,y\in\intefd{\Psi}$ \cite[Prop. 3.3]{Bauschke:Noll:2002}, hence left and right $D_\Psi$-projections not only do not coincide, but also have to be considered as a priori independent notions.
\item\label{rem.quasinonexpansive.ii} The difference in the strength of assumptions imposed to obtained analogous behaviour of left and right $D_\Psi$-projections/proximal maps/resolvents is caused by the limitation of the current  knowledge about right $D_\Psi$-projections: in practice, all of known results in the reflexive Banach space setting are obtained by the Euler--Legendre transformation of the corresponding properties of their left $D_\Psi$ variants. The most blatant manifestation of this approach is the use of Euler--Legendre transformation of $\lres^{\Psi}_{\lambda T}$ for the purpose of a \textit{definition} of $\rres^{\Psi}_{\lambda T}$. However, it is important to remember, that the above approach does not cover the whole possible range of the right Va\u{\i}nberg--Br\`{e}gman theory. (Cf. also Remarks \ref{remark.psi.results}.\ref{remark.psi.results.v}, \ref{remark.categorical}.\ref{remark.categorical.i}, and \ref{remark.categorical}.\ref{remark.categorical.v} for a further discussion of this theme.)
\item\label{rem.quasinonexpansive.iii} For $(X,\n{\cdot}_X)$ given by the Hilbert space, and $\Psi=\frac{1}{2}\n{\cdot}^2_X$, $\lres^\Psi_T=\res^\Psi_T:=(T+\id_X)^\inver$ was introduced in \cite[Lem. 2]{Zarantonello:1960} and \cite[Cor. (p. 344)]{Minty:1962}, $\lprox^{D_\Psi}_{1,f}=\rprox^{D_\Psi}_{1,f}=\prox^{d_{\n{\cdot}_{\H}}}_{1,f}:y\mapsto\arginff{x\in\efd(f)}{f(x)+\frac{1}{2}\n{x-y}_X^2}$ was introduced in \cite[p. 2897]{Moreau:1962:prox} \cite[p. 1069]{Moreau:1963:prox}, while the corresponding $\prox^{d_{\n{\cdot}_{\H}}}_{\lambda,f}=\lprox^{D_\Psi}_{\lambda,f}=\rprox^{D_\Psi}_{\lambda,f}$ for $\lambda\in\,]0,\infty[$ appeared in \cite[p. 539]{Attouch:1977}. For a Banach space $(X,\n{\cdot}_X)$ and $\Psi=\frac{1}{2}\n{\cdot}_X^2$, $\lres^\Psi_{\lambda T}$ was introduced first in \cite[Eqn. (1)]{Kassay:1985}. 
\item\label{rem.quasinonexpansive.iv} For $(X,\n{\cdot}_X)$ given by the Hilbert space $(\H,\s{\cdot,\cdot}_\H)$, and $\Psi=\frac{1}{2}\n{\cdot}^2_\H$, the definition of $\lfn(\Psi,K)$ operators takes the form \cite[Def. 6]{Browder:1967}
\begin{equation}
\s{x-y,T(x)-T(y)}_\H\geq\n{T(x)-T(y)}_\H^2\;\;\forall x,y\in K.
\end{equation}
\end{enumerate}
\end{remark}
\subsection{Quasigauge functions and Banach space geometry}\label{section.gauge.functions}
\subsubsection{Banach space geometry}
A Banach space $(X,\n{\cdot}_X)$ is said to satisfy the \df{Radon--Riesz--Shmul'yan property}\footnote{In the literature it is usually called either the \df{Radon--Riesz property} \cite[\S3]{Leonard:1976}, or an \df{$H$-property} \cite[Def. 2 (Ch. 7)]{Day:1958}, or the \df{Kadec--Klee property} \cite[p. 119]{Diestel:1975}. It was first considered by Radon \cite[p. 1363]{Radon:1913} and Riesz \cite[p. 182]{Riesz:1929} for $(L_{1/\gamma}(\X,\mu),\n{\cdot}_{1/\gamma})$, $\gamma\in\,]1,\infty[$. For the general Banach spaces it was first introduced and studied by Shmul'yan in \cite[Thm. 5]{Shmulyan:1939:geometrical}. Kadec \cite[p. 13]{Kadec:1958} explicitly refers to this work of Shmul'yan, while Klee \cite[pp. 25--27]{Klee:1960} explicitly refers to this paper of Kadec.} \cite[p. 1363]{Radon:1913} \cite[p. 182]{Riesz:1929} \cite[Thm. 5]{Shmulyan:1939:geometrical} if{}f, for any $\{x_n\in X\mid n\in\NN\}$, convergence of $x_n$ to $x\in X$ in weak topology together with $\lim_{n\ra\infty}\n{x_n}_X=\n{x}_X$ implies $\lim_{n\ra\infty}\n{x_n-x}_X=0$. A Banach space $(X,\n{\cdot}_X)$ is called: \df{strictly convex} \cite[p. 39]{Frechet:1925:affines} \cite[p. 404]{Clarkson:1936} \cite[p. 178]{Krein:1938} if{}f
\begin{equation}
	\forall x,y\in X\setminus\{0\}\;\;\n{x+y}_X=\n{x}_X+\n{y}_X\;\;\limp\;\;\exists\lambda>0\;\;y=\lambda x,
\end{equation}
which is equivalent \cite[Thm. 1.(1)]{Ruston:1949} with 
\begin{equation}
	\forall x,y\in S(X,\n{\cdot}_X)\;\;x\neq y\;\;\limp\;\;\textstyle\frac{1}{2}\n{x+y}_X<1;
\end{equation}
\df{Gateaux differentiable} \cite[p. 78]{Mazur:1933} if{}f $\n{\cdot}_X$ is Gateaux differentiable at every $x\in X\setminus\{0\}$ (or, equivalently, at every $x\in S(X,\n{\cdot}_X)$), which is equivalent \rpkmark{\cite{Mazur:1933,Mazur:1933:schwache}} to each point of $S(X,\n{\cdot}_X)$ having a unique supporting hyperplane\footnote{More specifically, if $x\in X\setminus\{0\}$, then $\DG\n{x}_X$ exists if{}f $\n{x}_XS(X,\n{\cdot}_X)$ has a unique supporting hyperplane at $x$.}, i.e. 
\[
	\forall x\in S(X,\n{\cdot}_X)\;\;\exists! y(x)\in X^\star\;\;(y(x))(x)=\n{y(x)}_{X^\star}\n{x}_X=\n{x}_X;
\]
\df{uniformly convex} \cite[Def. 1]{Clarkson:1936} if{}f
\begin{equation}
	\forall\epsilon_1>0\;\;\exists\epsilon_2>0\;\;\forall x,y\in S(X,\n{\cdot}_X)\;\;\n{x-y}_X\geq\epsilon_1\;\;\limp\;\;\textstyle\frac{1}{2}\n{x+y}_X\leq1-\epsilon_2;
\end{equation}
\df{uniformly Fr\'{e}chet differentiable} if{}f any of the equivalent properties hold:
\begin{enumerate}[nosep,label=(\roman*)]
\item\label{ufd.prop.i} \cite[Thm. 77.1]{Nakano:1951:book} $\forall\epsilon_1>0$ $\exists\epsilon_2>0$ $\forall x,y\in X$ $(\n{x}_X=1$, $\n{y}_X\leq\epsilon_2)$ $\limp$ $\n{x+y}_X+\n{x-y}_X\leq2+\epsilon_1\n{y}_X$,
\item\label{ufd.prop.ii} \cite[p. 375]{Day:1944} $\forall\epsilon_1>0$ $\exists\epsilon_2>0$ $\forall x,y\in S(X,\n{\cdot}_X)$ $\n{x-y}_X\leq\epsilon_1$ $\limp$ $1-\frac{1}{2}\n{x+y}_X\leq\epsilon_2\n{x-y}_X$,
\item\label{ufd.prop.iii} \cite[p. 645]{Shmulyan:1940:diff} the limit $\DG\n{h}_X(x)$ exists in uniform convergence as $x$ and $h$ vary over $S(X,\n{\cdot}_X)$;
\end{enumerate}
\df{Fr\'{e}chet differentiable} \cite[p. 129]{Mazur:1933:schwache} if{}f $\n{\cdot}_X$ is Fr\'{e}chet differentiable at every $x\in X\setminus\{0\}$ (or, equivalently, at every $x\in S(X,\n{\cdot}_X)$), i.e. for any fixed $x\in X\setminus\{0\}$ (or $x\in S(X,\n{\cdot}_X)$) $\DG\n{h}_X(x)$ exist in uniform convergence $\forall h\in S(X,\n{\cdot}_X)$; 
\df{locally uniformly convex} \cite[Def. 0.2]{Lovaglia:1955} if{}f
\begin{equation}
\forall\epsilon_1>0\;\;\forall x\in S(X,\n{\cdot}_X)\;\;\exists\epsilon_2>0\;\;\forall y\in S(X,\n{\cdot}_X)\;\;\n{x-y}_X\geq\epsilon_1\;\limp\;\textstyle\frac{1}{2}\n{x+y}_X\leq1-\epsilon_2;
\end{equation}
\df{$r$-uniformly convex} for $r\in[2,\infty[$ \cite[Def. 2)]{Assouad:1975} \cite[p. 468]{Ball:Carlen:Lieb:1994} if{}f
\begin{equation}
\exists\lambda>0\;\forall x,y\in X\;\;\n{x+y}_X^r+\n{x-y}_X^r\geq2(\n{x}_X^r+\n{\lambda^{-1}y}_X^r),
\label{eqn.r.unif.convex}
\end{equation}
or, equivalently, if{}f \cite[Def. 4.1, Thm. 4.3]{Reich:Xu:2003}
\begin{equation}
\exists\lambda>0\;\forall x,y\in X\;\;\n{x+y}_X^r+\lambda\n{x-y}_X^r\leq2^{r-1}(\n{x}_X^r+\n{y}_X^r),
\label{eqn.r.unif.convex.reich.xu}
\end{equation}
or, equivalently \textup{\cite[Thm. 2.5]{Woyczynski:1975}} \textup{\cite[Prop. 7]{Ball:Carlen:Lieb:1994}}, if{}f $\exists c>0$ $\delta(X,\n{\cdot}_X;\epsilon)\geq c\epsilon^r$, where \cite[Def. 1]{Clarkson:1936} \cite[p. 375]{Day:1944}
\begin{equation}
\rpkmarkmath{]0,2]}\ni\epsilon\mapsto\delta(X,\n{\cdot}_X;\epsilon):=\inf\left\{1-\textstyle\frac{1}{2}\n{x+y}_X\mid x,y\in S(X,\n{\cdot}_X),\,\n{x-y}_X\rpkmarkmath{\geq}\epsilon\right\}\in[0,1];
\end{equation}
\df{$r$-uniformly Fr\'{e}chet differentiable} for $r\in\,]1,2]$ \cite[Def. 1)]{Assouad:1975} \cite[p. 468]{Ball:Carlen:Lieb:1994} if{}f
\begin{equation}
\exists\lambda>0\;\forall x,y\in X\;\;\n{x+y}_X^r+\n{x-y}_X^r\leq2(\n{x}_X^r+\n{\lambda y}_X^r),
\label{eqn.r.unif.frechet}
\end{equation}
or, equivalently \cite[Thm. 2.2]{Reich:Xu:2003}, if{}f \cite[Def. 1]{Trubnikov:1987}
\begin{equation}
\exists\lambda>0\;\forall x,y\in X\;\;\n{x+y}_X^r+\lambda\n{x-y}_X^r\geq2^{r-1}(\n{x}_X^r+\n{y}_X^r),
\label{eqn.r.unif.frechet.trubnikov}
\end{equation}
or, equivalently \textup{\cite[Thm. 2.4]{HoffmannJoergensen:1974}} \textup{\cite[Lem. 1]{Assouad:1974}} \textup{\cite[Prop. 7]{Ball:Carlen:Lieb:1994}}, if{}f $\exists c>0$ $\rho(X,\n{\cdot}_X;\epsilon)\leq c\epsilon^r$, where \cite[p. 241]{Lindenstrauss:1963}
\begin{align}
\rpkmarkmath{]0,\infty[\,}\ni\epsilon\mapsto\rho(X,\n{\cdot}_X;\epsilon):=&\sup\left\{\textstyle\frac{1}{2}(\n{x+y}_X+\n{x-y}_X)-1\mid x,y\in X,\,\n{x}=1,\,\n{y}_X=\epsilon\right\}\nonumber\\
=&\sup\left\{\textstyle\frac{1}{2}(\n{x+\epsilon y}_X+\n{x-\epsilon y}_X)-1\mid x,y\in S(X,\n{\cdot}_X)\right\}\in\RR^+.
\end{align}
\subsubsection{Quasigauge $\varphi$}
A strictly increasing and continuous function $\varphi:\RR^+\ra\RR^+$ such that $\varphi(0)=0$ and $\lim_{t\ra\infty}\varphi(t)=\infty$ \cite[p. 407]{Beurling:Livingston:1962}\footnote{Under a weakening of `strictly increasing' to `nondecreasing', the corresponding function, called a \df{$\boldsymbol{\varphi}$-function}, is used in the Orlicz space theory since \cite[p. 349]{Matuszewska:1960}. In the context of duality map, this weakening, joined with dropping the condition $\lim_{t\ra\infty}\varphi(t)=\infty$, has been considered in \cite[Def. (p. 200)]{Asplund:1967}. All of these functions are special cases of a quasigauge.} will be called a \df{gauge} \cite[p. 348]{Browder:1965:multivalued}. A nondecreasing function $\varphi:\RR^+\ra[0,\infty]$ satisfying $\varphi\not\equiv0$ and $\exists u\in\RR^+$ $\lim_{t\ra^+u}\varphi(t)<\infty$ \cite[p. 367]{Zalinescu:1983} will be called a \df{quasigauge}. For any Banach space $(X,\n{\cdot}_X)$ and any quasigauge $\varphi$, a \df{duality map} on $(X,\n{\cdot}_X)$ is defined by \cite[Def. (p. 200)]{Asplund:1967} \cite[Eqn. (A.4)]{Zalinescu:1983}
\begin{equation}
j_\varphi:X\ni x\mapsto\{y\in X^\star\mid\duality{x,y}_{X\times X^\star}=\n{x}_X\n{y}_{X^\star},\;\lim_{t\ra^-\n{x}_X}\varphi(t)\leq\n{y}_{X^\star}\leq\lim_{t\ra^+\n{x}_X}\varphi(t)\}\subseteq X^\star,
\label{eqn.j.quasigauge.varphi.definition}
\end{equation}
with the convention $\lim_{t\ra^-0}\varphi(t)\equiv\varphi(0)$. For any gauge $\varphi$, \eqref{eqn.j.quasigauge.varphi.definition} turns into \cite[p. 407]{Beurling:Livingston:1962}
\begin{equation}
j_\varphi:X\ni x\mapsto\{y\in X^\star\mid\duality{x,y}_{X\times X^\star}=\n{x}_X\n{y}_{X^\star},\;\n{y}_{X^\star}=\varphi(\n{x}_X)\}\subseteq X^\star,
\label{eqn.j.varphi.definition}
\end{equation}
which is said to be \df{normalised}, and denoted as $j$, if{}f $\varphi(t)=t$ \cite[p. 35]{Klee:1953} \cite[p. 211]{Vainberg:1961}. For any quasigauge $\varphi$, let $j_\varphi^\star$ denote a duality map on $(X^\star,\n{\cdot}_{X^\star})$. Then, for any gauge \cite[Prop. 3]{DePrima:Petryshyn:1971}
\begin{equation}
j_\varphi=(j_{\varphi^{\inver}}^\star)^{\inver}\circ J_X\mbox{ and }(j_\varphi)^{\inver}= (J_X)^{\inver}\circ j_{\varphi^{\inver}}^\star.
\label{eqn.j.j.star.inverse}
\end{equation}
If $\varphi_1$ and $\varphi_2$ are gauges, then \cite[Prop. I.3.1.g)]{Cioranescu:1974}
\begin{equation}
	\varphi_2(\n{x}_X)j_{\varphi_1}(x)=\varphi_1(\n{x}_X)j_{\varphi_2}(x)\;\forall x\in X.
\label{eqn.two.gauges}
\end{equation}
For any quasigauge $\varphi$, 
\begin{equation}
X\ni x\mapsto\Psi_\varphi(x):=\int_0^{\n{x}_X}\dd t\,\varphi(t)\in\RR^+
\label{eqn.Psi.varphi}
\end{equation}
satisfies $\Psi_\varphi\in\pcl(X,\n{\cdot}_X)$ \cite[p. 281]{Zalinescu:1984}, $\Psi_\varphi(0)=0$ \cite[p. 368]{Zalinescu:1983}, as well as \cite[p. 200]{Asplund:1967} \cite[Prop. A.3.(i)]{Zalinescu:1983}
\begin{equation}
j_\varphi(x)=\partial\Psi_\varphi(x)\;\forall x\in X\setminus\{0\}\;\;\mbox{with}\;\;j_\varphi(0)=0.
\label{eqn.asplund}
\end{equation}
If $\varphi$ is a gauge, then $\Psi_\varphi$ is continuous (cf., e.g., \cite[Thm. 3.7.2.(i)]{Zalinescu:2002}). 

As an example, for $\alpha\in\,]0,\infty[$ and $\beta\in\,]0,1[$, $\varphi_{\alpha,\beta}(t):=\frac{1}{\alpha}t^{{1/\beta}-1}$ is a gauge. Application of \eqref{eqn.Psi.varphi} and \eqref{eqn.two.gauges} gives
\begin{equation}
\Psi_{\varphi_{\alpha,\beta}}(x)=\textstyle\frac{\beta}{\alpha}\n{x}_X^{1/\beta},\;\;j_{\varphi_{\alpha,\beta}}(x)=\textstyle\frac{1}{\alpha}\n{x}_X^{1/\beta-2}j(x)\;\;\forall x\in X.
\label{eqn.varphi.alpha.beta.Psi.j.varphi}
\end{equation}

For any nondecreasing function $f:\RR^+\ra[0,\infty]$, its \df{right} (resp., \df{left}) \df{inverse} function reads 
\begin{align}
\RR^+\ni t\mapsto f^\meet(t)&:=\sup\{s\geq0\mid f(s)\leq t\}=\inf\{s\geq0\mid f(s)>t\}\in[0,\infty]\\
\mbox{(resp., }\RR^+\ni t\mapsto f^\join(t)&:=\sup\{s\geq0\mid f(s)<t\}=\inf\{s\geq0\mid f(s)\geq t\}\in[0,\infty]\mbox{)}.
\end{align}
In general, $f^\join\leq f^\meet$ (cf., e.g., \cite[Eqn. (8)]{Klement:Mesiar:Pap:1999}). If $f:\RR^+\ra[0,\infty]$ is strictly increasing and continuous on $\RR^+$, then $f^\meet=f^{\inver}=f^\join$, where $f^{\inver}$ is an inverse function in the standard sense, i.e. $f\circ f^{\inver}=\id_{\RR^+}=f^{\inver}\circ f$ (cf., e.g., \cite[p. 5]{Klement:Mesiar:Pap:1999} and \cite[Rem. 1.(1)]{Embrechts:Hofert:2013}).

For any $f:\RR\ra[0,\infty]$, its \df{Young--Birnbaum--Orlicz dual} \cite[p. 226]{Young:1912} \cite[Eqn. (5)]{Birnbaum:Orlicz:1931} reads
\begin{equation}
\RR\ni y\mapsto f^\young(y):=\sup\{x\ab{y}-f(x)\mid x\geq0\}\in[0,\infty].
\end{equation}
If $f:\RR^+\ra[0,\infty]$ is proper, convex on $\efd(f)$, satisfies $f(0)=0$, $f\not\equiv0$, and is left continuous at $\sup(\efd(f))$ (i.e. $\lim_{t\ra^-\sup(\efd(f))}f(t)=\sup(\efd(f))$), then it is called a \df{Young function} \cite[\S2]{Zaanen:1949}. If $f$ is a Young function, then $f^{\young\young}=f^\young$ (cf., e.g., \cite[Prop. 2.4.5]{Harjulehto:Haestoe:2019}).
\subsubsection{Characterisation of Banach space geometry by $\Psi_\varphi$}
\begin{remark} 
In order to stress that the properties of $\Psi_\varphi$ depend on a choice of the specific norm $\n{\cdot}_X$ on $X$, while avoiding the notation ``$\Psi_{\varphi,\n{\cdot}_X}$'' (in order to avoid dealing with such terms as ``$\LPPP^{D_{\ell_\orlicz,\Psi_{\varphi_{\alpha,\beta},\n{\cdot}_{\orlicz,p}}}}_C$''), we will sometimes use the notation $\Psi_\varphi:(X,\n{\cdot}_X)\ra\RR^+$.
\end{remark}

\begin{proposition}\label{prop.psi.varphi.geometry}
Let $(X,\n{\cdot}_X)$ be a Banach space, $\varphi$ a gauge, and  $\Psi_\varphi:(X,\n{\cdot}_X)\ra\RR^+$. Then:
\begin{enumerate}[nosep,label=(\roman*)]
\item\label{prop.psi.varphi.geometry.i} $(X,\n{\cdot}_X)$ is a Hilbert space if{}f $j$ is linear \textup{\cite[Prop. 2]{Golomb:Tapia:1972}};
\item\label{prop.psi.varphi.geometry.ii} $(X,\n{\cdot}_X)$ is strictly convex:
\begin{enumerate}[nosep,label=\alph*)]
\item\label{prop.psi.varphi.geometry.ii.a} if{}f $\Psi_\varphi$ is strictly convex \textup{\cite[Prop. A.3.(iii)]{Zalinescu:1983}};
\item\label{prop.psi.varphi.geometry.ii.b} if{}f $j_\varphi$ is strictly monotone on $X$ \textup{\cite[Thm. 1]{DePrima:Petryshyn:1971}} (cf. \textup{\cite[Thm. 1]{Petryshyn:1970}} for $\varphi(t)=t$);
\end{enumerate}
\item\label{prop.psi.varphi.geometry.iii} $(X,\n{\cdot}_X)$ is Gateaux differentiable:
\begin{enumerate}[nosep,label=\alph*)]
\item\label{prop.psi.varphi.geometry.iii.a} if{}f $\Psi_\varphi$ is Gateaux differentiable on $X$ \textup{\cite[Prop. A.3.(ii)]{Zalinescu:1983}};
\item\label{prop.psi.varphi.geometry.iii.b} if{}f $j_\varphi$ is single-valued on $X$ \textup{\cite[Prop. A.3.(ii)]{Zalinescu:1983}} (cf. \textup{\cite[Rems. 5, 8]{Mazur:1933}} \textup{\cite[p. 130]{Mazur:1933:schwache}} \textup{\cite[Cor. 4.8]{Cudia:1964}} for $\varphi(t)=t$);
\end{enumerate}
\item\label{prop.psi.varphi.geometry.iv} $(X,\n{\cdot}_X)$ is Fr\'{e}chet differentiable:
\begin{enumerate}[nosep,label=\alph*)]
\item\label{prop.psi.varphi.geometry.iv.a} if{}f $\Psi_\varphi$ is Fr\'{e}chet differentiable on $X$ \textup{\cite[Prop. 3.7.4.(ii)]{Zalinescu:2002}};
\item\label{prop.psi.varphi.geometry.iv.b} if{}f $j_\varphi$ is single-valued and norm-to-norm continuous on $X$ \textup{\cite[Thm. II.2.9]{Cioranescu:1974}} (cf. \textup{\cite[Cor. 4.12]{Cudia:1964}} for $\varphi(t)=t$);
\end{enumerate}
\item\label{prop.psi.varphi.geometry.v} $(X,\n{\cdot}_X)$ is uniformly Fr\'{e}chet differentiable:
\begin{enumerate}[nosep,label=\alph*)]
\item\label{prop.psi.varphi.geometry.v.a} if{}f $\Psi_\varphi$ is uniformly Fr\'{e}chet differentiable on bounded subsets of $X$ \textup{\cite[Thm. 3.7.4.(iii)]{Zalinescu:2002}};
\item\label{prop.psi.varphi.geometry.v.b} if{}f $j_\varphi$ is single-valued and uniformly continuous on bounded subsets of $X$ \textup{\cite[Thm. II.2.10]{Cioranescu:1974}} (cf. \textup{\cite[Cor. 4.12]{Cudia:1964}} for $\varphi(t)=t$);
\end{enumerate}
\item\label{prop.psi.varphi.geometry.vi} $(X,\n{\cdot}_X)$ is uniformly convex:
\begin{enumerate}[nosep,label=\alph*)]
\item\label{prop.psi.varphi.geometry.vi.a} if{}f $\Psi_\varphi$ is uniformly convex on bounded subsets of $X$ \textup{\cite[Thm. 4.1.(ii)]{Zalinescu:1983}};
\item\label{prop.psi.varphi.geometry.vi.b} if{}f $\Psi_\varphi^\lfdual$ is uniformly Fr\'{e}chet differentiable on bounded subsets of $X^\star$ \textup{\cite[Thm. 3.7.9.(iv)]{Zalinescu:2002}};
\item\label{prop.psi.varphi.geometry.vi.c} if{}f $(j_\varphi)^{\inver}$ is single-valued and uniformly continuous on bounded subsets of $X^\star$  \textup{\cite[Cor. 4.2]{Zalinescu:1983}};
\item\label{prop.psi.varphi.geometry.vi.d} if{}f $j$ is $f$-uniformly monotone on $B(X,\n{\cdot}_X)$ \textup{\cite[Thm. 1]{Pruess:1981}} \textup{\cite[Rem. 2]{Xu:Roach:1991}} \textup{\cite[Cor. 3.(iii)]{Xu:1991}};
\item\label{prop.psi.varphi.geometry.vi.e} if{}f \textup{\cite[Thm. 3]{Bynum:1971}} $\forall t\in\,]0,2]$
\begin{equation}
\inf\{\duality{x-y,v-w}_{X\times X^\star}\mid x,y\in S(X,\n{\cdot}_X),\;v\in j(x),\;w\in j(y),\;\n{x-y}_X\geq t\}>0;
\end{equation}
\end{enumerate}
\item\label{prop.psi.varphi.geometry.vii} $(X,\n{\cdot}_X)$ is $r$-uniformly convex for $r\in[2,\infty[$:
\begin{enumerate}[nosep,label=\alph*)]
\item\label{prop.psi.varphi.geometry.vii.a} if{}f $\Psi_\varphi$ with $\varphi(t)=rt^{r-1}$ is uniformly convex on $X$ \textup{\cite[Thm. 1]{Xu:1991}};
\item\label{prop.psi.varphi.geometry.vii.b} if{}f $j_{\widetilde{\varphi}}$ with $\widetilde{\varphi}(t)=t^{r-1}$ is $f$-uniformly monotone on $X$ with $f(t)=\widetilde{\varphi}(t)$ \textup{\cite[Cor. 1.(ii)]{Xu:1991}}.
\end{enumerate}
In particular, $(X,\n{\cdot}_X)$ is $2$-uniformly convex:
\begin{enumerate}[nosep,label=\alph*)]
\setcounter{enumii}{2}
\item\label{prop.psi.varphi.geometry.vii.c} if{}f $j$ is strongly monotone on $X$ \textup{\cite[p. 203]{Xu:Roach:1991}} (cf. \textup{\cite[Prop. 2.11]{Poffald:Reich:1986}});
\item\label{prop.psi.varphi.geometry.vii.d} if{}f $j^\star$ is single-valued and Lipschitz continuous on $X^\star$ \textup{\cite[Thm. (p. 62)]{Zemek:1991}};
\end{enumerate}
\item\label{prop.psi.varphi.geometry.viii} $(X,\n{\cdot}_X)$ is $r$-uniformly Fr\'{e}chet differentiable for $r\in\,]1,2]$:
\begin{enumerate}[nosep,label=\alph*)]
\item\label{prop.psi.varphi.geometry.viii.a} if{}f $\Psi_\varphi$ with $\varphi(t)=rt^{r-1}$ is uniformly Fr\'{e}chet differentiable on $X$ \textup{\cite[Thm. 10]{Shioji:1994}} (= \textup{\cite[Thm. 6.4]{Shioji:1995}}) \textup{\cite[Thm. 2.2]{Borwein:Guirao:Hajek:Vanderwerff:2009}};
\item\label{prop.psi.varphi.geometry.viii.b} if{}f $j_\varphi$ with $\varphi(t)=rt^{r-1}$ is single-valued and $(r-1)$-Lipschitz--H\"{o}lder continuous on $X$ \textup{\cite[Thm. 2.2]{Borwein:Guirao:Hajek:Vanderwerff:2009}};
\item\label{prop.psi.varphi.geometry.viii.c} if{}f $j_{\tilde{\varphi}}$ with $\tilde{\varphi}(t)=t^{r-1}$ is single-valued and $(r-1)$-Lipschitz--H\"{o}lder continuous on $X$ \textup{\cite[Thm. 2.4]{HoffmannJoergensen:1974}} \textup{\cite[Rem. 5]{Xu:Roach:1991}};
\item\label{prop.psi.varphi.geometry.viii.d} if{}f, for $\widetilde{\varphi}(t)=t^{r-1}$, \textup{\cite[Cor. 1']{Xu:1991}}
\begin{equation}
\exists c>0\;\forall x,y\in X\;\forall v\in j_{\widetilde{\varphi}}(x)\;\forall w\in j_{\widetilde{\varphi}}(y)\;\;\duality{x-y,v-w}_{X\times X^\star}\leq c\n{x-y}_X^r.
\end{equation}
\end{enumerate}
In particular, $(X,\n{\cdot}_X)$ is $2$-uniformly Fr\'{e}chet differentiable:
\begin{enumerate}[nosep,label=\alph*)]
\setcounter{enumii}{4}
\item\label{prop.psi.varphi.geometry.viii.e} if{}f $j$ is single-valued and Lipschitz continuous on $X$ \textup{\cite[Lem. 2.4.(iv)]{Fabian:Whitfield:Zizler:1983}} \textup{\cite[Thm. (p. 62)]{Zemek:1991}};
\item\label{prop.psi.varphi.geometry.viii.f} if{}f $j^\star$ is strongly monotone on $X$ \textup{\cite[Thm. (p. 62)]{Zemek:1991}};
\end{enumerate}
\item\label{prop.psi.varphi.geometry.ix} $(X,\n{\cdot}_X)$ is locally uniformly convex if{}f $\Psi_\varphi$ is uniformly convex at each $x\in X$ \textup{\cite[Thm. 4.1.(i)]{Zalinescu:1983}} (cf. \textup{\cite[p. 232]{Asplund:1967:averaged}} for $\varphi(t)=t$);
\item\label{prop.psi.varphi.geometry.x} $(X,\n{\cdot}_X)$ is reflexive if{}f $j_\varphi$ is surjective (i.e. $X^\star=\bigcup_{x\in X}\ran(j_\varphi(x))$) \textup{\cite[Thm. III.7]{deFigueiredo:1967}} (cf. also \textup{\cite[Prop. A.3.(iv)]{Zalinescu:1983}});
\item\label{prop.psi.varphi.geometry.xi} if $(X,\n{\cdot}_X)$ is reflexive, then ($(X,\n{\cdot}_X)$ is strictly convex and has the Radon--Riesz--Shmul'yan property) if{}f $\Psi_\varphi$ is totally convex) \textup{\cite[Thm. 3.1, Thm. 3.3]{Resmerita:2004}};
\item\label{prop.psi.varphi.geometry.xii} if $(X,\n{\cdot}_X)$ is locally uniformly convex, then $\Psi_\varphi$ is totally convex on $X$ \textup{\cite[Thm. 3.1, Cor. 3.4]{Resmerita:2004}};
\item\label{prop.psi.varphi.geometry.xiii} $(X,\n{\cdot}_X)$ has the Radon--Riesz--Shmul'yan property if{}f at least one single-valued section $T$ of $j$ satisfies
\begin{equation}
\left\{
\begin{array}{l}
\lim_{n\ra\infty}\duality{x_n-x,T(x_n)-T(x)}_{X\times X^\star}=0\\
\lim_{n\ra\infty}\n{T(x_n)}_{X^\star}=\n{T(x)}_{X^\star}\;\limp\;\lim_{n\ra\infty}\duality{x,T(x_n)}_{X\times X^\star}=\duality{x,T(x)}_{X\times X^\star}
\end{array}
\right.
\end{equation}
for every $\{x_n\in X\mid n\in\NN\}$ with $x_n$ convergent to $x$ in the weak topology \textup{\cite[Prop. 1]{Petryshyn:1970}}.
\end{enumerate}
\end{proposition}

\begin{corollary}\label{cor.r.unif.convex.j.lambda}
$(X,\n{\cdot}_X)$ is $r$-uniformly convex for $r\in[2,\infty[$ if{}f $j_{\varphi_\lambda}$ is $f$-uniformly monotone with $f(t)=\varphi_\lambda(t)$, where $\varphi_\lambda(t)=\lambda t^{r-1}$ and $\lambda\in\,]0,\infty[$.
\end{corollary}
\begin{proof}
Follows from Proposition \ref{prop.psi.varphi.geometry}.\ref{prop.psi.varphi.geometry.vii}, together with \eqref{eqn.two.gauges} giving $j_{\varphi_\lambda}=\lambda j_{\widetilde{\varphi}}$, and cancellation of $\frac{1}{\lambda}$ on both sides of \eqref{eqn.f.uniformly.monotone}.
\end{proof}

\begin{corollary}\label{cor.j.varphi.bijection}
For any gauge $\varphi$ and any Banach space $(X,\n{\cdot}_X)$:
\begin{enumerate}[nosep,label=(\roman*)]
\item\label{cor.j.varphi.bijection.i} $j_\varphi$ is a bijection if{}f $(X,\n{\cdot}_X)$ is reflexive, strictly convex, and Gateaux differentiable \textup{\cite[Cor. b) (p. 105)]{Cioranescu:1974}};
\item\label{cor.j.varphi.bijection.ii} $j_\varphi$ is a norm-to-norm homeomorphism if{}f $(X,\n{\cdot}_X)$ reflexive, strictly convex, Fr\'{e}chet differentiable, and has the Radon--Riesz--Shmul'yan property \textup{\cite[Lem. 2.(i)]{Kolomy:1984}} (and, for $\varphi(t)=t$, \textup{\cite[Cor. (p. 189)]{Giles:1971}}).
\end{enumerate}
\end{corollary}
\begin{proof}
By Propositions \ref{prop.psi.varphi.geometry}.\ref{prop.psi.varphi.geometry.iii} and \ref{prop.psi.varphi.geometry}.\ref{prop.psi.varphi.geometry.x}, $(X,\n{\cdot}_X)$ is Gateaux differentiable and reflexive if{}f $j_\varphi$ is single-valued and surjective. By reflexivity of $(X,\n{\cdot}_X)$, strict convexity of $(X,\n{\cdot}_X)$ is equivalent with Gateaux differentiability of $(X^\star,\n{\cdot}_{X^\star})$, which is equivalent with single-valuedness and surjectivity of $j_{\varphi^{\inver}}^\star$. By reflexivity of $(X,\n{\cdot}_X)$ and \eqref{eqn.j.j.star.inverse}, $j_\varphi=(j_{\varphi^{\inver}}^\star)^{\inver}$. Hence, $j_\varphi$ is a bijection. Norm-to-norm homeomorphy of $j_\varphi$ follows the same way, by Proposition \ref{prop.psi.varphi.geometry}.\ref{prop.psi.varphi.geometry.iv}, and the fact \cite[Thm. 3.9]{Anderson:1960} that, if $(X,\n{\cdot}_X)$ is reflexive, then $(X^\star,\n{\cdot}_{X^\star})$ is Fr\'{e}chet differentiable if{}f ($(X,\n{\cdot}_X)$ is strictly convex and has the Radon--Riesz--Shmul'yan property).
\end{proof}

\begin{proposition}\label{prop.psi.quasigauge.geometry}
For any quasigauge $\varphi$ and any Banach space $(X,\n{\cdot}_X)$:
\begin{enumerate}[nosep,label=(\roman*)]
\item\label{prop.psi.quasigauge.geometry.i} ($(X,\n{\cdot}_X)$ is Gateaux differentiable and $\varphi$ is continuous on $[0,\sup(\efd(\varphi))[$) if{}f $j_\varphi$ is single-valued on $\intefd{\Psi_\varphi}=\INT(\sup(\efd(\varphi))B(X,\n{\cdot}_X))$ \textup{\cite[Prop. A.3.(ii)]{Zalinescu:1983}};
\item\label{prop.psi.quasigauge.geometry.ii} ($(X,\n{\cdot}_X)$ is strictly convex and $\varphi$ is strictly increasing on $\efd(\varphi)$) if{}f $\Psi_\varphi$ is strictly convex on $\efd(\Psi_\varphi)$ \textup{\cite[Prop. A.3.(iii)]{Zalinescu:1983}} \textup{\cite[Thm. 2.1.(viii)]{Zalinescu:1984}};
\item\label{prop.psi.quasigauge.geometry.iii} ($(X,\n{\cdot}_X)$ is locally uniformly convex and $\varphi$ is strictly increasing on $\efd(\varphi)$) if{}f $\Psi_\varphi$ is uniformly convex at any $x\in\intefd{\Psi_\varphi}=\INT(\sup(\efd(\varphi))B(X,\n{\cdot}_X))$ \textup{\cite[Thm. 4.1.(i)]{Zalinescu:1983}};
\item\label{prop.psi.quasigauge.geometry.iv} ($(X,\n{\cdot}_X)$ is uniformly convex and $\varphi$ is strictly increasing on $\efd(\varphi)$) if{}f $\Psi_\varphi$ is uniformly convex on $\lambda S(X,\n{\cdot}_X)$ $\forall\lambda\in\,]0,\sup(\efd(\varphi))]$ \textup{\cite[Thm. 4.1.(ii)]{Zalinescu:1983}};
\item\label{prop.psi.quasigauge.geometry.v} ($(X,\n{\cdot}_X)$ is reflexive and $\lim_{t\ra\infty}\varphi(t)=\infty$) if{}f $j_\varphi$ is surjective \textup{\cite[Prop. A.3.(iv)]{Zalinescu:1983}}.
\end{enumerate}
\end{proposition}
\section{Va\u{\i}nberg--Br\`{e}gman relative entropies, quasinonexpansive maps, and their extensions}\label{section.convex.new}
\subsection{$D_\Psi$ with $\Psi\in\pclg(X,\n{\cdot}_X)$ and reflexive $(X,\n{\cdot}_X)$}\label{section.convex.new.general}
\begin{proposition}\label{prop.right.pythagorean}
Let $(X,\n{\cdot}_X)$ be reflexive, $\Psi\in\pclg(X,\n{\cdot}_X)$, $\Psi^\lfdual$ Gateaux differentiable on $\varnothing\neq\DG\Psi(\intefd{\Psi})\subseteq\intefd{\Psi^\lfdual}$, $\varnothing\neq K\subseteq\intefd{\Psi}$, $\DG\Psi(K)$ closed and convex. If any of the following (generally, inequivalent) conditions holds:
\begin{enumerate}[nosep,label=\alph*)]
\item\label{prop.right.pythagorean.a} $\Psi^\lfdual$ is totally convex on $\efd(\Psi^\lfdual)$; or
\item\label{prop.right.pythagorean.b} $\Psi^\lfdual$ is strictly convex on $\efd(\Psi^\lfdual)$ and supercoercive; or
\item\label{prop.right.pythagorean.c} $\Psi^\lfdual$ is Euler--Legendre,
\end{enumerate}
then:
\begin{enumerate}[nosep,label=(\roman*)]
\item\label{prop.right.pythagorean.i} $K$ is right $D_\Psi$-Chebysh\"{e}v, and $D_\Psi$ is right pythagorean on $K$;
\item\label{prop.right.pythagorean.ii} $\RPPP_K^{D_\Psi}$ is zone consistent;
\item\label{prop.right.pythagorean.iii} $D_{\Psi^\lfdual}$ is an information on $X^\star$ (and, in the case c), $D_\Psi$ is an information on $X$);
\item\label{prop.right.pythagorean.iv} if $\DG\Psi(K)$ is affine, then
\begin{equation}
	D_\Psi(x,y)=D_\Psi(x,\RPPP^{D_\Psi}_K(x))+D_\Psi(\RPPP^{D_\Psi}_K(x),y)\;\;\forall(x,y)\in\intefd{\Psi}\times K.
\label{eqn.right.pyth.eqn}
\end{equation}
\end{enumerate} 
\end{proposition}
\begin{proof}
\begin{enumerate}[nosep]
\item[(i)] Proposition \ref{prop.left.pythagorean}.\ref{prop.left.pythagorean.i}, by application of \eqref{eqn.left.to.right}, implies that $K$ is right $D_\Psi$-Che\-by\-sh\"{e}v. As for $D_\Psi$ being right pythagorean, the proof is exactly the same as a proof of \cite[Prop. 4.11]{MartinMarquez:Reich:Sabach:2012}, with the following changes, extending the range of its validity: (1) an extension from $\Psi:X\ra\RR$ to $\Psi:X\ra\,]-\infty,\infty]$ is provided by using \eqref{eqn.left.to.right} instead of \cite[Prop. 7.1, Eqn. (79)]{Bauschke:Wang:Ye:Yuan:2009} (=\cite[Eqn. (28)]{MartinMarquez:Reich:Sabach:2012}), together with imposing the condition $\DG\Psi(\intefd{\Psi})\subseteq\intefd{\Psi^\lfdual}$, instead of $\efd(\Psi)=X$, as in \cite[Prop. 4.11]{MartinMarquez:Reich:Sabach:2012}, or instead of the Euler--Legendre property of $\Psi$, as in \cite[Prop. 7.1]{Bauschke:Wang:Ye:Yuan:2009}; (2) the assumption of total convexity of $\Psi^\lfdual$ on $\efd(\Psi^\lfdual)$ is weakened by allowing the alternative assumptions of the Euler--Legendre property or of (strict convexity on $\intefd{\Psi^\lfdual}$ and supercoercivity)  due to Proposition \ref{prop.left.pythagorean}.\ref{prop.left.pythagorean.i}. 
\item[(i')] \sloppy Alternatively, one may use \cite[Thm. 3.12]{Luo:Meng:Wen:Yao:2019} (relaxing its condition $\DG\Psi(\intefd{\Psi})=\intefd{\Psi^\lfdual}$ to one-sided inclusion from left to right, because we do not require characterisation of convexity of $\DG\Psi(K)$) and fulfill its condition of right $D_\Psi$-proximinality by Proposition \ref{prop.left.pythagorean}.\ref{prop.left.pythagorean.i} to obtain equivalence of Definition \ref{def.pythagorean}.\ref{def.pythagorean.b}.\ref{def.pythagorean.b.i} and \ref{def.pythagorean}.\ref{def.pythagorean.b}.\ref{def.pythagorean.b.ii}. \eqref{eqn.left.to.right} gives equivalence of \eqref{eqn.left.pyth.ii} and \eqref{eqn.left.pyth.iii}.
\item[(ii)] Follows from the assumption $K\subseteq\intefd{\Psi}$.
\item[(iii)] Follows directly from assumptions on $\Psi^\lfdual$ and Corollary \ref{cor.information}.
\item[(iv)] Follows from \eqref{eqn.right.pyth.ineq} in the same way as \eqref{eqn.left.pyth} follows from \eqref{eqn.left.pyth.iii}.
\end{enumerate}
\end{proof}

\begin{corollary}\label{corollary.dual}
Let $(X,\n{\cdot}_X)$ be reflexive, $\Psi\in\pclg(X,\n{\cdot}_X)$, $\Psi^\lfdual$ Gateaux differentiable on $\varnothing\neq\DG\Psi(\intefd{\Psi})\subseteq\intefd{\Psi^\lfdual}$. If $\varnothing\neq K\subseteq\intefd{\Psi}$ is convex, $\DG\Psi$-convex, closed, $\DG\Psi$-closed, and any of the following (generally, inequivalent) conditions holds, then $D_\Psi$ is left and right pythagorean on $K$:
\begin{enumerate}[nosep,label=\alph*)]
\item\label{corollary.dual.a} $\Psi$ is totally convex on $\efd(\Psi)$, $\Psi^\lfdual$ is strictly convex on $\efd(\Psi^\lfdual)$ and supercoercive; or
\item\label{corollary.dual.b} $\Psi$ is strictly convex on $\efd(\Psi)$ and supercoercive, $\Psi^\lfdual$ is totally convex on $\efd(\Psi^\lfdual)$; or
\item\label{corollary.dual.c} $\Psi$ is strictly convex on $\efd(\Psi)$ and supercoercive, $\Psi^\lfdual$ is strictly convex on $\efd(\Psi^\lfdual)$ and supercoercive; or
\item\label{corollary.dual.d} $\Psi$ is Euler--Legendre.
\end{enumerate}
Furthermore, if $\Psi$ is Fr\'{e}chet differentiable on $\intefd{\Psi}$, then an assumption of $\DG\Psi$-closure of $C$ is obsolete.
\end{corollary}
\begin{proof}
Follows directly from Propositions \ref{prop.left.pythagorean} and \ref{prop.right.pythagorean}. The case of ($\Psi$ totally convex on $\efd(\Psi)$ and $\Psi^\lfdual$ totally convex on $\efd(\Psi^\lfdual)$) reduces to d), because, for reflexive $(X,\n{\cdot}_X)$, total convexity of $\Psi$ on $\efd(\Psi)$ implies its essential strict convexity \cite[Prop. 2.1]{Resmerita:2004} \cite[Prop. 2.13]{Butnariu:Resmerita:2006} (the converse is not true \cite[p. 3]{Resmerita:2004}). For norm-to-norm continuity of $\DF\Psi$ on $\intefd{\Psi}$ see, e.g., \cite[Cor. 3.3.6]{Zalinescu:2002}.
\end{proof}

\begin{proposition}\label{prop.gen.pyth.thm.rprox}
Let $(X,\n{\cdot}_X)$ be reflexive, $\lambda\in\,]0,\infty[$, let $\Psi\in\pclg(X,\n{\cdot}_X)$ be Euler--Legendre, let $f:X\ra\,]-\infty,\infty]$ satisfy $f\circ\DG\Psi^\lfdual\in\pcl(X^\star,\n{\cdot}_{X^\star})$ and $\intefd{\Psi^\lfdual}\cap\efd(f\circ\DG\Psi^\lfdual)\neq\varnothing$. Then $\rprox^{D_\Psi}_{\lambda,f}$ is single-valued on $\intefd{\Psi}$ and satisfies
\begin{equation}
D_\Psi(x,y)\geq D_\Psi(x,\rprox^{D_\Psi}_{\lambda,f}(x))+D_\Psi(\rprox^{D_\Psi}_{\lambda,f}(x),y)\;\;\forall(x,y)\in\intefd{\Psi}\times\Fix(\rprox^{D_\Psi}_{\lambda,f}),
\label{eqn.gen.pyth.thm.rprox}
\end{equation}
where $\Fix(\rprox^{D_\Psi}_{\lambda,f})=\DG\Psi^\lfdual\circ(\intefd{\Psi^\lfdual}\cap\arginff{x\in X^\star}{f\circ\DG\Psi^\lfdual(x)})$.
\end{proposition}
\begin{proof}
\sloppy Follows from Proposition \ref{prop.res.pythagorean}.\ref{prop.res.pythagorean.iv}, by application of \eqref{eqn.bregman.fenchel.duality}, Proposition \ref{prop.prox.res.left.right}.\ref{prop.prox.res.left.right.ii}, $\Fix(\rprox^{D_\Psi}_{\lambda,f})=\DG\Psi^\lfdual(\Fix(\lprox^{D_{\Psi^\lfdual}}_{\lambda,f\circ\DG\Psi^\lfdual}))$ (which follows from \cite[Prop. 2.7.(iii)]{MartinMarquez:Reich:Sabach:2012}), and $\efd(\lprox^{D_{\Psi^\lfdual}}_{\lambda,f\circ\DG\Psi^\lfdual})=\intefd{\Psi^\lfdual}$ (which follows from \cite[Prop. 3.21.(vi)]{Bauschke:Borwein:Combettes:2003}).
\end{proof}

\begin{proposition}\label{prop.gen.pyth.thm.rres}
If $(X,\n{\cdot}_X)$ is reflexive and $\Psi\in\pclg(X,\n{\cdot}_X)$ is strictly convex on $\intefd{\Psi}$ and Euler--Legendre, then:
\begin{enumerate}[nosep,label=(\roman*)]
\item if $T:X\ra 2^{X^\star}$ is monotone and $\ran(\DG\Psi)\subseteq\ran(\DG\Psi+T)$, then $\rres^\Psi_T$ is single-valued on $\efd(\rres^\Psi_T)\subseteq\intefd{\Psi^\lfdual}$, and
\begin{equation}
D_{\Psi^\lfdual}(x,y)\geq D_{\Psi^\lfdual}(x,\rres^\Psi_T(x))+D_{\Psi^\lfdual}(\rres^\Psi_T(x),y)\;\;\forall(x,y)\in\intefd{\Psi^\lfdual}\times\Fix(\rres^\Psi_T),
\label{eqn.gen.pyth.thm.rres}
\end{equation}
with $\Fix(\rres^{\Psi}_T)=\DG\Psi(\intefd{\Psi}\cap T^\inver(0))$;
\item  if $T:X\ra 2^{X^\star}$ is maximally monotone, $\efd(T)\subseteq\intefd{\Psi}$, $\efd(\Psi^\lfdual)=X^\star$, and $\lambda\in\,]0,\infty[$, then $\rres^\Psi_{\lambda T}$ is single-valued on $\efd(\rres^\Psi_{\lambda T})\subseteq\intefd{\Psi^\lfdual}$, and
\begin{equation}
D_{\Psi^\lfdual}(x,y)\geq D_{\Psi^\lfdual}(x,\rres^\Psi_{\lambda T}(x))+D_{\Psi^\lfdual}(\rres^\Psi_{\lambda T}(x),y)\;\;\forall(x,y)\in\intefd{\Psi^\lfdual}\times\Fix(\rres^\Psi_{\lambda T}),
\label{eqn.gen.pyth.thm.rres}
\end{equation}
with $\Fix(\rres^{\Psi}_{\lambda T})=\DG\Psi(\intefd{\Psi}\cap T^\inver(0))$.
\end{enumerate}
\end{proposition}
\begin{proof}
Follows from Propositions \ref{prop.res.pythagorean}.\ref{prop.res.pythagorean.ii}, \ref{prop.res.pythagorean}.\ref{prop.res.pythagorean.iii}, and \ref{prop.res.pythagorean}.\ref{prop.res.pythagorean.v}, by application of \eqref{eqn.bregman.fenchel.duality}, Proposition \ref{prop.prox.res.left.right}.\ref{prop.prox.res.left.right.i}, and $\Fix(\rres^{\Psi}_{T})=\DG\Psi(\Fix(\lres^{\Psi}_{T}))$ (which follows from \cite[Prop. 2.7.(iii)]{MartinMarquez:Reich:Sabach:2012}).
\end{proof}

\begin{proposition}\label{prop.continuity.psi}
Let $(X,\n{\cdot}_X)$ be reflexive, $\varnothing\neq K\subseteq\intefd{\Psi}$, and $\Psi\in\pclg(X,\n{\cdot}_X)$ be Fr\'{e}chet differentiable on $\intefd{\Psi}$. Then:
\begin{enumerate}[nosep,label=(\roman*)]
\item\label{prop.continuity.psi.i} if $K$ is convex and closed, $\Psi$ is totally convex on $\efd(\Psi)$ and supercoercive, then $\LPPP^{D_\Psi}_K$ is norm-to-norm continuous on $\intefd{\Psi}$, while $\inf_{y\in K}\{D_\Psi(y,\,\cdot\,)\}$ is continuous on $\intefd{\Psi}$;
\item\label{prop.continuity.psi.ii} if $K$ is convex and closed, $\Psi$ is totally convex on bounded subsets of $X$, supercoercive, and Euler--Legendre, then $\LPPP^{D_\Psi}_K$ is norm-to-norm continuous on $\intefd{\Psi}$;
\item\label{prop.continuity.psi.iii} if $\DG\Psi(K)$ is convex and closed, $\varnothing\neq\DG\Psi(\intefd{\Psi})\subseteq\intefd{\Psi^\lfdual}$, $\Psi^\lfdual$ is totally convex on $\efd(\Psi^\lfdual)$, Fr\'{e}chet differentiable on $\intefd{\Psi^\lfdual}$, and supercoercive, then $\RPPP^{D_\Psi}_K$ is norm-to-norm continuous on $\intefd{\Psi}$;
\item\label{prop.continuity.psi.iv} if $\DG\Psi(K)$ is convex and closed, $\Psi^\lfdual$ is totally convex on bounded subsets of $X^\star$, Euler--Legendre, Fr\'{e}chet differentiable on $\intefd{\Psi^\lfdual}$, and supercoercive, then $\RPPP^{D_\Psi}_K$ is norm-to-norm continuous on $\intefd{\Psi}$.
\end{enumerate}
\end{proposition}
\begin{proof}
\begin{enumerate}[nosep]
\item[(i)] By \cite[Props. 4.2, 4.3]{Resmerita:2004}, if $\Psi\in\pcl(X,\n{\cdot}_X)$ is totally convex on $\efd(\Psi)$ and Fr\'{e}chet differentiable on $\intefd{\Psi}\neq\varnothing$, and if the set $\{y\in K\mid D_\Psi(x,y)\leq\lambda\}$ is bounded $\forall\lambda\in\,]0,\infty[$ (or, equivalently, by Corollary \ref{cor.information}, $\forall\lambda\in[0,\infty[$) $\forall y\in K$, then $x\mapsto\inf_{y\in K}\{D_\Psi(y,x)\}$ is continuous on $\intefd{\Psi}$ and $\LPPP^{D_\Psi}_K:\intefd{\Psi}\ra K$ is norm-to-norm continuous on $\intefd{\Psi}$. By \cite[Lem. 7.3.(vii)]{Bauschke:Borwein:Combettes:2001}, if $\Psi\in\pcl(X,\n{\cdot}_X)$ is supercoercive, and $x\in\intefd{\Psi}\neq\varnothing$, then $D^+_\Psi(x,\,\cdot\,)$ is coercive, which is equivalent (cf. \cite[Defs. 2.10, 4.1.(B3).(ii)]{Bauschke:Borwein:1997}) with boundedness of $\{y\in\intefd{\Psi}\mid D^+_\Psi(x,y)\leq\lambda\}$ $\forall\lambda\in[0,\infty[$ $\forall x\in\intefd{\Psi}$.
\item[\hypertarget{prop.continuity.psi.0}{(0)}] We will use the following fact: the Fr\'{e}chet differentiability of $\Psi\in\pclg(X,\n{\cdot}_X)$ on $\intefd{\Psi}$ (resp., $\Psi^\lfdual\in\pclg(X^\star,\n{\cdot}_{X^\star})$ on $\intefd{\Psi^\lfdual}$) is equivalent with norm-to-norm continuity of $\DG\Psi=\DF\Psi$ (resp.,  $\DG\Psi^\lfdual=\DF\Psi^\lfdual$) \cite[Prop. 2.8]{Phelps:1989} \cite[Cor. 3.3.6]{Zalinescu:2002}.
\item[(ii)] By \cite[Thm. 3.4.(ii)]{ZamaniEskandani:Azarmi:Raeisi:2020}, if $\Psi\in\pclg(X,\n{\cdot}_X)$ is totally convex on bounded subsets of $X$, supercoercive, and Euler--Legendre, and $\varnothing\neq K\subseteq\intefd{\Psi}$ is convex and closed, then the map \cite[Def. 1]{Pang:Naraghirad:Wen:2014}
\begin{equation}
\intefd{\Psi^\lfdual}\ni y\mapsto\arginff{x\in K}{D_\Psi(x,\DG\Psi^\lfdual(y))}\in K
\end{equation}
is norm-to-norm continuous. Setting $y=\DG\Psi(z)$ for $z\in\intefd{\Psi}$, and using Fr\'{e}chet differentiability of $\Psi$ on $\intefd{\Psi}$ together with \hyperlink{prop.continuity.psi.0}{(0)}, gives a result, since a composition of norm-to-norm continuous functions is norm-to-norm continuous. 
\item[(iii)] Follows from \ref{prop.continuity.psi.i} and an application of \eqref{eqn.left.to.right}, taking into account \hyperlink{prop.continuity.psi.0}{(0)}, so \eqref{eqn.left.to.right} is a composition of norm-to-norm continuous functions.
\item[(iv)] Follows from \ref{prop.continuity.psi.ii} in the same way as \ref{prop.continuity.psi.iii} follows from \ref{prop.continuity.psi.i}.
\end{enumerate}
\end{proof}

\begin{proposition}\label{prop.lsq.rsq.new}
If $(X,\n{\cdot}_X)$ is reflexive, $\Psi\in\pclg(X,\n{\cdot}_X)$, $\Psi^\lfdual$ is supercoercive, $\Psi$ is uniformly Fr\'{e}chet differentiable on bounded subsets of $\intefd{\Psi}$, then $\intefd{\Psi}=\efd(\Psi)=X$, and:
\begin{enumerate}[nosep,label=(\roman*)]
\item\label{prop.lsq.rsq.new.i} $\Psi$ is LSQ-compositional if any of (generally, inequivalent) conditions holds:
\begin{enumerate}[nosep,label=\alph*)]
\item\label{prop.lsq.rsq.new.i.a} $\Psi$ is totally convex on $X$; or
\item\label{prop.lsq.rsq.new.i.b} $\Psi$ is totally convex on bounded subsets of $X$ and is supercoercive;
\end{enumerate}
\item\label{prop.lsq.rsq.new.ii} $\Psi$ is RSQ-compositional if any of (generally, inequivalent) conditions holds:
\begin{enumerate}[nosep,label=\alph*)]
\item\label{prop.lsq.rsq.new.ii.a} $\Psi$ is totally convex on bounded subsets of $X$; or
\item\label{prop.lsq.rsq.new.ii.b} $\Psi$ is Euler--Legendre and supercoercive, $\Psi^\lfdual$ is (totally convex and uniformly Fr\'{e}chet differentiable) on bounded subsets of $\efd(\Psi^\lfdual)$;
\end{enumerate}
\item\label{prop.lsq.rsq.new.iii} $\Psi$ is LSQ-adapted on any closed and convex $\varnothing\neq K\subseteq X$ (and $\LPPP^{D_\Psi}_C$ is adapted for any $\varnothing\neq C\subseteq K$ with convex and closed $C$) if $\Psi$ is Euler--Legendre;
\item\label{prop.lsq.rsq.new.iv} $\Psi$ is RSQ-adapted on any $\varnothing\neq K\subseteq X$ (and $\RPPP^{D_\Psi}_C$ is adapted for any $\varnothing\neq C\subseteq K$ with convex and closed $\DG\Psi(C)$) if $\Psi$ is supercoercive and Euler--Legendre, and $\Psi^\lfdual$ is uniformly Fr\'{e}chet differentiable on bounded subsets of $\intefd{\Psi^\lfdual}\neq\varnothing$;
\item\label{prop.lsq.rsq.new.v} ($T\in\rsq(\Psi,K)$ if{}f $\DG\Psi\circ T\circ\DG\Psi^\lfdual\in\lsq(\Psi,\DG\Psi(K))$) if $\Psi$ is Euler--Legendre and supercoercive, $\varnothing\neq K\subseteq\intefd{\Psi}$, $T:K\ra\intefd{\Psi}$, and $\Psi^\lfdual$ is uniformly Fr\'{e}chet differentiable on bounded subsets of $\intefd{\Psi^\lfdual}\neq\varnothing$.
\end{enumerate}
\end{proposition}
\begin{proof}
\begin{enumerate}[nosep]
\item[\hypertarget{prop.lsq.rsq.new.0}{(0)}] For any $\Psi\in\pcl(X,\n{\cdot}_X)$, by \cite[Lem. 3.6.1]{Zalinescu:2002}, $\Psi^\lfdual$ is supercoercive if{}f ($\efd(\Psi)=X$ and $\Psi$ is bounded on bounded subsets). By \cite[Prop. 1.1.11]{Butnariu:Iusem:2000}, if $f:X\ra\RR$ is continuous and convex, then ($\partial f$ is bounded on bounded subsets of $\efd(\partial f)$ if{}f $f$ is bounded on bounded subsets of $X$). Hence, since $\intefd{\Psi}=\efd(\Psi)=X$ and $\Psi\in\pclg(X,\n{\cdot}_X)$, $\Psi$ is continuous and convex, and $\DG\Psi$ is bounded on bounded subsets of $\efd(\DG\Psi)=X$.
\item[(i).a)] By \cite[Prop. 3.6.3]{Zalinescu:2002} (cf. \cite[Prop. 2.1]{Reich:Sabach:2009}), $\Psi$ is (bounded and uniformly Fr\'{e}chet differentiable) on bounded subsets if{}f ($\Psi$ is Fr\'{e}chet differentiable on $X=\efd(\Psi)$ and $\DG\Psi$ is uniformly continuous on bounded subsets). By \cite[Prop. 2.3]{Butnariu:Iusem:Zalinescu:2003}, if $x\in\efd(\Psi)$ and $\Psi$ is Fr\'{e}chet differentiable at $x$, then $(\Psi$ is totally convex at $x$ if{}f $\Psi$ is uniformly convex at $x$). The rest follows from Proposition \ref{prop.lsq.rsq.old}.\ref{prop.lsq.rsq.old.i}.\ref{prop.lsq.rsq.old.i.a}, taking \hyperlink{prop.lsq.rsq.new.0}{(0)} into account.
\item[(i).b)] Since $\Psi^\lfdual\in\pcl(X,\n{\cdot}_X)$, it is (convex and) continuous on $\intefd{\Psi^\lfdual}=\efd(\Psi^\lfdual)=X^\star$ \cite[Cor. 7C]{Rockafellar:1966:level}. Hence, by \hyperlink{prop.lsq.rsq.new.0}{(0)}, applied to $\Psi^\lfdual$ instead of $\Psi$, $\DG(\Psi^\lfdual)$ is bounded on bounded subsets of $\efd(\Psi^\lfdual)=X^\star$ if{}f $\Psi$ is supercoercive. The rest follows from Proposition \ref{prop.lsq.rsq.old}.\ref{prop.lsq.rsq.old.i}.\ref{prop.lsq.rsq.old.i.b}.
\item[(ii)--(v)] Follow, respectively, from Proposition \ref{prop.lsq.rsq.old}.\ref{prop.lsq.rsq.old.ii}, \ref{prop.lsq.rsq.old}.\ref{prop.lsq.rsq.old.iii}, \ref{prop.lsq.rsq.old}.\ref{prop.lsq.rsq.old.iv}.\ref{prop.lsq.rsq.old.iv.b}, \ref{prop.lsq.rsq.old}.\ref{prop.lsq.rsq.old.v}, and Corollary \ref{cor.proj.adaptedness}, by the same technique as above.
\end{enumerate}
\end{proof}

\begin{proposition}\label{prop.norm.continuity.left.right.prox}
Let $\lambda\in\,]0,\infty[$, let $(X,\n{\cdot}_X)$ be reflexive, let $\Psi\in\pclg(X,\n{\cdot}_X)$ be strictly convex on $\intefd{\Psi}$, let $f\in\pcl(X,\n{\cdot}_X)$ be bounded from below, $\lim_{\n{x}_X\ra\infty}f(x)=\infty$, $\efd(f)\cap\efd(\Psi)\neq\varnothing$ and ($\efd(f)\cap\efd(\Psi)\subseteq\intefd{\Psi}$ or $\efd(\Psi)$ is open or $\efd(f)\subseteq\intefd{\Psi}$ or ($\intefd{\Psi}\cap\efd(f)\neq\varnothing$ and $\Psi$ is essentially Gateaux differentiable)). Then:
\begin{enumerate}[nosep,label=(\roman*)]
\item\label{prop.norm.continuity.left.right.prox.i} $\lprox^{D_\Psi}_{\lambda,f}$ is single-valued and norm-to-norm continuous on $X$;
\item\label{prop.norm.continuity.left.right.prox.ii} if $\Psi$ is Fr\'{e}chet differentiable on $\intefd{\Psi}$ and Euler--Legendre, and $\Psi^\lfdual$ is Fr\'{e}chet differentiable on $\intefd{\Psi^\lfdual}$, then $\rprox^{D_\Psi}_{\lambda,f}$ is single-valued and norm-to-norm continuous on $\intefd{\Psi}$.
\end{enumerate}
\end{proposition}
\begin{proof}
\begin{enumerate}[nosep]
\item[(i)] Follows from \cite[p. 186, Cor. 4.2]{Chen:Kan:Song:2012}, using the criteria for single-valuedness of $\lprox^{D_\Psi}_{\lambda,\Psi}$ provided by \cite[Props. 3.22.(ii).(d), 3.23]{Bauschke:Borwein:Combettes:2003}.
\item[(ii)] Follows from \ref{prop.norm.continuity.left.right.prox.i}, using Proposition \ref{prop.prox.res.left.right}.\ref{prop.prox.res.left.right.ii} together with the fact \hyperlink{prop.continuity.psi.0}{(0)} in the proof of Proposition \ref{prop.continuity.psi}.
\end{enumerate}
\end{proof}

\begin{proposition}\label{prop.lipschitz.hoelder.left.resolvent}
Let $(X,\n{\cdot}_X)$ be reflexive, $\lambda\in\,]0,\infty[$, $r\in\,]1,\infty[$, $s\in\,]0,1]$, let $\Psi\in\pclg(X,\n{\cdot}_X)$, $\efd(\Psi)=X$, and $\Psi^\lfdual\in\pclg(X^\star,\n{\cdot}_{X^\star})$. 
\begin{enumerate}[nosep,label=(\roman*)]
\item\label{prop.lipschitz.hoelder.left.resolvent.i} If $T:X\ra 2^{X^\star}$ is maximally monotone, $f\in\pcl(X,\n{\cdot}_X)$, and
 $\DG\Psi$ is $s$-Lipschitz--H\"{o}lder continuous on $X$ and $g$-uniformly monotone on $X$ with $g(t)=rt^{r-1}$, then $\lres^\Psi_{\lambda T}$ and $\lprox^{D_\Psi}_{\lambda,f}$ are single-valued and $\frac{s}{r-1}$-Lipschitz--H\"{o}lder continuous on $X$, while $\rres^\Psi_{\lambda T}$ is single-valued and $\frac{s^2}{(r-1)^2}$-Lipschitz--H\"{o}lder continuous on $X^\star$.
\item\label{prop.lipschitz.hoelder.left.resolvent.ii} If $f:X\ra\,]-\infty,\infty]$ satisfies $f\circ\DG\Psi^\lfdual\in\pcl(X^\star,\n{\cdot}_{X^\star})$, $w\in\,]0,1]$, $\DG\Psi$ is $w$-Lipschitz--H\"{o}lder continuous on $X$, and
 $\DG\Psi^\lfdual$ is $s$-Lipschitz--H\"{o}lder continuous on $X^\star$ and $g$-uniformly monotone on $X^\star$ with $g(t)=rt^{r-1}$, then $\rprox^{D_\Psi}_{\lambda,f}$ is single-valued and $\frac{s^2w}{r-1}$-Lipschitz--H\"{o}lder continuous on $X$.
\end{enumerate}
\end{proposition}
\begin{proof}
By \cite[Cor. 6.4]{Reem:Reich:2018}, $(\DG\Psi+\lambda T)^\inver$ is single-valued and $\frac{1}{r-1}$-Lipschitz--H\"{o}lder continuous on $X^\star$. Furthermore, we use Proposition \ref{prop.prox.res.left.right}.\ref{prop.prox.res.left.right.i}. The result for $\lres^{\Psi}_{\lambda T}$ and $\rres^{\Psi}_{\lambda T}$ follows from the fact that the composition of $r_1$-Lipschitz--H\"{o}lder map with $r_2$-Lipschitz--H\"{o}lder map, with both maps defined over all space, is $r_1r_2$-Lipschitz--H\"{o}lder map $\forall r_1,r_2\in\,]0,1]$. For $\lprox^{D_\Psi}_{\lambda,f}$ we use additionally Proposition \ref{prop.proj.prox.res}.\ref{prop.proj.prox.res.i}, while for $\rprox^{D_\Psi}_{\lambda,f}$ we use also Proposition \ref{prop.prox.res.left.right}.\ref{prop.prox.res.left.right.ii}.
\end{proof}

\begin{corollary}\label{cor.Lipschitz.Hoelder.projections.Psi}
Let $(X,\n{\cdot}_X)$ be reflexive, $r\in\,]1,\infty[$, $s\in\,]0,1]$, $\Psi\in\pclg(X,\n{\cdot}_X)$, $\efd(X)=\RR$, $\Psi^\lfdual\in\pclg(X^\star,\n{\cdot}_{X^\star})$, and $\varnothing\neq K\subseteq\intefd{\Psi}$. Then:
\begin{enumerate}[nosep,label=(\roman*)]
\item\label{cor.Lipschitz.Hoelder.projections.Psi.i} if $\DG\Psi$ is $s$-Lipschitz--H\"{o}lder continuous on $X$ and $g$-uniformly monotone on $X$ with $g(t)=rt^{r-1}$, and $K$ is convex and closed, then $\LPPP^{D_\Psi}_K$ is $\frac{s}{r-1}$-Lipschitz--H\"{o}lder continuous on $X$;
\item\label{cor.Lipschitz.Hoelder.projections.Psi.ii} if $w\in\,]0,1]$, $\DG\Psi$ is $w$-Lipschitz--H\"{o}lder continuous on $X$, $\DG\Psi^\lfdual$ is $s$-Lipschitz--H\"{o}lder continuous on $X^\star$ and $g$-uniformly monotone on $X^\star$ with $g(t)=rt^{r-1}$, $\varnothing\neq\DG(\intefd{\Psi})\subseteq\intefd{\Psi^\lfdual}$, and $\DG\Psi(K)\subseteq\intefd{\Psi^\lfdual}$ is convex and closed, then $\RPPP^{D_\Psi}_K$ is $\frac{s^2w}{r-1}$-Lipschitz--H\"{o}lder continuous on $X$.
\end{enumerate}
\end{corollary}
\begin{proof}
\begin{enumerate}[nosep]
\item[(i)] Follows from Propositions \ref{prop.lipschitz.hoelder.left.resolvent}.\ref{prop.lipschitz.hoelder.left.resolvent.i} and \ref{prop.proj.prox.res}.\ref{prop.proj.prox.res.ii}.
\item[(ii)] Follows from Propositions \ref{prop.lipschitz.hoelder.left.resolvent}.\ref{prop.lipschitz.hoelder.left.resolvent.ii}, \ref{prop.prox.res.left.right}.\ref{prop.prox.res.left.right.ii}, and \ref{prop.proj.prox.res}.\ref{prop.proj.prox.res.ii}.
\end{enumerate}
\end{proof}

\begin{lemma}\label{lemma.uFd.uc.duality}
If $\Psi\in\pcl(X,\n{\cdot}_X)$, as well as $\efd(\Psi)=X$ and $\efd(\Psi^\lfdual)\neq\{*\}$ (resp., $\efd(\Psi^\lfdual)=X^\star$ and $\efd(\Psi)\neq\{*\}$), then:
\begin{enumerate}[nosep,label=(\roman*)]
\item\label{lemma.uFd.uc.duality.i} $\Psi$ is uniformly Fr\'{e}chet differentiable (resp., uniformly convex) on $X$ if{}f $\Psi^\lfdual$ is uniformly convex (resp., uniformly Fr\'{e}chet differentiable) on $X^\star$;
\item\label{lemma.uFd.uc.duality.ii} if $\Psi$ is uniformly Fr\'{e}chet differentiable on bounded subsets (resp., uniformly convex on bounded subsets) of $X$, and supercoercive, then $\Psi^\lfdual$ is uniformly convex on bounded subsets (resp., uniformly Fr\'{e}chet differentiable on bounded subsets) of $X^\star$.
\end{enumerate}
\end{lemma}
\begin{proof}
Let $\Psi\in\pcl(X,\n{\cdot}_X)$. By \cite[Cor. 2.8]{Aze:Penot:1995}, if $\efd(\Psi)\neq\{*\}$ (resp., $\efd(\Psi^\lfdual)\neq\{*\}$), then $\Psi$ is uniformly convex (resp., uniformly Gateaux differentiable) on $X$ if{}f $\Psi^\lfdual$ is uniformly Gateaux differentiable (resp., uniformly convex) on $X^\star$. By \cite[p. 207]{Zalinescu:2002} (cf. \cite[p. 4]{Shioji:1994}(=\cite[p. 643]{Shioji:1995})), $\Psi$ is uniformly Fr\'{e}chet differentiable on any $\varnothing\neq K\subseteq\efd(\Psi)$ if{}f $\Psi$ is uniformly Gateaux differentiable on $K$. Setting $\efd(\Psi)=X$ gives \ref{lemma.uFd.uc.duality.i}. \ref{lemma.uFd.uc.duality.ii} follows directly from \ref{lemma.uFd.uc.duality.i} and \cite[Prop. 3.6.2.(i)]{Zalinescu:2002}.
\end{proof}

\begin{proposition}\label{prop.alber.decomposition}
Let $(X,\n{\cdot}_X)$ be reflexive, $\varnothing\neq K\subseteq X$, $\Psi\in\pclg(X,\n{\cdot}_X)$ be supercoercive and strictly convex, $\Psi^\lfdual\in\pclg(X^\star,\n{\cdot}_{X^\star})$ be supercoercive and strictly convex. For any convex set $C\subseteq X$, let $C^\circ:=\{y\in X^\star\mid\duality{x,y}_{X\times X^\star}\leq0\;\;\forall x\in C\}$. Then:
\begin{enumerate}[nosep,label=(\roman*)]
\item\label{prop.alber.decomposition.i} if $\Psi:X\ra\RR$, $\Psi^\lfdual(0)=0$, $(\DG\Psi^\lfdual)(0)=0$, $(\DG\Psi^\lfdual)(-y)=-(\DG\Psi^\lfdual)(y)$ $\forall y\in X^\star$, $\varnothing\neq K\subset X$ is a closed convex cone with a vertex at $0\in X$, then
\begin{equation}
\forall x\in X\;\;
\left\{
\begin{array}{l}
x = (\DG\Psi)^{\inver}\circ\hat{\PPP}^{\Psi^\lfdual}_{K^\circ}\circ\DG\Psi(x)+\LPPP^{D_\Psi}_K(x)\\
\duality{\LPPP^{D_\Psi}_K(x),\hat{\PPP}^{\Psi^\lfdual}_{K^\circ}\circ\DG\Psi(x)}_{X\times X^\star}=0,
\end{array}
\right.
\label{eqn.left.alber.decomposition}
\end{equation}
where 
\begin{equation}
\left\{
\begin{array}{l}
\hat{\PPP}^{\Psi^\lfdual}_{K^\circ}(y):=\arginff{z\in K^\circ}{\Psi^\lfdual(y-z)}\;\forall y\in X^\star\\
\hat{\PPP}^{\Psi^\lfdual}_{K^\circ}\circ\hat{\PPP}^{\Psi^\lfdual}_{K^\circ}(y)=\hat{\PPP}^{\Psi^\lfdual}_{K^\circ}(y)\;\forall y\in X^\star;
\end{array}
\right.
\end{equation}
\item\label{prop.alber.decomposition.ii} if $\Psi^\lfdual:X^\star\ra\RR$, $\Psi(0)=0$, $(\DG\Psi)(0)=0$, $(\DG\Psi)(-x)=-(\DG\Psi)(x)$ $\forall x\in X$, $\varnothing\neq\DG\Psi(K)\subset X^\star$ is a closed convex cone with a vertex at $0\in X^\star$, then
\begin{equation}
\forall y\in X\;\;
\left\{
\begin{array}{l}
y=\hat{\PPP}^\Psi_{(\DG\Psi(K))^\circ}(y)+\RPPP^{D_\Psi}_K(y)\\
\duality{(\DG\Psi^\lfdual)^{\inver}\circ\RPPP^{D_\Psi}_K(y),\hat{\PPP}^\Psi_{(\DG\Psi(K))^\circ}(y)}_{X\times X^\star}=0.
\end{array}
\right.
\label{eqn.right.alber.decomposition}
\end{equation}
\end{enumerate}
Furthermore: 
\begin{enumerate}[nosep,label=(\roman*)]
\setcounter{enumi}{2}
\item\label{prop.alber.decomposition.iii} if $\varnothing\neq K\subseteq X$ is a linear subspace instead of a closed convex cone, then \eqref{eqn.left.alber.decomposition} holds under replacement of $K^\circ$ with $K^\bot:=\{y\in X^\star\mid\duality{x,y}_{X\times X^\star}=0\;\forall x\in K\}$;
\item\label{prop.alber.decomposition.iv} if $\varnothing\neq\DG\Psi(K)\subseteq X^\star$ is a linear subspace  instead of a closed convex cone, then \eqref{eqn.right.alber.decomposition} holds under replacement of $(\DG\Psi(K))^\circ$ with $(\DG\Psi(K))^\bot$.
\end{enumerate}
\end{proposition}
\begin{proof}
\begin{enumerate}[nosep]
\item[(i)] This is \cite[Thm. 3.19]{Alber:2007}.
\item[(iii)] This is  \cite[Rem. 3.20]{Alber:2007}.
\item[(ii)+(iv)] Follows directly from \ref{prop.alber.decomposition.i} and \ref{prop.alber.decomposition.iii} combined with Proposition \ref{prop.Luo.Meng.Wen.Yao}.
\end{enumerate}
\end{proof}

\begin{remark}\label{remark.psi.results}
\begin{enumerate}[nosep,label=(\roman*)]
\item\label{remark.psi.results.i} Proposition \ref{prop.right.pythagorean}.\ref{prop.right.pythagorean.i} is a generalisation of \cite[Prop. 4.11]{MartinMarquez:Reich:Sabach:2012}, as indicated in the proof.
\item\label{remark.psi.results.ii} Regarding Corollary \ref{corollary.dual}, left and right projections onto sets which are both convex and $\DG\Psi$-convex (despite that $\DG\Psi$ is not an affine map) were considered earlier in \cite[p. 11, Probl. 3]{Bauschke:Macklem:Wang:2011}.
\item\label{remark.psi.results.iii} Proposition \ref{prop.continuity.psi}.\ref{prop.continuity.psi.i} provides a (new) special case of \cite[Props. 4.2, 4.3]{Resmerita:2004}, which is more suitable for our purposes, because it allows to derive Proposition \ref{prop.continuity.psi}.\ref{prop.continuity.psi.iii}, as well as Propositions \ref{prop.continuity}.\ref{prop.continuity.psi.ii} and \ref{prop.continuity}.\ref{prop.continuity.psi.iii}. $\Psi\in\pclg(X,\n{\cdot}_X)$ totally convex on $\efd(\Psi)$, such that $D_\Psi(x,\,\cdot\,)$ is coercive on $\intefd{\Psi}$ $\forall x\in\efd(\Psi)$ is called a \textit{Br\`{e}gman function} in \cite[Def. 2.1.1]{Butnariu:Iusem:2000}, and it provides a Banach space generalisation (and also a weakening) of the notion of Br\`{e}gman function introduced in \cite[Def. 2.1]{Censor:Lent:1981} (cf. also \cite[\S4]{Bauschke:Borwein:1997}). See \cite[Def. 4.2]{Reem:Reich:DePierro:2019} for further generalisation and discussion of this notion.
\item\label{remark.psi.results.iv} The direct relationship between the differing assumptions of Propositions \ref{prop.continuity.psi}.\ref{prop.continuity.psi.i} (resp., \ref{prop.continuity.psi}.\ref{prop.continuity.psi.iii}) and \ref{prop.continuity.psi}.\ref{prop.continuity.psi.ii} (resp., \ref{prop.continuity.psi}.\ref{prop.continuity.psi.iv}) is not clear at this level of generality. However, in a special case of $\Psi=\Psi_\varphi$, the former variants are essentially more general than the latter, see Remark \ref{remark.varphi}.\ref{remark.varphi.vi}.
\item\label{remark.psi.results.v} Proposition \ref{prop.continuity.psi}.\ref{prop.continuity.psi.iii} is essentially new in the Banach space setting. For $X=\RR^n$, $\varnothing\neq K\subseteq X$ convex closed, $K\cap\intefd{\Psi}\neq\varnothing$, $\Psi$ Euler--Legendre, $\Psi\in\mathrm{C}^2(\intefd{\Psi})$, $D_\Psi$ jointly convex, $D_\Psi(x,\,\cdot\,)$ strictly convex on $\intefd{\Psi}$ and coercive $\forall x\in\intefd{\Psi}$, norm-to-norm continuity of $\RPPP^{D_\Psi}_C$ has been established in \cite[Cor. 3.7]{Bauschke:Noll:2002}. While coerciveness of $D_\Psi(x,\,\cdot\,)$ follows from supercoerciveness of $\Psi$ \cite[Lem. 7.3.(viii)]{Bauschke:Borwein:Combettes:2001}, the rest of these conditions is noticeably different from the assumptions of Proposition \ref{prop.continuity.psi}.\ref{prop.continuity.psi.iii}.
\item\label{remark.psi.results.vi} The reason why Proposition \ref{prop.lsq.rsq.new}.\ref{prop.lsq.rsq.new.iv} omits case \ref{prop.lsq.rsq.old.iv.a} of Proposition \ref{prop.lsq.rsq.old}.\ref{prop.lsq.rsq.old.iv} will be explained in the Remark \ref{remark.varphi}.\ref{remark.varphi.x}.
\item\label{remark.psi.results.vii} Generalised pythagorean equations \eqref{eqn.right.pyth.eqn} and \eqref{eqn.left.pyth} are special cases of the generalised cosine equation \eqref{eqn.generalised.cosine}, obtained for
\begin{equation}
\duality{x-\RPPP^{D_\Psi}_K(x),\DG\Psi(y)-\DG\Psi(\RPPP^{D_\Psi}_K(x))}_{X\times X^\star}=0\;\;\forall(x,y)\in\intefd{\Psi}\times K,
\label{eqn.general.orthogonality.rppp.condition}
\end{equation}
and
\begin{equation}
\duality{x-\LPPP^{D_\Psi}_K(y),\DG\Psi(y)-\DG\Psi(\LPPP^{D_\Psi}_K(y))}_{X\times X^\star}=0\;\;\forall(x,y)\in K\times\intefd{\Psi},
\label{eqn.general.orthogonality.lppp.condition}
\end{equation}
respectively. One can see \eqref{eqn.general.orthogonality.rppp.condition}--\eqref{eqn.general.orthogonality.lppp.condition} as the conditions of orthogonality (at $\RPPP^{D_\Psi}_K(x)$ and $\LPPP^{D_\Psi}_K(y)$, respectively) between a vector joining the projected point with its projection, and a vector ranging from a projection into an arbitrary point within the constraint set $K$.
\item\label{remark.psi.results.viii} As compared to the original phrasing of \cite[Thm. 3.19]{Alber:2007}, Proposition \ref{prop.alber.decomposition}.\ref{prop.alber.decomposition.i} assumes additionally strict convexity of $\Psi^\lfdual$, since, by the definition of $\hat{\PPP}^{\Psi^\lfdual}$ \cite[Def. 3.1]{Alber:2007}, this condition is necessary for the uniqueness of $\hat{\PPP}^{\Psi^\lfdual}_{K^\circ}(y)$. Proposition \ref{prop.alber.decomposition}.\ref{prop.alber.decomposition.ii} is new.
\item\label{remark.psi.results.viii} For $n\in\NN$, $X=\RR^n$, Euler--Legendre $\Psi\in\pclg(X,\n{\cdot}_X)$, $f\in\pcl(X,\n{\cdot}_X)$, $\efd(f)\cap\intefd{\Psi}\neq\varnothing$, and some additional conditions on $\Psi$, the (norm-to-norm) continuity of $\lprox^{D_\Psi}_{1,f}$ and $\rprox^{D_\Psi}_{1,f}$ was established in \cite[Prop. 3.10]{Bauschke:Combettes:Noll:2006}.
\end{enumerate}
\end{remark}
\subsection{$D_\Psi$ with $\Psi=\Psi_\varphi$}\label{section.convex.new.varphi}
\begin{proposition}\label{prop.supercoercive.Psi.varphi}
For any gauge $\varphi$ and any Banach space $(X,\n{\cdot}_X)$, $\Psi_\varphi$ is supercoercive.
\end{proposition}
\begin{proof}
Let $\Psi\in\pcl(X,\n{\cdot}_X)$. By \cite[Lem. 3.6.1]{Zalinescu:2002}, $\Psi$ is supercoercive if{}f ($\efd(\Psi^\lfdual)=X^\star$ and $\Psi^\lfdual$ is bounded on bounded subsets). If $\intefd{\Psi^\lfdual}=X^\star$, then $\Psi^\lfdual$ is continuous on $X^\star$ \cite[Cor. 7C]{Rockafellar:1966:level}. If $\Psi^\lfdual$ is continuous on $X^\star$, then, by \cite[Prop. 1.1.11]{Butnariu:Iusem:2000}, $\Psi^\lfdual$ is bounded on bounded subsets if{}f $\partial(\Psi^\lfdual)$ is bounded on bounded subsets. By \cite[Thm. 3.7.2.(ii)]{Zalinescu:2002}, 
\begin{equation}
(\Psi_\varphi)^\lfdual(y)=\int_0^{\n{y}_{X^\star}}\dd t\,\varphi^{\inver}(t)\;\;\forall y\in X^\star,
\label{eqn.zalinescu.inverse.fenchel.dual}
\end{equation}
so $(\Psi_\varphi)^\lfdual$ has the same properties as $\Psi_\varphi$, since $\varphi^{\inver}$ is a gauge function (cf. \cite[p. 227]{Zalinescu:2002}). In particular, $(\Psi_\varphi)^\lfdual$ is convex and continuous on $X^\star$, with $\efd((\Psi_\varphi)^\lfdual)=X^\star=\intefd{(\Psi_\varphi)^\lfdual}$. Finally, from definition \eqref{eqn.j.varphi.definition} of $j_\varphi$, it follows:
\begin{equation}
\exists \lambda>0\;\forall z\in Z\subsetneq X^\star\;\;\n{z}_{X^\star}\leq\lambda\;\;\limp\;\;\forall y\in j_{\varphi^{\inver}}^\star(z)\;\;\n{y}_{X^\star{}^\star}=\varphi^{\inver}(\n{z}_{X^\star})\leq \varphi^{\inver}(\lambda),
\label{eqn.bounded.j.varphi.inver}
\end{equation}
where the last inequality holds since $\varphi^{\inver}$ is nondecreasing. Hence, $j_{\varphi^{\inver}}^\star$ maps bounded sets to bounded sets (cf., e.g., \cite[p. 176]{Lions:1969}). Since $j_{\varphi^{\inver}}^\star=\partial((\Psi_\varphi)^\lfdual)$ by \eqref{eqn.asplund} and \eqref{eqn.zalinescu.inverse.fenchel.dual}, this completes the proof.	
\end{proof}

\begin{proposition}\label{prop.legendre}
For any gauge $\varphi$, $\Psi_\varphi$ is Euler--Legendre if{}f $(X,\n{\cdot}_X)$ is strictly convex and Gateaux differentiable.
\end{proposition}
\begin{proof} 
\begin{enumerate}[nosep]
\item[1)] By Proposition \ref{prop.psi.varphi.geometry}.\ref{prop.psi.varphi.geometry.iii}, $(X,\n{\cdot}_X)$ is Gateaux differentiable if{}f $j_\varphi$ is single-valued on $X$. By \cite[Thm. 5.6.(i)-(ii)]{Bauschke:Borwein:Combettes:2001}, $\Psi_\varphi$ is essentially Gateaux differentiable if{}f ($\intefd{\Psi_\varphi}\neq\varnothing$ and $\partial\Psi_\varphi$ is single-valued on $\efd(\partial\Psi_\varphi)$).  By \cite[Thm. 1]{Asplund:1967}, $j_\varphi=\partial\Psi_\varphi$.  Since $\efd(\partial\Psi_\varphi)=X$ \cite[Obs. I.3.1]{Cioranescu:1974} and $\efd(\Psi_\varphi)=\intefd{\Psi_\varphi}=X$ for any gauge $\varphi$, we obtain: $(X,\n{\cdot}_X)$ is Gateaux differentiable if{}f $\Psi_\varphi$ is essentially Gateaux differentiable.
\item[2)] By Proposition \ref{prop.psi.varphi.geometry}.\ref{prop.psi.varphi.geometry.ii}, $(X,\n{\cdot}_X)$ is strictly convex if{}f $\Psi_\varphi$ is strictly convex. By \cite[Lemma 5.8]{Bauschke:Borwein:Combettes:2001}, if $\efd(\partial\Psi_\varphi)$ and $\efd((\Psi_\varphi)^\lfdual)$ are open, then ($\Psi_\varphi$ is strictly convex on $\intefd{\Psi_\varphi}$ if{}f $\Psi_\varphi$ is essentially strictly convex). By \cite[Obs. I.3.1]{Cioranescu:1974}, $\efd(\partial\Psi_\varphi)=X$. Furthermore, $\efd(\Psi_\varphi)=\intefd{\Psi_\varphi}=X$. From \eqref{eqn.zalinescu.inverse.fenchel.dual} it it follows that $\efd((\Psi_\varphi)^\lfdual)=X$, which is an open set, since every Banach space is (both a closed and) an open set.
\end{enumerate}
\end{proof}

\begin{corollary}\label{cor.superc.EL.totconv.varphi}
For any gauge $\varphi$ and any Banach space $(X,\n{\cdot}_X)$: 
\begin{enumerate}[nosep,label=(\roman*)]
\item\label{cor.superc.EL.totconv.varphi.i} if $(X,\n{\cdot}_X)$ is Gateaux differentiable, then
\begin{equation}
D_{\Psi_\varphi}(x,y)=\int_0^{\n{x}_X}\dd t\,\varphi(t)+\int_0^{\n{j_\varphi(y)}_{X^\star}}\dd t\,\varphi^\inver(t)-\duality{x,j_\varphi(y)}_{X\times X^\star}\;\forall x,y\in X;
\label{eqn.D.Psi.varphi.formula}
\end{equation}
\item\label{cor.superc.EL.totconv.varphi.ii} if $(X,\n{\cdot}_X)$ is Gateaux differentiable and $\Psi_\varphi$ is totally convex, then $\Psi_\varphi$ is Euler--Legendre;
\item\label{cor.superc.EL.totconv.varphi.iii} in particular, if $(X,\n{\cdot}_X)$ is locally uniformly convex and Gateaux differentiable, then $\Psi_\varphi$ is Euler--Legendre and totally convex.
\end{enumerate}
\end{corollary}
\begin{proof}
\begin{enumerate}[nosep]
\item[(i)] Follows from \eqref{eqn.bregman.fenchel}, \eqref{eqn.zalinescu.inverse.fenchel.dual}, and Proposition \ref{prop.psi.varphi.geometry}.\ref{prop.psi.varphi.geometry.iii}.
\item[(ii)] Follows from Proposition \ref{prop.legendre} combined with the fact \cite[Thm. 3.1, Cor. 3.4]{Resmerita:2004} that total convexity of $\Psi$ implies strict convexity of $(X,\n{\cdot}_X)$.
\item[(iii)] Follows from the fact that local uniform convexity of $(X,\n{\cdot}_X)$ implies its strict convexity, combined with Propositions \ref{prop.psi.varphi.geometry}.\ref{prop.psi.varphi.geometry.xii} and \ref{prop.legendre}.
\end{enumerate}
\end{proof}

\begin{corollary}\label{cor.left.D.psi.varphi.chebyshev.characterisation}
For any gauge $\varphi$, if $(X,\n{\cdot}_X)$ is reflexive, strictly convex, and Gateaux differentiable, and $\varnothing\neq K\subseteq X$ is weakly closed, then $K$ is left $D_{\Psi_\varphi}$-Chebysh\"{e}v if{}f $K$ is convex.
\end{corollary}
\begin{proof}
Follows from \cite[Cor. 1]{Volle:HiriartUrruty:2012} combined with Proposition \ref{prop.legendre}.
\end{proof}

\begin{proposition}\label{prop.left.right.psi.varphi}
For any gauge $\varphi$, if $\varnothing\neq K\subseteq X$, and $(X,\n{\cdot}_X)$ is reflexive, strictly convex, and Gateaux differentiable, then:
\begin{enumerate}[nosep,label=(\roman*)]
\item\label{prop.left.right.psi.varphi.i} if $K$ is convex and closed, then $K$ is left $D_{\Psi_\varphi}$-Chebysh\"{e}v, and $D_{\Psi_\varphi}$ is left pythagorean on $K$;
\item\label{prop.left.right.psi.varphi.ii} if $j_\varphi(K)$ is convex and closed, then $K$ is right $D_{\Psi_\varphi}$-Chebysh\"{e}v, and $D_{\Psi_\varphi}$ is right pythagorean on $K$;
\item\label{prop.left.right.psi.varphi.iii} $D_{\Psi_\varphi}$ (resp., $D_{\Psi_\varphi^\lfdual}$) is an information on $X$ (resp., $X^\star$);
\item\label{prop.left.right.psi.varphi.iv} $\LPPP^{D_{\Psi_\varphi}}_K$ and $\RPPP^{D_{\Psi_\varphi}}_K$ are zone consistent.
\end{enumerate}
\end{proposition}
\begin{proof}
\begin{enumerate}[nosep]
\item[(i)] By Proposition \ref{prop.supercoercive.Psi.varphi}, $\Psi_\varphi$ is supercoercive for any gauge $\varphi$. For any gauge $\varphi$, Proposition \ref{prop.psi.varphi.geometry}.\ref{prop.psi.varphi.geometry.xii} gives that $\Psi_{\varphi}$ is totally convex on any locally uniformly convex Banach space $(X,\n{\cdot}_X)$. The latter implies strict convexity, and the opposite implication is not true in general. Furthermore, by Proposition \ref{prop.psi.varphi.geometry}.\ref{prop.psi.varphi.geometry.xi}, if $(X,\n{\cdot}_X)$ is reflexive, then (it is strictly convex and has the Radon--Riesz--Shmul'yan property) if{}f $\Psi_\varphi$ is totally convex for any gauge $\varphi$. Hence, when applied to $\Psi_\varphi$ (and taking into account Proposition \ref{prop.legendre}), the weakest conditions to be assumed in Proposition \ref{prop.left.pythagorean}.\ref{prop.left.pythagorean.ii} are provided in Proposition \ref{prop.left.pythagorean}.\ref{prop.left.pythagorean.i}.\ref{prop.left.pythagorean.i.b} and \ref{prop.left.pythagorean}.\ref{prop.left.pythagorean.i}.\ref{prop.left.pythagorean.i.c}, which turn out to be equivalent in this situation.
\item[(ii)] Follows from \ref{prop.left.right.psi.varphi.i} and Proposition \ref{prop.right.pythagorean}.\ref{prop.right.pythagorean.i}, taking into account that, for reflexive $(X,\n{\cdot}_X)$, Gateaux differentiability of $(X,\n{\cdot}_X)$ (resp., $(X^\star,\n{\cdot}_{X^\star})$) implies strict convexity of $(X^\star,\n{\cdot}_{X^\star})$ (resp., $(X,\n{\cdot}_X)$) \cite[A.1.1]{Klee:1953}.
\item[(iii)] Follows from Corollary \ref{cor.information} and Proposition \ref{prop.right.pythagorean}.\ref{prop.right.pythagorean.iii}.
\item[(iv)] Follows from $\intefd{\Psi_\varphi}=X$.
\end{enumerate}
\end{proof}

\begin{lemma}\label{lem.quasigauge.inverses}
If $\varphi$ is a quasigauge, then:
\begin{enumerate}[nosep,label=(\roman*)]
\item\label{lem.quasigauge.inverses.i} $\varphi^\join$, $\varphi^\meet$, $(t\mapsto\lim_{s\ra^-t}\varphi(s))^\join$, and $(t\mapsto\lim_{s\ra^+t}\varphi(s))^\meet$ are quasigauges;
\item\label{lem.quasigauge.inverses.ii} $\varphi^\join$ (resp., $\varphi^\meet$) is left (resp., right) continuous.
\end{enumerate}
\end{lemma}
\begin{proof}
\begin{enumerate}[nosep]
\item[(i)] Nondecreasing of $f^\join$ and $f^\meet$ holds for any $f:\RR^+\ra[0,\infty]$ (cf., e.g., \cite[Lem. 2.3.9.a)]{Harjulehto:Haestoe:2019} for $f^\join$ and \cite[Lem. 1.(b)]{Wacker:2023}\footnote{This lemma is stated for $f:\RR^+\ra\RR^+$, however the extension of a proof to $f:\RR^+\ra[0,\infty]$ by an analogy with the proof of  \cite[Lem. 2.3.9]{Harjulehto:Haestoe:2019} is straightforward.}\addtocounter{footnote}{-1}\addtocounter{Hfootnote}{-1}). This implies existence of left and right limits of $f^\join(t)$ and $f^\meet(t)$ at any $t\in\RR^+$. $\varphi\not\equiv0$ (resp., $\varphi\not\equiv\infty$) implies $\varphi^\join\not\equiv\infty\not\equiv\varphi^\meet$ (resp., $\varphi^\join\not\equiv0\not\equiv\varphi^\meet$). Thus, $\exists s,t>0$ such that $\lim_{u\ra^+s}\varphi^\join(u)<\infty$ and $\lim_{u\ra^+t}\varphi^\meet(u)<\infty$. The same reasoning applies to $(t\mapsto\lim_{s\ra^-t}\varphi(s))^\join$ and $(t\mapsto\lim_{s\ra^+t}\varphi(s))^\meet$.
\item[(ii)] This holds for any nondecreasing $f:\RR^+\ra[0,\infty]$, cf., e.g., \cite[Lem. 2.3.9.c)]{Harjulehto:Haestoe:2019} for $f^\join$ and \cite[Lem. 1.(c)]{Wacker:2023}\footnotemark\ for $f^\meet$.
\end{enumerate}
\end{proof}

\begin{lemma}\label{lem.dual.quasigauge.integral}
If $\varphi$ is a quasigauge, and $f_\varphi(u):=\int_0^u\dd t\,\varphi(t)$ $\forall u\in\RR^+$, then
\begin{equation}
(f_\varphi)^\young(u)=\int_0^u\dd t\,(\lim_{s\ra^+t}\varphi(s))^\meet=\int_0^u\dd t\,(\lim_{s\ra^-t}\varphi(s))^\join\;\;\forall u\in\RR^+.
\label{eqn.dual.quasigauge.integral.limits}
\end{equation}
If, furthermore, $\varphi$ is right (resp., left) continuous on $\RR^+$, then\footnote{\eqref{eqn.dual.quasigauge.integral.meet} has been proved earlier, by a different method, in \cite[Thm. 2.11]{Shi:Shi:2019}.}
\begin{align}
(f_\varphi)^\young(u)&=\int_0^u\dd t\,\varphi^\meet(t)\;\;\forall u\in\RR^+
\label{eqn.dual.quasigauge.integral.meet}\\
\mbox{(resp., }(f_\varphi)^\young(u)&=\int_0^u\dd t\,\varphi^\join(t)\;\;\forall u\in\RR^+\mbox{)}.
\label{eqn.dual.quasigauge.integral.join}
\end{align}
\end{lemma}
\begin{proof}
Since $f_\varphi$ and $(f_\varphi)^\young$ are proper, convex, lower semicontinuous functions, taking value 0 at 0 \cite[pp. 367--368]{Zalinescu:1983}, we can use the representation \cite[Thm. 24.2]{Rockafellar:1970} of such type of functions, $g(u)=\int_0^u\dd t\,g_+'(t)=\int_0^u\dd t\,g_-'(t)$ $\forall u\in\RR^+$, and combine it with \cite[Prop. A.2.(i)]{Zalinescu:1983}
\begin{align}
((f_\varphi)^\young)_+'(t)&=(\lim_{s\ra^+t}\varphi(s))^\meet,\\
((f_\varphi)^\young)_-'(t)&=(\lim_{s\ra^-t}\varphi(s))^\join.
\end{align}
\end{proof}

\begin{lemma}\label{lem.fenchel.dual.psi.quasigauge}
For any quasigauge $\varphi$ and any Banach space $(X,\n{\cdot}_X)$:
\begin{enumerate}[nosep,label=(\roman*)]
\item\label{lem.fenchel.dual.psi.quasigauge.i} $(\Psi_\varphi)^\lfdual(y)=\int_0^{\n{y}_{X^\star}}\dd t\,(\lim_{s\ra^+t}\varphi(s))^\meet=\int_0^{\n{y}_{X^\star}}\dd t\,(\lim_{s\ra^-t}\varphi(s))^\join$ $\forall y\in X^\star$;
\item\label{lem.fenchel.dual.psi.quasigauge.ii} if $\varphi$ is right (resp., left) continuous on $\RR^+$, then
\begin{align}
(\Psi_\varphi)^\lfdual(y)&=\int_0^{\n{y}_{X^\star}}\dd t\,\varphi^\meet(t)\;\;\forall y\in X^\star
\label{eqn.fenchel.dual.psi.quasigauge.meet}\\
\mbox{(resp., }(\Psi_\varphi)^\lfdual(y)&=\int_0^{\n{y}_{X^\star}}\dd t\,\varphi^\join(t)\;\;\forall y\in X^\star\mbox{)},
\label{eqn.fenchel.dual.psi.quasigauge.join}\\
((\Psi_\varphi)^\lfdual)^\lfdual(z)&=\int_0^{\n{z}_{X^\star{}^\star}}\dd t\,(\varphi^\meet)^\meet(t)\;\;\forall z\in X^\star{}^\star\\
\mbox{(resp., }((\Psi_\varphi)^\lfdual)^\lfdual(z)&=\int_0^{\n{z}_{X^\star{}^\star}}\dd t\,(\varphi^\join)^\join(t)\;\;\forall z\in X^\star{}^\star\mbox{)},
\end{align}
and $\partial((\Psi_\varphi)^\lfdual)=j_{\varphi^\meet}^\star$ (resp., $\partial((\Psi_\varphi)^\lfdual)=j_{\varphi^\join}^\star$);
\item\label{lem.fenchel.dual.psi.quasigauge.iii} if $\varphi$ is right (resp., left) continuous on $\RR^+$ and $(X,\n{\cdot}_X)$ is reflexive, then $j_{(\varphi^\meet)^\meet}^{\star\star}=j_\varphi$ (resp., $j_{(\varphi^\join)^\join}^{\star\star}=j_\varphi$).
\end{enumerate}
\end{lemma}
\begin{proof}{\tiny\ \\}{\vspace{-0.4cm}}
\begin{enumerate}[nosep]
\item[(i)--(ii)] By \cite[Eqn. (A.6)]{Zalinescu:1983}, $(\Psi_\varphi)^\lfdual(y)=(f_\varphi)^\young(\n{y}_{X^\star})$ $\forall y\in X^\star$, where $f_\varphi(u)=\int_0^u\dd t\,\varphi(t)$ $\forall u\in\RR^+$. The rest follows from Lemmas \ref{lem.quasigauge.inverses} and \ref{lem.dual.quasigauge.integral}.
\item[(iii)] For any $f:\RR^+\ra[0,\infty]$, $f$ is (nondecreasing and right (resp., left) continuous) if{}f $f^{\meet\meet}=f$ (resp., $f^{\join\join}=f$) \cite[Lem. 2.4]{Shi:Shi:2019}\footnote{This lemma states only an implication from left to right, however an implication in the opposite direction can be provided by a direct analogue of the corresponding proof in \cite[Lem. 2.3.11]{Harjulehto:Haestoe:2019}.} (resp., \cite[Lem. 2.3.11]{Harjulehto:Haestoe:2019}). On the other hand, reflexivity of $(X,\n{\cdot}_X)$ and $\Psi_\varphi\in\pcl(X,\n{\cdot}_X)$ give $((\Psi_\varphi)^\lfdual)^\lfdual=\Psi_\varphi$. Hence, $j_{(\varphi^\meet)^\meet}^{\star\star}=\partial(((\Psi_\varphi)^\lfdual)^\lfdual)=\partial\Psi_\varphi=j_\varphi$ in the right continuous case, and analogously in the left continuous case.
\end{enumerate}
\end{proof}

\begin{corollary}\label{cor.VB.quasigauge.formulas}
If $(X,\n{\cdot}_X)$ is a Gateaux differentiable Banach space, and a quasigauge $\varphi$ is continuous on $[0,\sup(\efd(\varphi))[$, then $\forall(x,y)\in X\times\INT(\sup(\efd(\varphi))B(X,\n{\cdot}_X))$
\begin{align}
D_{\Psi_\varphi}(x,y)&=
\int_0^{\n{x}_X}\dd t\,\varphi(t)+\int_0^{\n{j_\varphi(y)}_{X^\star}}\dd t\,(\lim_{s\ra^+t}\varphi(s))^\meet-\duality{x,j_\varphi(y)}_{X\times X^\star}\label{eqn.D.varphi.quasigauge.formula.1}\\
&=\int_0^{\n{x}_X}\dd t\,\varphi(t)+\int_0^{\n{j_\varphi(y)}_{X^\star}}\dd t\,(\lim_{s\ra^-t}\varphi(s))^\join-\duality{x,j_\varphi(y)}_{X\times X^\star}.\label{eqn.D.varphi.quasigauge.formula.2}
\end{align}
If, furthermore, $\varphi$ is right (resp., left) continuous at $\sup(\efd(\varphi))$, then 
\begin{align}
D_{\Psi_\varphi}(x,y)&=
\int_0^{\n{x}_X}\dd t\,\varphi(t)+\int_0^{\n{j_\varphi(y)}_{X^\star}}\dd t\,\varphi^\meet(t)-\duality{x,j_\varphi(y)}_{X\times X^\star}\label{eqn.D.varphi.quasigauge.formula.3}\\
\mbox{(resp., }D_{\Psi_\varphi}(x,y)&=\int_0^{\n{x}_X}\dd t\,\varphi(t)+\int_0^{\n{j_\varphi(y)}_{X^\star}}\dd t\,\varphi^\join(t)-\duality{x,j_\varphi(y)}_{X\times X^\star}\mbox{)}\label{eqn.D.varphi.quasigauge.formula.4}
\end{align}
$\forall(x,y)\in X\times\INT(\sup(\efd(\varphi))B(X,\n{\cdot}_X))$.
\end{corollary}
\begin{proof}
Follows from Lemma \ref{lem.fenchel.dual.psi.quasigauge} and Proposition \ref{prop.psi.quasigauge.geometry}.\ref{prop.psi.quasigauge.geometry.i}, applied to \eqref{eqn.bregman.fenchel}, \eqref{eqn.Psi.varphi}, and \eqref{eqn.asplund}.
\end{proof}

\begin{proposition}\label{prop.supercoercive.Psi.quasigauge}
Let $\varphi$ be a quasigauge, and let $(X,\n{\cdot}_X)$ be a Banach space. Then:
\begin{enumerate}[nosep,label=(\roman*)]
\item\label{prop.supercoercive.Psi.quasigauge.i} $j_{\varphi^\meet}^\star$ is bounded;
\item\label{prop.supercoercive.Psi.quasigauge.ii} if $\varphi^\join$ is right continuous, then $j_{\varphi^\join}^\star$ is bounded;
\item\label{prop.supercoercive.Psi.quasigauge.iii} if either $\varphi^\meet$ is finite or $\varphi^\join$ is (right continuous and finite), then $\Psi_\varphi$ is supercoercive.
\end{enumerate}
\end{proposition}
\begin{proof}{\tiny\ \\}{\vspace{-0.4cm}}
\begin{enumerate}[nosep]
\item[(i)--(ii)] Lemma \ref{lem.quasigauge.inverses}.\ref{lem.quasigauge.inverses.i} allows us to use \eqref{eqn.asplund} to  \eqref{eqn.fenchel.dual.psi.quasigauge.meet} and \eqref{eqn.fenchel.dual.psi.quasigauge.join}. Replacing $\varphi^\inver$ in \eqref{eqn.bounded.j.varphi.inver} by $\varphi^\meet$ (resp., $\varphi^\join$), and using its right continuity together with \eqref{eqn.j.quasigauge.varphi.definition} instead of \eqref{eqn.j.varphi.definition}, gives the boundedness of $j_{\varphi^\meet}^\star$ (resp., $j_{\varphi^\join}^\star$).
\item[(iii)] The assumption of finiteness of $\varphi^\meet$ (resp., $\varphi^\join$) guarantees that $(f_\varphi)^\young(u)$, as defined by \eqref{eqn.dual.quasigauge.integral.meet} (resp., \eqref{eqn.dual.quasigauge.integral.join}), is finite $\forall u\in\RR^+$. Hence, $\efd((\Psi_\varphi)^\lfdual)=X^\star=\intefd{(\Psi_\varphi)^\lfdual}$. Using \ref{prop.supercoercive.Psi.quasigauge.i}--\ref{prop.supercoercive.Psi.quasigauge.ii}, together with \cite[Cor. 7C]{Rockafellar:1966:level}, \cite[Prop. 1.1.11]{Butnariu:Iusem:2000}, and \cite[Lem. 3.6.1]{Zalinescu:2002}, in the same way as in the proof of Proposition \ref{prop.supercoercive.Psi.varphi}, gives the result.
\end{enumerate}
\end{proof}

\begin{proposition}\label{prop.legendre.quasigauge}
Let $(X,\n{\cdot}_X)$ be a Banach space, and let $\varphi$ be a quasigauge strictly increasing on $\efd(\varphi)$, continuous on $[0,\sup(\efd(\varphi))[$, and let
\begin{equation}
\left\{\begin{array}{l}
j_\varphi\mbox{ is not single-valued on }\efd(j_\varphi)\setminus\INT(\sup(\efd(\varphi))B(X,\n{\cdot}_X))\\
\intefd{\Psi_\varphi}=\efd(\Psi_\varphi)\\
\efd(j_\varphi)\mbox{ is open}\\
\efd((\Psi_\varphi)^\lfdual)\mbox{ is open}.
\end{array}
\right.
\label{eqn.euler.legendre.quasigauge.conditions}
\end{equation}
Then $\Psi_\varphi$ is Euler--Legendre if{}f $(X,\n{\cdot}_X)$ is strictly convex and Gateaux differentiable.
\end{proposition}
\begin{proof}
Follows the same arguments as the proof of Proposition \ref{prop.legendre}, with the properties of a gauge $\varphi$ and of $\Psi_\varphi$ replaced by the above assumptions, and with the use of Proposition \ref{prop.psi.varphi.geometry}.\ref{prop.psi.varphi.geometry.ii}--\ref{prop.psi.varphi.geometry.iii} replaced by the use of Proposition \ref{prop.psi.quasigauge.geometry}.\ref{prop.psi.quasigauge.geometry.i}--\ref{prop.psi.quasigauge.geometry.ii}.
\end{proof}

\begin{proposition}\label{prop.left.right.psi.quasigauge}
Let $(X,\n{\cdot}_X)$ be a reflexive, strictly convex, and Gateaux differentiable Banach space, let $\varphi$ be a quasigauge strictly increasing on $\efd(\varphi)$ and continuous on $[0,\sup(\efd(\varphi))[$, let $K\cap\intefd{\Psi_\varphi}\neq\varnothing$. Then:
\begin{enumerate}[nosep,label=(\roman*)]
\item\label{prop.left.right.psi.quasigauge.i} if $K$ is convex and closed, and any of the following (inequivalent) conditions holds:
\begin{enumerate}[nosep,label=\alph*)]
\item\label{prop.left.right.psi.quasigauge.i.a} $\varphi^\meet$ is finite or $\varphi^\join$ is (right continuous and finite);
\item\label{prop.left.right.psi.quasigauge.i.b} \eqref{eqn.euler.legendre.quasigauge.conditions},
\end{enumerate}
then:
\begin{enumerate}[nosep,label=\arabic*)]
\item\label{prop.left.right.psi.quasigauge.i.1} $K$ is left $D_{\Psi_\varphi}$-Chebysh\"{e}v, $D_{\Psi_\varphi}$ is left pythagorean on $K$, and $D_{\Psi_\varphi}$ is an information on $X$;
\item\label{prop.left.right.psi.quasigauge.i.2} if \ref{prop.left.right.psi.quasigauge.i.b} or (\ref{prop.left.right.psi.quasigauge.i.a} and $K\subseteq\intefd{\Psi_\varphi}$) holds, then $\LPPP^{D_{\Psi_\varphi}}_K$ are zone consistent;
\end{enumerate}
\item\label{prop.left.right.psi.quasigauge.ii} if $K\subseteq\intefd{\Psi_\varphi}$, $j_\varphi(K)$ is convex and closed, and any of the following (inequivalent) conditions holds:
\begin{enumerate}[nosep,label=\alph*)]
\item\label{prop.left.right.psi.quasigauge.ii.a} $\left\{\begin{array}{l}
\mbox{\ref{prop.left.right.psi.quasigauge.ii.x}}\\
\varphi\mbox{ is finite or }(\varphi^\meet)^\join\mbox{ is right continuous and finite};
\end{array}\right.$
\item\label{prop.left.right.psi.quasigauge.ii.b} $\left\{\begin{array}{l}
\mbox{\ref{prop.left.right.psi.quasigauge.ii.y}}\\
(\varphi^\join)^\meet\mbox{ is finite or }\varphi\mbox{ is right continuous and finite};
\end{array}\right.$
\item\label{prop.left.right.psi.quasigauge.ii.c} $\left\{\begin{array}{l}
\mbox{\ref{prop.left.right.psi.quasigauge.ii.x}}\\
\varphi^\meet\mbox{ is continuous on }[0,\sup(\efd(\varphi^\meet))[\\
j_{\varphi^\meet}^\star\mbox{ is not single-valued on }\efd(j_{\varphi^\meet}^\star)\setminus\INT(\sup(\efd(\varphi^\meet))B(X^\star,\n{\cdot}_{X^\star}))\\
\intefd{\Psi_{\varphi^\meet}}=\efd(\Psi_{\varphi^\meet})\\
\efd(j_{\varphi^\meet}^\star)\mbox{ is open}\\
\efd(\Psi_\varphi)\mbox{ is open};
\end{array}\right.$
\item\label{prop.left.right.psi.quasigauge.ii.d} $\left\{\begin{array}{l}
\mbox{\ref{prop.left.right.psi.quasigauge.ii.y}}\\
\varphi^\join\mbox{ is continuous on }[0,\sup(\efd(\varphi^\join))[\\
j_{\varphi^\join}^\star\mbox{ is not single-valued on }\efd(j_{\varphi^\join}^\star)\setminus\INT(\sup(\efd(\varphi^\join))B(X^\star,\n{\cdot}_{X^\star}))\\
\intefd{\Psi_{\varphi^\join}}=\efd(\Psi_{\varphi^\join})\\
\efd(j_{\varphi^\join}^\star)\mbox{ is open}\\
\efd(\Psi_\varphi)\mbox{ is open},
\end{array}\right.$
\end{enumerate}
where:
\begin{enumerate}[nosep,label=\alph*)]
\setcounter{enumii}{23}
\item\label{prop.left.right.psi.quasigauge.ii.x} $\left\{\begin{array}{l}
\varphi\mbox{ is right continuous on }\RR^+\\
\varphi^\meet\mbox{ is strictly increasing on }\efd(\varphi^\meet)\\
\Psi_{\varphi^\meet}\mbox{ is Gateaux differentiable on }\varnothing\neq j_\varphi(\INT(\sup(\efd(\varphi))B(X,\n{\cdot}_X)))\subseteq\intefd{\Psi_{\varphi^\meet}};
\end{array}\right.$
\item\label{prop.left.right.psi.quasigauge.ii.y} $\left\{\begin{array}{l}
\varphi\mbox{ is left continuous on }\RR^+\\
\varphi^\join\mbox{ is strictly increasing on }\efd(\varphi^\join)\\
\Psi_{\varphi^\join}\mbox{ is Gateaux differentiable on }\varnothing\neq j_\varphi(\INT(\sup(\efd(\varphi))B(X,\n{\cdot}_X)))\subseteq\intefd{\Psi_{\varphi^\join}},
\end{array}\right.$
\end{enumerate}
then:
\begin{enumerate}[nosep,label=\arabic*)]
\item\label{prop.left.right.psi.quasigauge.ii.1} $K$ is right $D_{\Psi_\varphi}$-Chebysh\"{e}v, $D_{\Psi_\varphi}$ is right pythagorean on $K$, $D_{(\Psi_\varphi)^\lfdual}$ is an information on $X^\star$, and $\RPPP^{D_{\Psi_\varphi}}_K$ is zone consistent;
\item\label{prop.left.right.psi.quasigauge.ii.2} if either \ref{prop.left.right.psi.quasigauge.ii.c} or \ref{prop.left.right.psi.quasigauge.ii.d} holds, then $D_{\Psi_\varphi}$ is an information on $X$.
\end{enumerate}
\end{enumerate}
\end{proposition}
\begin{proof}
Follows directly from Propositions \ref{prop.supercoercive.Psi.quasigauge} and \ref{prop.legendre.quasigauge}, applied to Propositions \ref{prop.left.pythagorean} and \ref{prop.right.pythagorean}, and to Corollary \ref{cor.information}. In \ref{prop.left.right.psi.quasigauge.ii}.\ref{prop.left.right.psi.quasigauge.ii.a} (resp., \ref{prop.left.right.psi.quasigauge.ii}.\ref{prop.left.right.psi.quasigauge.ii.b}) we use $\varphi^{\meet\meet}=\varphi$ (resp., $\varphi^{\join\join}=\varphi$) for right (resp., left) continuous $\varphi$ (cf. the proof of Lemma \ref{lem.fenchel.dual.psi.quasigauge}), while in \ref{prop.left.right.psi.quasigauge.ii}.\ref{prop.left.right.psi.quasigauge.ii.c} and \ref{prop.left.right.psi.quasigauge.ii}.\ref{prop.left.right.psi.quasigauge.ii.d} we use also $((\Psi_\varphi)^\lfdual)^\lfdual=\Psi_\varphi$ that follows from $\Psi_\varphi\in\pcl(X,\n{\cdot}_X)$ and reflexivity of $(X,\n{\cdot}_X)$.
\end{proof}

\begin{corollary}\label{cor.prox.varphi.gen.pyth.thm}
If $(X,\n{\cdot}_X)$ is reflexive, strictly convex, and Gateaux differentiable, $\varphi$ is a gauge, and $\lambda\in\,]0,\infty[$, then:
\begin{enumerate}[nosep,label=(\roman*)]
\item if $f\in\pcl(X,\n{\cdot}_X)$, then $\lprox^{D_{\Psi_\varphi}}_{\lambda,f}$ is single-valued on $X$, and satisfies
\begin{equation}
D_{\Psi_\varphi}(x,y)\geq D_{\Psi_\varphi}(x,\lprox^{D_{\Psi_\varphi}}_{\lambda,f}(y))+D_{\Psi_\varphi}(\lprox^{D_{\Psi_\varphi}}_{\lambda,f}(y),y)\;\;\forall(x,y)\in\Fix(\lprox^{D_{\Psi_\varphi}}_{\lambda,f})\times X,
\label{eqn.lprox.varphi.gen.pyth.thm}
\end{equation}
where $\Fix(\lprox^{D_{\Psi_\varphi}}_{\lambda,f})=\arginff{x\in X}{f(x)}$;
\item if a proper $f:X\ra\,]-\infty,\infty]$ satisfies $f\circ(j_\varphi)^\inver\in\pcl(X^\star,\n{\cdot}_{X^\star})$, then $\rprox^{D_{\Psi_\varphi}}_{\lambda,f}$ is single-valued on $X$, and satisfies
\begin{equation}
D_{\Psi_\varphi}(x,y)\geq D_{\Psi_\varphi}(x,\rprox^{D_{\Psi_\varphi}}_{\lambda,f}(x))+D_{\Psi_\varphi}(\rprox^{D_{\Psi_\varphi}}_{\lambda,f}(x),y)\;\;\forall(x,y)\in X\times\Fix(\rprox^{D_{\Psi_\varphi}}_{\lambda,f}),
\label{eqn.rprox.varphi.gen.pyth.thm}
\end{equation}
where $\Fix(\rprox^{D_{\Psi_\varphi}}_{\lambda,f})=(j_\varphi)^\inver\circ\arginff{x\in X^\star}{f\circ(j_\varphi)^\inver(x)}$.
\end{enumerate}
\end{corollary}
\begin{proof}
\begin{enumerate}[nosep,label=(\roman*)]
\item Follows from Propositions \ref{prop.res.pythagorean}.\ref{prop.res.pythagorean.iv} and \ref{prop.legendre}.
\item Follows from Propositions \ref{prop.gen.pyth.thm.rprox} and \ref{prop.legendre}.
\end{enumerate}
\end{proof}

\begin{corollary}\label{cor.res.varphi.gen.pyth.thm}
Let $(X,\n{\cdot}_X)$ be reflexive, strictly convex, and Gateaux differentiable, let $\varphi$ be a gauge, $\lambda\in\,]0,\infty[$, and let $T:X\ra 2^{X^\star}$. If ($T$ is monotone and $\ran(j_\varphi+\lambda T)=X^\star$) or $T$ is maximally monotone, then:
\begin{enumerate}[nosep,label=(\roman*)]
\item $\lres^{\Psi_\varphi}_{\lambda T}$ is single-valued on $\efd(\lres^{\Psi_\varphi}_{\lambda T})$, $\Fix(\lres^{\Psi_\varphi}_{\lambda T})=T^\inver(0)$ is convex, and
\begin{equation}
D_{\Psi_\varphi}(x,y)\geq D_{\Psi_\varphi}(x,\lres^{\Psi_\varphi}_{\lambda T}(y))+D_{\Psi_\varphi}(\lres^{\Psi_\varphi}_{\lambda T}(y),y)\;\;\forall(x,y)\in\Fix(\lres^{\Psi_\varphi}_{\lambda T})\times X;
\label{eqn.lres.varphi.gen.pyth.thm}
\end{equation}
\item $\rres^{\Psi_\varphi}_{\lambda T}$ is single-valued on $\efd(\rres^{\Psi_\varphi}_{\lambda T})$, $\Fix(\rres^{\Psi_\varphi}_{\lambda T})=j_\varphi(T^\inver(0))$ is $j_\varphi$-convex, and
\begin{equation}
D_{(\Psi_\varphi)^\lfdual}(x,y)\geq D_{(\Psi_\varphi)^\lfdual}(x,\rres^{\Psi_\varphi}_{\lambda T}(x))+D_{(\Psi_\varphi)^\lfdual}(\rres^{\Psi_\varphi}_{\lambda T}(x),y)\;\;\forall(x,y)\in X^\star\times\Fix(\rres^{\Psi_\varphi}_{\lambda T}),
\label{eqn.rres.varphi.gen.pyth.thm}
\end{equation}
where $(\Psi_\varphi)^\lfdual$ is given by \eqref{eqn.zalinescu.inverse.fenchel.dual}.
\end{enumerate}
\end{corollary}
\begin{proof}
Follows from Propositions \ref{prop.res.pythagorean}.\ref{prop.res.pythagorean.i}--\ref{prop.res.pythagorean.iii}, \ref{prop.gen.pyth.thm.rres}, and \ref{prop.legendre}.
\end{proof}

\begin{proposition}\label{prop.continuity}
Let $\varnothing\neq K\subseteq X$. For any gauge $\varphi$, if $(X,\n{\cdot}_X)$ is strictly convex, Fr\'{e}chet differentiable, reflexive, and has the Radon--Riesz--Shmul'yan property, then:
\begin{enumerate}[nosep,label=(\roman*)]
\item\label{prop.continuity.i} if $K$ is convex and closed, then $\LPPP^{D_{\Psi_\varphi}}_K$ is norm-to-norm continuous on $X$, while $\inf_{y\in K}\{D_{\Psi_\varphi}(y,\,\cdot\,)\}$ is continuous on $X$;
\item\label{prop.continuity.ii} if $j_\varphi(K)$ is convex and closed, then  $\RPPP^{D_{\Psi_\varphi}}_{K}$ is norm-to-norm continuous on $X$.
\end{enumerate}
\end{proposition}
\begin{proof}
\begin{enumerate}[nosep]
\item[(i)] Taking into account Proposition \ref{prop.psi.varphi.geometry}.\ref{prop.psi.varphi.geometry.xi}, together with equivalence of Fr\'{e}chet differentiability of $(X,\n{\cdot}_X)$ with Fr\'{e}chet differentiability of $\Psi_\varphi$ (Proposition \ref{prop.psi.varphi.geometry}.\ref{prop.psi.varphi.geometry.iv}), we conclude that, for any gauge $\varphi$, if $(X,\n{\cdot}_X)$ is reflexive, strictly convex, Fr\'{e}chet differentiable, and has the Radon--Riesz--Shmul'yan property, then $\Psi_\varphi$ is totally convex and Fr\'{e}chet differentiable on $X$ (as well as supercoercive). Hence, Proposition \ref{prop.continuity.psi}.\ref{prop.continuity.psi.i} applies.
\item[(ii)] \sloppy By Proposition \ref{prop.psi.varphi.geometry}.\ref{prop.psi.varphi.geometry.iv}, $j_\varphi$ (resp., $j_{\varphi^{\inver}}^\star$) is norm-to-norm continuous if{}f $(X,\n{\cdot}_X)$ (resp., $(X^\star,\n{\cdot}_{X^\star})$) is Fr\'{e}chet differentiable (and, in such case, $j_\varphi=\DF\Psi_\varphi$ (resp., $j_{\varphi^{\inver}}^\star=\DF((\Psi_\varphi)^\lfdual)$)). By \cite[Thm. 3.9]{Anderson:1960}, if $(X,\n{\cdot}_X)$ is reflexive, then $(X^\star,\n{\cdot}_{X^\star})$ is Fr\'{e}chet differentiable if{}f ($(X,\n{\cdot}_X)$ is strictly convex and has the Radon--Riesz--Shmul'yan property). Hence, \eqref{eqn.left.to.right} is a composition of norm-to-norm continuous functions, and thus it is norm-to-norm continuous, under the same assumptions on $\Psi$ as in \ref{prop.continuity.ii}.
\item[(ii')] Follows from Proposition \ref{prop.continuity.psi}.\ref{prop.continuity.psi.iii}.
\end{enumerate}
\end{proof}

\begin{proposition}\label{prop.varphi.compositional}
Let $\varphi$ be any gauge.
\begin{enumerate}[nosep,label=(\roman*)]
\item\label{prop.varphi.compositional.i} If $(X,\n{\cdot}_X)$ is uniformly Fr\'{e}chet differentiable and strictly convex, then $\Psi_\varphi$ is LSQ-adapted on any convex and closed $\varnothing\neq K\subseteq X$, and $\LPPP^{D_{\Psi_\varphi}}_C$ is adapted for any $\varnothing\neq C\subseteq K$ with convex and closed $C$.
\item\label{prop.varphi.compositional.ii}\sloppy If $(X,\n{\cdot}_X)$ is uniformly Fr\'{e}chet differentiable, strictly convex, and has the Radon--Riesz--Shmul'yan property, then $\Psi_\varphi$ is LSQ-compositional and RSQ-compositional. 
\item\label{prop.varphi.compositional.iii} If $(X,\n{\cdot}_X)$  is uniformly Fr\'{e}chet differentiable and uniformly convex, then:
\begin{enumerate}[nosep,label=\alph*)]
\item\label{prop.varphi.compositional.iii.a} $\Psi_\varphi$ is RSQ-adapted on any $\varnothing\neq K\subseteq X$, and $\RPPP^{D_{\Psi_\varphi}}_C$ is adapted for any $\varnothing\neq C\subseteq K$ with convex and closed $j_\varphi(C)$;
\item\label{prop.varphi.compositional.iii.b} For any $\varnothing\neq K\subseteq X$ and $T:K\ra X$,
\begin{equation}
T\in\rsq(\Psi_\varphi,K)\;\;\iff\;\;j_\varphi\circ T\circ(j_\varphi)^{\inver}\in\lsq(\Psi_\varphi,j_\varphi(K)).
\end{equation}
\end{enumerate}
\end{enumerate}
\end{proposition}
\begin{proof}{\tiny\ \\}{\vspace{-0.4cm}}
\begin{enumerate}[nosep]
\item[(i)--(ii)] By Proposition \ref{prop.psi.varphi.geometry}.\ref{prop.psi.varphi.geometry.v}, $(X,\n{\cdot}_X)$ is uniformly Fr\'{e}chet differentiable if{}f $\Psi_\varphi$ is uniformly Fr\'{e}chet differentiable on bounded subsets of $X$. By Proposition \ref{prop.psi.varphi.geometry}.\ref{prop.psi.varphi.geometry.xi}, if $(X,\n{\cdot}_X)$ is reflexive, then (it is strictly convex and has the Radon--Riesz--Shmul'yan property) if{}f $\Psi_\varphi$ is totally convex on $X$. The rest follows from Proposition \ref{prop.lsq.rsq.new}.\ref{prop.lsq.rsq.new.i}.\ref{prop.lsq.rsq.new.i.a}, as well \ref{prop.lsq.rsq.new}.\ref{prop.lsq.rsq.new.i}.\ref{prop.lsq.rsq.new.i.b} and \ref{prop.lsq.rsq.new}.\ref{prop.lsq.rsq.new.ii}--\ref{prop.lsq.rsq.new.iii}, due to supercoercivity of $\Psi_\varphi$ (Proposition \ref{prop.supercoercive.Psi.varphi}) and because total convexity on $X$ implies total convexity on bounded subsets of $X$.
\item[(iii)] By Proposition \ref{prop.psi.varphi.geometry}.\ref{prop.psi.varphi.geometry.vi}, uniform convexity of $(X,\n{\cdot}_X)$ is equivalent with uniform Fr\'{e}chet differentiability of $(\Psi_\varphi)^\lfdual$ on bounded subsets of $(X^\star,\n{\cdot}_{X^\star})$. By Propositions \ref{prop.lsq.rsq.new}.\ref{prop.lsq.rsq.new.iv} and \ref{prop.lsq.rsq.new}.\ref{prop.lsq.rsq.new.v}, this gives, respectively, \ref{prop.varphi.compositional.iii}.\ref{prop.varphi.compositional.iii.a} and \ref{prop.varphi.compositional.iii}.\ref{prop.varphi.compositional.iii.b}.
\end{enumerate}
\end{proof}

\begin{proposition}\label{prop.j.star.hoelder}
Let $\gamma\in\,]0,\frac{1}{2}]$.
\begin{enumerate}[nosep,label=(\roman*)]
\item\label{prop.j.star.hoelder.i} If $\varphi(t)\in\left\{\frac{1}{1-\gamma}t^{\frac{\gamma}{1-\gamma}},t^{\frac{\gamma}{1-\gamma}}\right\}$, then $(X,\n{\cdot}_X)$ is $\frac{1}{\gamma}$-uniformly convex if{}f $j_\varphi^\star$ on $(X^\star,\n{\cdot}_{X^\star})$ is single-valued and $\frac{\gamma}{1-\gamma}$-Lipschitz--H\"{o}lder continuous on $X$.
\item\label{prop.j.star.hoelder.ii} Let $\varphi(t)\in\left\{(1-\gamma)^{\frac{1-\gamma}{\gamma}}t^{\frac{1-\gamma}{\gamma}},t^{\frac{1-\gamma}{\gamma}}\right\}$. If $(X,\n{\cdot}_X)$ is $\frac{1}{\gamma}$-uniformly convex, then $(j_\varphi)^{\inver}$ is single-valued and $\frac{\gamma}{1-\gamma}$-Lipschitz--H\"{o}lder continuous on $X$.
\end{enumerate}
\end{proposition}
\begin{proof}
\begin{enumerate}[nosep]
\item[(i)] Follows from equivalence of $\frac{1}{\gamma}$-uniform convexity of $(X,\n{\cdot}_X)$ and $\frac{1}{1-\gamma}$-uniform Fr\'{e}chet differentiability of $(X^\star,\n{\cdot}_{X^\star})$ \cite[p. 63 (Vol. 2)]{Lindenstrauss:Tzafriri:1977:1979} combined with Proposition \ref{prop.psi.varphi.geometry}.\ref{prop.psi.varphi.geometry.viii}.\ref{prop.psi.varphi.geometry.viii.b}--\ref{prop.psi.varphi.geometry.viii.c}.
\item[(ii)] Follows from \ref{prop.j.star.hoelder.i}, Proposition \ref{prop.psi.varphi.geometry}.\ref{prop.psi.varphi.geometry.viii}.\ref{prop.psi.varphi.geometry.viii.b}--\ref{prop.psi.varphi.geometry.viii.c}, and \eqref{eqn.j.j.star.inverse}.
\end{enumerate}
\end{proof}

\begin{proposition}\label{prop.varphi.prox.norm.continuity}
Let $\lambda\in\,]0,\infty[$, let $(X,\n{\cdot}_X)$ be reflexive, let $\varphi$ be a gauge, let $f\in\pcl(X,\n{\cdot}_X)$ be bounded from below, and $\lim_{\n{x}_X\ra\infty}f(x)=\infty$. Then:
\begin{enumerate}[nosep,label=(\roman*)]
\item\label{prop.varphi.prox.norm.continuity.i} if $(X,\n{\cdot}_X)$ is strictly convex and Gateaux differentiable, then $\lprox^{D_{\Psi_\varphi}}_{\lambda,f}$ is single-valued and continuous on $X$;
\item\label{prop.varphi.prox.norm.continuity.ii} if $(X,\n{\cdot}_X)$ is strictly convex, Fr\'{e}chet differentiable, and satisfies Radon--Riesz--Shmul'yan property, then $\rprox^{D_{\Psi_\varphi}}_{\lambda,f}$ is single-valued and continuous on $X$.
\end{enumerate}
\end{proposition}
\begin{proof}
Follows from Proposition \ref{prop.norm.continuity.left.right.prox}.\ref{prop.norm.continuity.left.right.prox.i}--\ref{prop.norm.continuity.left.right.prox.ii}, and the same reasoning as in the proof of Proposition \ref{prop.continuity}.\ref{prop.continuity.i}--\ref{prop.continuity.ii}.
\end{proof}

\begin{proposition}\label{prop.varphi.resolvent.norm.continuity}
Let $(X,\n{\cdot}_X)$ be a Banach space, let $\varphi$ be a gauge, and let $T:X\ra 2^{X^\star}$ be maximally monotone with $0\in\efd(T)$. Then:
\begin{enumerate}[nosep,label=(\roman*)]
\item\label{prop.varphi.resolvent.norm.continuity.i} if $(X,\n{\cdot}_X)$ is Fr\'{e}chet differentiable and uniformly convex, then $\lres^{\Psi_\varphi}_T=(T+j_\varphi)^\inver\circ j_\varphi$ maps $X$ on $\efd(T)$ and is norm-to-norm continuous on $X$;
\item\label{prop.varphi.resolvent.norm.continuity.ii} if $(X,\n{\cdot}_X)$ is strictly convex, uniformly Fr\'{e}chet differentiable, and has Radon--Riesz--Shmul'yan property, then $\rres^{\Psi_\varphi}_T=(j_\varphi)^\inver\circ(T+j_\varphi)^\inver$ maps $X^\star$ on $j_\varphi(\efd(T))$ and is norm-to-norm continuous on $X^\star$.
\end{enumerate}
\end{proposition}
\begin{proof}
\begin{enumerate}[nosep]
\item[(i)] By \cite[Thm. 5.(c)]{Browder:1968}, if $(X,\n{\cdot}_X)$ is Gateaux differentiable and uniformly convex, $\varphi$ is a gauge, and $T:X\ra2^{X^\star}$ is maximally monotone with $0\in\efd(T)$, then $(T+j_\varphi)^\inver$ is a norm-to-norm continuous map from $X^\star$ to $\efd(T)$. (While the explicit statement of this theorem assumes additionally that $T=\partial f$ for $f\in\pcl(X,\n{\cdot}_X)$, its proof does not depend on this assumption.) The rest follows from Proposition \ref{prop.psi.varphi.geometry}.\ref{prop.psi.varphi.geometry.iv}.\ref{prop.psi.varphi.geometry.iv.b} and composability of norm-to-norm continuous maps on Banach spaces.
\item[(ii)] From Proposition \ref{prop.prox.res.left.right}.\ref{prop.prox.res.left.right.i} it follows that $\rres^{\Psi_\varphi}_T=\DG\Psi_\varphi\circ\lres^{\Psi_\varphi}_T\circ\DG(\Psi_\varphi)^\lfdual$ with $\efd(\rres^{\Psi_\varphi}_T)=\DG\Psi_\varphi(\efd(\lres^{\Psi_\varphi}_T))$ and $\ran(\rres^{\Psi_\varphi}_T)=\DG\Psi_\varphi(\ran(\lres^{\Psi_\varphi}_T))$. 
Furthermore, due to \eqref{eqn.asplund} and \eqref{eqn.zalinescu.inverse.fenchel.dual}, $(\Psi_\varphi)^\lfdual=\Psi_{\varphi^\inver}$ and $(j_\varphi)^\inver=j_{\varphi^\inver}^\star$. The rest follows from \ref{prop.varphi.resolvent.norm.continuity.i}, combined with the equivalence of (Gateaux diferentiability and uniform convexity of $(X,\n{\cdot}_X)$) with (strict convexity and uniform Fr\'{e}chet differentiability of $(X^\star,\n{\cdot}_{X^\star})$), Proposition \ref{prop.psi.varphi.geometry}.\ref{prop.psi.varphi.geometry.iv}.\ref{prop.psi.varphi.geometry.iv.b} applied to $j_\varphi$ and $j_{\varphi^\inver}^\star$, and the fact \cite[Thm. 3.9]{Anderson:1960} that (reflexivity, strict convexity, and Radon--Riesz--Shmul'yan property of $(X,\n{\cdot}_X)$) implies Fr\'{e}chet differentiability of $(X^\star,\n{\cdot}_{X^\star})$.
\end{enumerate}
\end{proof}

\begin{proposition}\label{prop.varphi.resolvent.uniform.continuity}
Let $(X,\n{\cdot}_X)$ be a Banach space, let $\beta\in\,]0,1[$, let $T:X\ra 2^{X^\star}$ be maximally monotone, let $f\in\pcl(X,\n{\cdot}_X)$, let $g:X\ra\,]-\infty,\infty]$ satisfy $g\circ j_{\varphi_{1,1-\beta}}\in\pcl(X^\star,\n{\cdot}_{X^\star})$ for $\beta\in[\frac{1}{2},1[$, let $\lambda\in\,]0,\infty[$ and $r\in\,]1,2]$. Then:
\begin{enumerate}[nosep,label=(\roman*)]
\item\label{prop.varphi.resolvent.uniform.continuity.i} if $\beta\in\,]0,\frac{1}{2}]$ and $(X,\n{\cdot}_X)$ is $\frac{1}{\beta}$-uniformly convex and uniformly Fr\'{e}chet differentiable, then $\lres^{\Psi_{\varphi_{1,\beta}}}_{\lambda T}$ and $\lprox^{D_{\Psi_{\varphi_{1,\beta}}}}_{\lambda,f}$ are single-valued and uniformly continuous on bounded subsets of $X$;
\item\label{prop.varphi.resolvent.uniform.continuity.ii} if $\beta\in\,]0,\frac{1}{2}]$ and $(X,\n{\cdot}_X)$ is $\frac{1}{\beta}$-uniformly convex and $r$-uniformly Fr\'{e}chet differentiable, then $\lres^{\Psi_{\varphi_{1,\beta}}}_{\lambda T}$ and $\lprox^{D_{\Psi_{\varphi_{1,\beta}}}}_{\lambda,f}$ are single-valued and $(r-1)\frac{\beta}{1-\beta}$-Lipschitz--H\"{o}lder continuous on $X$;
\item\label{prop.varphi.resolvent.uniform.continuity.iii} if $\beta\in[\frac{1}{2},1[$ and $(X,\n{\cdot}_X)$ is uniformly convex and $\frac{1}{\beta}$-uniformly Fr\'{e}chet differentiable, then $\rres^{\Psi_{\varphi_{1,\beta}}}_{\lambda T}$ is single-valued and uniformly continuous on bounded subsets of $X^\star$, while $\rprox^{D_{\Psi_{\varphi_{1,\beta}}}}_{\lambda,g}$ is single-valued and uniformly continuous on bounded subsets of $X$;
\item\label{prop.varphi.resolvent.uniform.continuity.iv} if $\beta\in[\frac{1}{2},1[$ and $(X,\n{\cdot}_X)$ is $\frac{r}{r-1}$-uniformly convex and $\frac{1}{\beta}$-uniformly Fr\'{e}chet differentiable, then $\rres^{\Psi_{\varphi_{1,\beta}}}_{\lambda T}$ is single-valued and $(\frac{(1-\beta)(r-1)}{\beta})^2$-Lipschitz--H\"{o}lder continuous on bounded subsets of $X^\star$, while $\rprox^{D_{\Psi_{\varphi_{1,\beta}}}}_{\lambda,g}$ is single-valued and $(\frac{(1-\beta)(r-1)}{\beta})^2$-Lipschitz--H\"{o}lder continuous on bounded subsets of $X$.
\end{enumerate}
\end{proposition}
\begin{proof}
\begin{enumerate}[nosep]
\item[(i)] By \cite[Ex. 6.7]{Reem:Reich:2018}, if $(X,\n{\cdot}_X)$ is Gateaux differentiable and $\frac{1}{\beta}$-uniformly convex, then $(\DG\Psi_{\varphi_{1,\beta}}+\lambda T)^\inver$ is single-valued and $\frac{\beta}{1-\beta}$-Lipschitz--H\"{o}lder continuous on $X$. Combining this with Proposition \ref{prop.psi.varphi.geometry}.\ref{prop.psi.varphi.geometry.v}.\ref{prop.psi.varphi.geometry.v.b} gives the result for $\lres^{\Psi_{\varphi_{1,\beta}}}_{\lambda T}$. The result for $\lprox^{D_{\Psi_{\varphi_{1,\beta}}}}_{\lambda,f}$ follows in the same way as in Proposition \ref{prop.lipschitz.hoelder.left.resolvent}. (Equivalently, the same conclusion follows from Proposition \ref{prop.lipschitz.hoelder.left.resolvent}, combined with Propositions \ref{prop.psi.varphi.geometry}.\ref{prop.psi.varphi.geometry.v}.\ref{prop.psi.varphi.geometry.v.b} and \ref{prop.psi.varphi.geometry}.\ref{prop.psi.varphi.geometry.vii}.\ref{prop.psi.varphi.geometry.vii.b}.)
\item[(ii)] Follows from the proof of \ref{prop.varphi.resolvent.uniform.continuity.i}, Proposition \ref{prop.psi.varphi.geometry}.\ref{prop.psi.varphi.geometry.viii}.\ref{prop.psi.varphi.geometry.viii.c}, and the fact that the composition of Lipschitz--H\"{o}lder continuous maps has the exponent given by the multiplication of the exponents of the composite maps.
\item[(iii)] From Proposition \ref{prop.prox.res.left.right}.\ref{prop.prox.res.left.right.i}, we obtain
\begin{equation}
\rres_{\lambda T}^{\Psi_{\varphi_{1,\beta}}}=j_{\varphi_{1,\beta}}\circ\lres_{\lambda T}^{\Psi_{\varphi_{1,\beta}}}\circ(j_{\varphi_{1,\beta}})^\inver.
\label{eqn.left.right.resolvent.varphi.beta}
\end{equation}
The result for $\rres^{\Psi_{\varphi_{1,\beta}}}_{\lambda T}$ follows from combining \ref{prop.varphi.resolvent.uniform.continuity.i} with Propositions \ref{prop.psi.varphi.geometry}.\ref{prop.psi.varphi.geometry.vi}.\ref{prop.psi.varphi.geometry.vi.c} and \ref{prop.psi.varphi.geometry}.\ref{prop.psi.varphi.geometry.v}.\ref{prop.psi.varphi.geometry.v.b}, and with the fact that $\frac{1}{\beta}$-uniform Fr\'{e}chet differentiability of $(X,\n{\cdot}_X)$ implies its Fr\'{e}chet differentiability. The result for $\rprox^{D_{\Psi_{\varphi_{1,\beta}}}}_{\lambda,g}$ follows the same way, using Proposition \ref{prop.prox.res.left.right}.\ref{prop.prox.res.left.right.ii}.
\item[(iv)] By 
\begin{equation}
(j_{\varphi_{1,\beta}})^{\inver}=(\DG\Psi_{\varphi_{1,\beta}})^{\inver}=\DG((\Psi_{\varphi_{1,\beta}})^\lfdual)=j_{(\varphi_{1,\beta})^{\inver}}^\star=j^\star_{\varphi_{1,1-\beta}},
\label{eqn.inver.j.varphi.beta}
\end{equation}
\eqref{eqn.left.right.resolvent.varphi.beta} gives
\begin{equation}
\rres_{\lambda T}^{\Psi_{\varphi_{1,\beta}}}=j_{\varphi_{1,\beta}}\circ\lres_{\lambda T}^{\Psi_{\varphi_{1,\beta}}}\circ j_{\varphi_{1,1-\beta}}^\star.
\label{eqn.left.right.resolvent.varphi.beta.beta}
\end{equation}
By \cite[p. 63 (Vol. 2)]{Lindenstrauss:Tzafriri:1977:1979}, $\frac{1}{\gamma}$-uniform convexity of $(X,\n{\cdot}_X)$ is equivalent with $\frac{1}{1-\gamma}$-uniform Fr\'{e}chet differentiability of $(X^\star,\n{\cdot}_{X^\star})$ $\forall\gamma\in\,]0,\frac{1}{2}]$. By \cite[Cor. 2.35]{Kazimierski:2010}, if a Banach space $(X,\n{\cdot}_X)$ is $s$-uniformly Fr\'{e}chet differentiable, $s\in\,]1,2]$,\footnote{While \cite[Cor. 2.35]{Kazimierski:2010} uses $s\in\,]1,\infty]$, the limitation to $s\in\,]1,2]$ follows from the fact that there are no $s$-uniformly Fr\'{e}chet differentiable spaces for ($s<1$ as well as) $s>2$ (cf., e.g., \cite[pp. 8--9]{Yamada:2006} for a proof).} and $w\in\,]1,\infty[$, then $j_{\varphi_{1,1/w}}$ is single-valued and $\min\{w-1,s-1\}$-Lipschitz--H\"{o}lder continuous on bounded subsets of $X$. Setting $1-\beta=w$ for $\beta\in[\frac{1}{2},1[$ gives $w\in[2,\infty[$, hence $\min\{\frac{1}{1-\beta}-1,r-1\}=r-1$. Combining this with \ref{prop.varphi.resolvent.uniform.continuity.ii} and with Proposition \ref{prop.psi.varphi.geometry}.\ref{prop.psi.varphi.geometry.viii}.\ref{prop.psi.varphi.geometry.viii.c}, we obtain that the map \eqref{eqn.left.right.resolvent.varphi.beta.beta} is $\frac{1-\beta}{\beta}\left(\frac{1-\beta}{\beta}(r-1)\right)(r-1)$-Lipschitz--H\"{o}lder continuous on $X^\star$. The result for $\rprox^{D_{\Psi_{\varphi_{1,\beta}}}}_{\lambda,g}$ follows completely analogously.
\end{enumerate}
\end{proof}

\begin{proposition}\label{prop.varphi.uniform.continuity}
Let $(X,\n{\cdot}_X)$ be a Banach space, $\beta\in\,]0,1[$, $r\in\,]1,2]$, $\varnothing\neq K\subseteq X$. Then:
\begin{enumerate}[nosep,label=(\roman*)]
\item\label{prop.varphi.uniform.continuity.i} if $\beta\in\,]0,\frac{1}{2}]$, $(X,\n{\cdot}_X)$ is $\frac{1}{\beta}$-uniformly convex and uniformly Fr\'{e}chet differentiable, and $K$ is convex and closed, then $\LPPP^{D_{\Psi_{\varphi_{1,\beta}}}}_K$ is uniformly continuous on bounded subsets of $X$;
\item\label{prop.varphi.uniform.continuity.ii} if $\beta\in\,]0,\frac{1}{2}]$, $(X,\n{\cdot}_X)$ is $\frac{1}{\beta}$-uniformly convex and $r$-uniformly Fr\'{e}chet differentiable, and $K$ is convex and closed, then $\LPPP^{D_{\Psi_{\varphi_{1,\beta}}}}_K$ is $\frac{\beta(r-1)}{1-\beta}$-Lipschitz--H\"{o}lder continuous on $X$;
\item\label{prop.varphi.uniform.continuity.iii} if $\beta\in[\frac{1}{2},1[$, $(X,\n{\cdot}_X)$ is $\frac{1}{\beta}$-uniformly Fr\'{e}chet differentiable and uniformly convex, and $j_{\varphi_{1,\beta}}(K)$ is convex and closed, then $\RPPP^{D_{\Psi_{\varphi_{1,\beta}}}}_K$ is uniformly continuous on bounded subsets of $X$; 
\item\label{prop.varphi.uniform.continuity.iv} if $\beta\in[\frac{1}{2},1[$, $(X,\n{\cdot}_X)$ is $\frac{1}{\beta}$-uniformly Fr\'{e}chet differentiable and $\frac{r}{r-1}$-uniformly convex, and $j_{\varphi_{1,\beta}}(K)$ is convex and closed, then $\RPPP^{D_{\Psi_{\varphi_{1,\beta}}}}_K$ is $\frac{(1-\beta)^2}{\beta^2}(r-1)^2$-Lipschitz--H\"{o}lder continuous on bounded subsets of $X$.
\end{enumerate}
\end{proposition}
\begin{proof}{\tiny\ \\}{\vspace{-0.4cm}}
\begin{enumerate}[nosep]
\item[(i)--(ii)] Follows from Propositions \ref{prop.varphi.resolvent.uniform.continuity}.\ref{prop.varphi.resolvent.uniform.continuity.i}--\ref{prop.varphi.resolvent.uniform.continuity.ii} and \ref{prop.proj.prox.res}.\ref{prop.proj.prox.res.ii}, by setting $T=\partial\iota_K$ and $\lambda=1$.
\item[(iii)] Follows from \ref{prop.varphi.uniform.continuity.i}, taken together with \eqref{eqn.left.to.right}, analogously to the proof of Proposition \ref{prop.varphi.resolvent.uniform.continuity}.\ref{prop.varphi.resolvent.uniform.continuity.iii}, using \eqref{eqn.inver.j.varphi.beta}, and replacing \eqref{eqn.left.right.resolvent.varphi.beta.beta} by
\begin{equation}
\RPPP^{D_{\Psi_{\varphi_{1,\beta}}}}_K(x)=j^\star_{\varphi_{1,1-\beta}}\circ\LPPP^{D_{\Psi_{\varphi_{1,1-\beta}}}}_{j_{\varphi_{1,\beta}}(K)}\circ j_{\varphi_{1,\beta}}(x)\;\forall x\in X.
\label{eqn.left.right.varphi}
\end{equation}
\item[(iv)] Follows from \ref{prop.varphi.uniform.continuity.ii}, analogously to the proof of Proposition \ref{prop.varphi.resolvent.uniform.continuity}.\ref{prop.varphi.resolvent.uniform.continuity.iv}, with the same substitution as in \ref{prop.varphi.uniform.continuity.iii}.
\end{enumerate}
\end{proof}

\begin{definition}
Let $\varphi$ be a gauge. A Banach space $(X,\n{\cdot}_X)$ will be called \df{$\varphi$-uniformly convex} (resp., \df{$\varphi$-uniformly Fr\'{e}chet differentiable}) if{}f $\Psi_\varphi$ is uniformly convex (resp., uniformly Fr\'{e}chet differentiable) on $(X,\n{\cdot}_X)$.
\end{definition}

\begin{corollary}\label{cor.varphi.uniformly.convex}
For any gauge $\varphi$, a Banach space $(X,\n{\cdot}_X)$ is $\varphi$-uniformly convex (resp., $\varphi$-uniformly Fr\'{e}chet differentiable) if{}f $(X^\star,\n{\cdot}_{X^\star})$ is $\varphi$-uniformly Fr\'{e}chet differentiable (resp., $\varphi$-uni\-form\-ly convex).
\end{corollary}
\begin{proof}
Follows directly from Lemma \ref{lemma.uFd.uc.duality}.
\end{proof}

\begin{proposition}\label{prop.alber.decomposition.psi.varphi}
If $(X,\n{\cdot}_X)$ is reflexive, strictly convex, Gateaux differentiable, $\varphi$ is a gauge, $\varnothing\neq K_1\subset X$ is a closed convex cone with a vertex at $0\in X$, $\varnothing\neq K_2\subset X$, $j_\varphi(K_2)$ is a closed convex cone with a vertex at $0\in X^\star$, then
\begin{equation}
\forall x\in X\;\;\;
\left\{
\begin{array}{l}
x=j^\star_{\varphi^{\inver}}\circ\hat{\PPP}^{(\Psi_\varphi)^\lfdual}_{K_1^\circ}\circ j_\varphi(x)+\LPPP_{K_1}^{D_{\Psi_\varphi}}(x)\\
\duality{\LPPP^{D_{\Psi_\varphi}}_{K_1}(x),\hat{\PPP}_{K_1^\circ}^{(\Psi_\varphi)^\lfdual}\circ j_\varphi(x)}_{X\times X^\star}=0
\end{array}
\right.
\label{eqn.left.alber.varphi.decomposition}
\end{equation}
and
\begin{equation}
\forall y\in X\;\;\;
\left\{
\begin{array}{l}
y=\hat{\PPP}^{\Psi_\varphi}_{(j_\varphi(K_2))^\circ}(y)+\RPPP^{D_{\Psi_\varphi}}_{K_2}(y)\\
\duality{j^\star_{\varphi^{\inver}}\circ\RPPP^{D_{\Psi_\varphi}}_{K_2}(y),\hat{\PPP}^{\Psi_\varphi}_{(j_\varphi(K_2))^\circ}(y)}_{X\times X^\star}=0,
\end{array}
\right.
\label{eqn.right.alber.varphi.decomposition}
\end{equation}
where, for any strictly convex function $f$ and convex closed set $C$, $\hat{\PPP}^f_C(x):=\arginff{z\in C}{\Psi(x-z)}$ $\forall x\in X$ and $\hat{\PPP}^f_C\circ\hat{\PPP}^f_C(x)=\hat{\PPP}^f_C(x)$ $\forall x\in X$. Furthermore, if $\varnothing\neq K_1\subseteq X$ (resp., $\varnothing\neq j_\varphi(K_2)\subseteq X^\star$) is a linear subspace instead of a closed convex cone, then \eqref{eqn.left.alber.varphi.decomposition} (resp., \eqref{eqn.right.alber.varphi.decomposition}) holds under replacement of $(\cdot)^\circ$ with $(\cdot)^\bot$.
\end{proposition}
\begin{proof}
Follows from Propositions \ref{prop.alber.decomposition}, \ref{prop.psi.varphi.geometry}.\ref{prop.psi.varphi.geometry.ii}--\ref{prop.psi.varphi.geometry.iii}, \ref{prop.supercoercive.Psi.varphi}, and \eqref{eqn.j.j.star.inverse}.
\end{proof}

\begin{remark}\label{remark.varphi}
\begin{enumerate}[nosep,label=(\roman*)]
\item\label{remark.varphi.i} \eqref{eqn.D.Psi.varphi.formula} and \eqref{eqn.D.varphi.quasigauge.formula.1}--\eqref{eqn.D.varphi.quasigauge.formula.4} are new. The same holds for the equations in Lemma \ref{lem.fenchel.dual.psi.quasigauge}. The first implicit appearance of the formula equivalent to a statement of nonnegativity of $D_{\Psi_\varphi}$ (more precisely, of $D^\partial_{\Psi_\varphi}$) can be found in \cite[Thm. 1]{Asplund:1967}. (Nonnegativity of $D_\Psi$ as a condition characterising monotonicity of $\partial\Psi$ appeared earlier, in \cite[Thm. 3]{Kachurovskii:1966}.) For $\Psi=\Psi_{\varphi_{1/2,1/2}}$ and $\Psi=\Psi_{\varphi_{1,\beta}}$, $\beta\in\,]0,1[$, an identification of this formula as corresponding to the Va\u{\i}nberg--Br\`{e}gman functional, together with a study of $\LPPP^{D_\Psi}$, was made in \cite[\S7]{Alber:1993} and \cite[\S7]{Alber:1996}.
\item\label{remark.varphi.ii} Proposition \ref{prop.legendre} clarifies relationships between the Euler--Legendre property and total convexity of $\Psi_\varphi$: the former is equivalent to (strict convexity and Gateaux differentiability) of $(X,\n{\cdot}_X)$, while the latter is implied by the local uniform convexity of $(X,\n{\cdot}_X)$, hence it entails strict convexity and the Radon--Riesz--Shmul'yan property of $(X,\n{\cdot}_X)$ (implication of the latter property is proved in \cite[Prop. (p. 352)]{Vyborny:1956}). The lack of Gateaux differentiability in the latter case should be seen in the context of total convexity being defined by $\DG_+\Psi$ (and thus $D^+_\Psi$) instead of $\DG\Psi$ (and thus $D_\Psi$). For reflexive $(X,\n{\cdot}_X)$, $\Psi_\varphi$ is totally convex if{}f $(X,\n{\cdot}_X)$ is strictly convex and has the Radon--Riesz--Shmul'yan property \cite[Thms. 3.1, 3.3]{Resmerita:2004}. An example of a reflexive, strictly convex, Gateaux differentiable Banach space $(X,\n{\cdot}_X)$ which does not satisfy the Radon--Riesz--Shmul'yan property, so $\Psi_\varphi$ is Euler--Legendre but is not totally convex, is provided in \cite[Ex. 2.5]{Bauschke:Combettes:2003}.
\item\label{remark.varphi.iii} Let $(X,\n{\cdot}_X)$ be a Banach space, let $\varnothing\neq K\subseteq X$ be convex and closed, and consider a \df{metric projection}, defined as a set-valued map $\PPP^{d_{\n{\cdot}_X}}_K:X\ni x\mapsto\arginff{y\in K}{\n{y-x}_X}\subseteq K$. Then:
\begin{enumerate}[nosep,label=\alph*)]
\item\label{remark.varphi.iii.a} $\PPP^{d_{\n{\cdot}_X}}_K$ exists and is unique (i.e. $K$ is a Chebysh\"{e}v set: $\PPP^{d_{\n{\cdot}_X}}_K(x)=\{*\}$ $\forall x\in X$) if{}f $(X,\n{\cdot}_X)$ is strictly convex and reflexive \cite[p. 292]{Klee:1961}\footnote{An implication from right to left was proved earlier in \cite[Lem. (p. 316)]{Day:1941}. Two key components of the characterisation result were: 1) the characterisation of strict convexity in \cite[p. 179]{Krein:1938}, implying equivalence of strict convexity of $(X,\n{\cdot}_X)$ and uniqueness of $\PPP^{d_{\n{\cdot}_X}}_K$; 2) the characterisation of reflexivity in \cite[p. 167]{James:1957} \cite[Thm. 5]{James:1964}, implying equivalence of reflexivity of $(X,\n{\cdot}_X)$ and existence of $\PPP^{d_{\n{\cdot}_X}}_K$ \cite[p. 253]{Phelps:1960}. Cf., e.g., \cite[Thm. 2.9]{Megginson:1984} and \cite[Thm. 5.1.18, p. 436]{Megginson:1998} for more details.};
\item\label{remark.varphi.iii.b} $\PPP^{d_{\n{\cdot}_X}}_K$ is norm-to-norm (resp., norm-to-weak) continuous on $X$ if{}f  $(X,\n{\cdot}_X)$ is strictly convex, reflexive, and satisfies the Radon--Riesz--Shmul'yan property \cite[Thm. (p. 813)]{Vlasov:1981}\footnote{\label{footnote.Fan.Glicksberg}\cite[Thm. 8]{Fan:Glicksberg:1958} proved an implication from right to left, while \cite[Thm. (p. 813)]{Vlasov:1981} established equivalence of norm-to-norm continuity of $\PPP^{d_{\n{\cdot}_X}}_K$ with $(X,\n{\cdot}_X)$ being strictly convex and having the Efimov--Stechkin property (cf. also \cite[Thm. (p. 459), p. 466]{Oshman:1971} for an earlier, and equivalent, characterisation result). By \cite[Cor. 3]{Singer:1964} (cf. also \cite[Prop. 2.5]{Vlasov:1973}), this is equivalent to say that $(X,\n{\cdot}_X)$ is strictly convex, reflexive, and has the Radon--Riesz--Shmul'yan property. (The claim of a counterexample for this characterisation, stated in \cite[Thm. 2.1]{Lambert:1975}, has been shown \cite[\S5]{Deutsch:Lambert:1980} to contain an error, invalidating this claim. On the other hand, the claim of characterisation of norm-to-norm continuity of $\PPP^{d_{\n{\cdot}_X}}_K$ by (reflexivity and strict convexity) of $(X,\n{\cdot}_X)$, stated in \cite[Thm. E]{Li:2004}, is not equipped with any proof, and refers to a paper that has never been published or cited elsewhere.)} (resp., \cite[Thm. 2.16]{Megginson:1984}).
\end{enumerate}
Proposition \ref{prop.left.right.psi.varphi}.\ref{prop.left.right.psi.varphi.i}--\ref{prop.left.right.psi.varphi.ii} (resp., Proposition \ref{prop.continuity}) can be seen as a Va\u{\i}nberg--Br\`{e}gman analogue of implication from right to left in \ref{remark.varphi.iii.a} (resp., \ref{remark.varphi.iii.b}). Furthermore, Proposition \ref{prop.varphi.uniform.continuity} provides a Va\u{\i}nberg--Br\`{e}gman analogue of the facts:
\begin{enumerate}[nosep,label=\alph*)]
\setcounter{enumii}{2}
\item\label{remark.varphi.iii.c} if $(X,\n{\cdot}_X)$ is uniformly convex, then $\PPP^{d_{\n{\cdot}_X}}_K$ are uniformly continuous on bounded subsets of $X$ \cite[Thm. 4.1]{Penot:2005:continuity}\footnote{For uniformly convex and uniformly Fr\'{e}chet differentiable $(X,\n{\cdot}_X)$ this implication has been obtained earlier in \cite[Thm. 2.(ii)]{Xu:Roach:1992} \cite[Thm. 3.1, Rem. 3.2]{Alber:Notik:1995} \cite[Thm. 3.1, Rem. 3.4]{Alber:1996:bound}. For uniformly convex (resp., uniformly convex and uniformly Fr\'{e}chet differentiable) $(X,\n{\cdot}_X)$ the implication of uniform continuity of  $\PPP^{d_{\n{\cdot}_X}}_K$ on bounded neighbourhoods of $K$ has been obtained earlier in \cite[Lem. 2.5]{Benyamini:Lindenstrauss:2000} (resp., \cite[Thm. 4]{Xu:1991}).};
\item\label{remark.varphi.iii.d} if $(X,\n{\cdot}_X)$ is $\frac{1}{\beta}$-uniformly convex, with $\beta\in\,]0,\frac{1}{2}]$, then $\PPP^{d_{\n{\cdot}_X}}_K$ are $\beta$-Lipschitz--H\"{o}lder continuous on bounded neighbourhoods of $K$ \cite[Thm. 5.7]{Adzhiev:2009};
\item\label{remark.varphi.iii.e} if $(X,\n{\cdot}_X)$ is $\frac{1}{\beta}$-uniformly convex and $r$-uniformly Fr\'{e}chet differentiable, with $\beta\in\,]0,\frac{1}{2}]$ and $r\in\,]1,2]$, then $\PPP^{d_{\n{\cdot}_X}}_K$ are $r\beta$-Lipschitz--H\"{o}lder continuous on bounded neighbourhoods of $K$ \cite[Thm. 5.8]{Adzhiev:2009}.
\end{enumerate}
In general, the results on behaviour of $D_{\Psi_\varphi}$-projections (existence and uniqueness, norm-to-norm continuity, uniform continuity, Lipschitz--H\"{o}lder continuity) require stronger sufficient conditions on $\n{\cdot}_X$ than those which are sufficient for the corresponding properties of metric projections. In all of these cases the additional strengthening guarantees a suitable differentiability of $\Psi_\varphi$ (or its Mandelbrojt--Fenchel dual), which is equivalent with a suitable continuity of $j_\varphi$ (or, respectively, $j_{\varphi^{\inver}}^\star$). 
\item\label{remark.varphi.iii.plus} The characterisation results \ref{remark.varphi.iii}.\ref{remark.varphi.iii.a} and \ref{remark.varphi.iii}.\ref{remark.varphi.iii.b}, considered in parallel to the characterisation provided by Proposition \ref{prop.legendre}, leads us to ask:
\begin{enumerate}[nosep,label=\arabic*)]
\item\label{remark.varphi.iii.1} are the conditions for $K$ being left $D_{\Psi_\varphi}$-Chebysh\"{e}v (resp., for norm-to-norm continuity of $\LPPP^{D_{\Psi_\varphi}}_K$), imposed in Proposition \ref{prop.left.right.psi.varphi} (resp., Proposition \ref{prop.continuity}.\ref{prop.continuity.i}), not only sufficient but also necessary?
\end{enumerate}
Additionally, a comparison of \ref{remark.varphi.iii}.\ref{remark.varphi.iii.a} and \ref{remark.varphi.iii}.\ref{remark.varphi.iii.b} with \ref{remark.varphi.iii}.\ref{remark.varphi.iii.c}, as well as a comparison of \ref{remark.varphi.iii}.\ref{remark.varphi.iii.c} and Proposition \ref{prop.varphi.uniform.continuity}.\ref{prop.varphi.uniform.continuity.i}--\ref{prop.varphi.uniform.continuity.ii} with \ref{remark.varphi.iii}.\ref{remark.varphi.iii.d}--\ref{remark.varphi.iii}.\ref{remark.varphi.iii.e}, in the context of Proposition \ref{prop.psi.varphi.geometry}, leads us to ask:
\begin{enumerate}[nosep,label=\arabic*)]
\setcounter{enumii}{1}
\item\label{remark.varphi.iii.2} does uniform continuity of $\PPP^{d_{\n{\cdot}_X}}_K$ on bounded subsets of $(X,\n{\cdot}_X)$ imply (and, thus, characterise) uniform convexity of $(X,\n{\cdot}_X)$?;
\item\label{remark.varphi.iii.3} do the results \ref{remark.varphi.iii.d}--\ref{remark.varphi.iii.e} hold, with the same values of parameters, globally (i.e. for the Lipschitz--H\"{o}lder continuity of $\PPP^{d_{\n{\cdot}_X}}_K$ on bounded subsets of $(X,\n{\cdot}_X)$)?
\end{enumerate}
\item\label{remark.varphi.iv} For any Gateaux differentiable $(X,\n{\cdot}_X)$, $\beta\in\,]0,1[$, and $\alpha\in\,]0,\infty[$, \eqref{eqn.varphi.alpha.beta.Psi.j.varphi} gives us a special case of \eqref{eqn.D.Psi.varphi.formula}, 
\begin{equation}
D_{\Psi_{\varphi_{\alpha,\beta}}}(x,y)=\textstyle\frac{1}{\alpha}\left(\beta\n{x}_X^{1/\beta}+(1-\beta)\n{y}_X^{1/\beta}-\n{y}_X^{1/\beta-2}\duality{x,j(y)}_{X\times X^\star}\right)\;\forall x,y\in X.
\label{eqn.psi.varphi.alpha.beta}
\end{equation}
The formula \eqref{eqn.psi.varphi.alpha.beta} is a tiny generalisation of $D_{\Psi_{\varphi_{1,\beta}}}$. For some discussion of the properties of $\Psi_{\varphi_{\alpha,\beta}}$ see \cite[p. 616]{Iusem:GarcigaOtero:2001}.
\item\label{remark.varphi.v} Proposition \ref{prop.legendre} provides a generalisation of \cite[Lem. 6.2]{Bauschke:Borwein:Combettes:2001}. The latter is recovered for $\varphi=\varphi_{1,\beta}$. For any Gateaux differentiable $(X,\n{\cdot}_X)$, \eqref{eqn.varphi.alpha.beta.Psi.j.varphi} gives gives $\DG\Psi_{\varphi_{1,\beta}}(x)=\n{x}_X^{\frac{1}{\beta}-2}j(x)$ (cf. \rpkmark{\cite{Lorch:1953}}). The corresponding Va\u{\i}nberg--Br\`{e}gman functional appeared implicitly in \cite[p. 68]{Yurgelas:1982}, and was explicitly discussed, together with a study of $\LPPP^{D_{\Psi_{\varphi_{1,\beta}}}}_K$ for nonempty, closed, convex $K\subseteq X$, in \cite[pp. 14--15]{Alber:1993} and \cite[\S7]{Alber:1996}, as well as in \rpkmark{\cite{Schoepfer:Louis:Schuster:2006,Schoepfer:Schuster:Louis:2007,Schuster:Kaltenbach:Hofmann:Kazimierski:2012}}. For any $(X,\n{\cdot}_X)$ which is reflexive, strictly convex, and has the Radon--Riesz--Shmul'yan property, total convexity of $\Psi_{\varphi_{1,\beta}}$ follows directly from \cite[Thm. 3.1]{Resmerita:2004}. 
\item\label{remark.varphi.vi} Proposition \ref{prop.continuity}.\ref{prop.continuity.i} is a generalisation of \cite[Cor. 4.4]{Resmerita:2004}. The latter is recovered for $\varphi=\varphi_{\beta,\beta}$ (for a Gateaux differentiable $(X,\n{\cdot}_X)$, the corresponding $D_{\Psi_{\varphi_{\beta,\beta}}}$ was discussed in \rpkmark{\cite{Alber:Butnariu:1997,Butnariu:Resmerita:2001}}). The use of Proposition \ref{prop.continuity.psi}.\ref{prop.continuity.psi.i} (resp., \ref{prop.continuity.psi}.\ref{prop.continuity.psi.iii}) in the proof of Proposition \ref{prop.continuity}.\ref{prop.continuity.i} (resp., \ref{prop.continuity}.\ref{prop.continuity.ii}), instead of Proposition \ref{prop.continuity.psi}.\ref{prop.continuity.psi.ii} (resp., \ref{prop.continuity.psi}.\ref{prop.continuity.psi.iv}) is due to their larger generality in the $\Psi=\Psi_\varphi$ case. More precisely, since \ref{prop.continuity.psi}.\ref{prop.continuity.psi.ii} requires $\Psi_\varphi$ to be totally convex on bounded subsets of $X$, and any $\Psi\in\pcl(X,\n{\cdot}_X)$ is totally convex on bounded subsets of $X$ if{}f it is uniformly convex on bounded subsets of $X$ \cite[Prop. 4.2]{Butnariu:Iusem:Zalinescu:2003}, Proposition \ref{prop.psi.varphi.geometry}.\ref{prop.psi.varphi.geometry.vi} implies that $(X,\n{\cdot}_X)$ has to be uniformly convex. Thus, using Proposition \ref{prop.continuity.psi}.\ref{prop.continuity.psi.ii} instead of Proposition \ref{prop.continuity.psi}.\ref{prop.continuity.psi.i} in the proof would require us to strengthen an assumption of (reflexivity, strict convexity, and the Radon--Riesz--Shmul'yan property of $(X,\n{\cdot}_X)$) to uniform convexity. Analogous situation holds for the Proposition \ref{prop.continuity}.\ref{prop.continuity.ii}.
\item\label{remark.varphi.vii} If $\alpha=1$ and $\beta=\frac{1}{2}$, then $\varphi_{\alpha,\beta}(t)=t$ and $\Psi_{\varphi_{1,1/2}}(x)=\frac{1}{2}\n{x}^2_X$. If $(X,\n{\cdot}_X)$ is Gateaux differentiable, then $\DG\Psi_{\varphi_{1,1/2}}=\n{\cdot}_X\DG\n{\cdot}_X=j$ and we obtain a special case of \eqref{eqn.psi.varphi.alpha.beta}, given by \cite[p. 1035]{Alber:Notik:1984} \cite[p. 5]{Alber:1986} \cite[\S7]{Alber:1993} \cite[\S7]{Alber:1996} (cf. also \cite[Def. 1]{Zarantonello:1984})
\begin{equation}
D_{\Psi_{\varphi_{1,1/2}}}(x,y)=\textstyle\frac{1}{2}\n{x}_X^2+\textstyle\frac{1}{2}\n{y}_X^2-\duality{x,j(y)}_{X\times X^\star}\;\forall x,y\in X.
\label{eqn.D.Psi.varphi.one.half}
\end{equation}
In general, if $(X,\n{\cdot}_X)$ is not a Hilbert space, then neither left nor right $D_{\Psi_{\varphi_{1,1/2}}}$-projections coincide with metric projections (cf. \cite[p. 39]{Alber:Butnariu:1997} for a simple example). If $(X,\n{\cdot}_X)$ is reflexive, Gateaux differentiable, strictly convex, and $x\in X$, then: left $D_{\Psi_{\varphi_{1,1/2}}}$-projections $\LPPP^{D_{\Psi_{\varphi_{1,1/2}}}}_K(x)$ onto closed convex sets $K$ are characterised as $z\in X$ satisfying variational inequality \cite[Prop. 7.c]{Alber:1996}
\begin{equation}
\duality{z-y,j(x)-j(z)}_{X\times X^\star}\geq0\;\;
\forall y\in K,
\label{eqn.alber.characterisation}
\end{equation}
which is a special case of \eqref{eqn.left.pyth.ii}; if $K\subseteq X$ is left $D_{\Psi_{\varphi_{1,1/2}}}$-Chebysh\"{e}v then $K$ is convex if{}f it is weakly closed \cite[Cor. 4.2]{Li:Song:Yao:2010} (cf. Corollary \ref{cor.left.D.psi.varphi.chebyshev.characterisation}). See \cite{Alber:1993,Alber:1996,Alber:2007} for further properties of left $D_{\Psi_{\varphi_{1,1/2}}}$-projections in this case.
\item\label{remark.varphi.viii} It is quite noticeable that Proposition \ref{prop.supercoercive.Psi.varphi} and Corollary \ref{cor.superc.EL.totconv.varphi}.\ref{cor.superc.EL.totconv.varphi.iii} provide jointly all three key convexity properties of $\Psi_\varphi$ (i.e. $\Psi_\varphi$ being supercoercive, totally convex, and Euler--Legendre) without assuming reflexivity of $(X,\n{\cdot}_X)$. Additionally, the sum of Proposition \ref{prop.supercoercive.Psi.varphi} and Corollary \ref{cor.superc.EL.totconv.varphi}.\ref{cor.superc.EL.totconv.varphi.iii} can be considered as a far generalisation of Example \ref{ex.EL.totalconvex.Rn.Psi}.\ref{ex.EL.totalconvex.Rn.Psi.vi} from $\Psi=\Psi_{\varphi_{1,1/2}}$ and $X=\RR^n$ to $\Psi=\Psi_\varphi$ for any gauge $\varphi$ and any Gateaux differentiable, strictly convex, locally uniformly convex Banach space $(X,\n{\cdot}_X)$. The equation \eqref{eqn.Kivinen.Warmuth.Hassibi.D.Psi} is recovered from \eqref{eqn.D.Psi.varphi.one.half} by setting $(X,\n{\cdot}_X)=(L_{1/\gamma}(\X,\mu),\n{\cdot}_{1/\gamma})$ with purely atomic finite $(\X,\mu)$, and applying the formula for $j$ on $(L_{1/\gamma}(\X,\mu),\n{\cdot}_{1/\gamma})$, which reads \cite[p. 132]{Mazur:1933:schwache} $j(x)=\n{x}_{1/\gamma}^{2-1/\gamma}\ab{x}^{1/\gamma-1}\sgn(x)$.
\item\label{remark.varphi.ix} If $\alpha=1$, $\beta=\frac{1}{2}$, and $(X,\n{\cdot}_X)$ is a Hilbert space $(\H,\s{\cdot,\cdot}_\H)$, then $(\Psi_{\varphi_{1,1/2}})^\lfdual=\Psi_{\varphi_{1,1/2}}=\frac{1}{2}\n{\cdot}^2_\H$, $\DG\Psi_{\varphi_{1,1/2}}=\id_\H$ with $\DG\Psi_{\varphi_{1,1/2}}(y)(x)=\s{x,y}_\H$, and \eqref{eqn.D.Psi.varphi.one.half} turns into \cite[p. 1021]{Bregman:1966} \cite[\S2.1]{Bregman:1966:PhD}
\begin{equation}
D_{\Psi_{\varphi_{1,1/2}}}(x,y)=\textstyle\frac{1}{2}\n{x}^2_\H+\textstyle\frac{1}{2}\n{y}^2_\H-\s{x,y}_\H=\textstyle\frac{1}{2}\n{x-y}^2_\H\;\forall x,y\in\H.
\label{eqn.bregman.hilbert}
\end{equation}
In consequence, Chebysh\"{e}v, left $D_{\Psi_{\varphi_{1,1/2}}}$-Chebysh\"{e}v, and right $D_{\Psi_{\varphi_{1,1/2}}}$-Chebysh\"{e}v subsets of $\H$ coincide, with 
\begin{equation}
\LPPP^{D_{\Psi_{\varphi_{1,1/2}}}}_K(y)=\RPPP^{D_{\Psi_{\varphi_{1,1/2}}}}_K(y)=\PPP^{d_{\n{\cdot}_\H}}_K(y)\;\forall y\in\H\;\;\forall\mbox{ Chebysh\"{e}v }K\subseteq\H.
\end{equation}
In particular, for any convex closed $K\subseteq\H$, the metric projection $\PPP^{d_{\n{\cdot}_\H}}_K(y)$ is characterised as a map $T:\H\ra K$ satisfying \cite[p. 87]{Aronszajn:1950}
\begin{equation}
\s{y-T(x),x-T(x)}_\H\leq0\;\;\forall(x,y)\in\H\times K.
\label{eqn.aronszajn}
\end{equation}
If $K$ is affine, then the generalised pythagorean equation \eqref{eqn.left.pyth} turns into
\begin{equation}
\n{x-y}_\H^2=\n{x-\PPP^{d_{\n{\cdot}_\H}}_K(y)}_\H^2+\n{\PPP^{d_{\n{\cdot}_\H}}_K(y)-y}_\H^2\;\;\forall(x,y)\in K\times\H
\end{equation}
(with $\PPP^{d_{\n{\cdot}_\H}}_K$ given by the bounded linear projection operator $P_K:\H\ra K$ if $K$ is a linear subspace of $\H$), while the generalised cosine equation \eqref{eqn.generalised.cosine} turns into
\begin{equation}
\n{x-z}_\H^2=\n{x-y}_\H^2+\n{y-z}_\H^2-2\s{x-y,z-y}_\H\;\forall x,y,z\in\H
\label{eqn.cosine.hilbert}
\end{equation}
(which, for $\H=\RR^2$ with $\s{x,y}_\H=\sum_{i=1}^2x_iy_i$, gives the cartesian version of a planar cosine theorem of al-K\={a}sh\={a}n\={\i} \cite{alKashani:1427}\footnote{More precisely, al-K\={a}sh\={a}n\={\i} states $c=\sqrt{(a\sin\theta)^2+(b-a\cos\theta)^2}$, that is equivalent to $c^2=a^2+b^2-2ab\cos\theta$ via $(\sin\theta)^2+(\cos\theta)^2=1$. Cf. p. 143 of Russ. transl. or p. 31 of Vol. 2 of Engl. transl.}). On the other hand, left and right strongly quasinonexpansive operators with respect to $\Psi_{\varphi_{1,1/2}}$ on $\H$ do not coincide, in general, with the strongly $\n{\cdot}_\H$-nonexpansive operators of \cite[p. 459]{Bruck:Reich:1977}, although for $\H=\RR^n$ the latter are the subset of the former \cite[Rem. 3]{MartinMarquez:Reich:Sabach:2013}.
\item\label{remark.varphi.xix} While there is no general notion of an angle between two elements of a general Banach space, the relationship between \eqref{eqn.generalised.cosine} and \eqref{eqn.cosine.hilbert} allows us to introduce a \df{$\Psi$-angle} between nonzero vectors $x-y,z-y\in X$, for any reflexive and Gateaux differentiable Banach space $(X,\n{\cdot}_X)$:
\begin{equation}
\measuredangle_\Psi(x-y,z-y):=\arccos\left(\frac{\duality{x-y,\DG\Psi(z)-\DG\Psi(y)}_{X\times X^\star}}{2\n{x-y}_X\n{z-y}_X}\right).
\label{eqn.Psi.angle}
\end{equation}
\item\label{remark.varphi.x} Not much is known so far about weak sequential continuity of $j_\varphi=\DG\Psi_\varphi$ for arbitrary gauge $\varphi$ and arbitrary Banach spaces. It is known to hold for $\varphi=\varphi_{1,\beta}$ on sequence spaces $(l_{1/\beta},\n{\cdot}_{1/\beta})$ with $\beta\in\,]0,1[$ \cite[Lem. 5]{Browder:1966}, and on arbitrary infinite-dimensional Hilbert spaces if{}f $\beta=\frac{1}{2}$ \cite[Prop. 3.3]{Xu:Kim:Yin:2014}. On the other hand, it is known that $j_\varphi$ is not weakly sequentially continuous for arbitrary $\varphi$ on $(L_{1/\gamma}(\X,\mu),\n{\cdot}_{1/\gamma})$ spaces with $\gamma\in\,]0,1[\setminus\{\frac{1}{2}\}$ and nonatomic finite $(\X,\mu)$ \cite[Lem. 3, \S5]{Opial:1967} (cf. \cite[p. 268]{Browder:1966} for $\gamma=\frac{1}{4}$ case), and for $\varphi(t)=t$ on $(l_{1/\gamma},\n{\cdot}_{1/\gamma})$ spaces for $p\in\,]0,1[\setminus\{\frac{1}{2}\}$ \cite[Prop. 3.2]{Xu:Kim:Yin:2014}. These are quite severe limitations, appearing already at the range of elementary model spaces. By this reason, in Propositions \ref{prop.lsq.rsq.new}.\ref{prop.lsq.rsq.new.iv} and \ref{prop.varphi.compositional} we have omitted the case \ref{prop.lsq.rsq.old.iv.a} of Proposition \ref{prop.lsq.rsq.old}.\ref{prop.lsq.rsq.old.iv} in favour of case \ref{prop.lsq.rsq.old.iv.b}, which is much better behaved geometrically, and (as we will show in Section \ref{section.models}) admits a direct application to a range of noncommutative model spaces. In the broader perspective, dependence of weak sequential continuity of $\DG\Psi_\varphi$ on the specific choice of $\varphi$ makes it a property of a different character from all other properties of $\Psi_\varphi$ and $\DG\Psi_\varphi$ used in this paper for the case I--IV models. So, even if it would be available for a larger class of models, relying on it would break the invariance of our framework with respect to the choice of a gauge, and this would be a structurally undesirable feature. Nevertheless, in face of the relationships in Proposition \ref{prop.psi.varphi.geometry}, it is tempting to ask: what kind of differentiability property of $\Psi_\varphi$ (and of $(X,\n{\cdot}_X)$) is equivalent to weak sequential continuity of $j_\varphi$?
\item\label{remark.varphi.xi} Corollary \ref{cor.varphi.uniformly.convex} provides a generalisation of the duality between $\frac{1}{\gamma}$-uniformly convex and $\frac{1}{1-\gamma}$-uniformly Fr\'{e}chet differentiable Banach spaces, $\gamma:=\frac{1}{r}\in\,]0,1[$, as exhibited in Proposition \ref{prop.psi.varphi.geometry}.\ref{prop.psi.varphi.geometry.vii}--\ref{prop.psi.varphi.geometry.viii} (and originally stated in \cite[p. 63 (Vol. 2)]{Lindenstrauss:Tzafriri:1977:1979}; the case $\gamma=\frac{1}{2}$ goes back to \cite[Lem. 4]{Lindenstrauss:1963}). Proposition \ref{prop.psi.varphi.geometry}.\ref{prop.psi.varphi.geometry.viii}.\ref{prop.psi.varphi.geometry.viii.b} leads us to ask: is it possible to identify a specific type of uniform continuity of $j_\varphi$ which would be equivalent to $\varphi$-uniform Fr\'{e}chet differentiability of $(X,\n{\cdot}_X)$? And, if yes, then is it possible to use it to generalise Proposition \ref{prop.varphi.uniform.continuity} by replacing $\varphi_{1,\beta}$ by any gauge $\varphi$, together with replacing $\frac{1}{\beta}$-uniform convexity (resp., $\frac{1}{\beta}$-uniform Fr\'{e}chet differentiability) by $\varphi$-uniform convexity (resp., $\varphi$-uniform Fr\'{e}chet differentiability)? An analogous question of an extension rises with respect to \cite[Thm. 5]{Takahashi:Hashimoto:Kato:2002} (=\cite[Thms. 1, 2]{Kato:Takahashi:Hashimoto:2002}), which states that $r$-uniform convexity (resp., $r$-uniform Fr\'{e}chet differentiability) of a Banach space $(X,\n{\cdot}_X)$ is equivalent with $(X,\n{\cdot}_X)$ having a \textit{strong} type (resp., \textit{strong} cotype) $r$, as defined in \cite[Def. 2]{Takahashi:Hashimoto:Kato:2002} (=\cite[Rems. 2.(iii), 3.(iii)]{Kato:Takahashi:Hashimoto:2002}). Is it possible to define the corresponding notions of a strong $\varphi$-type (resp., strong $\varphi$-cotype), which would be equivalent to $\varphi$-uniform convexity (resp., $\varphi$-uniform Fr\'{e}chet differentiability)?
\item\label{remark.varphi.xii} The special case of Proposition \ref{prop.alber.decomposition.psi.varphi}, for $\varphi(t)=t$ and $\LPPP^{D_{\Psi_\varphi}}_{K_1}$, has been obtained in \cite[Thm. 2.4]{Alber:2000} (cf. \cite[Thm. 2.13]{Alber:2005} for its nontrivial consequence). For $\varphi(t)=t$ and $(X,\n{\cdot}_X)$ given by a Hilbert space, this result has been obtained in \cite[Prop. 1]{Moreau:1962:decomposition}.
\item\label{remark.varphi.xiii} If $(X,\n{\cdot}_X)$ is reflexive and Gateaux differentiable, and $\varphi$ is a gauge, then $\forall x,y\in X$
\begin{equation}
D_{\Psi_\varphi}(x,y)=\left\{
\begin{array}{ll}
\Psi_\varphi(x)+(\Psi_\varphi)^\lfdual(j_\varphi(y))-\lumer{x,y}_\varphi&\st y\neq0\\
\Psi_\varphi(x)&\st y=0,
\end{array}
\right.
\end{equation}
where
\begin{equation}
\lumer{\,\cdot\,,x}_\varphi:=\left\{
\begin{array}{ll}
\frac{\varphi(\n{x}_X)}{\n{x}_X}\lumer{\,\cdot\,,x}&\st x\neq0\\
0&\st x=0,
\end{array}
\right.
\end{equation}
and $\lumer{x,y}:=(j(y))(x)$. For any Banach space $(X,\n{\cdot}_X)$, if $x,y\in X$, then $x$ is said to be \df{orthogonal} to $y$ if{}f $\n{x+\lambda y}_X\geq\n{x}_X$ $\forall\lambda\in\RR$ \cite[p. 169]{Birkhoff:1935}. If $(X,\n{\cdot}_X)$ is Gateaux differentiable, then $x$ is orthogonal to $y$ if{}f $(j(y))(x)=0$ (cf., e.g., \cite[Prop. 1.4.4]{Fleming:Jamison:2002}). Hence, $\lumer{\,\cdot\,,\,\cdot\,}_\varphi$ can be seen as a generalised form of orthogonality. (The notation $\lumer{\,\cdot\,,\,\cdot\,}$ refers to Lumer's semi-inner product \cite[Def. 1]{Lumer:1961}, which in the case of Gateaux differentiable $(X,\n{\cdot}_X)$ is given uniquely by $(j(\cdot))(\cdot)$.) In particular, for $\Psi=\Psi_\varphi$, the formula \eqref{eqn.Psi.angle} turns into
\begin{equation}
\measuredangle_{\Psi_\varphi}(x-y,z-y):=\arccos\left(\frac{\lumer{x-y,z}_{\varphi}-\lumer{x-y,y}_{\varphi}}{2\n{x-y}_X\n{z-y}_X}\right).
\end{equation}
\item\label{remark.varphi.xiv} Corollary \ref{cor.VB.quasigauge.formulas} provides an alternative proof of Corollary \ref{cor.superc.EL.totconv.varphi}.\ref{cor.superc.EL.totconv.varphi.i}. Proposition \ref{prop.supercoercive.Psi.varphi} (resp., \ref{prop.legendre}; \ref{prop.left.right.psi.varphi}) is a special case of Proposition \ref{prop.supercoercive.Psi.quasigauge}.\ref{prop.supercoercive.Psi.quasigauge.iii} (resp., \ref{prop.legendre.quasigauge}; \ref{prop.left.right.psi.quasigauge}). We have separated these propositions in order to illustrate the differences showing up under generalisation from gauges to quasigauges. In principle, provided a quasigauge generalisation of Propositions \ref{prop.psi.varphi.geometry}.\ref{prop.psi.varphi.geometry.iv} and \ref{prop.psi.varphi.geometry}.\ref{prop.psi.varphi.geometry.xii}, one could use Propositions \ref{prop.continuity.psi} and \ref{prop.lsq.rsq.new}, combined with Propositions \ref{prop.psi.quasigauge.geometry}.\ref{prop.psi.quasigauge.geometry.iii}--\ref{prop.psi.quasigauge.geometry.iv}, to obtain a quasigauge generalisation of Propositions \ref{prop.continuity} and \ref{prop.varphi.compositional}. (Even a quasigauge analogue of Proposition \ref{prop.psi.varphi.geometry}.\ref{prop.psi.varphi.geometry.xii} would suffice, although in this case the corresponding generalisation of Proposition \ref{prop.continuity} would be less general, using local uniform convexity of $(X,\n{\cdot}_X)$ and $(X,\n{\cdot}_X)^\star$ instead of their Fr\'{e}chet differentiability.) For our current purposes it is sufficient to compare Proposition \ref{prop.left.right.psi.varphi} with Proposition \ref{prop.left.right.psi.quasigauge}: already at this level, there is a noticeable difference between conditions required for a quasigauge for either left or right pythagoreanity of $D_{\Psi_\varphi}$. Furthermore, in the left case there are three inequivalent conditions available, while in the right case there are six inequivalent conditions. An inspection of the proofs leading to this result, together with a look at Propositions \ref{prop.psi.quasigauge.geometry}.\ref{prop.psi.quasigauge.geometry.iii}--\ref{prop.psi.quasigauge.geometry.iv}, shows that the conditions imposed on $\varphi$ in the quasigauge analogues of Propositions \ref{prop.continuity} and \ref{prop.varphi.compositional} will not be the same as in Proposition \ref{prop.left.right.psi.quasigauge}. Thus, while there is no a priori constraints on the gauge functions used in the Propositions \ref{prop.left.right.psi.varphi}, \ref{prop.continuity}, and \ref{prop.varphi.compositional}, their quasigauge analogues introduce a substantial split of the assumptions on $\varphi$ used in each of the corresponding propositions. So, while the properties of case I--IV models are independent of the choice of a gauge, they are sensitive to the choice of a quasigauge.
\item\label{remark.varphi.xv} In principle, due to Lemma \ref{lem.quasigauge.inverses}.\ref{lem.quasigauge.inverses.i}, given a quasigauge $\varphi$, one can use $(\lim_{s\ra^+t}\varphi(s))^\meet$ (resp, $(\lim_{s\ra^-t}\varphi(s))^\join$) instead of $\varphi^\meet$ (resp., $\varphi^\join$), relying on \eqref{eqn.dual.quasigauge.integral.limits} and Lemma \ref{lem.fenchel.dual.psi.quasigauge}.\ref{lem.fenchel.dual.psi.quasigauge.i} instead of \eqref{eqn.dual.quasigauge.integral.meet} (resp., \eqref{eqn.dual.quasigauge.integral.join}) and \eqref{eqn.fenchel.dual.psi.quasigauge.meet} (resp., \eqref{eqn.fenchel.dual.psi.quasigauge.join}) in Propositions \ref{prop.supercoercive.Psi.quasigauge}.\ref{prop.supercoercive.Psi.quasigauge.iii}, \ref{prop.legendre.quasigauge}, and \ref{prop.left.right.psi.quasigauge}. However, while this would make these propositions a bit more general, it would also make them less readable.
\item\label{remark.varphi.xvi} If $\efd(j_\varphi)=\intefd{j_\varphi}=\efd(\Psi_\varphi)=\intefd{\Psi_\varphi}=\INT(\sup(\efd(\varphi))B(X,\n{\cdot}_X))$, then the assumptions \eqref{eqn.euler.legendre.quasigauge.conditions} simplify, since their first and third line become obsolete. However, in a general case, we know only that $\INT(\sup(\efd(\varphi))B(X,\n{\cdot}_X))\subseteq\efd(\Psi_\varphi)\subseteq\cl(\INT(\sup(\efd(\varphi))B(X,\n{\cdot}_X)))$ \cite[p. 369]{Zalinescu:1983}.
\item\label{remark.varphi.xvii} The result in Proposition \ref{prop.varphi.uniform.continuity}.\ref{prop.varphi.uniform.continuity.i} was obtained earlier, by a different method, in \cite[Prop. 6.27.(a)]{Schuster:Kaltenbach:Hofmann:Kazimierski:2012}. By \cite[Eqn. (6.100)]{Schuster:Kaltenbach:Hofmann:Kazimierski:2012}, under conditions on $(X,\n{\cdot}_X)$ given in \ref{prop.varphi.uniform.continuity}.\ref{prop.varphi.uniform.continuity.i}, $\exists\lambda>0$ $\forall x,y\in X$
\begin{equation}
\n{\LPPP^{D_{\Psi_{\varphi_{1,\beta}}}}_K(x)-\LPPP^{D_{\Psi_{\varphi_{1,\beta}}}}_K(y)}_X
\leq
\lambda\left(\max\left\{\n{\PPP^{d_{\n{\cdot}_X}}_K(0)}_X,\n{x}_X,\n{y}_X\right\}\n{j_{\varphi_{1,\beta}}(x)-j_{\varphi_{1,\beta}}(y)}_X\right)^\beta.
\label{eqn.SKHK}
\end{equation}
Assuming boundedness of $K$, boundedness of domain $Q\subseteq X$ of $\LPPP^{D_{\Psi_{\varphi_{1,\beta}}}}_K$, $r$-uniform Fr\'{e}chet differentiability of $(X,\n{\cdot}_X)$, and using Proposition \ref{prop.psi.varphi.geometry}.\ref{prop.psi.varphi.geometry.viii}.\ref{prop.psi.varphi.geometry.viii.c}, we get
\begin{equation}
\exists\lambda>0\;\forall x,y\in Q\;\;\n{\LPPP^{D_{\Psi_{\varphi_{1,\beta}}}}_K(x)-\LPPP^{D_{\Psi_{\varphi_{1,\beta}}}}_K(y)}_X\leq\lambda\n{x-y}_X^{\beta(r-1)}.
\label{eqn.Hoelder.left.breg.proj}
\end{equation}
Hence, for bounded $K$, and under all assumptions of Proposition \ref{prop.varphi.uniform.continuity}.\ref{prop.varphi.uniform.continuity.i}, this gives $\beta(r-1)$-Lipschitz--H\"{o}lder continuity of $\LPPP^{D_{\Psi_{\varphi_{1,\beta}}}}_K$ on bounded subsets of $X$. In comparison with Proposition \ref{prop.varphi.uniform.continuity}.\ref{prop.varphi.uniform.continuity.ii}, this conclusion is weaker regarding the value of the exponent of Lipschitz--H\"{o}lder continuity (since $1\geq\frac{\beta}{1-\beta}(r-1)>\beta(r-1)$) and regarding the assumptions on its domain (limitation to bounded subsets of $X$), while assuming more (boundedness of $K$). An analogue of a proof of Proposition \ref{prop.varphi.uniform.continuity}.\ref{prop.varphi.uniform.continuity.iv}, using:
\begin{enumerate}[nosep,label=\alph*)]
\item\label{remark.varphi.xvii.a} given a convex and bounded subset $U$ (resp., $W$) of a Banach space $(X_1,\n{\cdot}_{X_1})$ (resp., $(X_2,\n{\cdot}_{X_2})$), and a Banach space $(X_3,\n{\cdot}_{X_3})$, if $f:U\ra X_2$ is $s$-Lipschitz--H\"{o}lder continuous, $g:W\ra X_3$ is $\lambda$-Lipschitz--H\"{o}lder continuous, $f(U)\subseteq W$, and $s,\lambda\in\,]0,1]$, then $g\circ f$ is $s\lambda$-Lipschitz--H\"{o}lder continuous \cite[Thm. 4.3, Prop. 5.2]{delaLlave:Obaya:1999};
\item\label{remark.varphi.xvii.b} by definition of $j_\varphi$, it maps bounded sets to bounded sets (cf. the proof of Proposition \ref{prop.supercoercive.Psi.varphi}),
\end{enumerate}
gives $\frac{(1-\beta)^2}{\beta}(r-1)^2$-Lipschitz--H\"{o}lder continuity of $\RPPP^{D_{\Psi_{\varphi_{1,\beta}}}}_K$ on bounded and convex subsets of $X$, under the assumptions of Proposition \ref{prop.varphi.uniform.continuity}.\ref{prop.varphi.uniform.continuity.iv}, equipped with an additional requirement that $j_{\varphi_{1,\beta}}(K)$ is bounded. This conclusion is weaker than Proposition \ref{prop.varphi.uniform.continuity}.\ref{prop.varphi.uniform.continuity.iv}.
\item\label{remark.varphi.xviii} The continuity results in Proposition \ref{prop.continuity} cannot be improved using Proposition \ref{prop.varphi.prox.norm.continuity}, since $\lim_{\n{x}_X\ra\infty}\iota_K(x)\neq\infty$.
\end{enumerate}
\end{remark}
\subsection{$D_{\ell,\Psi}$}\label{section.convex.new.d.ell.psi}
\begin{definition}\label{def.d.ell.psi}
Given Banach spaces $(X,\n{\cdot}_X)$ and $(Y,\n{\cdot}_Y)$, $Z\subseteq Y$, $\Psi\in\pclg(X,\n{\cdot}_X)$, let $\ell:Z\ra\ell(Z)\subseteq X$ be a bijection such that $\ell(Z)\cap\intefd{\Psi}\neq\varnothing$. Then
\begin{equation}
D_{\ell,\Psi}(\phi,\psi):=D_\Psi(\ell(\phi),\ell(\psi))\;\;
\forall(\phi,\psi)\in Z\times Z
\label{eqn.d.ell.psi}
\end{equation}
will be called an \df{extended Va\u{\i}nberg--Br\`{e}gman functional}.
\end{definition}

\begin{definition}
Given Banach spaces $(X,\n{\cdot}_X)$ and $(Y,\n{\cdot}_Y)$, $\varnothing\neq Z\subseteq Y$, $\varnothing\neq K\subseteq X$, a bijection $\ell:Z\ra\ell(Z)\subseteq X$ with $K\subseteq\ell(Z)$, and a function $T:K\ra\ell(Z)$, the function $T^\ell:=\ell^{\inver}\circ T\circ\ell:\ell^{\inver}(K)\ra Z$ will be called an \df{$\ell$-operator}.
\end{definition}

\begin{definition}\label{def.d.ell.psi.notions}
Under assumptions on $(\ell,\Psi)$ as in Definition \ref{def.d.ell.psi}, and with $\varnothing\neq C\subseteq Z$,
\begin{enumerate}[nosep,label=(\roman*)]
\item\label{def.d.ell.psi.notions.i} $C$ will be called $\df{left}$ (resp., \df{right}) \df{$D_{\ell,\Psi}$-Chebysh\"{e}v} if{}f $\ell(C)$ is left (resp., right) $D_\Psi$-Chebysh\"{e}v, with the corresponding \df{left} (resp., \df{right}) \df{$D_{\ell,\Psi}$-projections} given by
\begin{align}
\LPPP^{D_{\ell,\Psi}}_C(\phi):=\ell^{\inver}\circ\LPPP^{D_\Psi}_{\ell(C)}\circ\ell(\phi)\;\;\forall\phi\in\ell^{\inver}(\intefd{\Psi}\cap\ell(Z))\\
\mbox{(resp., }
\RPPP^{D_{\ell,\Psi}}_C(\phi):=\ell^{\inver}\circ\RPPP^{D_\Psi}_{\ell(C)}\circ\ell(\phi)\;\;\forall\phi\in\ell^{\inver}(\intefd{\Psi}\cap\ell(Z))
\mbox{ )};
\end{align}
\item\label{def.d.ell.psi.notions.ii} $D_{\ell,\Psi}$ will be called \df{left} (resp., \df{right}) \df{pythagorean} on $C$ if{}f $D_\Psi$ is left (resp., right) pythagorean on $\ell(C)$;
\item\label{def.d.ell.psi.notions.iii} $C$ will be called \df{$\ell$-convex} (resp.,  \df{$\ell$-closed};  \df{$\ell$-affine};  \df{$\ell$-bounded};  \df{$\DG\Psi\circ\ell$-convex};  \df{$\DG\Psi\circ\ell$-closed};  \df{$\DG\Psi\circ\ell$-affine}) if{}f $\ell(C)$ is convex (resp.,  closed;  affine; bounded; $\DG\Psi$-convex;  $\DG\Psi$-closed;  $\DG\Psi$-affine);
\item\label{def.d.ell.psi.notions.iv} $\LPPP^{D_{\ell,\Psi}}_C$ (resp., $\RPPP^{D_{\ell,\Psi}}_C$) will be called \df{zone consistent} if{}f $\ell(C)\subseteq\intefd{\Psi}$ and  $\LPPP^{D_{\Psi}}_{\ell(C)}$ (resp., $\RPPP^{D_{\Psi}}_{\ell(C)}$) is zone consistent;
\item\label{def.d.ell.psi.notions.v} $\Psi\circ\ell:Z\ra\,]-\infty,\infty]$ and $\Psi^\lfdual\circ\DG\Psi\circ\ell:\ell^{\inver}(\intefd{\Psi}\cap\ell(Z))\ra\,]-\infty,\infty]$ will be called \df{$(\ell,\Psi)$-potentials};
\item\label{def.d.ell.psi.notions.vi} the topology on $Z$ induced by $\ell$ (resp., $\DG\Psi\circ\ell$) from the norm topology of $(X,\n{\cdot}_X)$ (resp., $(X^\star,\n{\cdot}_{X^\star})$) will be called \df{$\ell$-topology} (resp., \df{$\DG\Psi\circ\ell$-topology});
\item\label{def.d.ell.psi.notions.vii} An $\ell$-operator $T^\ell:C\ra\ell^{\inver}(\intefd{\Psi}\cap\ell(Z))$ will be called \df{left} (resp., \df{right}) \df{strongly quasinonexpansive} with respect to $(\ell,\Psi)$ and $C$ if{}f $\ell(C)\subseteq\intefd{\Psi}$ and $T$ is left (resp., right) strongly quasinonexpansive with respect to $\Psi$ and $\ell(C)$; the set of all $\ell$-operators which are left (resp., right) strongly quasinonexpansive with respect to $(\ell,\Psi)$ and $C$ will be denoted $\lsq(\ell,\Psi,C)$ (resp., $\rsq(\ell,\Psi,C)$);
\item\label{def.d.ell.psi.notions.viii} the set $\lsq(\ell,\Psi,C)$ (resp., $\rsq(\ell,\Psi,C)$) will be called \df{composable} if{}f $\lsq(\Psi,\ell(C))$ (resp., $\rsq(\Psi,\ell(C))$) is composable;
\item\label{def.d.ell.psi.notions.ix} $\LPPP^{D_{\ell,\Psi}}_C$ (resp., $\RPPP^{D_{\ell,\Psi}}_C$) will be called \df{adapted} if{}f $\LPPP^{D_\Psi}_C$ (resp., $\RPPP^{D_\Psi}_C$) is adapted;
\item\label{def.d.ell.psi.notions.x} an $\ell$-operator $T^\ell:C\ra\ell^{\inver}(\intefd{\Psi}\cap\ell(Z))$ will be called \df{completely nonexpansive} with respect to $(\ell,\Psi)$ and $C$ if{}f $\ell(C)\subseteq\intefd{\Psi}$ and $T$ is completely nonexpansive with respect to $\Psi$ and $\ell(C)$; the set of all $\ell$-operators which are completely nonexpansive with respect to $(\ell,\Psi)$ and $C$ will be denoted $\cn(\ell,\Psi,C)$;
\item\label{def.d.ell.psi.notions.xi} if $T:\ell(Z)\ra 2^{X^\star}$, $\Graph(T)\neq\varnothing$, and $\lambda\in\,]0,1[$, then the \df{left $D_{\ell,\Psi}$-resolvent} of $T$ is defined as $\lres^{\ell,\Psi}_{\lambda T}:=\ell^\inver\circ\lres^{\Psi}_{\lambda T}\circ\ell$.
\end{enumerate}
\end{definition}

\begin{corollary}\label{cor.d.ell.psi.properties}
If $(X,\n{\cdot}_X)$ and $(Y,\n{\cdot}_Y)$ are Banach spaces, $\Psi\in\pclg(X,\n{\cdot}_X)$, $\varnothing\neq C\subseteq Z\subseteq Y$, $\ell:Z\ra\ell(Z)\subseteq X$ is a bijection such that $\ell(Z)\cap\intefd{\Psi}\neq\varnothing$, then:
\begin{enumerate}[nosep,label=(\roman*)]
\item\label{cor.d.ell.psi.properties.i} if $D_\Psi$ is an information on $\ell(Z)$, then $D_{\ell,\Psi}$ is an information on $Z$;
\item\label{cor.d.ell.psi.properties.ii} $C$ is $\ell$-closed if{}f it is closed in the topology induced by $\ell$ from the norm topology of $(X,\n{\cdot}_X)$;
\item\label{cor.d.ell.psi.properties.iii} if $(X,\n{\cdot}_X)$ reflexive, $C$ is $\ell$-closed and $\ell$-convex, then:
\begin{enumerate}[nosep,label=\arabic*)]
\item\label{cor.d.ell.psi.properties.iii.1} if any of the following (generally, inequivalent) conditions holds:
\begin{enumerate}[nosep,label=\alph*)]
\item\label{cor.d.ell.psi.properties.iii.1.a} $\Psi$ is totally convex on $\efd(\Psi)$, $\ell(C)\subseteq\intefd{\Psi}$; or
\item\label{cor.d.ell.psi.properties.iii.1.b} $\Psi$ is strictly convex on $\efd(\Psi)$ and supercoercive, $\ell(C)\cap\intefd{\Psi}\neq\varnothing$;
\item\label{cor.d.ell.psi.properties.iii.1.c} $\Psi$ is Euler--Legendre, $\ell(C)\cap\intefd{\Psi}\neq\varnothing$,
\end{enumerate}
then $C$ is left $D_{\ell,\Psi}$-Chebysh\"{e}v and $D_{\ell,\Psi}$ is left pythagorean on $C$;
\item\label{cor.d.ell.psi.properties.iii.2} if (\ref{cor.d.ell.psi.properties.iii.1}.\ref{cor.d.ell.psi.properties.iii.1.c} holds) or (\ref{cor.d.ell.psi.properties.iii.1}.\ref{cor.d.ell.psi.properties.iii.1.a} or \ref{cor.d.ell.psi.properties.iii.1}.\ref{cor.d.ell.psi.properties.iii.1.b} holds, and $\ell(C)\subseteq\intefd{\Psi}$), then $\LPPP^{D_{\ell,\Psi}}_C$ is zone consistent;
\item\label{cor.d.ell.psi.properties.iii.3} if any of \ref{cor.d.ell.psi.properties.iii.1}.\ref{cor.d.ell.psi.properties.iii.1.a}--\ref{cor.d.ell.psi.properties.iii.1}.\ref{cor.d.ell.psi.properties.iii.1.c} holds, then $D_{\ell,\Psi}$ is an information on $Z$;
\item\label{cor.d.ell.psi.properties.iii.4} if \ref{cor.d.ell.psi.properties.iii.1}.\ref{cor.d.ell.psi.properties.iii.1.c} holds, then $D_{\DG\Psi\circ\ell,\Psi^\lfdual}$ is an information on $Z$;
\item\label{cor.d.ell.psi.properties.iii.5} if any of \ref{cor.d.ell.psi.properties.iii.1}.\ref{cor.d.ell.psi.properties.iii.1.a}--\ref{cor.d.ell.psi.properties.iii.1}.\ref{cor.d.ell.psi.properties.iii.1.c} holds, and $C$ is $\ell$-affine, then
\begin{equation}
	D_{\ell,\Psi}(\phi,\LPPP^{D_{\ell,\Psi}}_C(\psi))+D_{\ell,\Psi}(\LPPP^{D_{\ell,\Psi}}_C(\psi),\psi)=D_{\ell,\Psi}(\phi,\psi)\;\;\forall(\phi,y)\in C\times\ell^{\inver}(\intefd{\Psi}\cap\ell(Z));
\label{eqn.left.pyth.d.ell.psi}
\end{equation}
\item\label{cor.d.ell.psi.properties.iii.6} \sloppy if $\ell(C)\subseteq\intefd{\Psi}$, and $\Psi$ is Fr\'{e}chet differentiable on $\intefd{\Psi}$, totally convex on $\efd(\Psi)$, and supercoercive, then $\LPPP^{D_{\ell,\Psi}}_C$ is $\ell$-topology-to-$\ell$-topology continuous on $\ell^{\inver}(\intefd{\Psi}\cap\ell(Z))$, while $\inf_{\phi\in K}\{D_{\ell,\Psi}(\phi,\,\cdot\,)\}$ is continuous in $\ell$-topology on $\ell^{\inver}(\intefd{\Psi}\cap\ell(Z))$;
\end{enumerate}
\item\label{cor.d.ell.psi.properties.iv} if $(X,\n{\cdot}_X)$ is reflexive, $C$ is $\DG\Psi\circ\ell$-closed and $\DG\Psi\circ\ell$-convex, $\ell(C)\subseteq\intefd{\Psi}$, $\Psi^\lfdual$ is Gateaux differentiable on $\varnothing\neq\DG\Psi(\intefd{\Psi})\subseteq\intefd{\Psi^\lfdual}$, then:
\begin{enumerate}[nosep,label=\arabic*)]
\item\label{cor.d.ell.psi.properties.iv.1} if any of the following (generally, inequivalent) conditions holds:
\begin{enumerate}[nosep,label=\alph*)]
\item\label{cor.d.ell.psi.properties.iv.1.a} $\Psi^\lfdual$ is totally convex on $\efd(\Psi^\lfdual)$; or
\item\label{cor.d.ell.psi.properties.iv.1.b} $\Psi^\lfdual$ is strictly convex on $\efd(\Psi^\lfdual)$ and supercoercive; or
\item\label{cor.d.ell.psi.properties.iv.1.c} $\Psi^\lfdual$ is Euler--Legendre,
\end{enumerate}
then $C$ is right $D_{\ell,\Psi}$-Chebysh\"{e}v and $D_{\ell,\Psi}$ is right pythagorean on $C$;
\item\label{cor.d.ell.psi.properties.iv.2} if any of \ref{cor.d.ell.psi.properties.iv.1}.\ref{cor.d.ell.psi.properties.iv.1.a}--\ref{cor.d.ell.psi.properties.iv.1}.\ref{cor.d.ell.psi.properties.iv.1.c} holds, then $\RPPP^{D_{\ell,\Psi}}_C$ is zone consistent;
\item\label{cor.d.ell.psi.properties.iv.3} if \ref{cor.d.ell.psi.properties.iv.1}.\ref{cor.d.ell.psi.properties.iv.1.c} holds, then $D_{\ell,\Psi}$ and $D_{\DG\Psi\circ\ell,\Psi^\lfdual}$ are informations on $C$;
\item\label{cor.d.ell.psi.properties.iv.4} if any of \ref{cor.d.ell.psi.properties.iv.1}.\ref{cor.d.ell.psi.properties.iv.1.a}--\ref{cor.d.ell.psi.properties.iv.1}.\ref{cor.d.ell.psi.properties.iv.1.c} holds, and $C$ is $\DG\Psi\circ\ell$-affine, then
\begin{equation}
D_{\ell,\Psi}(\phi,\psi)=D_{\ell,\Psi}(\phi,\RPPP^{D_{\ell,\Psi}}_C(\phi))+D_{\ell,\Psi}(\RPPP^{D_{\ell,\Psi}}_C(\phi),\psi)\;\;\forall(\phi,\psi)\in\ell^{\inver}(\intefd{\Psi}\cap\ell(Z))\times C;
\label{eqn.right.pyth.eqn.d.ell.psi}
\end{equation}
\item\label{cor.d.ell.psi.properties.iv.5} if $\Psi^\lfdual$ is totally convex on $\efd(\Psi^\lfdual)$, Fr\'{e}chet differentiable on $\intefd{\Psi^\lfdual}$, and supercoercive, then $\RPPP^{D_{\ell,\Psi}}_C$ is $\ell$-topology-to-$\ell$-topology continuous on $\ell^{\inver}(\intefd{\Psi}\cap\ell(Z))$;
\end{enumerate}
\item\label{cor.d.ell.psi.properties.v} if $\ell$ is a norm-to-norm homeomorphism, then:
\begin{enumerate}[nosep,label=\arabic*)]
\item\label{cor.d.ell.psi.properties.v.1} $\ell$-closed sets in $(Y,\n{\cdot}_Y)$ coincide with closed sets;
\item\label{cor.d.ell.psi.properties.v.2} result in \ref{cor.d.ell.psi.properties.iii}.\ref{cor.d.ell.psi.properties.iii.6} is strengthened to norm-to-norm continuity of $\LPPP^{D_{\ell,\Psi}}_C$ onto closed $\ell$-convex $C$, and continuity of $\inf_{\phi\in C}\{D_{\ell,\Psi}(\phi,\,\cdot\,)\}$ in the norm topology;
\item\label{cor.d.ell.psi.properties.v.3} if $\Psi$ is Fr\'{e}chet differentiable, then $\DG\Psi\circ\ell$-closed sets in $(Y,\n{\cdot}_Y)$ coincide with closed sets;
\item\label{cor.d.ell.psi.properties.v.4} if $\Psi$ is Fr\'{e}chet differentiable, then the result in \ref{cor.d.ell.psi.properties.iv}.\ref{cor.d.ell.psi.properties.iv.5} is strengthened to norm-to-norm continuity of $\RPPP^{D_{\ell,\Psi}}_C$ onto $\DG\Psi\circ\ell$-convex set $C$, closed in the norm topology of $(Y,\n{\cdot}_Y)$;
\end{enumerate}
\item\label{cor.d.ell.psi.properties.vi} if $(X,\n{\cdot}_X)$ is reflexive, $\Psi:X\ra\RR$ is uniformly Fr\'{e}chet differentiable on bounded subsets of $X$, and $\Psi^\lfdual$ is supercoercive, then:
\begin{enumerate}[nosep,label=\arabic*)]
\item\label{cor.d.ell.psi.properties.vi.1} if any of the following (generally, inequivalent) condition holds:
\begin{enumerate}[nosep,label=\alph*)]
\item\label{cor.d.ell.psi.properties.vi.1.a} $\Psi$ is totally convex on $X$; or
\item\label{cor.d.ell.psi.properties.vi.1.b} $\Psi$ is totally convex on bounded subsets of $X$ and supercoercive,
\end{enumerate}
then $\lsq(\ell,\Psi,C)$ is composable;
\item\label{cor.d.ell.psi.properties.vi.2} if any of the following (generally, inequivalent) condition holds:
\begin{enumerate}[nosep,label=\alph*)]
\item\label{cor.d.ell.psi.properties.vi.2.a} $\Psi$ is totally convex on bounded subsets of $X$; or
\item\label{cor.d.ell.psi.properties.vi.2.b} $\Psi$ is Euler--Legendre and supercoercive, $\Psi^\lfdual$ is (totally convex and uniformly Fr\'{e}chet differentiable) on bounded subsets of $\efd(\Psi^\lfdual)=X^\star$,
\end{enumerate}
then $\rsq(\ell,\Psi,C)$ is composable;
\item\label{cor.d.ell.psi.properties.vi.3} if $\Psi$ is Euler--Legendre, and $C$ is $\ell$-convex and $\ell$-closed, then $\LPPP^{D_{\ell,\Psi}}_C$ is adapted;
\item\label{cor.d.ell.psi.properties.vi.4} if $\Psi$ is supercoercive and Euler--Legendre, and $\Psi^\lfdual$ is uniformly Fr\'{e}chet differentiable on bounded subsets of $\intefd{\Psi^\lfdual}\neq\varnothing$, then $\RPPP^{D_{\ell,\Psi}}_C$ is adapted for any $\DG\Psi\circ\ell$-convex $\DG\Psi\circ\ell$-closed $C$.
\end{enumerate}
\end{enumerate}
\end{corollary}
\begin{proof}
Follows from Definition \ref{def.d.ell.psi.notions} applied to Propositions \ref{prop.right.pythagorean}, \ref{prop.continuity.psi}, and \ref{prop.lsq.rsq.new}.
\end{proof}

\begin{corollary}\label{corollary.ell.psi.varphi}
Let $(Y,\n{\cdot}_Y)$ be a Banach space, $(X,\n{\cdot}_X)$ a reflexive, Gateaux differentiable, strictly convex Banach space, $\varphi$ a gauge, $\varnothing\neq C\subseteq Z\subseteq Y$, $\ell:Z\ra\ell(Z)\subseteq X$ a bijection. Then:
\begin{enumerate}[nosep,label=(\roman*)]
\item\label{corollary.ell.psi.varphi.i} $D_{\ell,\Psi_\varphi}$ and $D_{\DG\Psi_\varphi\circ\ell,\Psi_\varphi^\lfdual}$ are informations on $Z$;
\item\label{corollary.ell.psi.varphi.ii} if $C$ is $\ell$-convex and $\ell$-closed, then $C$ is left $D_{\ell,\Psi_\varphi}$-Chebysh\"{e}v, $D_{\ell,\Psi_\varphi}$ is left pythagorean on $C$, and $\LPPP^{D_{\ell,\Psi_\varphi}}_C$ are zone consistent;
\item\label{corollary.ell.psi.varphi.iii} if $C$ is $j_\varphi\circ\ell$-convex and $j_\varphi\circ\ell$-closed, then $C$ is right $D_{\ell,\Psi_\varphi}$-Chebysh\"{e}v, $D_{\ell,\Psi_\varphi}$ is right pythagorean on $C$, and $\RPPP^{D_{\ell,\Psi_\varphi}}_C$ are zone consistent;
\item\label{corollary.ell.psi.varphi.iv} if $(X,\n{\cdot}_X)$ is Fr\'{e}chet differentiable  and has the Radon--Riesz--Shmul'yan property, then:
\begin{enumerate}[nosep,label=\arabic*)]
\item\label{corollary.ell.psi.varphi.iv.1} if $C$ is $\ell$-convex and $\ell$-closed, then $\LPPP^{D_{\ell,\Psi_\varphi}}_C$ is $\ell$-topology-to-$\ell$-topology continuous on $Z$, while $\inf_{\phi\in C}\{D_{\ell,\Psi_\varphi}(\phi,\,\cdot\,)\}$ is continuous in $\ell$-topology on $Z$;
\item\label{corollary.ell.psi.varphi.iv.2} if $C$ is $j_\varphi\circ\ell$-convex and $j_\varphi\circ\ell$-closed, then $\ell$-topology coincides with $\DG\Psi_\varphi\circ\ell$-topology (so, $C$ is $\ell$-closed), and $\RPPP^{D_{\ell,\Psi_\varphi}}_C$ is $\ell$-topology-to-$\ell$-topology continuous;
\end{enumerate}
\item\label{corollary.ell.psi.varphi.vi} if $(X,\n{\cdot}_X)$ is uniformly Fr\'{e}chet differentiable and strictly convex, and $C$ is $\ell$-convex and $\ell$-closed, then $\LPPP^{D_{\ell,\Psi_\varphi}}_C$ is adapted;
\item\label{corollary.ell.psi.varphi.vii} if $(X,\n{\cdot}_X)$ is uniformly Fr\'{e}chet differentiable, strictly convex, and has the Radon--Riesz--Shmul'yan property, then the sets $\lsq(\ell,\Psi_\varphi,C)$ and $\rsq(\ell,\Psi_\varphi,C)$ are composable;
\item\label{corollary.ell.psi.varphi.viii} if $(X,\n{\cdot}_X)$ is uniformly Fr\'{e}chet differentiable and uniformly convex, and if $C$ is $j_\varphi\circ\ell$-convex $j_\varphi\circ\ell$-closed, then $\RPPP^{D_{\ell,\Psi}}_C$ is adapted;
\item\label{corollary.ell.psi.varphi.v} if $\ell$ is norm-to-norm homeomorphism, then $\ell$-topological closure and continuity in \ref{corollary.ell.psi.varphi.ii}--\ref{corollary.ell.psi.varphi.iv} and \ref{corollary.ell.psi.varphi.viii} coincide, respectively, with closure and continuity in the norm topology of $(Y,\n{\cdot}_Y)$.
\end{enumerate}
\end{corollary}
\begin{proof}
Follows from Definition \ref{def.d.ell.psi.notions} applied to Propositions \ref{prop.legendre}, \ref{prop.left.right.psi.varphi}, \ref{prop.continuity}, \ref{prop.varphi.compositional}.
\end{proof}
\subsection{Categories of $D_{\ell,\Psi}$-projections and strongly $D_{\ell,\Psi}$-quasinonexpansive maps}\label{section.categories.d.psi.proj}
\begin{definition}\label{def.cats.breg.proj}
Let $(X,\n{\cdot}_X)$ and $(Y,\n{\cdot}_Y)$ be Banach spaces, let $(X,\n{\cdot}_X)$ be reflexive, $\varnothing\neq W\subseteq Z\subseteq Y$, $\Psi\in\pclg(X,\n{\cdot}_X)$, $\ell:Z\ra\intefd{\Psi}$ be a bijection. Consider two conditions:
\begin{enumerate}[nosep]
\item[(\hypertarget{def.cats.breg.proj.L}{L})] $\Psi$ is totally convex on $\efd(\Psi)$ or Euler--Legendre or (strictly convex on $\efd(\Psi)$ and supercoercive);
\item[(\hypertarget{def.cats.breg.proj.R}{R})] $\Psi^\lfdual$ is Gateaux differentiable on $\varnothing\neq\DG\Psi(\intefd{\Psi})\subseteq\intefd{\Psi^\lfdual}$ and (totally convex on $\efd(\Psi^\lfdual)$ or Euler--Legendre or (strictly convex on $\efd(\Psi)$ and supercoercive)).
\end{enumerate}
Then:
\begin{enumerate}[nosep,label=(\roman*)]
\item\label{def.cats.breg.proj.i} if (\hyperlink{def.cats.breg.proj.L}{L}) holds, then $\lCvx(\ell,\Psi,W)$ is a category with: objects given by $\ell$-closed $\ell$-convex subsets of $W$, including $\varnothing$; morphisms given by left $D_{\ell,\Psi}$-projections onto $\ell$-closed $\ell$-convex subsets of these subsets (i.e. $\mathrm{Hom}_{\lCvx(\ell,\Psi,W)}(\,\cdot\,,C)$ consists of $\LPPP^{D_{\ell,\Psi}}_Q$ with $Q$ varying over all $\ell$-closed $\ell$-convex subsets of $C$), including $\varnothing$ (resulting in empty arrows, $\ulcorner\varnothing\urcorner\in\mathrm{Hom}_{\lCvx(\ell,\Psi,W)}(C_1,C_2)$); identity morphisms given by $\LPPP^{D_\Psi}_C(C)=C$; composition of morphisms given by 
\begin{equation}
\LPPP^{D_{\ell,\Psi}}_{C_2}\composition\LPPP^{D_{\ell,\Psi}}_{C_1}:=\LPPP^{D_{\ell,\Psi}}_{C_2\cap C_1},
\label{eqn.left.proj.composition}
\end{equation}
with composition of any morphism with empty arrow resulting in an empty arrow;
\item\label{def.cats.breg.proj.ii} if (\hyperlink{def.cats.breg.proj.L}{L}) holds, then $\lAff(\ell,\Psi,W)$ is a subcategory of $\lCvx(\ell,\Psi,W)$ obtained by restriction from $\ell$-closed $\ell$-convex to $\ell$-closed $\ell$-affine subsets of $W$;
\item\label{def.cats.breg.proj.iii} \sloppy if (\hyperlink{def.cats.breg.proj.L}{L}) holds, then $\lCvx^\subseteq(\ell,\Psi,W)$ (resp., $\lAff^\subseteq(\ell,\Psi,W)$) is a subcategory of $\lCvx(\ell,\Psi,W)$ (resp., $\lAff(\ell,\Psi,W)$) obtained by restriction of composition \eqref{eqn.left.proj.composition} by the condition $C_2\subseteq C_1$ (so the composition of morphisms not satisfying this condition results in $\ulcorner\varnothing\urcorner$);
\item\label{def.cats.breg.proj.iv} if (\hyperlink{def.cats.breg.proj.L}{L}) holds, then $\lCvx(\ell,\Psi)$ (resp., $\lAff(\ell,\Psi)$; $\lCvx^\subseteq(\ell,\Psi)$; $\lAff^\subseteq(\ell,\Psi)$) is defined as $\lCvx(\ell,\Psi,W)$ (resp., $\lAff(\ell,\Psi,W)$; $\lCvx^\subseteq(\ell,\Psi,W)$; $\lAff^\subseteq(\ell,\Psi,W)$) with $W=Z$;
\item\label{def.cats.breg.proj.v} if (\hyperlink{def.cats.breg.proj.L}{L}) holds, then $\lCvx(\Psi)$ (resp., $\lAff(\Psi)$; $\lCvx^\subseteq(\Psi)$; $\lAff^\subseteq(\Psi)$) is a category defined as $\lCvx(\ell,\Psi)$ (resp., $\lAff(\ell,\Psi)$; $\lCvx^\subseteq(\ell,\Psi)$; $\lAff^\subseteq(\ell,\Psi)$) with $(X,\n{\cdot}_X)=(Y,\n{\cdot}_Y)$ and $Z=\intefd{\Psi}$;
\item\label{def.cats.breg.proj.vi} if (\hyperlink{def.cats.breg.proj.R}{R}) holds, then $\rbarCvx(\ell,\Psi,W)$ is a category with objects given by $\DG\Psi\circ\ell$-closed $\DG\Psi\circ\ell$-convex subsets of $W$, including $\varnothing$; morphisms given by right $D_{\ell,\Psi}$-projections onto $\DG\Psi\circ\ell$-closed $\DG\Psi\circ\ell$-convex subsets of these subsets, including $\varnothing$; identity morphisms given by $\RPPP^{D_\Psi}_C(C)=C$; composition of morphisms given by 
\begin{equation}
\RPPP^{D_{\ell,\Psi}}_{C_2}\composition\RPPP^{D_{\ell,\Psi}}_{C_1}:=\ell^{\inver}\circ\DG\Psi^\lfdual\circ\left(\RPPP^{D_{\Psi^\lfdual}}_{(\DG\Psi\circ\ell)(C_2)}\composition\RPPP^{D_{\Psi^\lfdual}}_{(\DG\Psi\circ\ell)(C_1)}\right)\circ\DG\Psi\circ\ell;
\label{eqn.right.proj.composition}
\end{equation}
\item\label{def.cats.breg.proj.vii} if (\hyperlink{def.cats.breg.proj.R}{R}) holds, then $\rbarAff(\ell,\Psi)$, $\rbarCvx^\subseteq(\ell,\Psi)$, $\rbarAff^\subseteq(\ell,\Psi)$, $\rbarCvx(\Psi)$, $\rbarAff(\Psi)$, $\rbarCvx^\subseteq(\Psi)$, $\rbarAff^\subseteq(\Psi)$ are categories defined analogously as in (ii)--(v), with $\lCvx(\ell,\Psi)$ replaced by $\rbarCvx(\ell,\Psi)$;
\item\label{def.cats.breg.proj.viii} $\CN(\ell,\Psi,W)$ is a category with subsets of $W$ as objects, elements of $\cn(\ell,\Psi,W)$ as morphisms, identity maps of subsets as identity morphisms, and composition of morphisms given by composition of elements in $\cn(\ell,\Psi,W)$; $\cn(\ell,\Psi)$ will denote $\cn(\ell,\Psi,W)$ for $W=Z$;
\item\label{def.cats.breg.proj.ix} for any set $V$, let $\Pow(V)$ denote the category of all subsets of $V$ as objects, with functions between them as morphisms, and composition of functions as composition of morphisms.
\end{enumerate}
\end{definition}

\begin{definition}\label{def.functors.proj.cats}
Let $(X,\n{\cdot}_X)$ and $(Y,\n{\cdot}_Y)$ be Banach spaces, let $(X,\n{\cdot}_X)$ be reflexive, $Z\subseteq Y$, $\Psi\in\pclg(X,\n{\cdot}_X)$, $\ell:Z\ra\intefd{\Psi}$ be a bijection. Then:
\begin{enumerate}[nosep,label=(\roman*)]
\item\label{def.functors.proj.cats.i} if (\hyperlink{def.cats.breg.proj.L}{L}) and (\hyperlink{def.cats.breg.proj.R}{R}) hold, then $(\cdot)^{\Psi^\lfdual}:\lCvx(\ell,\Psi)\ra\rbarCvx(\ell,\Psi)$ is a functor, acting by $C\mapsto\ell^{\inver}\circ\DG\Psi\circ\ell(C)$ on objects $C\in\Ob(\lCvx(\ell,\Psi))$, with $(\varnothing)^{\Psi^\lfdual}:=\varnothing$, and by $T\mapsto\ell^{\inver}\circ\DG\Psi^\lfdual\circ T\circ\DG\Psi\circ\ell$ on morphisms $\ell^{\inver}\circ T\circ\ell\in\Arr(\lCvx(\ell,\Psi))$, with $(\ulcorner\varnothing\urcorner)^{\Psi^\lfdual}:=\ulcorner\varnothing\urcorner$;
\item\label{def.functors.proj.cats.ii} if (\hyperlink{def.cats.breg.proj.L}{L}) and (\hyperlink{def.cats.breg.proj.R}{R}) hold, then $(\cdot)^{\Psi}:\rbarCvx(\ell,\Psi)\ra\lCvx(\ell,\Psi)$ is a functor, acting by $C\mapsto\ell^{\inver}\circ\DG\Psi^\lfdual\circ\ell(C)$ on objects $C\in\Ob(\rbarCvx(\ell,\Psi))$, with $(\varnothing)^{\Psi}:=\varnothing$, and by $T\mapsto\ell^{\inver}\circ\DG\Psi\circ T\circ\DG\Psi^\lfdual\circ\ell$ on morphisms $\ell^{\inver}\circ T\circ\ell\in\Arr(\rbarCvx(\ell,\Psi))$, with $(\ulcorner\varnothing\urcorner)^{\Psi}:=\ulcorner\varnothing\urcorner$;
\item\label{def.functors.proj.cats.iii} if (\hyperlink{def.cats.breg.proj.L}{L}) holds, then $\overline{\co^\mathrm{L}_\Psi(\cdot)}^{w}:\Pow(X)\ra\lCvx(\Psi)$ is a functor, defined by:
\begin{enumerate}[nosep,label=\arabic*)]
\item\label{def.functors.proj.cats.iii.1} a map $\overline{\co^\mathrm{L}_\Psi(\cdot)}^{w}:\Ob(\Pow(X))\ra\Ob(\Pow(X))$, assigning to each subset $W$ of $X$ the closure $\overline{\,\cdot\,}^{w}$ of a convex hull $\co(\,\cdot\,)$ of $W\cap\intefd{\Psi}$ in the weak topology of $(X,\n{\cdot}_X)$ (it coincides with the norm closure) if $\overline{\co(W\cap\intefd{\Psi})}^{w}\subseteq\intefd{\Psi}$, and assigning $\varnothing$ otherwise;
\item\label{def.functors.proj.cats.iii.2} a map $\overline{\co^\mathrm{L}_\Psi(\cdot)}^{w}:\Arr(\Pow(X))\ra\Arr(\Pow(X))$, assigning to each function $f:W_1\ra W_2$ a map $\LPPP^{D_\Psi}_Q:\overline{\co^\mathrm{L}_\Psi(W_1)}^{w}\ra\overline{\co^\mathrm{L}_\Psi(W_2)}^{w}$, where $Q:=\overline{\co^\mathrm{L}_\Psi(f(W_1))}^{w}$, with $\LPPP^{D_\Psi}_Q=\ulcorner\varnothing\urcorner$ if either $Q=\varnothing$ or $\overline{\co^\mathrm{L}_\Psi(W_1)}^{w}=\varnothing$;
\end{enumerate}
\item\label{def.functors.proj.cats.iv} if (\hyperlink{def.cats.breg.proj.R}{R}) holds, then $\overline{\co^\mathrm{R}_\Psi(\cdot)}^{w}:\Pow(X)\ra\rbarCvx(\Psi)$ is a functor, defined by:
\begin{enumerate}[nosep,label=\arabic*)]
\item\label{def.functors.proj.cats.iv.1} a map $\overline{\co^\mathrm{R}_\Psi(\cdot)}^{w}:\Ob(\Pow(X))\ra\Ob(\Pow(X))$, assigning to each subset $W$ of $X$ the set $\DG\Psi^\lfdual\left(\overline{\co\left(\DG\Psi(W)\cap\intefd{\Psi^\lfdual}\right)}^{w}\right)$ (with $\overline{\,\cdot\,}^{w}$ denoting a closure in a weak topology of $(X^\star,\n{\cdot}_{X^\star})$) if $\overline{\co(\DG\Psi(W)\cap\intefd{\Psi^\lfdual})}^{w}\subseteq\intefd{\Psi^\lfdual}$, and assigning $\varnothing$ otherwise;
\item\label{def.functors.proj.cats.iv.2} a map $\overline{\co^\mathrm{R}_\Psi(\cdot)}^{w}:\Arr(\Pow(X))\ra\Arr(\Pow(X))$, assigning to each function $f:W_1\ra W_2$ a map $\RPPP^{D_\Psi}_Q:\overline{\co^\mathrm{R}_\Psi(W_1)}^{w}\ra\overline{\co^\mathrm{R}_\Psi(W_2)}^{w}$, where $Q:=\overline{\co^\mathrm{R}_\Psi(f(W_1))}^{w}$, with $\LPPP^{D_\Psi}_Q=\ulcorner\varnothing\urcorner$ if either $Q=\varnothing$ or $\overline{\co^\mathrm{R}_\Psi(W_1)}^{w}=\varnothing$;
\end{enumerate}
\item\label{def.functors.proj.cats.v} $\overline{\co^\mathrm{L}_{\ell,\Psi}(\cdot)}^{\ell}:\Pow(Z)\ra\lCvx(\ell,\Psi)$ and $\overline{\co^\mathrm{R}_{\ell,\Psi}(\cdot)}^{\ell}:\Pow(Z)\ra\rbarCvx(\ell,\Psi)$ are functors defined analogously to the corresponding functors in \ref{def.functors.proj.cats.iii}--\ref{def.functors.proj.cats.iv}, by the bijectivity of $\ell$;
\item\label{def.functors.proj.cats.vi} $\overline{\co^{\mathrm{L},\subseteq}_{\ell,\Psi}(\cdot)}^{\ell}:\Pow(Z)\ra\lCvx^\subseteq(\ell,\Psi)$ (resp., $\overline{\co^{\mathrm{R},\subseteq}_{\ell,\Psi}(\cdot)}^{\ell}:\Pow(Z)\ra\rbarCvx^\subseteq(\ell,\Psi)$) is defined as a restriction of the functor $\overline{\co^\mathrm{L}_{\ell,\Psi}(\cdot)}^{\ell}$ (resp., $\overline{\co^\mathrm{R}_{\ell,\Psi}(\cdot)}^{\ell}$) by the additional condition $\LPPP^{D_{\ell,\Psi}}_{\ell^{\inver}(Q)}=\ulcorner\varnothing\urcorner$ if $Q\not\subseteq\overline{\co^\mathrm{L}_{\ell,\Psi}(W_1)}^{\ell}$ (resp., $\RPPP^{D_{\ell,\Psi}}_{\ell^{\inver}(Q)}=\ulcorner\varnothing\urcorner$ if $Q\not\subseteq\overline{\co^\mathrm{R}_{\ell,\Psi}(W_1)}^{\ell}$);
\item\label{def.functors.proj.cats.vii} if (\hyperlink{def.cats.breg.proj.L}{L}) (resp.,(\hyperlink{def.cats.breg.proj.R}{R})) holds, then $\FrgSet^{\mathrm{L}}_{\ell,\Psi}:\lCvx(\ell,\Psi)\ra\Pow(Z)$ (resp., $\FrgSet^{\mathrm{R}}_{\ell,\Psi}:\rbarCvx(\ell,\Psi)\ra\Pow(Z)$) denotes a forgetful functor, forgetting all properties of domain category, except their structure as sets and functions between them; $\FrgSet_{\ell,\Psi}^{\mathrm{L},\subseteq}$ (resp., $\FrgSet_{\ell,\Psi}^{\mathrm{R},\subseteq}$) will denote a restriction of $\FrgSet_{\ell,\Psi}^{\mathrm{L}}$ (resp., $\FrgSet_{\ell,\Psi}^{\mathrm{R}}$) to the category $\lCvx^\subseteq(\ell,\Psi)$ (resp., $\rbarCvx^\subseteq(\ell,\Psi)$).
\end{enumerate}
\end{definition}

\begin{corollary}\label{cor.functors.breg.proj}
\begin{enumerate}[nosep,label=(\roman*)]
\item\label{cor.functors.breg.proj.i}
\sloppy Functors $(\cdot)^\Psi$ and $(\cdot)^{\Psi^\lfdual}$ establish equivalence of categories $\lCvx(\ell,\Psi)$ and $\rbarCvx(\ell,\Psi)$.
\item\label{cor.functors.breg.proj.ii} There are the following adjunctions of functors: $\overline{\co^{\mathrm{L}}_{\ell,\Psi}(\cdot)}^{\ell}\adj\FrgSet^{\mathrm{L}}_{\ell,\Psi}$, $\overline{\co^{\mathrm{L},\subseteq}_{\ell,\Psi}(\cdot)}^{\ell}\adj\FrgSet^{\mathrm{L},\subseteq}_{\ell,\Psi}$, $\overline{\co^{\mathrm{R}}_{\ell,\Psi}(\cdot)}^{\ell}\adj\FrgSet^{\mathrm{R}}_{\ell,\Psi}$, $\overline{\co^{\mathrm{R},\subseteq}_{\ell,\Psi}(\cdot)}^{\ell}\adj\FrgSet^{\mathrm{R},\subseteq}_{\ell,\Psi}$.
\end{enumerate} 
\end{corollary}
\begin{proof}
\begin{enumerate}[nosep]
\item[(i)] Follows from \eqref{eqn.right.proj.composition}.
\item[(ii)] Follows from the definition of the forgetful functor.
\end{enumerate} 
\end{proof}

\begin{proposition}\label{prop.d.psi.as.functor}
Let $[0,\infty]$ denote a category consisting of one object, $\bullet$, with morphisms given by the elements of the set $\RR^+\cup\{\infty\}$, and their composition defined by addition \textup{\cite[p. 140]{Lawvere:1973}}. Let $\catname{2}$ denote the category consisting of two objects, one arrow between them, and identity arrows on both of the objects. The category $[0,\infty]^\catname{2}$ has morphisms of $[0,\infty]$ as objects, commutative squares in $[0,\infty]$ as morphisms, and commutative compositions of these squares as compositions. Let $Q$ be a closed affine subset of a reflexive Banach space $(X,\n{\cdot}_X)$, $\Psi\in\pclg(X,\n{\cdot}_X)$ satisfies (\hyperlink{def.cats.breg.proj.L}{L}), $\phi\in Q$, and let $\lAff^\subseteq_Q(\Psi)$ denotes a subcategory of  $\lAff^\subseteq(\Psi)$ with objects restricted to sets $C\in\Ob(\lAff^\subseteq(\Psi))$ such that $Q\subseteq C$. Then $D_\Psi(\phi,\,\cdot\,)$ determines a contravariant functor $\lAff^\subseteq_Q(\Psi)\ra[0,\infty]^\catname{2}$ as well as a family of natural transformations in the category of functors $\lAff^\subseteq_Q(\Psi)\ra[0,\infty]$. Analogous statement holds for $\lAff^\subseteq(\Psi)$ (resp., $D_\Psi(\phi,\,\cdot\,)$) replaced by $\lAff^\subseteq(\ell,\Psi)$ (resp., $D_{\ell,\Psi}(\phi,\,\cdot\,)$), or by $\rbarAff^\subseteq(\Psi)$ (resp., $D_\Psi(\,\cdot\,,\phi)$), or by $\rbarAff^\subseteq(\ell,\Psi)$ (resp., $D_{\ell,\Psi}(\,\cdot\,,\phi)$) (last two cases require also to replace (\hyperlink{def.cats.breg.proj.L}{L}) by (\hyperlink{def.cats.breg.proj.R}{R})).
\end{proposition}
\begin{proof}
Let $K_1,K_2,K_3,K,L\in\Ob(\lAff^\subseteq_Q(\Psi))$, $K\subseteq K_2$ and $L\subseteq K_3$. For each $\phi\in Q$, left pythagorean equation implies commutativity of the diagram 

\begin{equation}
\xymatrix{%
\bullet \ar[rrrr]^{D_\Psi(\phi,x)} &&&&
\bullet \\
\bullet \ar[u]^{0}\ar[rrrr]^{D_\Psi(\phi,\LPPP^{D_\Psi}_K(x))} &&&&
\bullet \ar[u]_{D_\Psi(\LPPP^{D_\Psi}_K(x),x)}\\
\bullet \ar[u]^{0}\ar[rrrr]^{D_\Psi(\phi,\LPPP^{D_\Psi}_L\composition\LPPP^{D_\Psi}_K(x))} &&&&
\;\bullet, \ar[u]_{D_\Psi(\LPPP^{D_\Psi}_L\composition\LPPP^{D_\Psi}_K(x),\LPPP^{D_\Psi}_K(x))}\\
}%
\end{equation}

and hence also of

\begin{equation}
\xymatrix{%
x \ar@{|->}[rr] \ar@{|->}[d]|{\LPPP^{D_\Psi}_K} &&
\left(\;\bullet\;\; \ar[rrrr]_{D_\Psi(\phi,x)}\right. &&&&
\left.\;\;\bullet\;\right) \\
\LPPP^{D_\Psi}_K(x) \ar@{|->}[rr] \ar@{|->}[d]|{\LPPP^{D_\Psi}_L} &&
\left(\;\bullet\;\; \ar[u]^{0}\ar[rrrr]^{D_\Psi(\phi,\LPPP^{D_\Psi}_K(x))}\right. &&&&
\left.\;\;\bullet \ar[u]|{D_\Psi(\LPPP^{D_\Psi}_K(x),x)}\;\right)\\
\LPPP^{D_\Psi}_L\composition\LPPP^{D_\Psi}_K(x) \ar@{|->}[rr] &&
\left(\;\bullet\;\; \ar[u]^{0}\ar[rrrr]^{D_\Psi(\phi,\LPPP^{D_\Psi}_L\composition\LPPP^{D_\Psi}_K(x))}\right. &&&&
\left.\;\;\;\;\bullet \ar[u]|{D_\Psi(\LPPP^{D_\Psi}_L\composition\LPPP^{D_\Psi}_K(x),\LPPP^{D_\Psi}_K(x))}\;\right).\\
}%
\label{gpt.qpl.diag}
\end{equation}
This defines a contravariant functor $D_\Psi(\phi,\cdot):\lAff^\subseteq_Q(\Psi)\ra[0,\infty]^\texttt{2}$.

For any two categories $\texttt{C}$ and $\texttt{D}$, cartesian closedness of the category $\texttt{Cat}$ of all small categories (with natural transformations as morphisms) implies that any functor $\texttt{C}\ra\texttt{D}^\texttt{2}$ corresponds to a natural transformation in $\texttt{D}^\texttt{C}$. 
\end{proof}

\begin{definition}\label{def.LSQ.RSQ.cats}
Let $(X,\n{\cdot}_X)$ and $(Y,\n{\cdot}_Y)$ be Banach spaces, let $(X,\n{\cdot}_X)$ be reflexive, $\varnothing\neq W\subseteq Z\subseteq Y$, $\Psi\in\pclg(X,\n{\cdot}_X)$, $\ell:Z\ra\intefd{\Psi}$ be a bijection. Then:
\begin{enumerate}[nosep,label=(\roman*)]
\item\label{def.LSQ.RSQ.cats.i} if $\Psi$ is LSQ-compositional on $\ell(W)$, then $\LSQcvxsub(\Psi,\ell(W))$ is a category with: objects given by convex closed subsets of $\ell(W)$, including $\varnothing$; morphisms given by
\begin{equation}
	\Hom_{\LSQcvxsub(\Psi,\ell(W))}(K_1,K_2):=
	\left\{
	\begin{array}{ll}
	\ulcorner\varnothing\urcorner&:\;K_2\not\subseteq K_1\\
	\{T_i\in\lsq(\Psi,K_1)\mid K_2\supseteq\ran(T_i)\}&:\;K_2\subseteq K_1;
	\end{array}
	\right.
\end{equation}
composition of morphisms $(f:K_2\ra K_3)\composition(g:K_1\ra K_2)$ given by
\begin{equation}
\left\{
\begin{array}{ll}
(f\circ g):K_1\ra K_2&:\;\aFix(f\circ g)=\aFix(f)\cap\aFix(g)\neq\varnothing\\
\ulcorner\varnothing\urcorner&:\;\mathrm{otherwise};
\end{array}
\right.
\end{equation}
identity given by $\{\id_K:K\ra K\}\in\lsq(\Psi,K)$;
\item\label{def.LSQ.RSQ.cats.ii} $\LSQcvxsub(\ell,\Psi,W)$ is a category defined by pulling back $\LSQcvxsub(\Psi,\ell(W))$ along $\ell$; $\LSQcvxsub(\Psi)$ (resp., $\LSQcvxsub(\ell,\Psi)$) is defined as $\LSQcvxsub(\Psi,\ell(W))$ (resp., $\LSQcvxsub(\ell,\Psi,W)$) with $W=Z$;
\item\label{def.LSQ.RSQ.cats.iii} if $\Psi$ is RSQ-compositional, then the categories $\RbarSQcvxsub(\Psi,\ell(W))$, $\RbarSQcvxsub(\ell,\Psi,W)$, $\RbarSQcvxsub(\ell,\Psi)$, and $\RbarSQcvxsub(\Psi)$ are defined analogously to \ref{def.LSQ.RSQ.cats.i}--\ref{def.LSQ.RSQ.cats.ii}, by replacing $\lsq(\Psi,K)$ with $\rsq(\Psi,K)$, and replacing convex closed subsets of $\ell(W)$ by $\DG\Psi$-convex $\DG\Psi$-closed subsets of $\ell(W)$.
\end{enumerate}
\end{definition}

\begin{definition}\label{def.LSQ.RSQ.functors}
Let $(X,\n{\cdot}_X)$ and $(Y,\n{\cdot}_Y)$ be Banach spaces, let $(X,\n{\cdot}_X)$ be reflexive, $Z\subseteq Y$, $\Psi\in\pclg(X,\n{\cdot}_X)$, $\ell:Z\ra\intefd{\Psi}$ be a bijection. Then:
\begin{enumerate}[nosep,label=(\roman*)]
\item\label{def.LSQ.RSQ.functors.i} if $\Psi$ is LSQ-compositional and RSQ-compositional, then $(\cdot)^{\Psi^\lfdual}:\LSQcvxsub(\ell,\Psi)\ra\RbarSQcvxsub(\ell,\Psi)$ denotes a functor, acting by $C\mapsto\ell^{\inver}\circ\DG\Psi\circ\ell(C)$ on objects $C\in\Ob(\LSQcvxsub(\ell,\Psi))$, with $(\varnothing)^{\Psi^\lfdual}:=\varnothing$, and by $T\mapsto\ell^{\inver}\circ\DG\Psi^\lfdual\circ\DG\Psi\circ\ell$ on morphisms $\ell^{\inver}\circ T\circ\ell\in\Arr(\LSQcvxsub(\ell,\Psi))$, with $(\ulcorner\varnothing\urcorner)^{\Psi^\lfdual}:=\ulcorner\varnothing\urcorner$;
\item\label{def.LSQ.RSQ.functors.ii} if $\Psi$ is LSQ-compositional and RSQ-compositional, then $(\cdot)^{\Psi}:\RbarSQcvxsub(\ell,\Psi)\ra\LSQcvxsub(\ell,\Psi)$ denotes a functor, acting by $C\mapsto\ell^{\inver}\circ\DG\Psi^\lfdual\circ\ell(C)$ on objects $C\in\Ob(\RbarSQcvxsub(\ell,\Psi))$, with $(\varnothing)^\Psi:=\varnothing$, and by $T\mapsto\ell^{\inver}\circ\DG\Psi\circ T\circ\DG\Psi^\lfdual\circ\ell$ on morphisms $\ell^{\inver}\circ T\circ\ell\in\Arr(\RbarSQcvxsub(\ell,\Psi))$, with $(\ulcorner\varnothing\urcorner)^{\Psi}:=\ulcorner\varnothing\urcorner$;
\item\label{def.LSQ.RSQ.functors.iii} if (\hyperlink{def.cats.breg.proj.L}{L}) holds, $\Psi$ is LSQ-compositional, and $\Psi$ is LSQ-adapted on any convex closed $\varnothing\neq K\subseteq\intefd{\Psi}$, then $\iota_{\ell,\Psi}^{\mathrm{L},\subseteq}:\lCvx^\subseteq(\ell,\Psi)\hookrightarrow\LSQcvxsub(\ell,\Psi)$ denotes an embedding functor;
\item\label{def.LSQ.RSQ.functors.iv} if (\hyperlink{def.cats.breg.proj.R}{R}) holds, $\Psi$ is RSQ-compositional, and $\Psi$ is RSQ-adapted on any $\DG\Psi$-convex $\DG\Psi$-closed $\varnothing\neq K\subseteq\intefd{\Psi}$, then $\iota_{\ell,\Psi}^{\mathrm{R},\subseteq}:\rbarCvx^\subseteq(\ell,\Psi)\hookrightarrow\RbarSQcvxsub(\ell,\Psi)$ denotes an embedding functor;
\item\label{def.LSQ.RSQ.functors.v} if (\hyperlink{def.cats.breg.proj.L}{L}) holds and $\Psi$ is LSQ-compositional, then $\Fix_{\ell,\Psi}^{\mathrm{L},\subseteq}:\LSQcvxsub(\ell,\Psi)\ra\lCvx^\subseteq(\ell,\Psi)$ denotes a functor, acting as an identity map on objects, and as an assignment $\ell^{\inver}\circ T\circ\ell\mapsto\LPPP^{D_{\ell,\Psi}}_{\ell^{\inver}(\Fix(T))}$ to each $\ell^{\inver}\circ T\circ\ell\in\Arr(\LSQcvxsub(\ell,\Psi))$;
\item\label{def.LSQ.RSQ.functors.vi} if (\hyperlink{def.cats.breg.proj.R}{R}) holds and $\Psi$ is RSQ-compositional, then $\Fix_{\ell,\Psi}^{\mathrm{R},\subseteq}:\RbarSQcvxsub(\ell,\Psi)\ra\rbarCvx^\subseteq(\ell,\Psi)$ denotes a functor, acting as an identity map on objects, and as an assignment $\ell^{\inver}\circ T\circ\ell\mapsto\RPPP^{D_{\ell,\Psi}}_{\ell^{\inver}(\Fix(T))}$ to each $\ell^{\inver}\circ T\circ\ell\in\Arr(\RbarSQcvxsub(\ell,\Psi))$.
\end{enumerate}
\end{definition}

\begin{proposition}\label{prop.adjointness.cvx}
Let $(X,\n{\cdot}_X)$ and $(Y,\n{\cdot}_Y)$ be Banach spaces, let $(X,\n{\cdot}_X)$ be reflexive, $Z\subseteq Y$, $\Psi\in\pclg(X,\n{\cdot}_X)$, $\ell:Z\ra\intefd{\Psi}$ be a bijection. Then:
\begin{enumerate}[nosep,label=(\roman*)]
\item\label{prop.adjointness.cvx.i} if $\Psi$ is LSQ-compositional, RSQ-compositional, Euler--Legendre, $\DG\Psi$ is (uniformly continuous and bounded) on open subsets of $\intefd{\Psi}$, and $\DG\Psi^\lfdual$ is (uniformly continuous and bounded) on open subsets of $\intefd{\Psi^\lfdual}$, then $(\cdot)^\Psi$ and $(\cdot)^{\Psi^\lfdual}$ establish an equivalence of categories $\LSQcvxsub(\ell,\Psi)$ and $\RbarSQcvxsub(\ell,\Psi)$;
\item\label{prop.adjointness.cvx.ii} if (\hyperlink{def.cats.breg.proj.L}{L}) holds, $\Psi$ is LSQ-adapted on any convex closed $\varnothing\neq K\subseteq\intefd{\Psi}$, and LSQ-compositional, then there are adjunctions $\iota_{\ell,\Psi}^{\mathrm{L},\subseteq}\adj\Fix_{\ell,\Psi}^{\mathrm{L},\subseteq}$ and $\iota_{\ell,\Psi}^{\mathrm{L},\subseteq}\circ\overline{\co^{\mathrm{L},\subseteq}_{\ell,\Psi}(\cdot)}^{\ell}\adj\FrgSet^{\mathrm{L},\subseteq}_{\ell,\Psi}\circ\Fix_{\ell,\Psi}^{\mathrm{L},\subseteq}$, with a monad $\Fix_{\ell,\Psi}^{\mathrm{L},\subseteq}\circ\iota_{\ell,\Psi}^{\mathrm{L},\subseteq}$ on $\lCvx^\subseteq(\ell,\Psi)$, and a comonad $\overline{\co^{\mathrm{L},\subseteq}_{\ell,\Psi}(\cdot)}^{\ell}\circ\FrgSet^{\mathrm{L},\subseteq}_{\ell,\Psi}$ on $\lCvx^\subseteq(\ell,\Psi)$;
\item\label{prop.adjointness.cvx.iii} if (\hyperlink{def.cats.breg.proj.R}{R}) holds, $\Psi$ is RSQ-adapted on any $\DG\Psi$-convex $\DG\Psi$-closed 
$\varnothing\neq K\subseteq\intefd{\Psi}$, and RSQ-compositional, then there are adjunctions $\iota_{\ell,\Psi}^{\mathrm{R},\subseteq}\adj\Fix_{\ell,\Psi}^{\mathrm{R},\subseteq}$ and $\iota_{\ell,\Psi}^{\mathrm{R},\subseteq}\circ\overline{\co^{\mathrm{R},\subseteq}_{\ell,\Psi}(\cdot)}^{\ell}\adj\FrgSet^{\mathrm{R},\subseteq}_{\ell,\Psi}\circ\Fix_{\ell,\Psi}^{\mathrm{R},\subseteq}$, with a monad $\Fix_{\ell,\Psi}^{\mathrm{R},\subseteq}\circ\iota_{\ell,\Psi}^{\mathrm{R},\subseteq}$ on $\rbarCvx^\subseteq(\ell,\Psi)$, and a comonad $\overline{\co^{\mathrm{R},\subseteq}_{\ell,\Psi}(\cdot)}^{\ell}\circ\FrgSet^{\mathrm{R},\subseteq}_{\ell,\Psi}$ on $\rbarCvx^\subseteq(\ell,\Psi)$;
\item\label{prop.adjointness.cvx.iv} if (\hyperlink{def.cats.breg.proj.L}{L}) and (\hyperlink{def.cats.breg.proj.R}{R}) hold, $\Psi$ is LSQ-adapted on any convex closed nonempty subset of $\intefd{\Psi}$, RSQ-adapted on any $\DG\Psi$-convex $\DG\Psi$-closed nonempty subset of $\intefd{\Psi}$, LSQ-compositional, and RSQ-compositional, Euler--Legendre, $\DG\Psi$ is (uniformly continuous and bounded) on bounded subsets of $\intefd{\Psi}$, and $\DG\Psi^\lfdual$ is (uniformly continuous and bounded) on bounded subsets of $\intefd{\Psi^\lfdual}$, then the following diagram holds (with the horizontal arrows denoting adjoint functors, and vertical arrows denoting equivalences of categories):
\begin{equation}
	\xymatrix{
\Pow(Z)
		\ar@/^1pc/[rr]^{\overline{\co_{\ell,\Psi}^{\mathrm{L},\subseteq}(\cdot)}^{\ell}\;\;\;\;\;\;\;\;}_{}="cc"
		&&
		\lCvx^\subseteq(\ell,\Psi)
		\ar@(ul,ur)[]^{\;\;\;\;\overline{\co_{\ell,\Psi}^{\mathrm{L},\subseteq}(\cdot)}^{\ell}\,\circ\,\mathrm{FrgSet}^{\mathrm{L},\subseteq}_{\ell,\Psi}}
		\ar@(dr,dl)[]^{\Fix^{\mathrm{L},\subseteq}_{\ell,\Psi}\,\circ\,\iota^{\mathrm{L},\subseteq}_{\ell,\Psi}}
		\ar@/^1pc/[ll]^{\mathrm{FrgSet}^{\mathrm{L},\subseteq}_{\ell,\Psi}}_{\ }="dd"
				\ar@/^1pc/[rr]^{\iota^{\mathrm{L},\subseteq}_{\ell,\Psi}}_{}="ee"
				\ar@/^2.4pc/[dd]^{(\cdot)^{\Psi^\lfdual}}_{\ }
				&&
		\LSQcvxsub(\ell,\Psi)
					\ar@/^1pc/[ll]^{\Fix^{\mathrm{L},\subseteq}_{\ell,\Psi}}_{\ }="ff"
						\ar@{}^\bot"dd";"cc"
								\ar@{}^{\bot}"ff";"ee"
								\ar@/^1pc/[dd]^{(\cdot)^{\Psi^\lfdual}}_{\ }\\
&&&&\\			
		&&
		\rbarCvx^\subseteq(\ell,\Psi)
\ar@(ul,ur)[]^{\Fix^{\mathrm{R},\subseteq}_{\ell,\Psi}\,\circ\,\iota^{\mathrm{R},\subseteq}_{\ell,\Psi}}
		\ar@(dr,dl)[]^{(\cdot)^{\Psi^\lfdual}\,\circ\,\overline{\co_{\ell,\Psi}^{L,\subseteq}(\cdot)}^{\ell}\,\circ\,\mathrm{FrgSet}^{\mathrm{L},\subseteq}_{\ell,\Psi}\,\circ\,(\cdot)^{\Psi}\;\;\;\;\;\;\;\;\;\;\;\;\;\;\;\;\;\;\;\;\;\;\;\;\;\;\;\;\;\;}
				\ar@/^1pc/[rr]^{\iota^{\mathrm{R},\subseteq}_{\ell,\Psi}}_{}="e"
				\ar@/^2.4pc/[uu]^{(\cdot)^{\Psi}}_{\ }
				&&
		\RbarSQcvxsub(\ell,\Psi).
					\ar@/^1pc/[ll]^{\Fix^{\mathrm{R},\subseteq}_{\ell,\Psi}}_{\ }="f"
								\ar@{}^{\;\;\;\;\bot}"f";"e"
								\ar@/^1pc/[uu]^{(\cdot)^{\Psi}}_{\ }
	}
\label{eqn.bregmanian.cat.adjointness}
\end{equation}
\end{enumerate}
\end{proposition}
\begin{proof}
\begin{enumerate}[nosep]
\item[(i)] Follows from Proposition \ref{prop.lsq.rsq.old}.\ref{prop.lsq.rsq.old.v}.
\item[(ii)] The adjunction $\iota_{\ell,\Psi}^{\mathrm{L},\subseteq}\adj\Fix_{\ell,\Psi}^{\mathrm{L},\subseteq}$ follows from Definition \ref{def.lsq.rsq.adapted}, while the composite adjunction follows from Corollary \ref{cor.functors.breg.proj}.\ref{cor.functors.breg.proj.ii}. The corresponding monad and comonad are determined by the latter adjunction. 
\item[(iii)] Follows from Definition \ref{def.lsq.rsq.adapted} and Corollary \ref{cor.functors.breg.proj}.\ref{cor.functors.breg.proj.ii}.
\item[(iv)] Follows from \ref{prop.adjointness.cvx.i}--\ref{prop.adjointness.cvx.iii}.
\end{enumerate}
\end{proof}

\begin{corollary}\label{cor.LSQ.RSQ.adj.Psi.varphi}
Let $(Y,\n{\cdot}_Y)$ be a Banach space, let $(X,\n{\cdot}_X)$ be a uniformly Fr\'{e}chet differentiable, strictly convex Banach space with the Radon--Riesz--Shmul'yan property, let $Z\subseteq Y$, $\ell:Z\ra\intefd{\Psi}$ be a bijection, let $\varphi$ be a gauge. Then:
\begin{enumerate}[nosep,label=(\roman*)]
\item\label{cor.LSQ.RSQ.adj.Psi.varphi.i} there are adjunctions
\begin{equation}
\xymatrix{
\Pow(Z)
		\ar@/^1pc/[rr]^{\overline{\co_{\ell,\Psi_\varphi}^{\mathrm{L},\subseteq}(\cdot)}^{\ell}\;\;\;\;\;\;\;\;}_{}="cc"
		&&
		\lCvx^\subseteq(\ell,\Psi_\varphi)
		\ar@/^1pc/[ll]^{\mathrm{FrgSet}^{\mathrm{L},\subseteq}_{\ell,\Psi_\varphi}}_{\ }="dd"
				\ar@/^1pc/[rr]^{\iota^{\mathrm{L},\subseteq}_{\ell,\Psi_\varphi}}_{}="ee"
				&&
		\LSQcvxsub(\ell,\Psi_\varphi)
					\ar@/^1pc/[ll]^{\Fix^{\mathrm{L},\subseteq}_{\ell,\Psi_\varphi}}_{\ }="ff"
						\ar@{}^\bot"dd";"cc"
								\ar@{}^{\bot}"ff";"ee"
}
\end{equation}
and
\begin{equation}
\xymatrix{
\Pow(Z)
		\ar@/^1pc/[rr]^{\overline{\co_{\ell,\Psi_\varphi}^{\mathrm{R},\subseteq}(\cdot)}^{\ell}\;\;\;\;\;\;\;\;}_{}="cc"
		&&
		\lCvx^\subseteq(\ell,\Psi_\varphi),
		\ar@/^1pc/[ll]^{\mathrm{FrgSet}^{\mathrm{R},\subseteq}_{\ell,\Psi_\varphi}}_{\ }="dd"
								\ar@{}^\bot"dd";"cc"
}
\end{equation}
with corresponding monads and comonads, and there are also equivalences 
\begin{equation}
\xymatrix{
\rbarCvx^\subseteq(\ell,\Psi_\varphi)
\ar@/^1pc/[rr]^{(\cdot)^{\Psi_\varphi}}
&&
\lCvx^\subseteq(\ell,\Psi_\varphi),
\ar@/^1pc/[ll]^{(\cdot)^{(\Psi_\varphi)^\lfdual}}
}
\end{equation}
\begin{equation}
\xymatrix{
\RbarSQcvxsub(\ell,\Psi_\varphi)
\ar@/^1pc/[rr]^{(\cdot)^{\Psi_\varphi}}
&&
\LSQcvxsub(\ell,\Psi_\varphi);
\ar@/^1pc/[ll]^{(\cdot)^{(\Psi_\varphi)^\lfdual}}
}
\end{equation}
\item\label{cor.LSQ.RSQ.adj.Psi.varphi.ii} if $(X,\n{\cdot}_X)$ is uniformly convex, then there is an adjunction
\begin{equation}
\xymatrix{
		\rbarCvx^\subseteq(\ell,\Psi_\varphi)
				\ar@/^1pc/[rr]^{\iota^{\mathrm{R},\subseteq}_{\ell,\Psi_\varphi}}_{}="ee"
				&&
		\RbarSQcvxsub(\ell,\Psi_\varphi),
					\ar@/^1pc/[ll]^{\Fix^{\mathrm{R},\subseteq}_{\ell,\Psi_\varphi}}_{\ }="ff"
								\ar@{}^{\bot}"ff";"ee"
}
\end{equation}
together with the corresponding monads and comonads.
\end{enumerate}
\end{corollary}
\begin{proof}
Follows from Propositions \ref{prop.adjointness.cvx} and \ref{prop.varphi.compositional}.
\end{proof}

\begin{remark}\label{remark.categorical}
\begin{enumerate}[nosep,label=(\roman*)]
\item\label{remark.categorical.i} Regarding Definition \ref{def.cats.breg.proj}:
\begin{enumerate}[nosep,label=\alph*)]
\item\label{remark.categorical.i.a} the restriction to subsets of $\ell^{\inver}(\intefd{\Psi})$, together with (\hyperlink{def.cats.breg.proj.L}{L}) (resp., (\hyperlink{def.cats.breg.proj.R}{R})), guarantees composability of $\composition$ by means of zone consistency of $\LPPP^{D_{\ell,\Psi}}_C$ (resp.,  $\RPPP^{D_{\ell,\Psi}}_C$), via Corollary \ref{cor.d.ell.psi.properties}.\ref{cor.d.ell.psi.properties.iii}.\ref{cor.d.ell.psi.properties.iii.2} (resp. Corollary \ref{cor.d.ell.psi.properties}.\ref{cor.d.ell.psi.properties.iv}.\ref{cor.d.ell.psi.properties.iv.2});
\item\label{remark.categorical.i.b} $C_1\cap C_2=C_2\cap C_1$ implies commutativity of $\composition$;
\item\label{remark.categorical.i.c} we interpret an empty (resp., identity) arrow as an inference corresponding to overdetermination (resp., underdetermination) of constraints (cf. \cite[p. 35]{Jaynes:1985} for a related discussion);
\item\label{remark.categorical.i.d} the notation $\rCvx(\ell,\Psi)$ (resp., $\rAff(\ell,\Psi)$) is kept reserved for the category of right $D_{\ell,\Psi}$-projections onto $\ell$-closed and $\ell$-convex (resp., $\ell$-closed and $\ell$-afine) subsets of $Z$. An example of $\rCvx(\Psi)$ is provided by \cite[p. 192, Def. 3.1, Lem. 3.5]{Bauschke:Noll:2002}: if $n\in\NN$, $X=\RR^n$, $\Psi$ is Euler--Legendre with $\DG\DG\Psi$ continuous on $\intefd{\Psi}$, $D_\Psi$ is jointly convex, $D_\Psi(x,\,\cdot\,)$ is strictly convex on $\intefd{\Psi}$ $\forall x\in\intefd{\Psi}$, and $\efd(\Psi^\lfdual)$ is open, then every convex closed $K\subseteq X$ with $K\cap\intefd{\Psi}\neq\varnothing$ is a right $D_\Psi$-Chebysh\"{e}v set, with $\RPPP^{D_\Psi}_K(y)\in\intefd{\Psi}$ $\forall y\in\intefd{\Psi}$.
\end{enumerate}
\item\label{remark.categorical.ii}
\begin{enumerate}[nosep,label=\alph*)]
\item\label{remark.categorical.ii.a} The composition rule $\composition$ for left $D_\Psi$-projections has a range of well defined computational meanings. Its quantitative evaluation can be performed by means of an algorithm given in \cite[Rem. 4.5, Alg. 5.1, Cor. 5.2]{Bauschke:Combettes:2003} (valid for any $\{K_i\mid i\in I\}$ with a countable set $I$ and any Euler--Legendre $\Psi$ that is totally convex\footnote{\cite[Rem. 4.5]{Bauschke:Combettes:2003} assumes uniform convexity on bounded subsets of $(X,\n{\cdot}_X)$, yet it is equivalent with total convexity on bounded subsets of $(X,\n{\cdot}_X)$ due to \cite[Thm. 2.10]{Butnariu:Resmerita:2006}.} on bounded subsets of $(X,\n{\cdot}_X)$, hence, in particular, for any LSQ-adapted $\Psi$), or by means of  \cite[Thm. 3.2]{Bauschke:Lewis:2000} \cite[Alg. 2.4, Thm. 3.1]{Bregman:Censor:Reich:1999} \cite[Alg. B]{Dhillon:Tropp:2007} (valid for $\dim X\in\NN$, a finite family $\{K_i\mid i\in\{1,\ldots,n\}, n\in\NN\}$, and Euler--Legendre $\Psi$ satisfying some additional conditions). For $(X,\n{\cdot}_X)$ given by the Hilbert space $(\H,\n{\cdot}_\H)$ and $\Psi_{\varphi_{1,1/2}}=\frac{1}{2}\n{\cdot}_\H^2$, the former algorithm turns to Haugazeau's \cite[Thm. 3-2]{Haugazeau:1968} algorithm, while the latter turns to Dykstra's algorithm \cite[p. 838, Thm. 3.2]{Dykstra:1983} \cite[\S2, Thm. 4.7]{Han:1988} (valid for $\dim\H\in\NN\cup\aleph_0$ \cite[p. 32, Thm. 2]{Boyle:Dykstra:1986}, and extendable to $\{K_i\mid i\in I\}$ with a countable set $I$ \cite[\S2]{Hundal:Deutsch:1997}). Under further restriction of $\{K_i\mid i\in\{1,\ldots,n\}, n\in\NN\}$ to a finite family of closed linear subspaces of $\H$, $\LPPP^{D_{\Psi_{\varphi_{1,1/2}}}}_{K_i}$ turn into orthogonal projection operators $P_{K_i}:\H\ra K_i$, while Dykstra's algorithm turns into Halperin's theorem \cite[Thm. 1]{Halperin:1962} on strong convergence of a cyclic repetition of $P_{K_n}\cdots P_{K_1}$ to $P_{K_1\cap\ldots\cap K_n}$, i.e. 
\begin{equation}
	\lim_{k\ra\infty}\n{\left((P_{K_n}\cdots P_{K_1})^k-P_{K_1\cap\ldots\cap K_n}\right)\xi}_\H=0\;\;\forall\xi\in\H.
\label{eqn.halperin}
\end{equation}
When only two projections are considered, corresponding to a composition $\LPPP^{D_{\Psi_{\varphi_{1,1/2}}}}_{K_1}\composition\LPPP^{D_{\Psi_{\varphi_{1,1/2}}}}_{K_2}$ for linear subspaces $K_1$ and $K_2$ of $\H$, \eqref{eqn.halperin} becomes the von Neumann--Kakutani theorem \cite[Thm. 13.7]{vonNeumann:1933} \cite[pp. 42--44]{Kakutani:1940:Nakano}.
\item\label{remark.categorical.ii.b} All of above algorithms provide evaluation of the left $D_\Psi$-projection $\LPPP^{D_\Psi}_{K_1\cap\ldots\cap K_i}(x)$, for $i\in I$ where $I$ is a finite or countable set, in terms of a norm convergence of a cyclic sequence of algorithmic steps to the unique limit point. The differences in definitions of those algorithms correspond to different ranges of generality and computational effectivity. In particular, while the direct extension on the von Neumann--Kakutani algorithm to closed convex sets converges weakly to an element in the nonempty intersection of $K_1$ and $K_2$ \cite[Thm. 1]{Bregman:1965} (Kaczmarz's algorithm \cite[pp. 355--357]{Kaczmarz:1937} is a special case of this extension, obtained for hyperplanes and $\dim\H\in\NN$), the limit point may be not equal to a projection onto $K_1\cap K_2$ \cite[Fig. 2]{Combettes:1993} and the norm convergence generally does not hold \cite[Thm. 1]{Hundal:2004} (although the latter holds always for $\dim\H\in\NN$, and can be guaranteed under additional conditions for $\dim\H=\aleph_0$ \cite[Thms. 1, 2]{Gurin:Polyak:Raik:1967}). On the other hand, the direct extension of Halperin's theorem to linear  projections, of norm equal to 1, onto subspaces of uniformly convex Banach space is norm convergent and returns a projection, of norm equal to 1, onto an intersection \cite[Thm. 2.1]{Bruck:Reich:1977}. For noncyclic algorithms, see \cite{Browder:1958,Prager:1960,Amemiya:Ando:1965,Bruck:1982,Hundal:Deutsch:1997,Bauschke:Borwein:1997,Bauschke:Combettes:2003}.
\end{enumerate}
\item\label{remark.categorical.iii} If  $(X,\n{\cdot}_X)$ is separable, then $\lAff(\Psi)$ has objects given by the countable sets of polynomial equations, which can be interpreted as \textit{data types}, with morphisms between them interpreted as \textit{programs} (algorithms). More generally, if $(X,\n{\cdot}_X)$ is a separable Banach space, then every convex closed subset $C\subseteq X$ is the intersection of the countable number of its supporting closed half-spaces \cite[Cor. 3]{Bishop:Phelps:1963}, i.e. it is a (countable) polyhedron, which is the set of solutions for a countable system of linear inequalities (see \cite{Borwein:Vanderwerff:2004} for a discussion of the nonseparable case). Hence, also $\lCvx(\Psi)$, at least for separable $(X,\n{\cdot}_X)$, can be represented as a category of specific data types and computations between them.
\item\label{remark.categorical.iv} Functor $\overline{\co^\mathrm{L}_{\ell,\Psi}(\cdot)}^{\ell}$ (resp., $\overline{\co^\mathrm{R}_{\ell,\Psi}(\cdot)}^{\ell}$) can be seen as implementing the following procedure: given the collection of `raw' data sets and maps between them, produce the category of $\ell$-convex $\ell$-closed (resp., $\DG\Psi\circ\ell$-convex $\DG\Psi\circ\ell$-closed) information state spaces and inferences between them, provided by left (resp., right) $D_{\ell,\Psi}$-projections. The construction of objects by means of these functors implements maximum absolute entropy procedure, along the lines of \cite{Elsasser:1937,Stratonovich:1955,Jaynes:1957,Jaynes:1979:where:do:we:stand}, while the construction of morphisms corresponds to maximum relative entropy procedure, along the lines of \cite{Sanov:1957,Kullback:1959,Bregman:1966,Chencov:1968,Hobson:1969}.
\item\label{remark.categorical.v}
\begin{enumerate}[nosep,label=\alph*)]
\item\label{remark.categorical.v.a} The equivalence in Corollary \ref{cor.functors.breg.proj} may seem trivial, since it is a direct consequence of the definition of $\rbarCvx(\Psi)$. Yet, we see it is as a top of an iceberg: currently it is an open question whether $\rCvx(\Psi)$ is a subcategory of $\rbarCvx(\Psi)$ or is it an independent structure (see \cite[Ex. 7.5]{Bauschke:Wang:Ye:Yuan:2009} for an example of $\RPPP^{D_\Psi}_C$ with convex $\DG\Psi(C)$ and nonconvex $C$), the equivalence between LSQ$(\Psi)$ and RSQ$(\Psi)$ classes holds only under special conditions (see Proposition \ref{prop.lsq.rsq.old}.\ref{prop.lsq.rsq.old.v}), and there is an important difference between availability of LSQ- vs RSQ-adaptedness in models (see Propositions \ref{prop.lsq.rsq.new} and \ref{prop.varphi.compositional}). Furthermore, while $\LPPP^{D_{\ell,\Psi}}$ for $D_{\ell,\Psi}=D_1$ correspond to Sanov-type theorems \cite[Thms. 10--13]{Sanov:1957} \cite[Thm. 1]{Csiszar:1984} \cite[Thm. 2]{Bjelakovic:Deuschel:Krueger:Seiler:SiegmundSchulze:Szkola:2005} (and, more generally, for any Csisz\'{a}r--Morimoto information $D_\fff$ \cite[p. 86]{Csiszar:1963} \cite[p. 329]{Morimoto:1963}, $\LPPP^{D_\fff}$ corresponds to conditional laws of large numbers, cf., e.g., \cite[\S7]{Leonard:2010:entropic}), $\RPPP^{D_{\ell,\Psi}}$ correspond to minimum contrast (e.g., maximum likelihood) estimation \cite[pp. 328--330]{Chencov:1968} \cite[\S22]{Chencov:1972} \cite[\S1]{Eguchi:1983} \cite[p. 93 (Engl. rev. ed.)]{Amari:Nagaoka:1993}. 
\item\label{remark.categorical.v.b} The inequality 
\begin{equation}
\s{y-T(x),x-T(x)}_\H\leq0\;\forall(x,y)\in\H\times C
\label{eqn.Aronszajn.characterisation}
\end{equation}
characterises \cite[p. 87]{Aronszajn:1950} metric ($=D_{\Psi_{\varphi_{1,1/2}}}$-) projections, $T=\PPP^{d_{\n{\cdot}_\H}}_C$ onto convex closed subsets $C$ in Hilbert space $\H$.
In general, the dichotomy between $\LPPP^{D_\Psi}$ and $\RPPP^{D_\Psi}$ can be seen as $D_\Psi$-version of the left/right split of \eqref{eqn.Aronszajn.characterisation} under a passage from $\H$ to Banach spaces. More precisely, if $(X,\n{\cdot}_X)$ is a Gateaux differentiable Banach space and $\varnothing\neq C\subseteq X$ is convex and closed, then
\begin{equation}
\duality{y-T(x),j(x-T(x))}_{X\times X^\star}\leq0\;\;\forall(x,y)\in X\times C
\label{eqn.metric.proj}
\end{equation}
characterises metric projections on $C$  \rpkmark{\cite{Deutsch:1965}} \cite[Thm. (p. 711)]{Rubinshtein:1965} \cite[Eqn. (5.11)]{Lions:1969}, while
\begin{equation}
\duality{x-T(x),j(y-T(x))}_{X\times X^\star}\leq0\;\;\forall(x,y)\in X\times C
\label{eqn.sunny.retractions}
\end{equation}
characterises sunny completely $\n{\cdot}_X$-nonexpansive retractions\footnote{Given  nonempty subsets $K_1$ and $K_2$ of a Banach space $(X,\n{\cdot}_X)$, a function $T:K_1\ra K_2$ is called: a \df{retraction} if{}f $T(x)=x$ $\forall x\in K_2$ \cite[\S2]{Borsuk:1931}; \df{sunny} if{}f $T(x)=y$ $\limp$ $T(y+t(x-y))=y$ $\forall x\in K_1$ $\forall t\geq0$ \cite[p. 19]{Efimov:Stechkin:1958}.} on $C$ \cite[Lem. 2.7]{Reich:1973}. On the other hand, if $(X,\n{\cdot}_X)$ is reflexive, Gateaux differentiable, and strictly convex, then \eqref{eqn.alber.characterisation} characterises left $D_{\Psi_{\varphi_{1,1/2}}}$-projections \cite[Prop. 7.c]{Alber:1996}, while, if $(X,\n{\cdot}_X)$ is reflexive, $\Psi$ is totally convex on $\efd(\Psi)$ and Euler--Legendre, with $\efd(\Psi^\lfdual)=X^\star$, then right $D_\Psi$-projections are characterised as sunny quasinonexpansive $D_\Psi$-retractions \cite[Cor. 4.6]{MartinMarquez:Reich:Sabach:2012}.
\item\label{remark.categorical.v.c} This suggests us a tentative conjecture that the Euler--Legendre transform in the Va\u{\i}nberg--Br\`{e}gman setting, under a suitable choice of categories (e.g., a category of left $D_\Psi$-projections and a category of sunny quasinonexpansive right $D_\Psi$-retractions), may be an adjunction, with the above equivalence as a special case (arising as a relationship between reflective and coreflective subcategory of an above adjunction). Can this conjecture be approached via the notion of a nucleus of profunctor, as in \cite[\S5]{Willerton:2015}?
\end{enumerate}
\item\label{remark.categorical.vi} Regarding Proposition \ref{prop.d.psi.as.functor}, dependence of $D_\Psi(\phi,\cdot)$ on $Q$ can be factored out by reducing considerations to singletons $Q=\{\phi\}$ (understood as 0-dimensional closed affine spaces). In (some) analogy to \cite[Thm. 7]{Baez:Fritz:2014} \cite[Thm. 4.4]{Gagne:Panangaden:2018}, this allows us to state a problem of characterisation of $D_\Psi$ as a functor (or a natural transformation) $D_\Psi(\phi,\cdot)$.
\item\label{remark.categorical.vii} Given any $Q\in\Ob(\lCvx(\Psi))$, $\Hom_{\lCvx(\Psi)}(\cdot,Q)$ can be equipped with the structure of a commutative partially ordered monoid (i.e. a monoid $(M,\monoid)$, satisfying $x\monoid y = y\monoid x$ $\forall x,y\in M$, and equipped with a partial order $\leq$ such that $x\leq y$ $\limp$ $z\monoid x\leq z\monoid y$ $\forall x,y,z\in M$), with $\LPPP^{D_\Psi}_{Q_1}\composition\LPPP^{D_\Psi}_{Q_2}=\LPPP^{D_\Psi}_{Q_1}\monoid\LPPP^{D_\Psi}_{Q_2}:=\LPPP^{D_\Psi}_{Q_1\cap Q_2}$, $\LPPP^{D_\Psi}_{Q_1}\leq\LPPP^{D_\Psi}_{Q_2}:=Q_1\subseteq Q_2$, and a distinguished zero object, given by $\LPPP^{D_\Psi}_Q$. (Hence, each $\Hom_{\lCvx(\ell,\Psi)}(\cdot,Q)$ forms a resource theory in the sense of \cite[\S3]{Fritz:2017} (the latter generalises, in particular, the approaches of \cite{Lieb:Yngvason:1999} and \cite{Devetak:Harrow:Winter:2008}).) Viewing the order of extended positive reals as a feature distinct from their composition by addition turns $[0,\infty]$ into a commutative partially ordered monoid (with $x+\infty:=\infty=:\infty+x$ $\forall x\in[0,\infty]$). Thus, each functor $D_{\Psi}(\phi,\cdot)$ can be seen as a morphism $\Hom_{\lAff^\subseteq_Q(\Psi)}(\cdot,Q)\ra[0,\infty]$ inside the category of commutative partially ordered monoids.
\item\label{remark.categorical.viii} The extension from the monoid structure of the composable sets $\lsq(\Psi,K)$ and $\rsq(\Psi,K)$ to the corresponding categories $\LSQcvxsub(\Psi,K)$ and $\RbarSQcvxsub(\Psi,K)$ (which depends on the associativity of composition) is possible due to an observation that the proofs of \cite[Lems. 1, 2]{Reich:1996} and \cite[Props. 3.3, 4.4, 6.6, Fact 6.5]{MartinMarquez:Reich:Sabach:2013:BSN} do not depend on the action of $T_i$ on the domain $K\setminus\ran(T_{i-1})$ for $i\in\{2,\ldots,m\}$, $m\in\NN$. Hence, these results hold in larger generality, namely for $m$-tuples of maps $(T_1:K\ra K,\;T_2:\ran(T_1)\ra K,\;T_3:\ran(T_2)\ra K,\;\ldots,\;T_m:\ran(T_{m-1})\ra K)$.
\item\label{remark.categorical.ix} The restriction of considerations from the category $\RSQcvx^\subseteq(\ell,\Psi)$ (with objects given by any nonempty subsets of $\intefd{\Psi}$) to $\RbarSQcvxsub(\ell,\Psi)$, as exhibited in Definitions \ref{def.LSQ.RSQ.cats} and \ref{def.LSQ.RSQ.functors}, Proposition \ref{prop.adjointness.cvx}, and Corollary \ref{cor.LSQ.RSQ.adj.Psi.varphi}, is due to requirement of compatibility with $\rbarCvx(\ell,\Psi)$, as well as with the use of Proposition \ref{prop.lsq.rsq.old}.\ref{prop.lsq.rsq.old.v}. Thus, the discussion in \ref{remark.categorical.v} applies, mutatis mutandis, to $\RbarSQcvxsub(\ell,\Psi)$.
\item\label{remark.categorical.x} Under the additional assumptions on $\Psi$ provided by Proposition \ref{prop.continuity.psi}, and assuming norm-to-norm continuity of $\ell$, we obtain the categories $\lCvx_{\mathrm{cont}}(\ell,\Psi)$ and $\rbarCvx_{\mathrm{cont}}(\ell,\Psi)$ of norm-to-norm continuous left and right $D_{\ell,\Psi}$-projections, respectively. By Proposition \ref{prop.continuity}, if $\Psi=\Psi_\varphi$ for a gauge $\varphi$, then the above assumptions on $\Psi$ are equivalent with assuming that $(X,\n{\cdot}_X)$ is reflexive, strictly convex, Fr\'{e}chet differentiable, and has the Radon--Riesz--Shmul'yan property. Analogously, Corollary \ref{cor.Lipschitz.Hoelder.projections.Psi} leads to the category $\lCvx_{\mathrm{LH},\;X}^{s/(r-1)}(\Psi)$ (resp., $\rbarCvx_{\mathrm{LH},\;X}^{s^2w/(r-1)}(\Psi)$ of $\frac{s}{r-1}$(resp., $\frac{s^2w}{r-1}$)-Lipschitz--H\"{o}lder continuous left (resp., right) $D_{\ell,\Psi}$-projections on $X$, while Proposition \ref{prop.varphi.uniform.continuity} leads to the category $\lCvx_{\mathrm{LH},\;X}^{\beta(r-1)/(1-\beta)}(\Psi_{\varphi_{1,\beta}})$ (resp., $\rbarCvx_{\text{LH, bound}(X)}^{(1-\beta)^2(r-1)^2/\beta^2}(\Psi_{\varphi_{1,\beta}})$) of $\frac{\beta(r-1)}{1-\beta}$(resp., $\frac{(1-\beta)^2(r-1)^2}{\beta^2}$)-Lipschitz--H\"{o}lder continuous left (resp., right) $D_{\Psi_{\varphi_{1,\beta}}}$-projections on $X$ (resp., bounded subsets of $X$).
\end{enumerate}
\end{remark}
\section{Some models}\label{section.models}
\subsection{$X$ = $L_{1/\gamma}$ space, $\ell$ = Mazur map, $\Psi$ = $\Psi_\varphi$}\label{section.L.gamma.Mazur}
\begin{proposition}\label{prop.dope.gauge.gamma.wstar}
Let $\N$ be a W$^*$-algebra, $\varphi$ a gauge, $\gamma\in\,]0,1[$, $\lambda\in\,]0,\infty[$, $\varnothing\neq C\subseteq\N_\star$, $\B$ a ball in $\N_\star$ (e.g., $\B=B(\N_\star,\n{\cdot}_1)$). Then:
\begin{enumerate}[nosep,label=(\roman*)]
\item\label{prop.dope.gauge.gamma.wstar.i} $D_{\lambda\ell_\gamma,\Psi_\varphi}:\N_\star\times\N_\star\ra[0,\infty]$ is an information on $\N_\star$;
\item\label{prop.dope.gauge.gamma.wstar.ii} if $C$ is $\lambda\ell_\gamma$-convex $\lambda\ell_\gamma$-closed, then $D_{\lambda\ell_\gamma,\Psi_\varphi}$ is left pythagorean on $C$, $\LPPP^{D_{\lambda\ell_\gamma,\Psi_\varphi}}_C$ is zone consistent, and adapted;
\item\label{prop.dope.gauge.gamma.wstar.iii} if $C\subseteq\N_\star^+\cup\B$ is $\lambda\ell_\gamma$-convex closed, then $D_{\lambda\ell_\gamma,\Psi_\varphi}$ is left pythagorean on $C$, $\LPPP^{D_{\lambda\ell_\gamma,\Psi_\varphi}}_C$ is zone consistent, adapted, and norm-to-norm continuous on $\N_\star$, and $\inf_{y\in C}\{D_{\lambda\ell_\gamma,\Psi_\varphi}(y,\,\cdot\,)\}$ is continuous on $\N_\star$;
\item\label{prop.dope.gauge.gamma.wstar.iv} if $C$ is $\lambda\ell_{1-\gamma}$-convex $\lambda\ell_{1-\gamma}$-closed, then $D_{\lambda\ell_\gamma,\Psi_\varphi}$ is right pythagorean on $C$, and $\RPPP^{D_{\lambda\ell_\gamma,\Psi_\varphi}}_{C}$ is zone consistent, and adapted;
\item\label{prop.dope.gauge.gamma.wstar.v} if $C\subseteq\N_\star^+\cup\B$ is $\lambda\ell_{1-\gamma}$-convex closed, then $D_{\lambda\ell_\gamma,\Psi_\varphi}$ is right pythagorean on $C$, and $\RPPP^{D_{\lambda\ell_\gamma,\Psi_\varphi}}_{C}$ is zone consistent, adapted, and norm-to-norm continuous on $\N_\star$;
\item\label{prop.dope.gauge.gamma.wstar.vi} the sets $\lsq(\lambda\ell_\gamma,\Psi_\varphi,C)$ and $\rsq(\lambda\ell_\gamma,\Psi_\varphi,C)$ are composable;
\item\label{prop.dope.gauge.gamma.wstar.vii} the categories $\lCvx^\subseteq(\lambda\ell_\gamma,\Psi_\varphi)$, $\rbarCvx^\subseteq(\lambda\ell_\gamma,\Psi_\varphi)$, $\LSQcvxsub(\lambda\ell_\gamma,\Psi_\varphi)$, and $\RbarSQcvxsub(\lambda\ell_\gamma,\Psi_\varphi)$ satisfy the functorial adjunctions and equivalences given by Corollary \ref{cor.LSQ.RSQ.adj.Psi.varphi}.\ref{cor.LSQ.RSQ.adj.Psi.varphi.i}--\ref{cor.LSQ.RSQ.adj.Psi.varphi.ii} with $Z=\N_\star$;
\item\label{prop.dope.gauge.gamma.wstar.viii} if $T:L_{1/\gamma}(\N)\ra 2^{L_{1/(1-\gamma)}(\N)}$ is maximally monotone with $0\in\efd(T)$, then $\lres^{\Psi_\varphi}_T$ maps $L_{1/\gamma}(\N)$ on $\efd(T)$ and is norm-to-norm continuous on $(L_{1/\gamma}(\N),\n{\cdot}_{1/\gamma})$, $\rres^{\Psi_\varphi}_T$ maps $L_{1/(1-\gamma)}(\N)$ on $j_\varphi(\efd(T))$ and is norm-to-norm continuous on $(L_{1/(1-\gamma)}(\N),\n{\cdot}_{1/(1-\gamma)})$, and $\lres_T^{\ell_\gamma,\Psi_\varphi}$ maps $\N_\star$ on ${\ell_\gamma}^\inver(\efd(T))$ and is norm-to-norm continuous on $\N_\star$.
\end{enumerate}
\end{proposition}
\begin{proof}
Since $\intefd{\Psi_\varphi}=L_{1/\gamma}(\N)$, we have $(\lambda\ell_\gamma)^{\inver}(\intefd{\Psi_\varphi})=\N_\star$. Zone consistency of left and right $D_{\lambda\ell_\gamma,\Psi_\varphi}$-projections follows from Proposition \ref{prop.left.right.psi.varphi}.\ref{prop.left.right.psi.varphi.iv}. For any $\gamma\in\,]0,1[$, $(L_{1/\gamma}(\N),\n{\cdot}_{1/\gamma})$ is uniformly convex (as proved for $\gamma\in\,]0,\frac{1}{2}]$ in \cite[Lem. 5]{Dixmier:1953} and for $\gamma\in\,]0,1[$ in \cite[Cor. 2.1]{Cleaver:1973}\footnote{\cite[Thm. 2.7]{McCarthy:1967} is often cited in this case, however its proof is incorrect (see \cite[p. 299]{Fack:Kosaki:1986}).} for type I $\N$, for $\gamma\in\,]0,1[$ in \cite[Lems. 3.12, 3.22, p. 262]{Zsido:1980}, for $\gamma\in\,]0,1[$ in \cite[Prop. 8.2, Lem. 9.1]{Araki:Masuda:1982} and \cite[Thm. 4.2]{Kosaki:1984:ncLp} for countably finite $\N$, for $\gamma\in\,]0,\frac{1}{2}]$ in \cite[Lem. 3.4.2.(i)]{Kosaki:1980:PhD} \cite[Prop. 31]{Terp:1981} (cf. \cite[Lem. 1.18]{Haagerup:1979:ncLp}) \cite[Lem. 9]{Hilsum:1981} and for $\gamma\in\,]0,1[$ in \cite[Lems. 8.1, 8.2]{Masuda:1983} \cite[Thm. 5.3]{Fack:Kosaki:1986} for any $\N$), hence uniformly Fr\'{e}chet differentiable (due to $(L_{1/\gamma}(\N),\n{\cdot}_{1/\gamma})^\star\iso(L_{1/(1-\gamma)}(\N),\n{\cdot}_{1/(1-\gamma)})$ $\forall\gamma\in\,]0,1[$, proved in \cite[p. 580]{Schatten:vonNeumann:1948} for type I $\N$, \cite[Thm. 7]{Dixmier:1953} and \cite[Thm. 4.2]{Yeadon:1975} for semifinite $\N$, and in \cite[Thm. 3.4.3]{Kosaki:1980:PhD} \cite[Thm. 10.(2)]{Hilsum:1981} \cite[Thm. 32.(2)]{Terp:1981} (cf. \cite[Thm. 1.19]{Haagerup:1979:ncLp}) for arbitary $\N$). Since uniform convexity entails both strict convexity and the Radon--Riesz--Shmul'yan property, while uniform Fr\'{e}chet differentiability entails (Fr\'{e}chet differentiability and thus) Gateaux differentiability, $\Psi_\varphi$ is Euler--Legendre on any $(L_{1/\gamma}(\N),\n{\cdot}_{1/\gamma})$ by means of Proposition \ref{prop.legendre}, and is left (resp., right) pythagorean on convex closed (resp., $\DG\Psi_\varphi$-convex $\DG\Psi_\varphi$-closed) $K=\lambda\ell_\gamma(C)$ by means of Proposition \ref{prop.left.right.psi.varphi} (this proposition implies also zone consistency in \ref{prop.dope.gauge.gamma.wstar.ii} and \ref{prop.dope.gauge.gamma.wstar.iii}). By Proposition \ref{prop.psi.varphi.geometry}.\ref{prop.psi.varphi.geometry.iv}, Fr\'{e}chet differentiability of $(X,\n{\cdot}_X)$ is equivalent with norm-to-norm continuity of $j_\varphi$ for any gauge $\varphi$. Furthermore, for any $\gamma\in\,]0,1[$, $\ell_\gamma$ is a norm-to-norm homeomorphism from the positive cone of $(\N_\star,\n{\cdot}_1)$ to the positive cone of $(L_{1/\gamma}(\N),\n{\cdot}_{1/\gamma})$ \cite[Thm. 4.2]{Kosaki:1984:uniform}, and a uniform homeomorphism from any ball in $\N_\star$ \cite[Lem. 3.2]{Raynaud:2002}. Since $\lambda$ is a multiplicative constant, the same conclusion follows for $\lambda\ell_\gamma$. Since $\DG\Psi_\varphi$-convexity is a convexity in $L_{1/(1-\gamma)}(\N)$, while $j_\varphi=\DG\Psi_\varphi$ is norm-to-norm continuous, $(\DG\Psi_\varphi\circ\lambda\ell_\gamma)$-convex $(\DG\Psi_\varphi\circ\lambda\ell_\gamma)$-closed sets in $\N_\star$ coincide with $\lambda\ell_{1-\gamma}$-convex $\lambda\ell_{1-\gamma}$-closed sets in $\N_\star$. Since uniform convexity implies also the Radon--Riesz--Shmul'yan property \cite[Thm. 5]{Shmulyan:1939:geometrical}, norm-to-norm continuity in \ref{prop.dope.gauge.gamma.wstar.iii} and \ref{prop.dope.gauge.gamma.wstar.v} follows from Proposition \ref{prop.continuity}. Adaptedness in \ref{prop.dope.gauge.gamma.wstar.ii} and \ref{prop.dope.gauge.gamma.wstar.v} and composability in \ref{prop.dope.gauge.gamma.wstar.vi}, follow from Proposition \ref{prop.varphi.compositional}, and these imply \ref{prop.dope.gauge.gamma.wstar.vii}. \ref{prop.dope.gauge.gamma.wstar.viii} follows from Proposition \ref{prop.varphi.resolvent.norm.continuity}. 
\end{proof}

\begin{proposition}\label{prop.lambda.gamma.alpha.beta}
For $\varphi=\varphi_{\alpha,\beta}$, $\alpha,\lambda\in\,]0,\infty[$, $\beta,\gamma\in\,]0,1[$, any W$^*$-algebra $\N$, and $\phi,\psi\in\N_\star$,
\begin{equation}
D_{\lambda\ell_\gamma,\Psi_{\varphi_{\alpha,\beta}}}(\phi,\psi)={\textstyle\frac{\lambda^{1/\beta}}{\alpha}}\left(\beta\n{\phi}^{\gamma/\beta}_1+(1-\beta)\n{\psi}_1^{\gamma/\beta}-\n{\psi}_1^{\gamma(\frac{1}{\beta}-\frac{1}{\gamma})}\re\int u_\phi\ab{\phi}^\gamma u_\psi\ab{\psi}^{1-\gamma}\right).
\label{eqn.lambda.gamma.alpha.beta}
\end{equation}
\end{proposition}
\begin{proof}
By a direct calculation from \eqref{eqn.psi.varphi.alpha.beta}, using \cite[Lem. 3.1]{Kosaki:1984:uniform} and $L_{1/\gamma}(\N)\ni x\mapsto j(x)=\n{x}^{2-1/\gamma}_{1/\gamma}u_x\ab{x}^{1/\gamma-1}\in L_{1/(1-\gamma)}(\N)$, the latter following from \cite[Prop. 24]{Terp:1981} \cite[p. 162]{Hilsum:1981} (cf. also \cite[Eqn. (11)]{Jencova:2005}).
\end{proof}

\begin{corollary}\label{cor.d.gamma}
For $\varphi=\varphi_{\alpha,\beta}$, $\alpha,\lambda\in\,]0,\infty[$, $\beta,\gamma\in\,]0,1[$, any W$^*$-algebra $\N$, and $\phi,\psi\in\N_\star$:
\begin{enumerate}[nosep,label=(\roman*)]
\item\label{cor.d.gamma.i}
$D_{\lambda\ell_\gamma,\Psi_{\varphi_{\alpha,\beta}}}=D_{\ell_\gamma,\Psi_{\varphi_{\alpha\lambda^{-1/\beta},\beta}}}$;
\item\label{cor.d.gamma.ii} if $\lambda=1$, $\beta=\gamma$, $\alpha=\gamma(1-\gamma)$, then $\varphi_{\gamma(1-\gamma),\gamma}(t)=\frac{1}{\gamma(1-\gamma)}t^{1/\gamma-1}$, $\Psi_{\varphi_{\gamma(1-\gamma),\gamma}}(x)=\frac{1}{1-\gamma}\n{x}_{1/\gamma}^{1/\gamma}$, and
\begin{align}
D_{\ell_\gamma,\Psi_{\varphi_{\gamma(1-\gamma),\gamma}}}(\phi,\psi)&=D_{\frac{1}{\gamma}\ell_\gamma,\Psi_{\varphi_{\gamma^{1-1/\gamma}(1-\gamma),\gamma}}}(\phi,\psi)
\label{eqn.psi.gamma.equals.psi.gamma}
\\
&={\textstyle\frac{1}{1-\gamma}}\n{\phi}_1+{\textstyle\frac{1}{\gamma}}\n{\psi}_1+{\textstyle\frac{1}{\gamma(1-\gamma)}}\re\int u_\phi\ab{\phi}^\gamma u_\psi\ab{\psi}^{1-\gamma}=:D_\gamma(\phi,\psi);
\label{eqn.d.gamma}
\end{align}
\item\label{cor.d.gamma.iii} $\cptp(\N_\star)\subseteq\cn(\ell_\gamma,\Psi_{\varphi_{\gamma(1-\gamma),\gamma}})$, where $\cptp(\N_\star)$ denotes the set of completely positive trace-preserving maps from $\N_\star$ to $\N_\star$.
\end{enumerate}
\end{corollary}
\begin{proof}
\ref{cor.d.gamma.i} follows directly from \eqref{eqn.lambda.gamma.alpha.beta}, while \ref{cor.d.gamma.ii} is a special case of \ref{cor.d.gamma.i}. \ref{cor.d.gamma.iii} follows from \ref{cor.d.gamma.ii}, combined with the result \cite[(ii) (p. 288)]{Jencova:2005} obtained for $D_\gamma$ on $\N_\star$.
\end{proof}

\begin{proposition}\label{prop.uniform.Lp}
Let $\N$ be a W$^*$-algebra, $\beta,\gamma\in\,]0,1[$, $\varnothing\neq K\subseteq L_{1/\gamma}(\N)$, $\varnothing\neq C\subseteq B(\N_\star,\n{\cdot}_1)$, let $T:L_{1/\gamma}(\N)\ra 2^{L_{1/(1-\gamma)}(\N)}$ and $W:B(L_{1/\gamma}(\N),\n{\cdot}_{1/\gamma})\ra 2^{j_{\varphi_{1,\beta}}(B(L_{1/\gamma}(\N),\n{\cdot}_{1/\gamma}))}$ be maximally monotone, $f\in\pcl(L_{1/\gamma}(\N),\n{\cdot}_{1/\gamma})$, let $g:L_{1/\gamma}(\N)\ra\,]-\infty,\infty]$ satisfy $g\circ j_{\varphi_{1,1-\beta}}\in\pcl(L_{1/(1-\gamma)}(\N),\n{\cdot}_{1/(1-\gamma)})$ for $\beta\in[\frac{1}{2},1[$, and $\lambda\in\,]0,1[$. Then:
\begin{enumerate}[nosep,label=(\roman*)]
\item\label{prop.uniform.Lp.i} if $\gamma\in[\frac{1}{2},1[$ and $\beta\in\,]0,\frac{1}{2}]$, $K$ is convex and closed (resp., $C$ is $\ell_\gamma$-convex and closed), then $\LPPP^{D_{\Psi_{\varphi_{1,\beta}}}}_K$ (resp., $\LPPP^{D_{\ell_\gamma,\Psi_{\varphi_{1,\beta}}}}_C$) is uniformly continuous on bounded subsets of $L_{1/\gamma}(\N)$ (resp., $\ell_\gamma$-bounded subsets of $B(\N_\star,\n{\cdot}_1)$), and $\frac{\beta(\gamma-1)}{1-\beta}$-(resp., $\frac{\beta(\gamma-1)\gamma}{1-\beta}$-)Lipschitz--H\"{o}lder continuous on $L_{1/\gamma}(\N)$ (resp., $B(\N_\star,\n{\cdot}_1)$);
\item\label{prop.uniform.Lp.ii} if $\gamma\in\,]0,\frac{1}{2}]$ and $\beta\in\,]0,\gamma]$, $K$ is convex and closed (resp., $C$ is $\ell_\gamma$-convex and closed), then $\LPPP^{D_{\Psi_{\varphi_{1,\beta}}}}_K$ (resp., $\LPPP^{D_{\ell_\gamma,\Psi_{\varphi_{1,\beta}}}}_C$) is uniformly continuous on bounded subsets of $L_{1/\gamma}(\N)$ (resp., $\ell_\gamma$-bounded subsets of $B(\N_\star,\n{\cdot}_1)$) and $\frac{\beta}{1-\beta}$-(resp., $\frac{\beta\gamma}{1-\beta}$-)Lipschitz--H\"{o}lder continuous on $L_{1/\gamma}(\N)$ (resp., $B(\N_\star,\n{\cdot}_1)$); 
\item\label{prop.uniform.Lp.iii} if $\gamma\in[\frac{1}{2},1[$ and $\beta\in[\gamma,1[$, $j_{\varphi_{1,\gamma}}(K)$ is convex and closed (resp., $C$ is $(j_{\varphi_{1,\gamma}}\circ\ell_\gamma)$-convex and closed), then $\RPPP^{D_{\Psi_{\varphi_{1,\beta}}}}_K$ (resp., $\RPPP^{D_{\ell_\gamma,\Psi_{\varphi_{1,\beta}}}}_C$) is $(\frac{1-\beta}{\beta})^2$-(resp., $\gamma(\frac{1-\beta}{\beta})^2$-)Lipschitz--H\"{o}lder continuous on bounded subsets of $L_{1/\gamma}(\N)$ (resp., $\ell_\gamma$-bounded subsets of $B(\N_\star,\n{\cdot}_1)$);
\item\label{prop.uniform.Lp.iv} if $\gamma\in\,]0,\frac{1}{2}]$ and $\beta\in[\frac{1}{2},1[$, $j_{\varphi_{1,\beta}}(K)$ is convex and closed (resp., $C$ is $(j_{\varphi_{1,\beta}}\circ\ell_\gamma)$-convex and closed), then $\RPPP^{D_{\Psi_{\varphi_{1,\beta}}}}_K$ (resp., $\RPPP^{D_{\ell_\gamma,\Psi_{\varphi_{1,\beta}}}}_C$) is $(\frac{(1-\beta)\gamma}{\beta(1-\gamma)})^2$-(resp., $\gamma^3(\frac{(1-\beta)}{\beta(1-\gamma)})^2$-)Lipschitz--H\"{o}lder continuous on bounded subsets of $L_{1/\gamma}(\N)$ (resp., $\ell_\gamma$-bounded subsets of $B(\N_\star,\n{\cdot}_1)$);
\item\label{prop.uniform.Lp.v} if $\gamma\in[\frac{1}{2},1[$ and $\beta\in\,]0,\frac{1}{2}]$, then $\lres^{\Psi_{\varphi_{1,\beta}}}_{\lambda T}$ and $\lprox^{D_{\Psi_{\varphi_{1,\beta}}}}_{\lambda,f}$ (resp., $\lres^{\ell_\gamma,\Psi_{\varphi_{1,\beta}}}_{\lambda W}$) are (resp., is) single-valued and uniformly continuous on bounded subsets of $L_{1/\gamma}(\N)$ (resp., $\ell_\gamma$-bounded subsets of $B(\N_\star,\n{\cdot}_1)$), as well as single-valued and $\frac{\beta(\gamma-1)}{1-\beta}$-(resp., $\frac{\beta(\gamma-1)\gamma}{1-\beta}$-)Lipschitz--H\"{o}lder continuous on $L_{1/\gamma}(\N)$ (resp., $B(\N_\star,\n{\cdot}_1)$);
\item\label{prop.uniform.Lp.vi}\sloppy if $\gamma\in\,]0,\frac{1}{2}]$ and $\beta\in\,]0,\gamma]$, then $\lres^{\Psi_{\varphi_{1,\beta}}}_{\lambda T}$ and $\lprox^{D_{\Psi_{\varphi_{1,\beta}}}}_{\lambda,f}$ (resp., $\lres^{\ell_\gamma,\Psi_{\varphi_{1,\beta}}}_{\lambda W}$) are (resp., is) single-valued and uniformly continuous on bounded subsets of $L_{1/\gamma}(\N)$ (resp., $\ell_\gamma$-bounded subsets of $B(\N_\star,\n{\cdot}_1)$), as well as single-valued and $\frac{\beta}{1-\beta}$-(resp., $\frac{\beta\gamma}{1-\beta}$-)Lipschitz--H\"{o}lder continuous on $L_{1/\gamma}(\N)$ (resp., $B(\N_\star,\n{\cdot}_1)$);
\item\label{prop.uniform.Lp.vii} if $\gamma\in[\frac{1}{2},1[$ and $\beta\in[\gamma,1[$, then $\rres^{\Psi_{\varphi_{1,\beta}}}_{\lambda T}$ is single-valued and $(\frac{1-\beta}{\beta})^2$-Lipschitz--H\"{o}lder continuous on bounded subsets of $L_{1/(1-\gamma)}(\N)$, and $\rprox^{D_{\Psi_{\varphi_{1,\beta}}}}_{\lambda,g}$ is single-valued and $(\frac{1-\beta}{\beta})^2$-Lipschitz--H\"{o}lder continuous on bounded subsets of $L_{1/\gamma}(\N)$;
\item\label{prop.uniform.Lp.viii} if $\gamma\in\,]0,\frac{1}{2}]$ and $\beta\in[\frac{1}{2},1[$, then $\rres^{\Psi_{\varphi_{1,\beta}}}_{\lambda T}$ is single-valued and $(\frac{(1-\beta)\gamma}{\beta(1-\gamma)})^2$-Lipschitz--H\"{o}lder continuous on bounded subsets of $L_{1/(1-\gamma)}(\N)$, and $\rprox^{D_{\Psi_{\varphi_{1,\beta}}}}_{\lambda,g}$ is single-valued and $(\frac{(1-\beta)\gamma}{\beta(1-\gamma)})^2$-Lipschitz--H\"{o}lder continuous on bounded subsets of $L_{1/\gamma}(\N)$.
\end{enumerate}
\end{proposition}
\begin{proof}
Clarkson inequality for $(L_{1/\gamma}(\N),\n{\cdot}_{1/\gamma})$ spaces (\cite[Thm. 2]{Clarkson:1936} for commutative $\N$, \cite[Cor. 2.1]{Cleaver:1973} for type I $\N$, \cite[Lems. 3.21, 3.22, p. 262]{Zsido:1980} for semifinite $\N$, \cite[Prop. 5.3]{Kosaki:1984:ncLp} for countably finite $\N$, \cite[Thm. 5.3]{Fack:Kosaki:1986} for any $\N$) implies: if $\gamma\in\,]0,\frac{1}{2}]$, then $(L_{1/\gamma}(\N),\n{\cdot}_{1/\gamma})$ is $\frac{1}{\gamma}$-uniformly convex and $\frac{1}{1-\gamma}$-uniformly Fr\'{e}chet differentiable. Furthermore, if $\gamma\in[\frac{1}{2},1[$, then $(L_{1/\gamma}(\N),\n{\cdot}_{1/\gamma})$ is $2$-uniformly convex (\cite[Rem. (p. 244)]{Hanner:1956} \cite[Eqns. (10a)--(10b)]{Kadec:1956} for commutative $\N$, \cite[Thm. 2.2]{TomczakJaegermann:1974} for type I $\N$, \cite[Thm. 1, p. 466]{Ball:Carlen:Lieb:1994} for semifinite $\N$, \cite[Thm. 5.3]{Pisier:Xu:2003} for any $\N$), hence $(L_{1/(1-\gamma)}(\N),\n{\cdot}_{1/(1-\gamma)})$ is 2-uniformly Fr\'{e}chet differentiable. Thus, $(L_{1/\gamma}(\N),\n{\cdot}_{1/\gamma})$ is $\max\{2,\frac{1}{\gamma}\}$-uniformly convex and $\min\{2,\frac{1}{\gamma}\}$-uniformly Fr\'{e}chet differentiable $\forall\gamma\in\,]0,1[$. Combining this with Proposition \ref{prop.varphi.uniform.continuity}, and with the fact that $r_1$-uniform convexity (resp., $r_1$-uniform Fr\'{e}chet differentiability) implies $r_2$-uniform convexity (resp., $r_2$-uniform Fr\'{e}chet differentiability) for $2\leq r_1\leq r_2<\infty$ (resp., $1<r_2\leq r_1\leq 2$) (cf., e.g., \cite[Props. 2.1, 2.2.(iii)]{Yamada:2006} for a proof), we obtain the statements for left and right $D_{\Psi_\varphi}$-projections on $L_{1/\gamma}(\N)$. The corresponding statements for left and right $D_{\ell_\gamma,\Psi_\varphi}$-projections follow from Lipschitz continuity of $(\ell_\gamma)^{\inver}$ on $B(L_{1/\gamma}(\N),\n{\cdot}_{1/\gamma})$ and $\gamma$-Lipschitz--H\"{o}lder continuity of $\ell_\gamma$ on $B(\N_\star,\n{\cdot}_1)$, for any $\gamma\in\,]0,1[$ \cite[Thm. (p. 37)]{Ricard:2015}. \ref{prop.uniform.Lp.v}--\ref{prop.uniform.Lp.viii} follows from Proposition \ref{prop.varphi.resolvent.uniform.continuity} by an analogous reasoning.
\end{proof}

\begin{proposition}\label{prop.mazur}
\sloppy Let $A$ be a semifinite JBW-algebra, $\tau$ a faithful normal semifinite trace on $A$, $\gamma\in\,]0,1[$, $\lambda\in\,]0,\infty[$. Then $\lambda\ell_\gamma$ is a norm-to-norm homeomorphism between $(A_\star,\n{\cdot}_1)^+$ and $(L_{1/\gamma}(A,\tau),\n{\cdot}_{1/\gamma})^+$.
\end{proposition}
\begin{proof}
\begin{enumerate}[nosep]
\item[1)] Consider $x\in A_\star^+$ and a sequence $\{x_n\in A_\star^+\mid n\in\NN\}$ such that $\lim_{n\ra\infty}\n{x_n-x}_1=0$. From inequality $\n{x^\gamma-y^\gamma}_{1/\gamma}^{1/\gamma}\leq\n{x-y}_1$ $\forall x,y\in A^+_\star$ \cite[Prop. 9.(ii)]{Iochum:1986} it follows that $\lim_{n\ra\infty}\n{x_n^\gamma-x^\gamma}_{1/\gamma}^{1/\gamma}=0$, i.e. $\lim_{n\ra\infty}\n{\ell_\gamma(x_n)-\ell_\gamma(x)}_{1/\gamma}=0$.
\item[2)] The uniform convexity of $(L_{1/\gamma}(A,\tau),\n{\cdot}_{1/\gamma})$, proved in \cite[Thm. V.3.2]{Iochum:1984} for $\gamma\in\,]0,\frac{1}{2}]$ and in \cite[Thm. 2.5]{Ayupov:1986} and \cite[Cors. 12, 13]{Iochum:1986} for $\gamma\in\,]0,1[$, taken together with the duality 
 \cite[Thm. 2.1.10]{Abdullaev:1984} \cite[Thm. V.3.2]{Iochum:1984}
\begin{equation}
(L_{1/\gamma}(A,\tau),\n{\cdot}_{1/\gamma})^\star\iso (L_{1/(1-\gamma)}(A,\tau),\n{\cdot}_{1/(1-\gamma)})\;\;\forall\gamma\in\,]0,1[,
\label{eqn.jbw.lp.lq.duality}
\end{equation}
\sloppy implies uniform Fr\'{e}chet differentiability of $(L_{1/\gamma}(A,\tau),\n{\cdot}_{1/\gamma})$, which in turn implies its Fr\'{e}chet differentiability. Taking $\varphi(t)=t^{1/\gamma-1}$ (i.e. $\varphi_{1,\beta}$ with $\beta=\gamma$), corresponding to $\duality{x,j_\varphi(x)}_{L_{1/\gamma}(A,\tau)\times L_{1/(1-\gamma)}(A,\tau)}=\n{x}_{1/\gamma}^{1/\gamma}$ and $\n{j_\varphi(x)}_{1/(1-\gamma)}=\n{x}_{1/\gamma}^{1/\gamma-1}$, gives norm-to-norm continuity of
\begin{equation}
j_\varphi:L_{1/\gamma}(A,\tau)\ni x=s_x\jordan\ab{x}\mapsto s_x\jordan\ab{x}^{1/\gamma-1}\in L_{1/(1-\gamma)}(A,\tau),
\end{equation}
with $\jordan$ denoting the nonassociative Jordan product in $A$. (The above expression for $j_\varphi$ can be explicitly deduced by noticing that the formula $\n{x}_{1/\gamma}^{1-1/\gamma}s_x\jordan\ab{x}^{1/\gamma-1}$ \cite[p. 51]{Abdullaev:1984} \cite[Lem. V.3.3.2$^o$]{Iochum:1984} (cf. \cite[p. 101]{Ayupov:1986} and \cite[p. 420]{Iochum:1986}) equals to $\DF\n{x}_{1/\gamma}$ by \cite[Lem. 14]{Iochum:1986}. Then, using the general formulas $j(x)=\frac{1}{2}\DF(\n{x}_X^2)=\n{x}_X\DF\n{x}_X$ and $j_\varphi(x)=\n{x}^{-1}_X\varphi(\n{x}_X)j(x)$, valid for any gauge $\varphi$ and any Fr\'{e}chet differentiable $(X,\n{\cdot}_X)$, we obtain $j(x)=\n{x}_{1/\gamma}^{2-1/\gamma}s_x\jordan\ab{x}^{1/\gamma-1}$ and $j_\varphi(x)=s_x\jordan\ab{x}^{1/\gamma-1}$ for $\varphi(t)=t^{1/\gamma-1}$ and $(X,\n{\cdot}_X)=(L_{1/\gamma}(A,\tau),\n{\cdot}_{1/\gamma})$.) The latter, taken together with the nonassociative Rogers--H\"{o}lder inequality $\n{xy}_1\leq\n{x}_{1/\gamma}\n{y}_{1/(1-\gamma)}$ $\forall x,y\in A_\tau$ \cite[Lem. 2.1.1.(a)]{Abdullaev:1984} \cite[Lem. V.3.3.1$^o$]{Iochum:1984} (cf. \cite[Prop. IV.2.4.(i)]{Ayupov:1986} and \cite[Lem. 3.(i)]{Iochum:1986}), entails norm-to-norm continuity of $j_\varphi(x)\jordan x=x^{1/\gamma-1}\jordan x=x^{1/\gamma}$ $\forall x\in L_{1/\gamma}(A,\tau)^+$.
\item[3)] Since $\lambda$ is a multiplicative positive constant, the above reasoning follows from $\ell_\gamma$ to $\lambda\ell_\gamma$.
\end{enumerate}
\end{proof}

\begin{proposition}\label{prop.r.unif.cont.na.Lp}
Let $A$ be a semifinite JBW-algebra, $\tau$ a faithful normal semifinite trace on $A$. If $\gamma\in[\frac{1}{2},1[$, then $(L_{1/\gamma}(A,\tau),\n{\cdot}_{1/\gamma})$ is 2-uniformly convex and $\frac{1}{\gamma}$-uniformly Fr\'{e}chet differentiable.
\end{proposition}
\begin{proof}
$\frac{1}{1-\gamma}$-uniform convexity and $\frac{1}{\gamma}$-uniform Fr\'{e}chet differentiability of $L_{1/\gamma}(A,\tau)$ for $\gamma\in[\frac{1}{2},1[$ is established explicitly in \cite[p. 102]{Ayupov:1986} and \cite[p. 427]{Iochum:1986}. 2-uniform convexity follows from the inequality \cite[Thm. 1, p. 466]{Ball:Carlen:Lieb:1994} \cite[Thm. 5.3]{Pisier:Xu:2003}
\begin{equation}
\n{x}_{1/\gamma}^2+\left({\textstyle\frac{1}{\gamma}}-1\right)\n{y}_{1/\gamma}^2\leq\left({\textstyle\frac{1}{2}}\left(\n{x+y}_{1/\gamma}^{1/\gamma}+\n{x-y}_{1/\gamma}^{1/\gamma}\right)\right)^{2\gamma}\;\;\forall x,y\in L_{1/\gamma}(\N),
\end{equation}
for any W$^*$-algebra $\N$, taken together with an extension of the Shirshov--Cohn theorem to semifinite JBW-algebras with weights \cite[Rem. (p. 94)]{Ayupov:1986}.
\end{proof}

\begin{proposition}\label{prop.dope.gauge.gamma.jbw}
Let $A$ be a semifinite JBW-algebra, $\tau$ a faithful normal semifinite trace on $A$, $\varphi$ a gauge, $\gamma\in\,]0,1[$, $\lambda\in\,]0,\infty[$, $\varnothing\neq C\subseteq A_\star$. Then:
\begin{enumerate}[nosep,label=(\roman*)]
\item\label{prop.dope.gauge.gamma.jbw.i} $D_{\lambda\ell_\gamma,\Psi_\varphi}:A_\star\times A_\star\ra[0,\infty]$ is an information on $A_\star$, independent of the choice of $\tau$;
\item\label{prop.dope.gauge.gamma.jbw.ii} the sets $\lsq(\lambda\ell_\gamma,\Psi_\varphi,C)$ and $\rsq(\lambda\ell_\gamma,\Psi_\varphi,C)$ are composable;
\item\label{prop.dope.gauge.gamma.jbw.iii} if $C$ is $\lambda\ell_\gamma$-convex $\lambda\ell_\gamma$-closed, then $D_{\lambda\ell_\gamma,\Psi_\varphi}$ is left pythagorean on $C$, and $\LPPP^{D_{\lambda\ell_\gamma,\Psi_\varphi}}_C$ is zone consistent and adapted;
\item\label{prop.dope.gauge.gamma.jbw.iv} if $C$ is $\lambda\ell_{1-\gamma}$-convex $\lambda\ell_\gamma$-closed, then $D_{\lambda\ell_\gamma,\Psi_\varphi}$ is right pythagorean on $C$, and $\RPPP^{D_{\lambda\ell_\gamma,\Psi_\varphi}}_C$ is zone consistent and adapted;
\item\label{prop.dope.gauge.gamma.jbw.v} if $C\subseteq A_\star^+$ is $\lambda\ell_\gamma$-convex closed, then (iii) holds, $\LPPP^{D_{\lambda\ell_\gamma,\Psi_\varphi}}_C$ is norm-to-norm continuous on $A_\star^+$ (with respect to $\n{\cdot}_1$), and $\inf_{y\in C}\{D_{\lambda\ell_\gamma,\Psi_\varphi}(y,\,\cdot\,)\}$ is continuous on  $A_\star^+$ (with respect to $\n{\cdot}_1$); 
\item\label{prop.dope.gauge.gamma.jbw.vi} if $C\subseteq A_\star^+$ is $\lambda\ell_{1-\gamma}$-convex closed, then (iv) holds, and $\RPPP^{D_{\lambda\ell_\gamma,\Psi_\varphi}}_C$ is norm-to-norm continuous on $A_\star^+$ (with respect to $\n{\cdot}_1$);
\item\label{prop.dope.gauge.gamma.jbw.vii} the categories $\lCvx^\subseteq(\lambda\ell_\gamma,\Psi_\varphi)$, $\rbarCvx^\subseteq(\lambda\ell_\gamma,\Psi_\varphi)$, $\LSQcvxsub(\lambda\ell_\gamma,\Psi_\varphi)$, and $\RbarSQcvxsub(\lambda\ell_\gamma,\Psi_\varphi)$ satisfy the functorial adjunctions and equivalences given by Corollary \ref{cor.LSQ.RSQ.adj.Psi.varphi}.\ref{cor.LSQ.RSQ.adj.Psi.varphi.i}--\ref{cor.LSQ.RSQ.adj.Psi.varphi.ii} with $Z=A_\star$;
\item\label{prop.dope.gauge.gamma.jbw.viii} if $T:L_{1/\gamma}(A,\tau)\ra 2^{L_{1/(1-\gamma)}(A,\tau)}$ is maximally monotone with $0\in\efd(T)$, then $\lres_T^{\Psi_\varphi}$ maps $L_{1/\gamma}(A,\tau)$ on $\efd(T)$ and is norm-to-norm continuous on $(L_{1/\gamma}(A,\tau),\n{\cdot}_{1/\gamma})$, and $\rres_T^{\Psi_\varphi}$ maps $L_{1/(1-\gamma)}(A,\tau)$ on $j_\varphi(\efd(T))$ and is norm-to-norm continuous on $(L_{1/(1-\gamma)}(A,\tau),\n{\cdot}_{1/(1-\gamma)})$; 
\item\label{prop.dope.gauge.gamma.jbw.ix} if $T:(L_{1/\gamma}(A,\tau))^+\ra 2^{(L_{1/(1-\gamma)}(A,\tau))^+}$ is maximally monotone with $0\in\efd(T)$, then $\lres^{\ell_\gamma,\Psi_\varphi}_T$ maps ${A_\star}^+$ on ${\ell_\gamma}^\inver(\efd(T))$ and is norm-to-norm continuous on ${A_\star}^+$ (with respect to $\n{\cdot}_1$). 
\end{enumerate}
\end{proposition}
\begin{proof}
Due to uniform convexity and uniform Fr\'{e}chet differentiability of $(L_{1/\gamma}(A,\tau),\n{\cdot}_{1/\gamma})$, $\gamma\in\,]0,1[$, Proposition \ref{prop.legendre} implies that $\Psi_\varphi$ is Euler--Legendre on $(L_{1/\gamma}(A,\tau),\n{\cdot}_{1/\gamma})$, while Proposition \ref{prop.left.right.psi.varphi} implies that $D_{\Psi_\varphi}$ is an information. The proof that $D_{\lambda\ell_\gamma,\Psi_\varphi}$ is left (resp., right) pythagorean on $\lambda\ell_\gamma$-convex (resp., $\lambda\ell_{1-\gamma}$-convex) closed sets $C$ is exactly the same as in the proof of Proposition \ref{prop.dope.gauge.gamma.wstar}. By \cite[Cor. 1]{Ayupov:Abdullaev:1989} (cf. also \cite[Cor. 2]{Ayupov:Chilin:Abdullaev:2012} with the choice of an N-function $\orlicz(t)=\gamma\ab{t}^{1/\gamma}$), $(L_{1/\gamma}(A,\tau_1),\n{\cdot}_{1/\gamma})$ and $(L_{1/\gamma}(A,\tau_2),\n{\cdot}_{1/\gamma})$ are isometrically isomorphic for any two faithful normal semifinite traces $\tau_1$ and $\tau_2$ on $A$, with $\gamma\in\,]0,1]$. Hence, $D_{\lambda\ell_\gamma,\Psi_\varphi}$ does not depend on the choice of $\tau$. The rest follows from Propositions \ref{prop.mazur}, \cite[Prop. 4.6]{Kostecki:2017}, \ref{prop.continuity}, \ref{prop.varphi.compositional}, and \ref{prop.varphi.resolvent.uniform.continuity} in the same way as in Proposition \ref{prop.dope.gauge.gamma.wstar}.
\end{proof}

\begin{corollary}\label{cor.r.unif.cont.proj.na.Lp}
Let $A$ be a semifinite JBW-algebra with a faithful normal semifinite trace $\tau$. Then all of the statements of Proposition \ref{prop.uniform.Lp}.\ref{prop.uniform.Lp.i}--\ref{prop.uniform.Lp.viii} hold for $L_{1/\gamma}(\N)$ (resp., $B(\N_\star,\n{\cdot}_1)$) replaced by $L_{1/\gamma}(A,\tau)$ (resp., $(B(A_\star,\n{\cdot}_1))^+$).
\end{corollary}
\begin{proof}
Follows from Propositions \ref{prop.varphi.resolvent.uniform.continuity}, \ref{prop.varphi.uniform.continuity}, \cite[Prop. 4.6]{Kostecki:2017}, and \ref{prop.r.unif.cont.na.Lp}, directly along the lines of the proof of Proposition \ref{prop.uniform.Lp}.
\end{proof}

\begin{remark}\label{remark.lp.breg.proj}
\begin{enumerate}[nosep,label=(\roman*)]
\item\label{remark.lp.breg.proj.i} Identification of $D_\gamma$ as $D_{\ell_\gamma,\Psi_{\varphi_{\gamma(1-\gamma),\gamma}}}$, provided in Corollary \ref{cor.d.gamma}, is new. The right hand side of \eqref{eqn.psi.gamma.equals.psi.gamma} corresponds to $D_{\Psi_\varphi}$ with $\varphi(t)=\frac{1}{1-\gamma}(\gamma t)^{1/\gamma-1}$. Up to reformulation in weight-independent terms, provided in \cite[Eqn. (41)]{Kostecki:2011:OSID}, the formula \eqref{eqn.d.gamma} was obtained in \cite[\S8]{Jencova:2005} (cf. also \cite[Eqn. (42)]{Ojima:2004}) as $D_\Psi(\frac{1}{\gamma}\ell_\gamma(\phi),\frac{1}{\gamma}\ell_\gamma(\psi))$ with $\Psi$ equal to $\Psi_{\varphi_{\gamma^{1-1/\gamma}(1-\gamma),\gamma}}$ (however, it was not identified there as an example of $\Psi_\varphi$, although the corresponding $D_\Psi$ was explicitly identified as a Va\u{\i}nberg--Br\`{e}gman functional). Proposition \ref{prop.dope.gauge.gamma.wstar}.\ref{prop.dope.gauge.gamma.wstar.ii} provides a generalisation of \cite[Props. 8.1.(i)--(ii), 8.2.(ii)]{Jencova:2005} to $\Psi_\varphi$ with any gauge $\varphi$.
\item\label{remark.lp.breg.proj.ii} Another special case of $D_{\Psi_{\varphi_{\alpha,\beta}}}(x,y)$ on $L_{1/\gamma}(\N)$ spaces, with $\alpha=1$, $\beta=\frac{1}{2}$, $\gamma\in\,]0,1[$, and $\N$ limited to type I$_n$ W$^*$-algebras, was considered in \cite[p. 377]{Nock:Magdalou:Briys:Nielsen:2013}. For $\gamma=\frac{1}{2}$, and for any W$^*$-algebra $\N$, they are also a special case of $4D_{\Psi_{\varphi_{\gamma(1-\gamma),\gamma}}}(x,y)=D_{4\Psi_{\varphi_{\gamma(1-\gamma),\gamma}}}(x,y)$, and take the form $2\n{x-y}_{L_2(\N)}$.
\item\label{remark.lp.breg.proj.iii} Plugging the formula for $j:L_{1/\gamma}(A,\tau)\ra L_{1/(1-\gamma)}(A,\tau)$ from the proof of Proposition \ref{prop.mazur} into \eqref{eqn.psi.varphi.alpha.beta}, we obtain a family  belonging to the class $D_{\ell_\gamma,\Psi_\varphi}$ on $A_\star$, which is a nonassociative analogue of \eqref{eqn.lambda.gamma.alpha.beta}, with $(\lambda,\alpha,\beta,\gamma)\in\,]0,\infty[^2\,\times\,]0,1[^2$ and $\phi,\psi\in A_\star$:
\begin{equation}
\hspace{-0.8cm}D_{\lambda\ell_\gamma,\Psi_{\varphi_{\alpha,\beta}}}(\phi,\psi)=\textstyle\frac{\lambda^{1/\beta}}{\alpha}\left(\beta(\tau(\phi))^{\gamma/\beta}+(1-\beta)(\tau(\psi))^{\gamma/\beta}-(\tau(\psi))^{\gamma/\beta-1}\tau((s_\phi\jordan\ab{\phi}^\gamma)\jordan(s_\psi\jordan\ab{\psi}^{1-\gamma}))\right),
\label{eqn.lambda.alpha.beta.gamma.jbw}
\end{equation}
When restricted to $\phi,\psi\in A_\star^+$, corresponding to $s_\phi=\II=s_\psi$, \eqref{eqn.lambda.alpha.beta.gamma.jbw} satisfies the conditions of Proposition \ref{prop.dope.gauge.gamma.jbw}.\ref{prop.dope.gauge.gamma.jbw.v} and \ref{prop.dope.gauge.gamma.jbw}.\ref{prop.dope.gauge.gamma.jbw.vi}.
\item\label{remark.lp.breg.proj.iv} Since Corollary \ref{cor.d.gamma}.\ref{cor.d.gamma.i} applies to \eqref{eqn.lambda.alpha.beta.gamma.jbw} as well, in what follows we will set $\lambda=1$ in both JBW- and W$^*$-algebraic cases.
\item\label{remark.lp.breg.proj.v} Since \cite[Lem. 3.2]{Raynaud:2002} establishes \textit{local} uniform homeomorphy of $\ell_\gamma$ on $(\N_\star,\n{\cdot}_1)$, the statements about uniform continuity of $\LPPP^{D_{\ell_\gamma,\Psi_{\varphi_{1,\beta}}}}_C$ and $\RPPP^{D_{\ell_\gamma,\Psi_{\varphi_{1,\beta}}}}_C$ in Proposition \ref{prop.uniform.Lp} hold for $\ell_\gamma$-bounded subsets of any closed ball in $\N_\star$ (this variant was explicitly used in Proposition \rpkmark{\ref{prop.dope.gauge.gamma.wstar}}).
\item\label{remark.lp.breg.proj.vi} Propositions \ref{prop.dope.gauge.gamma.wstar}.\ref{prop.dope.gauge.gamma.wstar.i} and \ref{prop.dope.gauge.gamma.jbw}.\ref{prop.dope.gauge.gamma.jbw.i}, \ref{prop.dope.gauge.gamma.wstar}.\ref{prop.dope.gauge.gamma.wstar.ii}--\ref{prop.dope.gauge.gamma.wstar.iii} and \ref{prop.dope.gauge.gamma.jbw}.\ref{prop.dope.gauge.gamma.jbw.iii} in their part on left pythagoreanity and zone consistency, as well as \ref{prop.dope.gauge.gamma.wstar}.\ref{prop.dope.gauge.gamma.wstar.iv}--\ref{prop.dope.gauge.gamma.wstar.v} in their part on right pythagoreanity and zone consistency hold also under replacing a gauge $\varphi$ by a quasigauge $\varphi$, provided the latter satisfies the respective conditions of Proposition \ref{prop.left.right.psi.quasigauge}.
\item\label{remark.lp.breg.proj.vii} For an arbitrary W$^*$-algebra $\N$ and $\gamma_1,\gamma_2\in\,]0,\infty[$, the noncommutative Mazur map,
\begin{equation}
\ell_{\gamma_1,\gamma_2}:L_{1/\gamma_1}(\N)\ni x=u_x\ab{x}\mapsto u_x\ab{x}^{\gamma_2/\gamma_1}\in L_{1/\gamma_2}(\N),
\label{eqn.noncommutative.mazur}
\end{equation}
has appeared implicitly in \cite[Prop. 1.9]{Haagerup:1979:ncLp} (cf. \cite[Prop. 12]{Terp:1981}), and then explicitly in \cite[Thm. 4.2]{Kosaki:1984:uniform} (with $u_x=\II$ and $\gamma_1,\gamma_2\in\,]0,1]$) and \cite[p. 58]{Raynaud:2002} (in full generality). In commutative case, this map has been introduced in \cite[p. 83]{Mazur:1929}. As proved independently in \cite[Thm. 4.5, Rem. 4.3]{Adzhiev:2014:I} (for semifinite $\N$ and $\gamma_1,\gamma_2\in\,]0,1[$) and \cite[Thm. (p. 37)]{Ricard:2015} (for any $\N$ and $\gamma_1,\gamma_2\in\,]0,1]$), $\ell_{\gamma_1,\gamma_2}$ is $\min\{\frac{\gamma_2}{\gamma_1},1\}$-Lipschitz--H\"{o}lder from $B(L_{1/\gamma_1}(\N),\n{\cdot}_{1/\gamma_1})$ to $B(L_{1/\gamma_2}(\N),\n{\cdot}_{1/\gamma_2})$, hence, by \cite[Lem. 3.1]{Adzhiev:2014:I}, also from $S(L_{1/\gamma_1}(\N),\n{\cdot}_{1/\gamma_1})$ to $S(L_{1/\gamma_2}(\N),\n{\cdot}_{1/\gamma_2})$. The nonassociative Mazur map with $\gamma_1=1$ and $\gamma_2\in\,]1,\infty[$ has appeared implicitly in \cite[p. 68]{Abdullaev:1984}, and was introduced in full generality in \cite[Def. 4.5]{Kostecki:2017}. 
\item\label{remark.lp.breg.proj.viii} Due to availability of several different (although equivalent) definitions of noncommutative $L_{1/\gamma}$ spaces over general W$^*$-algebras, we should specify the default meaning of this notion (since for general W$^*$-algebras it is defined using the tools essentially beyond the range of the integration theory on semifinite W$^*$-algebras). In \eqref{eqn.noncommutative.mazur}, as well as everywhere else, we use the definition of $L_{1/\gamma}(\N)$ (and thus the functional analytic meaning of the symbol `$\ab{x}^{\gamma_2/\gamma_1}$') as given in \cite[p. 196]{Falcone:Takesaki:2001}. The sign of integral, appearing in \eqref{eqn.lambda.gamma.alpha.beta} below and further, is understood in the sense of \cite[Eqn. (3.12')]{Falcone:Takesaki:2001}.
\item\label{remark.lp.breg.proj.ix} Proposition \ref{prop.mazur} is a generalisation of \cite[Thm. 4.2]{Kosaki:1984:uniform} from (any) W$^*$-algebras to (semifinite) JBW-algebras. The last step of the second part of the proof  (using Rogers--H\"{o}lder inequality for $xx^{1/\gamma-1}$) is exactly the same here as there, however we prove the earlier part (corresponding to \cite[Lem. 4.1]{Kosaki:1984:uniform}), differently, using directly the properties of the duality map, instead of a multiple use of the Rogers--H\"{o}lder inequality. \cite[Lem. 3.2]{Raynaud:2002} establishes local uniform continuity of $\ell_\gamma$ on $\N_\star$ (this term means \cite[p. 70]{Raynaud:1998} uniform continuity on every closed ball in $\N_\star$). Due to structural analogies between noncommutative and nonassociative integration theories, and in accordance with a tradition of \cite[Conj. V.3.10]{Iochum:1984}, we conjecture that $\ell_\gamma$ is norm-to-norm continuous also on unit balls of preduals of JBW-algebras. If this is true, then the continuity results of Proposition \ref{prop.dope.gauge.gamma.jbw}.\ref{prop.dope.gauge.gamma.jbw.v}--\ref{prop.dope.gauge.gamma.jbw.vi} hold for $A_\star$ replacing $A_\star^+$.
\end{enumerate}
\end{remark}
\subsection{Other models with $\Psi$ = $\Psi_\varphi$}\label{section.other.models.Psi.varphi}
\begin{proposition}\label{prop.oscr}
\sloppy If $(X,\n{\cdot}_X)$ is a uniformly convex and uniformly Fr\'{e}chet differentiable Banach function space over a localisable measure space $(\X,\mu)$, ${\ell_X}^{\inver}:S(X,\n{\cdot}_X)\ni x\mapsto\ab{j(x)}x\in S(L_1(\X,\mu),\n{\cdot}_1)$, $\varphi$ is a gauge, $\Psi_\varphi:(X,\n{\cdot}_X)\ra\RR$, and $\varnothing\neq C\subseteq S(L_1(\X,\mu),\n{\cdot}_1)$, $\lambda\in\,]0,\infty[$, $T:S(X,\n{\cdot}_X)\ra 2^{j_{\varphi_{1,\beta}}(S(X,\n{\cdot}_X))}$ is maximally monotone for $\beta\in\,]0,\frac{1}{2}]$, then:
\begin{enumerate}[nosep,label=(\roman*)]
\item\label{prop.oscr.i} \sloppy $\ell_X$ is a bijection, and $\ell_X|_{(S(L_1(\X,\mu),\n{\cdot}_1))^+}=\RPPP^{D_1}_{(S(X,\n{\cdot}_X))^+}$, with $D_1$ understood as a map $(L_\infty(\X,\mu))^+\times X\ra[0,\infty]$;
\item\label{prop.oscr.ii} $D_{\ell_X,\Psi_\varphi}:S(L_1(\X,\mu),\n{\cdot}_1)\times S(L_1(\X,\mu),\n{\cdot}_1)\ra[0,\infty]$ is an information on $S(L_1(\X,\mu),\n{\cdot}_1)$;
\item\label{prop.oscr.iii} if $C$ is $\ell_X$-convex closed set, then $D_{\ell_X,\Psi_\varphi}$ is zone consistent, left pythagorean on $C$, $\LPPP^{D_{\ell_X,\Psi_\varphi}}_C$ is norm-to-norm continuous on $S(L_1(\X,\mu),\n{\cdot}_1)$ and adapted, and $\inf_{y\in C}\{D_{\ell_X,\Psi_\varphi}(y,\,\cdot\,)\}$ is continuous on $S(L_1(\X,\mu),\n{\cdot}_1)$;
\item\label{prop.oscr.iv} if $C$ is $(j_\varphi\circ\ell_X)$-convex closed set, then $D_{\ell_X,\Psi_\varphi}$ is zone consistent, right pythagorean on $C$, and $\RPPP^{D_{\ell_X,\Psi_\varphi}}_C$ are norm-to-norm continuous on $S(L_1(\X,\mu),\n{\cdot}_1)$ and adapted;
\item\label{prop.oscr.v} the sets $\lsq(\ell_X,\Psi_\varphi,C)$ and $\rsq(\ell_X,\Psi_\varphi,C)$ are composable;
\item\label{prop.oscr.vi} the categories $\lCvx^\subseteq(\ell_X,\Psi_\varphi)$, $\rbarCvx^\subseteq(\ell_X,\Psi_\varphi)$, $\LSQcvxsub(\ell_X,\Psi_\varphi)$, and $\RbarSQcvxsub(\ell_X,\Psi_\varphi)$ satisfy the functorial adjunctions and equivalences given by Corollary \ref{cor.LSQ.RSQ.adj.Psi.varphi}.\ref{cor.LSQ.RSQ.adj.Psi.varphi.i}--\ref{cor.LSQ.RSQ.adj.Psi.varphi.ii} with $Z=S(L_1(\X,\mu),\n{\cdot}_1)$;
\item\label{prop.oscr.vii} if $(X,\n{\cdot}_X)$ is $\frac{1}{\beta}$-uniformly convex with $\beta\in\,]0,\frac{1}{2}]$, then:
\begin{enumerate}[nosep,label=\alph*)]
\item\label{prop.oscr.vii.a} if $C$ is $\ell_X$-convex and closed, then $\LPPP^{D_{\ell_X,\Psi_{\varphi_{1,\beta}}}}_C$ is uniformly continuous on $\ell_X$-bounded subsets of $S(L_1(\X,\mu),\n{\cdot}_1)$;
\item\label{prop.oscr.vii.b} $\lres^{\ell_X,\Psi_{\varphi_{1,\beta}}}_{\lambda T}$ is uniformly continuous on $\ell_X$-bounded subsets of $S(L_1(\X,\mu),\n{\cdot}_1)$;
\end{enumerate}
\item\label{prop.oscr.viii} if $(X,\n{\cdot}_X)$ is $\frac{1}{1-\beta}$-uniformly Fr\'{e}chet differentiable with $\beta\in\,]0,\frac{1}{2}]$, and $C$ is $(j_{\varphi_{1,1-\beta}}\circ\ell_X)$-convex and closed, then \rpkmark{$\RPPP^{D_{\ell_X,\Psi_{\varphi_{1,1-\beta}}}}_C$} is uniformly continuous on $\ell_X$-bounded subsets of $S(L_1(\X,\mu),\n{\cdot}_1)$;
\item\label{prop.oscr.ix} if $(X,\n{\cdot}_X)$ is $\frac{1}{\beta}$-uniformly convex and $\frac{1}{\gamma}$-uniformly Fr\'{e}chet differentiable, with $\beta\in\,]0,\frac{1}{2}]$ and $\gamma\in[\frac{1}{2},1[$, then:
\begin{enumerate}[nosep,label=\alph*)]
\item\label{prop.oscr.ix.a} if $C$ is closed and $\ell_X$-convex, then $\LPPP^{D_{\ell_X,\Psi_{\varphi_{1,\beta}}}}_C$ is $\frac{\beta^2(1-\gamma)^2}{\gamma^2(1-\beta)}$-Lipschitz--H\"{o}lder continuous on $S(L_1(\X,\mu),\n{\cdot}_1)$;
\item\label{prop.oscr.ix.b} if $C$ is closed and $(j_{\varphi_{1,\gamma}}\circ\ell_X)$-convex, then $\RPPP^{D_{\ell_X,\Psi_{\varphi_{1,\beta}}}}_C$ is $\frac{\beta^3(1-\gamma)^3}{\gamma^3(1-\beta)^2}$-Lipschitz--H\"{o}lder continuous on $\ell_X$-bounded subsets of $S(L_1(\X,\mu),\n{\cdot}_1)$;
\item\label{prop.oscr.ix.c} $\lres^{\ell_X,\Psi_{\varphi_{1,\beta}}}_{\lambda T}$ is $\frac{\beta^2(1-\gamma)^2}{\gamma^2(1-\beta)}$-Lipschitz--H\"{o}lder continuous on $\ell_X$-bounded subsets of $S(L_1(\X,\mu),\n{\cdot}_1)$;
\end{enumerate}
\item\label{prop.oscr.x} the properties \ref{prop.oscr.ii}--\ref{prop.oscr.ix} hold also for $\ell_X$ (and, resp., $S(L_1(\X,\mu),\n{\cdot}_1)$) replaced by $\widetilde{\ell_X}(x):=\n{x}_1\ell_X\left(\frac{x}{\n{x}_1}\right)$ for $x\in B(L_1(\X,\mu),\n{\cdot}_1)\setminus\{0\}$ and $\widetilde{\ell_X}(0):=0$ (and, resp., $B(L_1(\X,\mu),\n{\cdot}_1)$).
\end{enumerate}
\end{proposition}
\begin{proof}
(i) comes straight from the definition of $\ell_X$ \cite[p. 261]{Odell:Schlumprecht:1994} \cite[pp. 16--17, 20]{Chaatit:1995} (see \cite[\S5]{ChavezDominguez:2023} for an explicit discussion). By \cite[Prop. 2.6]{Odell:Schlumprecht:1994} \cite[Prop. 2.9]{Chaatit:1995} (cf. also \cite[Thm. 12]{Raynaud:1997}), ${\ell_X}^{\inver}$ is a uniform homeomorphism. Hence, by \cite[Prop. 2.9]{Odell:Schlumprecht:1994}, $\widetilde{\ell_X}$ is a uniform homeomorphism on $B(L_1(\X,\mu),\n{\cdot}_1)$. $\frac{1-\gamma}{\gamma}$-(resp., $\beta$-)Lipschitz--H\"{o}lder continuity of $\ell_X$ (resp., ${\ell_X}^{\inver}$) for $\frac{1}{\beta}$-uniformly convex and $\frac{1}{\gamma}$-uniformly Fr\'{e}chet differentiable $(X,\n{\cdot}_X)$ has been established in \cite[Thm. 4.2]{Adzhiev:2014:I} (=\cite[Thm. 5.6]{Adzhiev:2020}) (cf. \cite[Rem. 3.1.b)]{Adzhiev:2009} for equivalence of  $\frac{1}{\beta}$-uniform convexity (resp., $\frac{1}{\gamma}$-uniform Fr\'{e}chet differentiability) with $(\frac{1}{\beta},h_c)$-uniform convexity (resp., $(\frac{1}{\gamma},h_s)$-uniform Fr\'{e}chet differentiability) used in \cite[Thm. 4.2]{Adzhiev:2014:I}). The rest follows from Corollary \ref{corollary.ell.psi.varphi} and Propositions \ref{prop.varphi.resolvent.uniform.continuity} and \ref{prop.varphi.uniform.continuity}. The part of \ref{prop.oscr.x} that refers to \ref{prop.oscr.ii}--\ref{prop.oscr.viii} relies on equivalence of uniform continuity of $\ell$ on $S(Y,\n{\cdot}_Y)$ with uniform continuity of $\widetilde{\ell}$ on $B(Y,\n{\cdot}_Y)$ for any Banach space $(Y,\n{\cdot}_Y)$, as provided by \eqref{eqn.homeo.ext} and \eqref{eqn.unif.homeo.induced.on.spheres}, while the part of \ref{prop.oscr.x} that refers to \ref{prop.oscr.ix} relies on the analogous equivalence for Lipschitz--H\"{o}lder continuity, proved in \cite[Lem. 3.1]{Adzhiev:2014:I}.
\end{proof}

\begin{proposition}\label{prop.nc.oscr}
If $(X,\n{\cdot}_X)$ is a uniformly convex and uniformly Fr\'{e}chet differentiable noncommutative rearrangement invariant space over a type I$_n$ W$^*$-algebra $\N$ with $n\in\NN$, ${\ell_X}^{\inver}:(S(X,\n{\cdot}_X))^+\ni x\mapsto\ab{j(x)}x\in (S(\N_\star,\n{\cdot}_1))^+$, $\varphi$ is a gauge, $\Psi_\varphi:(X,\n{\cdot}_X)\ra\RR$, $\varnothing\neq C\subseteq (S(\N_\star,\n{\cdot}_1))^+$, $\beta\in\,]0,\frac{1}{2}]$, $\lambda\in\,]0,\infty[$, and $T:(S(X,\n{\cdot}_X))^+\ra 2^{j_{\varphi_{1,\beta}}((S(X,\n{\cdot}_X))^+)}$ is maximally monotone, then:
\begin{enumerate}[nosep,label=(\roman*)]
\item\label{prop.nc.oscr.i} $\ell_X$ is a bijection, and $\ell_X=\RPPP^{D_1}_{(S(X,\n{\cdot}_X))^+}$;
\item\label{prop.nc.oscr.ii} $D_{\ell_X,\Psi_\varphi}:((S(\N_\star,\n{\cdot}_1))^+\times(S(\N_\star,\n{\cdot}_1))^+\ra[0,\infty]$ is an information on $(S(\N_\star,\n{\cdot}_1))^+$;
\item\label{prop.nc.oscr.iii} if $C$ is $\ell_X$-convex closed set, then $D_{\ell_X,\Psi_\varphi}$ is zone consistent, left pythagorean on $C$, $\LPPP^{D_{\ell_X,\Psi_\varphi}}_C$ is norm-to-norm continuous on $(S(\N_\star,\n{\cdot}_1))^+$ and adapted, and $\inf_{y\in C}\{D_{\ell_X,\Psi_\varphi}(y,\,\cdot\,)\}$ is continuous on $(S(\N_\star,\n{\cdot}_1))^+$;
\item\label{prop.nc.oscr.iv} if $C$ is $(j_\varphi\circ\ell_X)$-convex closed set, then $D_{\ell_X,\Psi_\varphi}$ is zone consistent, right pythagorean on $C$, and $\RPPP^{D_{\ell_X,\Psi_\varphi}}_C$ are norm-to-norm continuous on $(S(\N_\star,\n{\cdot}_1))^+$ and adapted;
\item\label{prop.nc.oscr.v} the sets $\lsq(\ell_X,\Psi_\varphi,C)$ and $\rsq(\ell_X,\Psi_\varphi,C)$ are composable;
\item\label{prop.nc.oscr.vi} the categories $\lCvx^\subseteq(\ell_X,\Psi_\varphi)$, $\rbarCvx^\subseteq(\ell_X,\Psi_\varphi)$, $\LSQcvxsub(\ell_X,\Psi_\varphi)$, and $\RbarSQcvxsub(\ell_X,\Psi_\varphi)$ satisfy the functorial adjunctions and equivalences given by Corollary \ref{cor.LSQ.RSQ.adj.Psi.varphi}.\ref{cor.LSQ.RSQ.adj.Psi.varphi.i}--\ref{cor.LSQ.RSQ.adj.Psi.varphi.ii} with $Z=(S(\N_\star,\n{\cdot}_1))^+$;
\item\label{prop.nc.oscr.vii} if $(X,\n{\cdot}_X)$ is $\frac{1}{\beta}$-uniformly convex, then:
\begin{enumerate}[nosep,label=(\roman*)]
\item\label{prop.nc.oscr.vii.a} if $C$ is $\ell_X$-convex and closed, then $\LPPP^{D_{\ell_X,\Psi_{\varphi_{1,\beta}}}}_C$ is uniformly continuous on $\ell_X$-bounded subsets of $(S(\N_\star,\n{\cdot}_1))^+$;
\item\label{prop.nc.oscr.vii.b} $\lres^{\ell_X,\Psi_{\varphi_{1,\beta}}}_{\lambda T}$ is is uniformly continuous on $\ell_X$-bounded subsets of $(S(\N_\star,\n{\cdot}_1))^+$;
\end{enumerate}
\item\label{prop.nc.oscr.viii} if $(X,\n{\cdot}_X)$ is $\frac{1}{1-\beta}$-uniformly Fr\'{e}chet differentiable, and $C$ is $(j_{\varphi_{1,1-\beta}}\circ\ell_X)$-convex and closed, then  $\RPPP^{D_{\ell_X,\Psi_{\varphi_{1,1-\beta}}}}_C$ is uniformly continuous on $\ell_X$-bounded subsets of $(S(\N_\star,\n{\cdot}_1))^+$.
\end{enumerate}
\end{proposition}
\begin{proof}
(i) is just \cite[Def. 5.3]{ChavezDominguez:2023}. By \cite[Props. 5.6, 5.7, Lem. 5.8]{ChavezDominguez:2023}, ${\ell_X}^{\inver}$ is a uniform homeomorphism. The rest follows from Corollary \ref{corollary.ell.psi.varphi} and Propositions \ref{prop.varphi.resolvent.uniform.continuity} and \ref{prop.varphi.uniform.continuity}.
\end{proof}

\begin{proposition}\label{prop.weakly.compact.base}
If $(V,\n{\cdot}_V)$ is a radially compact base normed space with a weakly compact base $K$, $\varphi$ is any gauge, $\varnothing\neq C\subseteq V$, $\Psi_\varphi:V\ra\RR$ is Euler--Legendre, $\ell:V\ra V$ is any automorphism of $V$, then:
\begin{enumerate}[nosep,label=(\roman*)]
\item\label{prop.weakly.compact.base.i} $D_{\ell,\Psi_\varphi}$ is an information on $V$;
\item\label{prop.weakly.compact.base.ii} if $C$ is $\ell$-convex and $\ell$-closed, then it is left $D_{\ell,\Psi_\varphi}$-Chebysh\"{e}v, $D_{\ell,\Psi_\varphi}$ is left pythagorean on $C$, and $\LPPP^{D_{\ell,\Psi_\varphi}}_C$ are zone consistent;
\item\label{prop.weakly.compact.base.iii} if $C$ is $j_\varphi\circ\ell$-convex and $j_\varphi\circ\ell$-closed, then it is right $D_{\ell,\Psi_\varphi}$-Chebysh\"{e}v, $D_{\ell,\Psi_\varphi}$ is right pythagorean on $C$, and $\RPPP^{D_{\ell,\Psi_\varphi}}_C$ are zone consistent;
\item\label{prop.weakly.compact.base.iv} if each norm-exposed face of $K$ is projective, then the pair $((V,\n{\cdot}_V),(V^\star,\n{\cdot}_{V^\star}))$ is in spectral duality.
\end{enumerate}
\end{proposition}
\begin{proof}
A radially compact base normed space is reflexive if{}f its base is weakly compact \cite[Lem. 8.71]{Alfsen:Shultz:2003}. Hence, \ref{prop.weakly.compact.base.i}--\ref{prop.weakly.compact.base.iii} follow from Proposition \ref{prop.legendre} and Corollary \ref{corollary.ell.psi.varphi}.\ref{corollary.ell.psi.varphi.i}--\ref{corollary.ell.psi.varphi.iii}. On the other hand, \ref{prop.weakly.compact.base.iv} is just \cite[Prop. 2.5]{Alfsen:Shultz:1979}.
\end{proof}

\begin{proposition}\label{prop.gen.spin.factors}
Let $(V=X\oplus\RR,\n{\cdot}_V)$ be a radially compact base normed space with a base $K$, and a reflexive real Banach space $(X,\n{\cdot}_X)$ such that
\begin{equation}
\forall v=(x,\lambda)\in V\;\;\left\{\begin{array}{l}
v\geq0\;:\iff\;\lambda\geq\n{x}_X\\
\n{v}_V:=\max\{\ab{\lambda},\n{x}_X\},
\end{array}\right.
\end{equation}
let $\ell_{/\RR}:K\ni v=(x,1)\mapsto x\in B(X,\n{\cdot}_X)$, and let $\varphi$ be any gauge. If any of equivalent conditions holds:
\begin{enumerate}[nosep,label=\arabic*)]
\item\label{prop.gen.spin.factors.1} the pair $((V,\n{\cdot}_V),(V^\star,\n{\cdot}_{V^\star}))$ is in spectral duality;
\item\label{prop.gen.spin.factors.2} $(X,\n{\cdot}_X)$ is strictly convex and Gateaux differentiable;
\item\label{prop.gen.spin.factors.3} $\Psi_\varphi:B(X,\n{\cdot}_X)\ra\RR$ is Euler--Legendre,
\end{enumerate}
then:
\begin{enumerate}[nosep,label=(\roman*)]
\item\label{prop.gen.spin.factors.i} $D_{\ell_{/\RR},\Psi_\varphi}$ is an information on $K$;
\item\label{prop.gen.spin.factors.ii} if $\varnothing\neq C\subseteq K$ is $\ell_{/\RR}$-convex and $\ell_{/\RR}$-closed, then it is left $D_{\ell_{/\RR},\Psi_\varphi}$-Chebysh\"{e}v, $D_{\ell_{/\RR},\Psi_\varphi}$ is left pythagorean on $C$, and $\LPPP^{D_{\ell_{/\RR},\Psi_\varphi}}_C$ are zone consistent;
\item\label{prop.gen.spin.factors.iii} if $\varnothing\neq C\subseteq K$ is $j_\varphi\circ\ell_{/\RR}$-convex and $j_\varphi\circ\ell_{/\RR}$-closed, then it is right $D_{\ell_{/\RR},\Psi_\varphi}$-Chebysh\"{e}v, $D_{\ell_{/\RR},\Psi_\varphi}$ is right pythagorean on $C$, and $\RPPP^{D_{\ell_{/\RR},\Psi_\varphi}}_C$ are zone consistent.
\end{enumerate}
Furthermore, if $(X,\n{\cdot}_X)$ satisfies any of \ref{prop.gen.spin.factors.1}--\ref{prop.gen.spin.factors.3} above, is uniformly Fr\'{e}chet differentiable, and has the Radon--Riesz--Shmul'yan property, then:
\begin{enumerate}[nosep,label=(\roman*)]
\setcounter{enumi}{3}
\item\label{prop.gen.spin.factors.iv} $\lsq(\ell_{/\RR},\Psi_\varphi,C)$ and $\rsq(\ell_{/\RR},\Psi_\varphi,C)$ are composable for any $\varnothing\neq C\subseteq K$;
\item\label{prop.gen.spin.factors.v} $\LPPP^{D_{\ell_{/\RR},\Psi_\varphi}}_C$ are adapted and $\ell_{/\RR}$-topology-to-$\ell_{/\RR}$-topology continuous on any $\ell_{/\RR}$-convex and $\ell_{/\RR}$-closed $\varnothing\neq C\subseteq K$;
\item\label{prop.gen.spin.factors.vi} $\RPPP^{D_{\ell_{/\RR},\Psi_\varphi}}_C$ are $\ell_{/\RR}$-topology-to-$\ell_{/\RR}$-topology continuous on any $(j_\varphi\circ\ell_{/\RR})$-convex and $\ell_{/\RR}$-closed $\varnothing\neq C\subseteq K$;
\item\label{prop.gen.spin.factors.vii} the categories  $\lCvx^\subseteq(\ell_{/\RR},\Psi_\varphi)$, $\rbarCvx^\subseteq(\ell_{/\RR},\Psi_\varphi)$, $\LSQcvxsub(\ell_{/\RR},\Psi_\varphi)$, and $\RbarSQcvxsub(\ell_{/\RR},\Psi_\varphi)$ satisfy the functorial adjunctions and equivalences given by Corollary \ref{cor.LSQ.RSQ.adj.Psi.varphi}.\ref{cor.LSQ.RSQ.adj.Psi.varphi.i} with $Z=K$;
\item\label{prop.gen.spin.factors.viii} if $T:B(X,\n{\cdot}_X)\ra 2^{j_\varphi(B(X,\n{\cdot}_X))}$ is maximally monotone with $0\in\efd(T)$, then $\rres_T^{\Psi_\varphi}$ maps $j_\varphi(B(X,\n{\cdot}_X))$ pn $j_\varphi(\efd(T))$ and is norm-to-norm continuous on $j_\varphi(B(X,\n{\cdot}_X))$;
\item\label{prop.gen.spin.factors.ix} if $(X,\n{\cdot}_X)$ is uniformly convex, then:
\begin{enumerate}[nosep,label=\alph*)]
\item\label{prop.gen.spin.factors.ix.a} $\RPPP^{D_{\ell_{/\RR},\Psi_\varphi}}_C$ are adapted for any $(j_\varphi\circ\ell_{/\RR})$-convex and $\ell_{/\RR}$-closed $\varnothing\neq C\subseteq K$;
\item\label{prop.gen.spin.factors.ix.b} the categories $\rbarCvx^\subseteq(\ell_{/\RR},\Psi_\varphi)$ and $\RbarSQcvxsub(\ell_{/\RR},\Psi_\varphi)$ satisfy the functorial adjunction of Corollary \ref{cor.LSQ.RSQ.adj.Psi.varphi}.\ref{cor.LSQ.RSQ.adj.Psi.varphi.ii};
\item\label{prop.gen.spin.factors.ix.c} if $T$ is as in \ref{prop.gen.spin.factors.viii}, then $\lres_{T}^{\Psi_\varphi}$ (resp., $\lres_T^{\ell_{/\RR},\Psi_\varphi}$) maps $B(X,\n{\cdot}_X)$ on $\efd(T)$ (resp., $K$ on ${\ell_{/\RR}}^\inver(\efd(T))$) and is norm-to-norm continuous on $B(X,\n{\cdot}_X)$ (resp., $\ell_{/\RR}$-topology-to-$\ell_{/\RR}$-topology continuous on $K$).
\end{enumerate} 
\end{enumerate}
\end{proposition}
\begin{proof}
Equivalence of \ref{prop.gen.spin.factors.1} and \ref{prop.gen.spin.factors.2} was established in \cite[Thm. 1]{Berdikulov:Odilov:1995} (and recently rediscovered in \cite[Thm. 6.6]{Jencova:Pulmannova:2021}). Equivalence of \ref{prop.gen.spin.factors.2} and \ref{prop.gen.spin.factors.3} follows from Proposition \ref{prop.legendre}. \ref{prop.gen.spin.factors.i}--\ref{prop.gen.spin.factors.iii} follow from Corollary \ref{corollary.ell.psi.varphi}.\ref{corollary.ell.psi.varphi.i}--\ref{corollary.ell.psi.varphi.iii}. \ref{prop.gen.spin.factors.iv}--\ref{prop.gen.spin.factors.vii} and \ref{prop.gen.spin.factors.ix}.\ref{prop.gen.spin.factors.ix.a}--\ref{prop.gen.spin.factors.ix.b} follow from Corollaries \ref{corollary.ell.psi.varphi}.\ref{corollary.ell.psi.varphi.iv}, \ref{corollary.ell.psi.varphi}.\ref{corollary.ell.psi.varphi.vi}--\ref{corollary.ell.psi.varphi.viii}, and \ref{cor.LSQ.RSQ.adj.Psi.varphi}. \ref{prop.gen.spin.factors.viii} and \ref{prop.gen.spin.factors.ix}.\ref{prop.gen.spin.factors.ix.c} follow from Proposition \ref{prop.varphi.resolvent.norm.continuity}.
\end{proof}

\begin{remark}\label{rem.orther.varphi.models}
\begin{enumerate}[nosep,label=(\roman*)]
\item\label{rem.orther.varphi.models.i} Using \cite[Thm. 2.1]{Odell:Schlumprecht:1994} \cite[Thm. 2.1]{Chaatit:1995} \cite[Thm. 6.2]{Adzhiev:2014:II} (resp., \cite[Thm. 6.4]{ChavezDominguez:2023}), combined together with the Maurey--Pisier theorem \cite[p. 46]{Maurey:Pisier:1976} and a fact that $q$-uniform convexity with $q\geq2$ implies cotype $q$ \cite[Thm. 1.e.16.(i) (Vol. 2)]{Lindenstrauss:Tzafriri:1977:1979}, one can state an analogue of Proposition \ref{prop.oscr}.\ref{prop.oscr.ii}--\ref{prop.oscr.x} (resp., Proposition \ref{prop.nc.oscr}.\ref{prop.nc.oscr.ii}--\ref{prop.nc.oscr.viii}) for $q$-uniformly convex $(X,\n{\cdot}_X)$ with $q\geq2$. (In the case of Proposition \ref{prop.nc.oscr}.\ref{prop.nc.oscr.ii}--\ref{prop.nc.oscr.viii}, the resulting proposition includes also a generalisation of $\N$ from type I$_n$ to separable factors of type I.) However, in both cases, the corresponding uniform homeomorphism $\ell$, as well as $\ell^{\inver}$, is constructed via renorming of the convexification of $(X,\n{\cdot}_X)$, and as a result it lacks an explicit formula. 
\item\label{rem.orther.varphi.models.ii} Let $(\X,\mu)$ be a localisable measure space, and let $E(\X,\mu)\subseteq L_0(\X,\mu)$ be a complete Banach vector lattice. Let $E(\X,\mu)$ satisfy also the Fatou property, i.e. if $\{x_n\in (E(\X,\mu))^+\mid n\in\NN\}$ is increasing and $\sup_{n\in\NN}\n{x_n}_{E(\X,\mu)}<\infty$, then there exists $\sup_{n\in\NN}x_n=:x\in E(\X,\mu)$ and $\sup_{n\in\NN}\n{x_n}_{E(\X,\mu)}=\n{x_n}_{E(\X,\mu)}$ \cite[$(\beta)_3$ (p. 45)]{Ogasawara:1944} \cite[(1)--(2) (p. 1)]{Luxemburg:1955}. (In particular, if $E(\X,\mu)$ is reflexive, then it satisfies the Fatou property \cite[Thm. \S3.1]{Ogasawara:1944}.) Lozanovski\u{\i} proved \cite[Thm. 6.3]{Lozanovskii:1972} that $\forall x\in L_1(\X,\mu)$ $\exists (y,z)\in E(\X,\mu)\times(E(\X,\mu))^\koethe$ such that $x=yz$ and $\n{x}_1=\n{y}_{E(\X,\mu)}\n{z}_{(E(\X,\mu))^\koethe}$. The uniqueness of this factorisation, under additional assumption $\supp(x)=\supp(y)=\supp(z)$, has been established in \cite[\S3.(a)]{Gillespie:1981} (cf. also \cite[1117]{Lozanovskii:1969:VIII}). The map $(S(E(\X,\mu),\n{\cdot}_{E(\X,\mu)}))^+\ni y\mapsto\ab{j(y)}y\in(S(L_1(\X,\mu)))^+$, together with the explicit formula for its inverse, as well as with the proof of its uniform homeomorphy, has been given by Odell and Schlumprecht in \cite[Prop. 2.6]{Odell:Schlumprecht:1994} for uniformly convex, uniformly Fr\'{e}chet differentiable $E(\X,\mu)$ with a 1-unconditional basis and atomic infinite $(\X,\mu)$. The assumptions of 1-unconditional basis and atomic infinite $(\X,\mu)$ were shown to be obsolete in \cite[Prop. 2.9]{Chaatit:1995}. The formula given in Proposition \ref{prop.oscr}.\ref{prop.oscr.i} has been established in \cite[\S5]{ChavezDominguez:2023}. So, while the right $D_1$-projection was introduced by Chencov \cite[Eqn. (16)]{Chencov:1968} at nearly the same time as Lozanovski\u{\i} introduced his factorisation, it took over a half of century to have the special case of the former map identified as the inverse of the latter.
\item\label{rem.orther.varphi.models.iii} Proposition \ref{prop.oscr} (when restricted to the positive parts of unit spheres) and Proposition \ref{prop.nc.oscr} deal with the special cases of left and right $D_{\ell,\Psi}$-projections with $\ell$ given by another special case of right $D_{\ell,\Psi}$-projection, i.e. $\LPPP^{D_{\RPPP^{D_1}_{(S(X,\n{\cdot}_X))^+},\Psi_\varphi}}_C$ and $\RPPP^{D_{\RPPP^{D_1}_{(S(X,\n{\cdot}_X))^+},\Psi_\varphi}}_C$. This leads to a question: is it possible to use the latter projection as $\ell$ in some appropriate context?
\item\label{rem.orther.varphi.models.iv} In Propositions \ref{prop.oscr} and  \ref{prop.nc.oscr}, one can replace $j$ with $j_{\widetilde{\varphi}}$ for any gauge $\widetilde{\varphi}$ satisfying $\widetilde{\varphi}(1)=1$, since \cite[Thm. 12]{Raynaud:1997} and \cite[Props. 5.6, 5.7, Lem. 5.8]{ChavezDominguez:2023} are proved by evaluation of $j$ exclusively on the unit spheres (or their positive parts).
\item\label{rem.orther.varphi.models.v} Examples of radially compact base normed spaces with weakly compact base include all finite dimensional base normed spaces, type I$_2$ JBW-factors \cite[Prop. 3.38]{Alfsen:Shultz:2003} (their reflexivity was established already in \cite[Ex. 2]{Lowdenslager:1957}), and state spaces of orthomodular posets satisfying the \rpkmark{Jordan--Hahn decomposition property} \cite[Thm. 3]{Fischer:Ruettimann:1978:geometry}.
\item\label{rem.orther.varphi.models.vi} Order unit spaces $(V^\star,\n{\cdot}_{V^\star})$ satisfying the condition \ref{prop.gen.spin.factors.2} in Proposition \ref{prop.gen.spin.factors} were introduced in \cite[Def. 4]{Berdikulov:Odilov:1995}, and named the \df{generalised spin factors} (cf. also \cite[\S2.3.3]{Lami:2017}, where their finite dimensional version was rediscovered as `centrally symmetric models'). For $X$ given by a real Hilbert space with a scalar product $\s{\cdot,\cdot}_X$, $(X^\star\oplus\RR,\n{\cdot}_{X^\star\oplus\RR})$ turn into spin factors.
\item\label{rem.orther.varphi.models.vii} Relationship between the properties of $\Psi_\varphi$ and the spectral properties of $(V,\n{\cdot}_V)$ differ strongly between Propositions \ref{prop.weakly.compact.base} and \ref{prop.gen.spin.factors}: in the first case there is no relationship between them, in the second case it is a characterisation.
\item\label{rem.orther.varphi.models.viii} Proposition \ref{prop.gen.spin.factors}.\ref{prop.gen.spin.factors.i}--\ref{prop.gen.spin.factors.iii} holds also under replacing a gauge $\varphi$ by a quasigauge $\varphi$, provided the latter satisfies the respective conditions of Proposition \ref{prop.left.right.psi.quasigauge}.
\end{enumerate}
\end{remark}
\subsection{Some models with $\Psi\neq\Psi_\varphi$}\label{section.lewis.type.models}
\begin{proposition}\label{prop.hilbert.operator.psi}
Let $\H$ (resp., $\ell_{1/2}$) be either an $(L_2(\N))^\sa$ space for any W$^*$-algebra $\N$ (resp., a map $\N_\star^\sa\ni\phi=u_\phi\ab{\phi}\mapsto u_\phi\ab{\phi}^{1/2}\in(L_2(\N))^\sa$) or an $L_2(A,\tau)$ space for a semifinite JBW-algebra $A$ with a faithful normal semifinite trace $\tau$ (resp., a map $A_\star\iso L_1(A,\tau)\ni s_\phi\jordan\ab{\phi}\mapsto s_\phi\jordan\ab{\phi}^{1/2}\in L_2(A,\tau)$), let $\B$ be a ball in $\N_\star$ (e.g., $\B=B(\N_\star,\n{\cdot}_1)$), and let $T:\H\ra\H$ be a continuous linear map such that $\exists\lambda>0$ $\forall x,y\in\H$ $\s{Ty-Tx,y-x}_\H\geq\lambda\n{x-y}_\H^2$ (or, equivalently, $\inf\{\s{T\xi,\xi}_\H\mid\xi\in\H,\;\n{\xi}_\H=1\}\geq1$). Then:
\begin{enumerate}[nosep,label=(\roman*)]
\item\label{prop.hilbert.operator.psi.i} $T$ is invertible;
\item\label{prop.hilbert.operator.psi.ii} $\Psi_T:\H\ra\RR$ given by $\Psi_T(x):=\frac{1}{2}\s{Tx,x}_\H$ is Euler--Legendre;
\item\label{prop.hilbert.operator.psi.iii} $\Psi_T^\lfdual=\frac{1}{2}\s{T^{\inver}y,y}_\H$ $\forall y\in\H$, $\DG\Psi_T=T$, $\DG\Psi^\lfdual_T=T^{\inver}$;
\item\label{prop.hilbert.operator.psi.iv} $D_{\Psi_T}(x,y)=\frac{1}{2}\s{Tx-Ty,x-y}_\H$;
\item\label{prop.hilbert.operator.psi.v} $D_{\ell_{1/2},\Psi_T}$ is an information on $\N_\star^\sa$ (resp., $A_\star$);
\item\label{prop.hilbert.operator.psi.vi} if $\varnothing\neq C\subseteq\N_\star^+\cup\B$ (resp., $\varnothing\neq C\subseteq A_\star^+$) is $\ell_{1/2}$-convex and closed, then it is left $D_{\ell_{1/2},\Psi_T}$-Chebysh\"{e}v, $D_{\ell_{1/2},\Psi_T}$ is left pythagorean on $C$, and $\LPPP^{D_{\ell_{1/2},\Psi_T}}_C$ are zone consistent;
\item\label{prop.hilbert.operator.psi.vii} if $\varnothing\neq C\subseteq\N_\star^+\cup\B$ (resp., $\varnothing\neq C\subseteq A_\star^+$) is $T\circ\ell_{1/2}$-convex and $T$-closed, then it is right $D_{\ell_{1/2},\Psi_T}$-Chebysh\"{e}v, $D_{\ell_{1/2},\Psi_T}$ is right pythagorean on $C$, and $\RPPP^{D_{\ell_{1/2},\Psi_T}}_C$ are zone consistent.
\end{enumerate}
\end{proposition}
\begin{proof}
\ref{prop.hilbert.operator.psi.i}--\ref{prop.hilbert.operator.psi.iv} are special cases of results which hold for any real Hilbert space $\H$. \ref{prop.hilbert.operator.psi.i} follows from \cite[Thm. 2.1]{Lax:Milgram:1954}; \ref{prop.hilbert.operator.psi.ii} follows from \cite[p. 64]{Reem:Reich:DePierro:2019}; \ref{prop.hilbert.operator.psi.iii} follows from \cite[Ex. 3.2]{Reem:Reich:2018}; \ref{prop.hilbert.operator.psi.iv} follows by a direct calculation; \ref{prop.hilbert.operator.psi.v} follows from conjunction of \ref{prop.hilbert.operator.psi.ii}, Corollary \ref{cor.information}, and \ref{cor.d.ell.psi.properties}.\ref{cor.d.ell.psi.properties.i}; \ref{prop.hilbert.operator.psi.vi} follows from Corollary \ref{cor.d.ell.psi.properties}.\ref{cor.d.ell.psi.properties.iii} and norm-to-norm continuity of $\ell_{1/2}$ on $\N_\star^+\cup\B$ (resp., $A_\star^+$) \cite[Thm. 4.2]{Kosaki:1984:uniform} \cite[Lem. 3.2]{Raynaud:2002} (resp., Proposition \ref{prop.mazur}); \ref{prop.hilbert.operator.psi.vii} follows from Corollary \ref{cor.d.ell.psi.properties}.\ref{cor.d.ell.psi.properties.iv} and norm-to-norm continuity of $\ell_{1/2}$.
\end{proof}

\begin{proposition}\label{prop.finite.spectral.EL.case.I}
Let $\H$ be a Hilbert space with $\dim\H\in\NN$, let $\dim((\schatten_2(\H))^\sa)=:n$, and let $\Psi\in\pclg((\schatten_2(\H))^\sa,\n{\cdot}_2)$ be spectral, with $\Psi=f\circ\lewis$, where $f\in\pclg(\RR^n,\n{\cdot}_{\RR^n})$ is Euler--Legendre. Let $\ell$ (resp., $C\neq\varnothing$; $K\neq\varnothing$) be given either by a bijection $\ell:Y:=(\schatten_2(\H))^\sa\ra(\schatten_2(\H))^\sa$ (resp., a closed $\ell$-convex subset of $(\schatten_2(\H))^\sa$; a closed $(\grad\Psi)\circ\ell$-convex subset of $(\schatten_2(\H))^\sa$) or by a bijection $\ell_{1/2}:Y:=(\schatten_1(\H))^\sa\ni x\mapsto x^{1/2}\in(\schatten_2(\H))^\sa$ (resp., a closed $\ell_{1/2}$-convex subset of $(\schatten_1(\H))^\sa$; a closed $(\grad\Psi)\circ\ell_{1/2}$-convex subset of $(\schatten_1(\H))^\sa$). Let $\ell(C)\subseteq\intefd{\Psi}\supseteq\ell(K)$. Then:
\begin{enumerate}[nosep,label=(\roman*)]
\item\label{prop.finite.spectral.EL.case.I.i} $D_{\ell,\Psi}$ is an information on $Y$;
\item\label{prop.finite.spectral.EL.case.I.ii} $C$ is left $D_{\ell,\Psi}$-Chebysh\"{e}v, $D_{\ell,\Psi}$ is left pythagorean on $C$, and $\LPPP^{D_{\ell,\Psi}}_C$ are zone consistent;
\item\label{prop.finite.spectral.EL.case.I.iii} $K$ is right $D_{\ell,\Psi}$-Chebysh\"{e}v, $D_{\ell,\Psi}$ is right pythagorean on $K$, and $\RPPP^{D_{\ell,\Psi}}_K$ are zone consistent.
\end{enumerate}
\end{proposition}
\begin{proof}
Follows from the fact that $f$ is Euler--Legendre if{}f $f\circ\lewis$ is Euler--Legendre (see Example \ref{ex.spectral.convex}), combined with Propositions \ref{prop.left.pythagorean}.\ref{prop.left.pythagorean.i}.\ref{prop.left.pythagorean.i.c}, \ref{prop.left.pythagorean}.\ref{prop.left.pythagorean.iii}, \ref{prop.right.pythagorean}.\ref{prop.right.pythagorean.c}, and Corollary \ref{cor.information}.
\end{proof}

\begin{corollary}\label{cor.finite.spectral.EL}
Proposition \ref{prop.finite.spectral.EL.case.I} holds, in particular, for $\Psi=f\circ\lewis$ and $D_\Psi=D_{f\circ\lewis}$ given in Example \ref{ex.spectral.convex}.\ref{ex.spectral.convex.i}--\ref{ex.spectral.convex.iii}, as well as for:
\begin{enumerate}[nosep,label=(\roman*)]
\setcounter{enumi}{3}
\item\label{cor.finite.spectral.EL.iv} $f(x)=\sum_{i=1}^n(x_i\log(x_i)+(1-x_i)\log(1-x_i))$ on $\efd(f)=[0,1]^n$ and $f(x)=\infty$ otherwise, which gives spectral Euler--Legendre $(f\circ\lewis)(\xi)=\tr_\H(\xi\log(\xi)+(\II-\xi)\log(\II-\xi))$ for $\xi\in\efd(f\circ\lewis)=B(\schatten_2(\H),\n{\cdot}_2)\cap(\schatten_2(\H))^+$ and $(f\circ\lewis)(\xi)=\infty$ otherwise. The corresponding Va\u{\i}nberg--Br\`{e}gman functional reads 
\begin{equation}
D_{f\circ\lewis}(\xi,\zeta)=\tr_\H(\xi(\log\xi-\log\zeta)+(\II-\xi)(\log(\II-\xi)-\log(\II-\zeta)))
\label{eqn.spectral.Fermi.Dirac}
\end{equation}
for $(\xi,\zeta)\in\efd(f\circ\lewis)\times\intefd{f\circ\lewis}$, and $D_{f\circ\lewis}(\xi,\zeta)=\infty$ otherwise;
\item\label{cor.finite.spectral.EL.v} $f$ given by $\Psi_\alpha$ in \eqref{eqn.Psi.alpha}, which gives spectral Euler--Legendre
\begin{equation}
(f\circ\lewis)(\xi)=\left\{
\begin{array}{ll}
\frac{1}{\alpha-1}\tr_\H(\xi^\alpha-1)
&\st\xi\in(\schatten_2(\H))^+,\;\alpha\in\,]0,1[\\
\frac{1}{1-\alpha}\tr_\H(\xi^\alpha-1)
&\st\xi\in(\schatten_2(\H))^+_0,\;\alpha\in\,]-\infty,0[\\
\infty
&\st\mbox{otherwise}.
\end{array}
\right.
\end{equation}
The corresponding Va\u{\i}nberg--Br\`{e}gman functional $D_{f\circ\lewis}(\xi,\zeta)$ reads
\begin{equation}
\left\{
\begin{array}{ll}
\frac{1}{1-\alpha}\tr_\H(-\xi^\alpha+(1-\alpha)\zeta^\alpha+\alpha\zeta^{\alpha-1}\xi)
&\st(\xi,\zeta)\in(\schatten_2(\H))^+\times(\schatten_2(\H))^+_0,\;\alpha\in\,]0,1[\\
\frac{1}{\alpha-1}\tr_\H(-\xi^\alpha+(1-\alpha)\zeta^\alpha+\alpha\zeta^{\alpha-1}\xi)
&\st(\xi,\zeta)\in(\schatten_2(\H))^+_0\times(\schatten_2(\H))^+_0,\;\alpha\in\,]-\infty,0[\\
\infty
&\st\mbox{otherwise}.
\end{array}
\right.
\label{eqn.Bregman.spectral.Psi.alpha}
\end{equation}
\end{enumerate}
\end{corollary}
\begin{proof}
\ref{cor.finite.spectral.EL.iv}--\ref{cor.finite.spectral.EL.v} follow by application of $\lewis$ to Examples \ref{ex.EL.totalconvex.Rn.Psi}.\ref{ex.EL.totalconvex.Rn.Psi.iv}--\ref{ex.EL.totalconvex.Rn.Psi.v}.
\end{proof}

\begin{corollary}
Let $n\in\NN$, let $\N$ be a type I$_n$ W$^*$-algebra, and $f(x)=\sum_{i=1}^n(x_i\log(x_i)-x_i)$ if $x\geq0$ and $f(x)=\infty$ $\forall x\in\RR^n\setminus(\RR^n)^+$. Then $\cptp(\N_\star)\subseteq\cn(\ell_{1/2},f\circ\lewis)$.
\end{corollary}
\begin{proof}
Follows from Corollary \ref{cor.finite.spectral.EL} for $\Psi=f\circ\lewis$ from Example \ref{ex.spectral.convex}.\ref{ex.spectral.convex.ii}, combined with \cite[Thm. (p. 149)]{Lindblad:1975}.
\end{proof}

\begin{proposition}\label{prop.BRLZ.spectral.EL.just.EL}
Let  $(l,\n{\cdot}_l)$ be a reflexive separable rearrangement invariant sequence space, and let $(\schatten(\H),\n{\cdot}_{\schatten(\H)})$ be a rearrangement invariant space of compact operators on a separable Hilbert space $\H$, corresponding to $(l,\n{\cdot}_l)$  via $\schatten(\H)=\{x\in\mathfrak{C}(\H)\mid\rearr{x}{\tau}\in l\}$ and $\n{\cdot}_{\schatten(\H)}=\n{\rearr{(\cdot)}{\tau}}_{l}$, where $\rearr{x}{\tau}$ denotes a decreasing rearrangement of eigenvalues of $x$, while $\mathfrak{C}(\H)$ denotes the space of compact operators on $\H$. For any $x\in l$ consider sets $I_>(x):=\{i\in\NN\mid x_i>0\}$, $I_=(x):=\{i\in\NN\mid x_i=0\}$, $I_<(x):=\{i\in\NN\mid x_i<0\}$. Let $\widehat{\lewis}:l\ra l$ be defined by the following procedure:
\begin{enumerate}[nosep,label=(\arabic*)]
\item\label{prop.BRLZ.spectral.EL.just.EL.0} let $j:=1$;
\item\label{prop.BRLZ.spectral.EL.just.EL.1} if $I_>(x)\neq\varnothing$, then:
\begin{enumerate}[nosep,label=(\roman*)]
\item\label{prop.BRLZ.spectral.EL.just.EL.1.i} choose $i\in I_>(x)$ maximising $x_i$;
\item\label{prop.BRLZ.spectral.EL.just.EL.1.ii} $(\widehat{\lewis})_j:=x_i$;
\item\label{prop.BRLZ.spectral.EL.just.EL.1.iii} redefine $\left\{\begin{array}{l}I_>(x):=I_>(x)\setminus\{i\}\\j:=j+1;\end{array}\right.$
\end{enumerate}
\item\label{prop.BRLZ.spectral.EL.just.EL.2} if $I_=(x)\neq\varnothing$, then:
\begin{enumerate}[nosep,label=(\roman*)]
\item\label{prop.BRLZ.spectral.EL.just.EL.2.i} choose $i\in I_=(x)$;
\item\label{prop.BRLZ.spectral.EL.just.EL.2.ii} $(\widehat{\lewis})_j:=0$;
\item\label{prop.BRLZ.spectral.EL.just.EL.2.iii} redefine $\left\{\begin{array}{l}I_=(x):=I_=(x)\setminus\{i\}\\j:=j+1;\end{array}\right.$
\end{enumerate}
\item\label{prop.BRLZ.spectral.EL.just.EL.3} if $I_<(x)\neq\varnothing$, then:
\begin{enumerate}[nosep,label=(\roman*)]
\item\label{prop.BRLZ.spectral.EL.just.EL.3.i} choose $i\in I_<(x)$ minimising $x_i$;
\item\label{prop.BRLZ.spectral.EL.just.EL.3.ii} $(\widehat{\lewis})_j:=x_i$;
\item\label{prop.BRLZ.spectral.EL.just.EL.3.iii} redefine $\left\{\begin{array}{l}I_<(x):=I_<(x)\setminus\{i\}\\j:=j+1;\end{array}\right.$
\end{enumerate}
\item\label{prop.BRLZ.spectral.EL.just.EL.4} go to \ref{prop.BRLZ.spectral.EL.just.EL.1}.
\end{enumerate}
Let $\widetilde{\lewis}:(\schatten(\H))^\sa\ra l$ be defined by $\widetilde{\lewis}(y):=\widehat{\lewis}(\widetilde{y})$, where $\widetilde{y}$ is any sequence of eigenvalues of $y\in(\schatten(\H))^\sa$, counted with multiplicities. Let $f\in\pclg(l,\n{\cdot}_l)$, with $\intefd{f^\lfdual}\neq\varnothing$, $\intefd{f\circ\widetilde{\lewis}}\neq\varnothing$, $\intefd{f^\lfdual\circ\widetilde{\lewis}}\neq\varnothing$, and $(f\circ\widetilde{\lewis})(uxu^*)=(f\circ\widetilde{\lewis})(x)$ $\forall x\in(\schatten(\H))^\sa$ $\forall$ unitary $u\in\BH$. Then $f\circ\widetilde{\lewis}$ is Euler--Legendre if{}f $f$ is Euler--Legendre, with $(f\circ\widetilde{\lewis})^\lfdual=f^\lfdual\circ\widetilde{\lewis}$.
\end{proposition}
\begin{proof}
Follows directly from \cite[Thms. 3.3, 5.9]{Borwein:Read:Lewis:Zhu:2000}, combined with the characterisation of Euler--Legendre functions on reflexive Banach spaces given in \cite[\S2.1]{Reich:Sabach:2009}. Reflexivity of $((\schatten(\H))^\sa,\n{\cdot}_{(\schatten(\H))^\sa})$ follows from \textup{\cite[Prop. 6.8.15]{Dodds:dePagter:Sukochev:2023}} (cf. also \textup{\cite[p. 153]{Arazy:1981:ball}}). The equivalence of single-valuedness of $\partial f$ on $\efd(\partial f)$ with single-valuedness of $\partial(f\circ\widetilde{\lewis})$ on $\efd(f\circ\widetilde{\lewis})$, as well as the corresponding property of their Mandelbrojt--Fenchel duals, is a direct consequence of the proof of \cite[Thm. 5.9]{Borwein:Read:Lewis:Zhu:2000}, when used without the restriction to the points of Gateaux differentiability. While \cite[Thm. 5.9]{Borwein:Read:Lewis:Zhu:2000} is stated only for $((\schatten_{1/\gamma}(\H))^\sa,\n{\cdot}_{(\schatten_{1/\gamma}(\H))^\sa})$, $\gamma\in\,]0,1[$, it holds also in the more general case considered here, since the only property of $((\schatten_{1/\gamma}(\H))^\sa,$ $\n{\cdot}_{(\schatten_{1/\gamma}(\H))^\sa})$ it relies upon (apart from reflexivity), is \cite[Prop. 5.3]{Borwein:Read:Lewis:Zhu:2000}. However, by \cite[Thm. 3.5]{Schatten:vonNeumann:1948}, the latter holds in every rearrangement invariant $((\schatten(\H))^\sa,\n{\cdot}_{(\schatten(\H))^\sa})$.
\end{proof}

\begin{proposition}\label{prop.BRLZ.spectral.EL}
Let  $(l,\n{\cdot}_l)$ be a reflexive separable rearrangement invariant sequence space, and let $(\schatten(\H),\n{\cdot}_{\schatten(\H)})$ be a rearrangement invariant space of compact operators on a separable Hilbert space $\H$, corresponding to $(l,\n{\cdot}_l)$ as in Proposition \ref{prop.BRLZ.spectral.EL.just.EL}. Let $\widetilde{\lewis}$ be defined as in Proposition \ref{prop.BRLZ.spectral.EL.just.EL}. Let $f\in\pclg(l,\n{\cdot}_l)$ be Euler--Legendre, $\intefd{f\circ\widetilde{\lewis}}\neq\varnothing$, $\intefd{f^\lfdual\circ\widetilde{\lewis}}\neq\varnothing$, and $(f\circ\widetilde{\lewis})(uxu^*)=(f\circ\widetilde{\lewis})(x)$ $\forall x\in(\schatten(\H))^\sa$ $\forall$ unitary $u\in\BH$. If one of the following conditions holds:
\begin{enumerate}[nosep,label=\alph*)]
\item\label{prop.BRLZ.spectral.EL.a} $\ell=\ell_{1/\gamma}|_{(\schatten_{1}(\H))^\sa}$, $(\schatten(\H),\n{\cdot}_{\schatten(\H)})=(\schatten_{1/\gamma}(\H),\n{\cdot}_{1/\gamma})$, $\varnothing\neq C\subseteq(\schatten_{1}(\H))^+\cup B((\schatten_1(\H))^\sa,\n{\cdot}_1)$, $Z=(\schatten_{1}(\H))^\sa$;
\item\label{prop.BRLZ.spectral.EL.b} $\ell=\ell_X$, with a finite dimensional uniformly convex and uniformly Fr\'{e}chet differentiable rearrangement invariant space $(X,\n{\cdot}_X)=(\schatten(\H),\n{\cdot}_{\schatten(\H)})$, $\varnothing\neq C\subseteq(B(\schatten_1(\H),\n{\cdot}_1))^+$, $Z=(B(\schatten_1(\H),\n{\cdot}_1))^+$,
\end{enumerate}
then:
\begin{enumerate}[nosep,label=(\roman*)]
\item\label{prop.BRLZ.spectral.EL.i} $D_{\ell,f\circ\widetilde{\lewis}}$ is an information on $Z$;
\item\label{prop.BRLZ.spectral.EL.ii} if $C$ is $\ell$-convex and closed, then $C$ is left $D_{\ell,f\circ\widetilde{\lewis}}$-Chebysh\"{e}v, $D_{\ell,f\circ\widetilde{\lewis}}$ is left pythagorean on $C$, and $\LPPP^{D_{\ell,f\circ\widetilde{\lewis}}}_C$ are zone consistent;
\item\label{prop.BRLZ.spectral.EL.iii} if $C$ is $(\DG(f\circ\widetilde{\lewis}))\circ\ell$-convex and $(\DG(f\circ\widetilde{\lewis}))$-closed, then $C$ is right $D_{\ell,f\circ\widetilde{\lewis}}$-Chebysh\"{e}v, $D_{\ell,f\circ\widetilde{\lewis}}$ is right pythagorean on $C$, and $\RPPP^{D_{\ell,f\circ\widetilde{\lewis}}}_C$ are zone consistent.
\end{enumerate}
\end{proposition}
\begin{proof}
Follows from Proposition \ref{prop.BRLZ.spectral.EL.just.EL}, combined with Propositions \ref{prop.left.pythagorean}.\ref{prop.left.pythagorean.i}.\ref{prop.left.pythagorean.i.c}, \ref{prop.left.pythagorean}.\ref{prop.left.pythagorean.iii}, \ref{prop.right.pythagorean}.\ref{prop.right.pythagorean.c}, and Corollary \ref{cor.information}. Homeomorphy of $\ell_\gamma$ follows from combination of \cite[Thm. 4.2]{Kosaki:1984:uniform} and \cite[Lem. 3.2]{Raynaud:2002}, while homeomorphy of $\ell_X$ follows from \cite[Props. 5.6, 5.7, Lem. 5.8]{ChavezDominguez:2023}.
\end{proof}

\begin{remark}\label{rem.models.non.varphi}
\begin{enumerate}[nosep,label=(\roman*)]
\item\label{rem.models.non.varphi.i} The condition on $T:\H\ra\H$ assumed in Proposition \ref{prop.hilbert.operator.psi} holds, in particular, when $T$ is positive semidefinite and invertible \cite[Ex. 3.2]{Reem:Reich:2018}, as well as when $\H$ is finite dimensional and $T$ is symmetric and positive definite \cite[p. 64]{Reem:Reich:DePierro:2019} (this case goes back to \cite[p. 15]{Bregman:1966:PhD}). For $T=\II_\H$ with arbitrary dimensional $\H$, one obtains $D_{\Psi_T}=D_{\varphi_{1,1/2}}$ (see Remark \ref{remark.varphi}.\ref{remark.varphi.ix}), which was considered as an example (for $\dim\H<\infty$) already in \cite[p. 1021]{Bregman:1966} and \cite[\S2.1]{Bregman:1966:PhD}.
\item\label{rem.models.non.varphi.ii} The formula \eqref{eqn.spectral.Fermi.Dirac} has appeared earlier in \cite[p. 376]{Nock:Magdalou:Briys:Nielsen:2013}, however without using spectral convex functions (the Taylor expansion formula, and the finite dimensional quantum analogue $D^{\tr_\H}_\Psi$ of the Brunk--Ewing--Utz functional were used instead). Formula \eqref{eqn.Bregman.spectral.Psi.alpha} is new.
\item\label{rem.models.non.varphi.iii} The main open problem posed by Proposition \ref{prop.BRLZ.spectral.EL}, is it provide examples of unitarily invariant Euler--Legendre functions $f\circ\widetilde{\lewis}$ on the suitable separable reflexive spaces $(\schatten(\H),\n{\cdot}_{\schatten(\H)})$. In particular, is it so for the countable extensions of functions listed in Proposition \ref{cor.finite.spectral.EL}?
\end{enumerate}
\end{remark}
\section*{Acknowledgments}
\addcontentsline{toc}{section}{{\vspace{-0.3cm}}Acknowledgments}
I would like to thank Francesco Buscemi and Marcin Marciniak for hospitality, as well as Micha{\l} Eckstein and Pawe{\l} Horodecki for support. This work was partially founded by MAB/2018/5 grant of Foundation for Polish Science, 2015/18/E/ST2/00327 and 2021/42/A/ST2/00356 grants of Polish National Center of Science, IMPULZ IM-2023-79 and VEGA 2/0164/25 grants of Slovak Academy of Sciences, APVV-22-0570 grant of Slovak Research and Development Agency, and FellowQUTE 2024-02 program of Slovak National Center for Quantum Technologies. Part of this research was conducted at Department of Mathematical Informatics, Graduate School of Information Science, Nagoya University, on an academic leave from University of Gda\'{n}sk.%
%
\section*{References}
\addcontentsline{toc}{section}{{\vspace{-0.3cm}}References}
\begin{spacing}{0.7}
{\scriptsize%
The symbol * in front of a bibliographic item indicates that I have not seen this work. Unless explicitly stated otherwise, all citations refer to the first editions of the corresponding texts in their original language. All Cyrillic Russian names and words were transliterated from original using the following system (which is bijective due to the lack of {\fontencoding{T2A}\selectfont ыа} and {\fontencoding{T2A}\selectfont ыу} combinations): {\fontencoding{T2A}\selectfont ц} = c, {\fontencoding{T2A}\selectfont ч} = ch, {\fontencoding{T2A}\selectfont х} = kh, {\fontencoding{T2A}\selectfont ж} = zh, {\fontencoding{T2A}\selectfont ш} = sh, {\fontencoding{T2A}\selectfont щ} = \v{s}, {\fontencoding{T2A}\selectfont и} = i, {\fontencoding{T2A}\selectfont й} = \u{\i}, i = \={\i}, {\fontencoding{T2A}\selectfont ы} = y, {\fontencoding{T2A}\selectfont ю} = yu, {\fontencoding{T2A}\selectfont я} = ya, {\fontencoding{T2A}\selectfont ё} = \"{e}, {\fontencoding{T2A}\selectfont э} = \`{e}, {\fontencoding{T2A}\selectfont ъ} = `, {\fontencoding{T2A}\selectfont ь} = ', and analogously for capitalised letters, with an exception of {\fontencoding{T2A}\selectfont Х} = H at the beginnings of words. When\-ev\-er possible, Chinese Mandarin (resp., Cantonese) names and titles were nonbijectively romanised from original, using p\={\i}ny\={\i}n (resp., toneless Yale).}
\end{spacing}

{\scriptsize
\begingroup
\raggedright

\ifx\VBver\sthundefined
\bibliographystyle{rpkbib}
\renewcommand\refname{\vskip -1cm}
\bibliography{rpk}  
\else
\bibliographystyle{../../rpkbib}
\renewcommand\refname{\vskip -1cm}
\bibliography{../../rpk}  
\fi

\endgroup        


\end{document}